\documentclass[10 pt,a4paper]{article}
\usepackage{amsmath,amssymb,amsthm, amscd,amsfonts,stackrel,mathrsfs,color}
\usepackage{mathrsfs}
\usepackage[all]{xy}
\usepackage{verbatim}
\usepackage{color}

\newtheorem{theorem}{Theorem}[section]
\newtheorem{definition}[theorem]{Definition}
\newtheorem{proposition}[theorem]{Proposition}
\newtheorem{lemma}[theorem]{Lemma}
\newtheorem{remark}[theorem]{Remark}
\newtheorem{corollary}[theorem]{Corollary}
\newtheorem{example}[theorem]{Example}
\newtheorem{conjecture}[theorem]{Conjecture}

\newtheorem*{question}{Question}
\newtheorem{problem}{Problem}

\numberwithin{equation}{section}
\numberwithin{subsection}{section}
\allowdisplaybreaks

\newcommand{\sym}{\mathrm{Sym}}

\begin{document}


\newcommand{\riem}{(M^m, \langle \, , \, \rangle)}
\newcommand{\Hess}{\mathrm{Hess}\, }
\newcommand{\hess}{\mathrm{hess}\, }
\newcommand{\cut}{\mathrm{cut}}
\newcommand{\ind}{\mathrm{ind}}
\newcommand{\ess}{\mathrm{ess}}
\newcommand{\longra}{\longrightarrow}
\newcommand{\eps}{\varepsilon}
\newcommand{\ra}{\rightarrow}
\newcommand{\vol}{\mathrm{vol}}
\newcommand{\di}{\mathrm{d}}
\newcommand{\R}{\mathbb R}
\newcommand{\C}{\mathbb C}
\newcommand{\Z}{\mathbb Z}
\newcommand{\N}{\mathbb N}
\newcommand{\HH}{\mathbb H}
\newcommand{\esse}{\mathbb S}
\newcommand{\bull}{\rule{2.5mm}{2.5mm}\vskip 0.5 truecm}
\newcommand{\binomio}[2]{\genfrac{}{}{0pt}{}{#1}{#2}} 
\newcommand{\metric}{\langle \, , \, \rangle}
\newcommand{\metricN}{( \, , \, )}
\newcommand{\lip}{\mathrm{Lip}}
\newcommand{\loc}{\mathrm{loc}}
\newcommand{\diver}{\mathrm{div}}
\newcommand{\disp}{\displaystyle}
\newcommand{\rad}{\mathrm{rad}}
\newcommand{\mmetric}{\langle\langle \, , \, \rangle\rangle}
\newcommand{\sn}{\mathrm{sn}}
\newcommand{\cn}{\mathrm{cn}}
\newcommand{\ink}{\mathrm{in}}
\newcommand{\hol}{\mathrm{H\ddot{o}l}}
\newcommand{\capac}{\mathrm{cap}}
\newcommand{\bmo}{\{b <0\}}
\newcommand{\bmuo}{\{b \le 0\}}
\newcommand{\Fk}{\mathcal{F}_k}
\newcommand{\dist}{\mathrm{dist}}
\newcommand{\gr}{\mathcal{G}}
\newcommand{\grg}{\mathcal{G}^{(g)}}
\newcommand{\Ricc}{\mathrm{Ric}}
\newcommand{\foc}{\mathrm{foc}}
\newcommand{\F}{\mathcal{F}}
\newcommand{\Cf}{\mathcal{C}_f}
\newcommand{\cutf}{\mathrm{cut}_{f}}
\newcommand{\Cn}{\mathcal{C}_n}
\newcommand{\cutn}{\mathrm{cut}_{n}}
\newcommand{\Ca}{\mathcal{C}_a}
\newcommand{\cuta}{\mathrm{cut}_{a}}
\newcommand{\cutc}{\mathrm{cut}_c}
\newcommand{\cutcf}{\mathrm{cut}_{cf}}
\newcommand{\rk}{\mathrm{rk}}
\newcommand{\crit}{\mathrm{crit}}
\newcommand{\diam}{\mathrm{diam}}
\newcommand{\haus}{\mathcal{H}}
\newcommand{\po}{\mathrm{po}}
\newcommand{\gp}{\mathcal{G}_p}
\newcommand{\FF}{\mathcal{F}}
\newcommand{\gru}{\nabla u}
\newcommand{\grr}{\nabla r}
\newcommand{\grho}{\nabla \psi}
\newcommand{\vp}{\varphi}
\renewcommand{\div}[1]{{\mathop{\mathrm div}}\left(#1\right)}
\newcommand{\divphi}[1]{{\mathop{\mathrm div}}\bigl(\vert \nabla #1
\vert^{-1} \varphi(\vert \nabla #1 \vert)\nabla #1   \bigr)}
\newcommand{\nablaphi}[1]{\vert \nabla #1\vert^{-1}
\varphi(\vert \nabla #1 \vert)\nabla #1}
\newcommand{\modnabla}[1]{\vert \nabla #1\vert }
\newcommand{\modnablaphi}[1]{\varphi\bigl(\vert \nabla #1 \vert\bigr)
\vert \nabla #1\vert }
\newcommand{\cL}{\mathcal{L}}
\newcommand{\essem}{\mathds{S}^m}
\newcommand{\erre}{\mathds{R}}
\newcommand{\errem}{\mathds{R}^m}
\newcommand{\enne}{\mathds{N}}
\newcommand{\acca}{\mathds{H}}
\newcommand{\cvett}{\Gamma(TM)}
\newcommand{\cinf}{C^{\infty}(M)}
\newcommand{\sptg}[1]{T_{#1}M}
\newcommand{\partder}[1]{\frac{\partial}{\partial {#1}}}
\newcommand{\partderf}[2]{\frac{\partial {#1}}{\partial {#2}}}
\newcommand{\ctloc}{(\mathcal{U}, \varphi)}
\newcommand{\fcoord}{x^1, \ldots, x^n}
\newcommand{\ddk}[2]{\delta_{#2}^{#1}}
\newcommand{\christ}{\Gamma_{ij}^k}
\newcommand{\ricc}{\operatorname{Ricc}}
\newcommand{\supp}{\operatorname{supp}}
\newcommand{\sgn}{\operatorname{sgn}}
\newcommand{\rg}{\operatorname{rg}}
\newcommand{\inv}[1]{{#1}^{-1}}
\newcommand{\id}{\operatorname{id}}
\newcommand{\jacobi}[3]{\sq{\sq{#1,#2},#3}+\sq{\sq{#2,#3},#1}+\sq{\sq{#3,#1},#2}=0}
\newcommand{\lie}{\mathfrak{g}}
\newcommand{\wedgedot}{\wedge\cdots\wedge}
\newcommand{\rp}{\erre\mathds{P}}
\newcommand{\II}{\operatorname{II}}
\newcommand{\gradh}[1]{\nabla_{H^m}{#1}}
\newcommand{\absh}[1]{{\left|#1\right|_{H^m}}}
\newcommand{\mob}{\mathrm{M\ddot{o}b}}
\newcommand{\mab}{\mathfrak{m\ddot{o}b}}
\newcommand{\TT}{\mathcal T}
\newcommand{\HHH}{\mathcal{H}}
\newcommand{\Sph}{\mathbb{S}}
\newcommand{\inte}{\mathrm{Int}}
\newcommand{\KOzsigma}{\text{(KO$_0$)}(\sigma)}
\newcommand{\KOisigma}{\text{(KO$_\infty$)}(\sigma)}
\newcommand{\wu}{\overline{w}_R}

\newcommand{\smp}{(\mathrm{SMP}_\infty)}
\newcommand{\wmp}{(\mathrm{WMP}_\infty)}
\newcommand{\owmp}{(\mathrm{OWMP}_\infty)}
\newcommand{\fmp}{(\mathrm{FMP})}
\newcommand{\csp}{(\mathrm{CSP})}
\newcommand{\lio}{(\mathrm{L})}
\newcommand{\slio}{(\mathrm{SL})}
\newcommand{\feller}{(\mathrm{FE})}

\newcommand{\rmi}{(\rm i)\,}
\newcommand{\rmii}{(\rm ii)\,}
\newcommand{\rmiii}{(\rm iii)\,}
\newcommand{\rmiv}{(\rm iv)\,}
\newcommand{\rmv}{(\rm v)\,}
\newcommand{\rmvi}{(\rm vi)\,}
\newcommand{\rmvii}{(\rm vii)\,}
\newcommand{\rmviii}{(\rm viii)\,}

\title{\textbf{On the interplay among maximum principles, compact support principles and Keller-Osserman conditions on manifolds.}}

\author{Bruno Bianchini \ \ \ Luciano Mari \ \ \ Patrizia Pucci \ \ \ Marco Rigoli}

\date{\today}

\maketitle

\scriptsize \begin{center} Dipartimento di Matematica Pura e Applicata, Universit\`a degli Studi di Padova\\
Via Trieste 63, I-35121 Padova (Italy)\\
E-mail address: bianchini@dmsa.unipd.it
\end{center}
\scriptsize \begin{center} Scuola Normale Superiore\\
P.za dei Cavalieri 7, 56126 Pisa (Italy)\\
E-mail address: luciano.mari@sns.it
\end{center}
\scriptsize \begin{center} Dipartimento di Matematica e Informatica, Universit\`a degli Studi di Perugia\\
Via Vanvitelli 1, I-06123 Perugia (Italy)\\
E-mail address: patrizia.pucci@unipg.it
\end{center}

\scriptsize \begin{center} Dipartimento di Matematica,
Universit\`a degli Studi di Milano\\
Via Saldini 50, I-20133 Milano (Italy)\\
E-mail address: marco.rigoli55@gmail.com
\end{center}

\normalsize

\vspace{0.3cm}

\begin{abstract}
\footnote{\textbf{Mathematic subject classification 2010}: primary 35R01, 35B50, 35B53, 58J65, 53C42; secondary 58J05, 35B08, 35B45, 35J08, 35R45. \par
}

This paper is about the influence of Geometry on the qualitative behaviour of solutions of quasilinear PDEs on Riemannian manifolds. Motivated by examples arising, among others, from the theory of submanifolds, we study classes of coercive differential inequalities of the form 
$$
\diver\left( \frac{\varphi(|\nabla u|)}{|\nabla u|} \nabla u\right) \ge b(x)f(u) l(|\nabla u|) \qquad \text{(respectively, $\le$ or $=$)}
$$
on domains of a manifold $M$, for suitable $\varphi,b,f,l$, with emphasis on mean curvature type operators. We investigate the validity of strong maximum principles, compact support principles and Liouville type theorems; in particular, the goal is to identify sharp thresholds, involving curvatures or volume growth of geodesic balls in $M$, to guarantee the above properties under appropriate Keller-Osserman type conditions, and to discuss the geometric reasons behind  the existence of such thresholds. The paper also aims to give a unified view of recent results in the literature. The bridge with Geometry is realized by studying the validity of weak and strong maximum principles at infinity, in the spirit of Omori-Yau's Hessian and Laplacian principles and subsequent improvements.

\end{abstract}

\normalsize

\vspace{0.3cm}

\tableofcontents

\section{Introduction}

The study of quasilinear differential inequalities of the type
\begin{equation}\label{eq_base}
\diver \mathcal{A}(x,u,\nabla u) \ge \mathcal{B}(x,u,\nabla u)
\end{equation}
on Euclidean space $\R^m$ is a classical subject, and a great deal of work has been devoted to the analysis of the qualitative properties of solutions. The literature is vast, and we restrict ourselves to the special case
\begin{equation}\label{standardcase}
\mathcal{B}(x,u,\nabla u) = b(x)f(u) l(|\nabla u|),
\end{equation}
for continuous $b,f,l$. With no claim of completeness, we quote \cite{bandlegrecoporru, damfamise, farinaserrin1, farinaserrin2, fprarch, fprgrad, greco, lili, mitpoho, pucciserrin_2}, and for similar inequalities in the sub-Riemannian setting of Carnot groups, \cite{bm, bm2, DAmbrMit, bordofilipucci, mmmr, AMR}. The results in the references above will be related to those in our work in a more precise way in due course in the article.\par

Motivated by geometrical and physical problems, there has recently been an increasing interest in the study of some classes of quasilinear PDEs on complete Riemannian manifolds. As it is well known, the behaviour of their solutions strongly depends on the underlying space. A typical example is the minimal (hyper-)surface equation, which admits no non-constant positive solutions on $\R^m$ while plenty of bounded solutions exist in the hyperbolic space $\HH^m$. To the best of our knowledge, only a few authors have analyzed the influence of geometry on the behaviour of solutions of \eqref{eq_base}, \eqref{standardcase} in a general setting, for instance see \cite{PuRS, maririgolisetti, mmmr, AMR_book}, leaving however the picture still fragmentary, especially in case where $l$ in \eqref{standardcase} is a non-constant function. As one of the main purposes of the present work, we aim to give a detailed account of how geometry comes into play at the global level. Nevertheless, many interesting questions and problems remain open.\par
From now on, we let $(M, \metric)$ be a Riemannian manifold of dimension $m \ge 2$. We shall assume throughout the paper that $M$ is non-compact. To avoid excessive technicalities, while still keeping a good amount of generality, we study the following subclass of \eqref{eq_base}: we consider a quasilinear operator $\Delta_\varphi$, called the $\varphi$-Laplacian, weakly defined by
$$
\Delta_\varphi u = \diver\left( \frac{\varphi(|\nabla u|)}{|\nabla u|} \nabla u\right),
$$
where we assume
\begin{equation}\label{assumptions}
\varphi \in C(\R^+_0), \quad \varphi(0)=0, \quad \varphi(t) > 0 \ \text{ on } \, \R^+; \end{equation}
hereafter, $\R^+_0 = [0,\infty)$ and $\R^+ = (0,\infty)$. Different choices of $\varphi$ give rise to well-known, geometrically relevant operators, for instance
\begin{itemize}
\item[-] the $p$-Laplacian $\Delta_p$, $p>1$, where $\varphi(t) = t^{p-1}$;
\item[-] the mean curvature operator, describing the mean curvature of the graph hypersurface $\{(x, v(x)) : x \in M\}$ into the Riemannian product $M \times \R$. In this case, $\varphi(t) = t(1+t^2)^{-1/2}$;
\item[-] the operator of exponentially harmonic functions, where $\varphi(t) = t\exp\big(t^2\big)$, considered for instance in \cite{DE};
\item[-] the operator associated to $\varphi(t) = t^{p-1} + t^{q-1}$, $1<q<p$, that appears in quantum physics, see \cite{benci};
\end{itemize}
and many more. 
We focus our attention on the differential inequalities
\begin{equation}\label{problems}
\begin{array}{rll}
(P_\ge) & \quad \Delta_\varphi u \ge b(x)f(u)l(|\nabla u|) \\[0.2cm]
(P_\le) & \quad \Delta_\varphi u \le b(x)f(u)l(|\nabla u|) \\[0.2cm]
(P_=) & \quad \Delta_\varphi u = b(x)f(u)l(|\nabla u|) & \\[0.2cm]
\end{array}
\end{equation}
in a \emph{connected} open set, that is, a domain $\Omega \subset M$. Typically, we do not require $\Omega$ relatively compact, it might coincide with $M$ or with an end of $M$, in other words, with a non relatively compact connected component of $M\backslash K$, for some compact set $K$.\\
Throughout the paper we fix the basic assumptions on $b,f,l$:
\begin{equation}\label{assumptions_bfl}
\begin{array}{l}
b \in C(M), \qquad b > 0 \ \text{ on } \, M, \\[0.1cm]
f \in C(\R), \\[0.1cm]
l \in C(\R_0^+), \qquad l>0 \ \text{on } \, \R^+.
\end{array}
\end{equation}
Because of the positivity of $b$ and $l$, if $f \ge 0$ the problems in \eqref{problems} are called completely coercive in the recent literature (see \cite{farinaserrin1, farinaserrin2, dambrosiomitidieri_2}). Obviously, solutions of \eqref{problems} are considered in the weak sense and, in view of geometric applications, we confine ourselves to locally Lipschitz or $C^1$ solutions. It should be stressed that relaxing their regularity class is by no means a trivial or just a technical issue. For instance, under  our requirements on $\varphi, b,l$, we are not aware of the validity of weak Harnack inequalities for \eqref{problems}, and solutions may not be even locally bounded.

\begin{definition}
A function $u : \Omega \ra \R$ is a \emph{$C^1$ solution} (respectively, \emph{$\lip_\loc$ solution})  of $(P_\ge)$ in \eqref{problems} if $u \in C^1(\Omega)$ (resp., $u \in \lip_\loc(\Omega)$) and satisfies $(P_\ge)$ in the weak sense, that is,
$$
- \int_\Omega \frac{\varphi(|\nabla u|)}{|\nabla u|}\langle \nabla u, \nabla \psi \rangle \ge \int_\Omega b(x)f(u)l(|\nabla u|) \psi \qquad \text{for each } \psi \in C^\infty_c(\Omega), \ \psi \ge 0,
$$
where integration is performed with respect to the Riemannian measure. The analogous statement, with the reverse inequality, defines $C^1$ and $\lip_\loc$ solutions of $(P_\le)$.
\end{definition}

\subsection{Bernstein type theorems}\label{subsec_bernstein}

The original motivation for the present paper was the investigation of Bernstein type theorems for graphs in general ambient spaces. The classical Bernstein theorem states that entire minimal graphs in $\R^{m+1}$, described by solutions $u : \R^m \ra \R$ of the minimal hypersurface equation
$$
\diver \left( \frac{\nabla u}{\sqrt{1+|\nabla u|^2}} \right) = 0 \qquad \text{on } \, \R^m,
$$
are affine functions if $m \le 7$. Its solution and the proof of the sharpness of the dimension restriction favoured the flourishing of Geometric Analysis (cf. \cite{giusti}). Needless to say, the possible existence of a similar result for global graph hypersurfaces in different ambient spaces $\bar M^{m+1}$ heavily depends on the geometry of $\bar M$. Suppose that $(\bar M^{m+1}, \metricN )$ has a nowhere-vanishing conformal vector field $X$ and a distinguished slice $(M, \metric)$ orthogonal to the flow lines of $X$. Typical examples include the warped product structures
\begin{equation}\label{warpedproducts}
\begin{array}{lll}
(1) & \bar M = \R \times_h M, & \quad \text{with metric} \qquad \metricN = \di s^2 + h(s)^2 \metric, \quad \text{and} \\[0.2cm]
(2) & \bar M = M \times_h \R, & \quad \text{with metric} \qquad \metricN = \metric + h(x)^2\di s^2,
\end{array}
\end{equation}
for positive $h \in C^\infty(\R)$ and $h \in C^\infty(M)$, respectively. In the first case, $X = h(s) \partial_s$ is a conformal field with geodesic flow lines, while in the second $X=\partial_s$ is Killing, and the distance between the two flow lines equals the distance between their projections on any slice $\{s = \mathrm{const}\}$. For this reason, in case $(2)$ the flow lines of $X$ are called equidistant curves.

\begin{example}\label{ex_Hmp1}
\emph{Consider the upper half-space model of the hyperbolic space:
$$
\HH^{m+1} = \Big\{ (x_0,x) \in \R \times \R^{m} \ \ : \ \ x_{0}>0\Big\}
$$
with metric
$$
\metricN = \frac{1}{x_0^2}\Big( \di x_0^2 + \metric_{\R^m}\Big).
$$
With the change of variables $s = -\log x_0$ we express $\metricN$ as the warped product of type $1$ 
$$
\HH^{m+1} = \R \times_{e^{s}} \R^m, \qquad \metricN = \di s^2 + e^{2s} \metric_{\R^m},
$$
whose slices $\{s= \mathrm{const} \}$ are called horospheres. Similarly, we can view $\HH^{m+1}$ as a different warped product of type $1$, along totally umbilical hyperspheres:
\begin{equation}\label{Hmp1hyperspheres}
\HH^{m+1} = \R \times_{\cosh s} \HH^m, \qquad  \metricN = \di s^2 + \cosh^2 s \metric_{\HH^m}.
\end{equation}
On the other hand, $\HH^{m+1}$ admits a warped product of type $2$ via equidistant curves, of the type
\begin{equation}\label{Hmp1hyperspheres}
\HH^{m+1} = \HH^m \times_{\cosh \rho} \R, \qquad  \metricN = \metric_{\HH^m} + \cosh^ 2\big(r(x)\big)\di s^2, 
\end{equation}
where $r : \HH^m \ra \R$ is the distance from a fixed origin in $\HH^m$. This corresponds, in the upper half-space model, to the fibration of $\HH^{m+1}$ via euclidean lines orthogonal to the totally geodesic hypersphere $\{ x_0^2 + |x|^2 = 1\}$.
}
\end{example}

Given a function $v : M \ra \R$, one can consider the graph 
$$
\Sigma^m = \Big\{ (s,x) \in \R \times M, \ \ s = v(x)\Big\},
$$
that according to whether $\bar M$ is of type $1$ or $2$ we call, respectively, geodesic or equidistant graph. In the setting of the hyperbolic space, M.P. Do Carmo and H.B. Lawson proved in \cite{docarmolawson} the following beautiful result:

\begin{theorem}\label{teo_docarmolawson}
Let $\Sigma^{m} \ra \HH^{m+1}$ be the geodesic graph of a function $v$ over a horosphere or a hypersphere $M$. Suppose that $\Sigma$ has constant mean curvature $H \in [-1,1]$. More precisely,
\begin{itemize}
\item[(i)] If $M$ is a horosphere, then $H= \pm 1$ and $\Sigma$ is a horosphere;
\item[(ii)] If $M$ is a hypersphere, then $H \in (-1,1)$ and $\Sigma$ is a hypersphere.
\end{itemize}
\end{theorem}

Note that, differently from the Euclidean case, no restriction appears on $m$. The result applies in fact to the much more general setting of properly embedded hypersurfaces, but its proof, relying on the moving plane method, seems difficult to adapt to graphs with variable mean curvature. This originates our search for alternative arguments and motivates the study of \eqref{problems}: indeed, as we shall see in a moment, the prescribed mean curvature equation for $\Sigma$ leads to the study of inequalities $(P_\ge), (P_\le)$ and $(P_=)$. We let $\Phi_t$ be the flow of $X$ and, for convenience, we express the prescribed mean curvature equation in terms of the function $u(x) = t (v(x))$. Let $\nabla$ be the connection on $(M, \metric)$.\\[0.2cm]
\noindent \textbf{Geodesic graphs.} In this case, the flow parameter $t$ satisfies
\begin{equation}\label{bonitinho}
t = \int_0^s \frac{\di \sigma}{h(\sigma)}, \qquad t : \R \ra t(\R) = I.
\end{equation}
Set $\lambda(t) = h\big( s(t)\big)$. If $H$ is the normalized mean curvature of $\Sigma$ with respect to the upward-pointing unit normal
\begin{equation}\label{normal_geodesic}
\nu = \frac{1}{\lambda(u)\sqrt{1+|\nabla u|^2}} \Big( \partial_t - (\Phi_u)_* \nabla u \Big),
\end{equation}
then a computation in \cite{dhl} shows that $u : M \ra I$ solves
\begin{equation}\label{prescribed_geodesic}
\diver \left( \frac{\nabla u}{\sqrt{1+|\nabla u|^2}} \right) = m \lambda(u) H + m \frac{\lambda_t(u)}{\lambda(u)} \frac{1}{\sqrt{1+|\nabla u|^2}} \qquad \text{on } \, M,
\end{equation}
where $\lambda_t$ is the derivative of $\lambda$ with respect to $t$.\\[0.2cm]
\noindent \textbf{Equidistant graphs.} In this case, the flow parameter is $t=s$. Having defined the normal direction
\begin{equation}\label{normal_equidistant}
\nu = \frac{1}{h \sqrt{1+h^2|\nabla u|^2}} \Big( \partial_t - h^2 (\Phi_u)_* \nabla u\Big),
\end{equation}
a computation in \cite{dajczerlira} shows that $u : M \ra I =\R$ solves
\begin{equation}\label{prescribed_equidistant}
\diver \left( \frac{h \nabla u}{\sqrt{1+ h^2|\nabla u|^2}} \right) = mH - \left\langle \frac{h \nabla u}{\sqrt{1+ h^2|\nabla u|^2}} , \frac{\nabla h}{h} \right\rangle \qquad \text{on } \, M. 
\end{equation}
If we consider the conformal deformation 
$$
\overline{\metric} = h^{-2} \metric,
$$
and we denote with $\|\cdot \|, \, \bar \nabla$ and $\overline{\diver}$, respectively, the norm, connection and divergence in the metric $\overline{\metric}$, then \eqref{prescribed_equidistant} is equivalent to 
\begin{equation}\label{prescribed_equidistant_confo}
\overline{\diver}_h \left( \frac{\bar \nabla u}{\sqrt{1+\|\bar \nabla u\|^ 2}} \right) = m H h^2 \qquad \text{on } \, \big( M, \overline{\metric}\big),  
\end{equation}
where 
$$
\overline{\diver}_h Y = h^{m-1} \overline{\diver} \big( h^{1-m} Y\big)
$$
is a weighted divergence.\\[0.2cm]
For suitable choices of $H$ including the minimal case $H=0$, equations \eqref{prescribed_geodesic} and \eqref{prescribed_equidistant_confo} can be put into the form $(P_=)$ in \eqref{problems} (in the second case, with a weight in the driving differential operator). A further interesting example is that of mean curvature flow (MCF) solitons. We recall that a family of hypersurfaces $f_t : \Sigma^m \ra \bar M^{m+1}$ moving by mean curvature flow
$$
\partial_t f_t = \overrightarrow{H}(f_t)
$$
(with $\overrightarrow{H}$ the unnormalized mean curvature vector) is said to be a mean curvature soliton if there exists a conformal vector field $Y$ on $\bar M$ such that $f_t(M) = \Psi_{\tau(t)}(f_0(M))$, where $\Psi_{\tau}$ is the flow of $Y$ and $\tau(t)$ is a time reparametrization. Equivalently, a soliton satisfies the identity 
$$
\overrightarrow{H} = Y^\perp,
$$
where $\perp$ is the orthogonal projection on the normal bundle. Solitons in $\R^{m+1}$ with respect to the homothetically shrinking and to the translating vector fields $Y$ give rise, respectively, to classical self-shrinkers and self-translators, that model the singularities developed under the MCF (cf. \cite{mantegazza}). Bernstein type theorems for shrinkers that are graphs over $\R^m$ have been proved in \cite{eckerhuisken, wang}, and for translators in \cite{baoshi}. Although solitons in more general ambient spaces can no longer describe the blow-up of a singularity, nevertheless they are still relevant since they act as barries for the MCF evolution. Suppose that $\Sigma$ is the graph of $u : M \ra I$ along the flow lines of a conformal field $X$, and note that $\overrightarrow{H} = mH \nu$. If the soliton field coincides with $\pm X$ ($= \pm \partial_t$), 
\begin{itemize}
\item[(1)] for geodesic graphs, by \eqref{normal_geodesic} equation \eqref{prescribed_geodesic} specifies to 
\begin{equation}\label{soliton_geodesic}
\diver\left( \frac{\nabla u}{\sqrt{1+|\nabla u|^2}}\right) = \left[ \frac{m \lambda_t(u) \pm \lambda^3(u)}{\lambda(u)} \right] \frac{1}{\sqrt{1+|\nabla u|^2}};
\end{equation}
\item[(2)] for equidistant graphs, by \eqref{normal_equidistant} equation \eqref{prescribed_equidistant_confo} becomes 
\begin{equation}\label{soliton_equidistant_confo}
\overline{\diver}_h \left( \frac{\bar \nabla u}{\sqrt{1+\|\bar \nabla u\|^2}} \right) = \pm h(x)^3 \frac{1}{\sqrt{1+\|\bar \nabla u\|^2}} \qquad \text{on } \, \big( M, \overline{\metric}\big).  
\end{equation}
\end{itemize}

\subsection{Main properties under investigation}

We make a preliminary observation. Suppose that $f$ has at least a zero on $\R$: then, by the translation invariance of \eqref{problems} with respect to $u$, without loss of generality we can assume that $f(0)=0$. The function $u \equiv 0$ is then a solution of $(P_=)$, and the reduction principle in \cite{dambrosiomitidieri_2} (or Lemma \ref{lem_pasting}, see also \cite{le} and the appendix of \cite{AMR}) guarantees that
$$
\begin{array}{c}
u_+ = \max\{u,0\} \quad \text{solves $(P_\ge)$ weakly on $\Omega$.} \\[0.1cm]
\end{array}
$$
Therefore, when $f$ has a zero, without loss of generality we can restrict ourselves to investigate $(P_\ge)$ under the further assumption $f(0)=0$ and, if $u>0$ somewhere, we can also suppose $u \ge 0$ on $\Omega$.

%
%

\begin{definition}\label{def_properties}
{\rm We say that:
\begin{itemize}
\item the \emph{{compact support principle}} (shortly, $\csp$) holds for $(P_\ge)$ if each non-negative $C^1$ solution of~$(P_\ge)$ on an end $\Omega$ of $M$, satisfying $u(x) \ra 0$ as $x \ra \infty$ in $\Omega$, has compact support, that is, $u\equiv 0$ outside some compact set;
\item the \emph{{finite maximum principle}} (shortly, $\fmp$) holds for $(P_\le)$ on the domain $\Omega \subset M$ if any non-negative $C^1$ solution of~$(P_\le)$ for which $u(x_0) =0$ at some $x_0 \in \Omega$, satisfies $u\equiv 0$ on $\Omega$;
\item the \emph{{strong Liouville property}} (shortly, $\slio$) holds for $(P_\ge)$ if there exist no non-negative, non-constant $C^1$ solutions of~$(P_\ge)$ on all of $M$.
\item the \emph{{Liouville property}} (shortly, $\lio$) holds for $(P_\ge)$ if there exist no non-negative, non-constant, \emph{bounded} $\lip_\loc$ solutions of~$(P_\ge)$ on all of  $M$.
\end{itemize}}
\end{definition}
\begin{remark}
\emph{Besides the regularity required on $u$, we emphasize that the only difference between properties $\lio$ and $\slio$ is that, in $\lio$, we require that the solution of~$(P_\ge)$ be \emph{a-priori} bounded.
}
\end{remark}

\begin{remark}[\textbf{Constant solutions}]
 \emph{It is clear, by the properties of $b,f,l$ in \eqref{assumptions_bfl}, that a constant $u=u^*$ solves $(P_\ge)$ if and only if
$$
\begin{array}{l}
\disp l(0)=0, \quad \text{independently of $f$,} \ \text{ or } \, \\[0.3cm]
\disp l(0)>0, \quad f(u^*) \le 0.
\end{array}
$$
Therefore, in what follows we will always concentrate on \emph{non-constant} solutions.}
\end{remark}

Note that $\fmp$ is of a local nature, and thus its validity should not depend on the considered manifold. On the other hand, $\csp$, $\lio$ and $\slio$ are \emph{global} properties, and for this reason they are expected to depend on the geometry at infinity of~$M$ and not only on the structure of the operator related to $\varphi,b,f,l$. More precisely, the next scheme summarizes what occurs in general:
\begin{equation}\label{scheme}
\left\{\begin{array}{c}
\text{geometric conditions} \\[0.1cm]
\text{related both to $b$} \\[0.1cm]
\text{and to $\varphi, f,l$}
\end{array}\right\}
\quad + \quad \left\{\begin{array}{c}
\text{condition on} \\[0.1cm]
\text{$\varphi,f,l$} \end{array}\right\} \quad \Longrightarrow \quad \left\{\begin{array}{c}
\text{\text{either $\slio$,}} \\[0.1cm]
\text{or $\lio$,} \\[0.1cm]
\text{or $\csp$} \end{array} \right\}.
\end{equation}

We now describe the requirements on $\varphi,f,l$ needed in order to possibly obtain  $\slio$ or $\csp$, and next we will consider the role of geometry and of $b$. We assume

\begin{equation}\label{assum_secODE_altreL_intro_prima}
\left\{ \begin{array}{l}
\disp \mbox{$\varphi \in C^1(\R^+)$, $\qquad \varphi'>0$ on $\R^+$,}\\[0.3cm]
\disp \frac{t \varphi'(t)}{l(t)} \in L^1(0^+).
\end{array}\right.
\end{equation}
Then, the function
\begin{equation}\label{def_K}
K(t) = \int_0^t \frac{s\varphi'(s)}{l(s)}\di s
\end{equation}
realizes a homeomorphism of $\R^+_0$ onto its image $[0, K_\infty)$, with inverse $K^{-1} : [0, K_\infty) \ra \R^+_0$. Unless otherwise specified, we set
\begin{equation}\label{def_Fe_intro}
F(t) = \int_0^t f(s) \di s.
\end{equation}
To deal with $\fmp$ and $\csp$, we further suppose that
\begin{equation}\label{ipo_base_f_CSP}
f \ge 0 \quad \text{ on some $\ [0,\eta_0)$, $\ \eta_0 >0$}.
\end{equation}
The validity of $\fmp$ and $\csp$ is related to the next integrability requirement:
\begin{equation}\label{KO_zero_intro}\tag{KO$_0$}
\frac{1}{K^{-1}\circ F} \in L^1(0^+).
\end{equation}
More precisely, $\fmp$ depends on the failure of \eqref{KO_zero_intro} while $\csp$ on its validity. Regarding $\slio$, the relevant condition becomes an integrability at infinity, that to be expressed needs the further assumption
\begin{equation} \label{assum_KO_nonl1}
\disp \frac{t \varphi'(t)}{l(t)} \not \in L^1(\infty),
\end{equation}
in order for $K^{-1}$ to be defined on $\R^+_0$ (i.e. $K_\infty = \infty$). If we now suppose that
\begin{equation}\label{ipo_base_f_SL}
f \ge 0 \ \text{ on $\R^ +$,}
\end{equation}
%
then $\slio$ depends on the requirement
\begin{equation}\label{KO_infinity_intro}\tag{KO$_\infty$}
\frac{1}{K^{-1}\circ F} \in L^1(\infty).
\end{equation}
If $l\equiv 1$, $K$ coincides with the function
\begin{equation}\label{def_Hs}
H(t)=t\varphi (t)-\int_{0}^{t} \varphi (s)\di s, \qquad t\geq 0,
\end{equation}
that represents the pre-Legendre trasform of
$$
\Phi (t)=\int_{0}^{t} \varphi (s)\di s,
$$
and in this case we recover the necessary and sufficient conditions for $\csp$, $\fmp$ and $\slio$ thoroughly investigated in \cite{nu, pucciserrin_revisited, pucciserrin, pucciserrinzou} on $\R^m$, see also the references therein. In the case of the $p$-Laplacian where $\varphi(t) = t^{p-1}$, and for $l \equiv 1$, \eqref{KO_zero_intro} and \eqref{KO_infinity_intro} take, respectively, the well-known form
\begin{equation}\label{notKO_zero_plapla}
\frac{1}{F^{1/p}} \in L^1(0^+), \qquad \frac{1}{F^{1/p}} \in L^1(\infty).
\end{equation}
Condition \eqref{KO_infinity_intro} originated from the works of J.B. Keller and R. Osserman \cite{keller, osserman} for the prototype case
\begin{equation}\label{semilinear_prototype}
\Delta u \ge f(u)
\end{equation}
on $\R^m$: in particular, Osserman introduced \eqref{KO_infinity_intro} in his investigation on the conformal type of a Riemann surface. For convenience, in what follows we name both \eqref{KO_zero_intro} and \eqref{KO_infinity_intro} the \emph{Keller-Osserman conditions}. To our knowledge, the study of the relations between Keller-Osserman conditions and the geometry of $M$ initiated with the influential paper \cite{chengyau} by S.Y. Cheng and S.T. Yau, for the semilinear example \eqref{semilinear_prototype}. An important feature of \eqref{KO_zero_intro} and \eqref{KO_infinity_intro} to notice is their independence on the underlying space and on the weight $b$. The way geometry relates to the Keller-Osserman conditions in order to give $\slio$ and $\csp$ is one of the primary concerns of the present work, and will be expressed in terms of sharp curvature or volume growth bounds on $M$, and sharp estimates for $b$. In this respect, in many instances even when $l$ is constant such interplay is still partially unclear.\par
The bridge between geometry and the properties in Definition \ref{def_properties} is provided, at least in this paper, by the validity of the weak and strong maximum principles at infinity, that we now define:

\begin{definition}\label{def_maxprinc}
\rm{Assume \eqref{assumptions} and fix $b,l$ satisfying \eqref{assumptions_bfl}. We say that
\begin{itemize}
\item $(bl)^{-1}\Delta_\varphi$ satisfies the \emph{{weak maximum principle at infinity}}, shortly, $\wmp$, if for each non-constant $u \in \lip_\loc(M)$ such that $u^* = \sup_M u < \infty$, and for each $\eta<u^*$,
$$
\inf_{\Omega_\eta} \Big\{\Big(b(x)l(|\nabla u|)\Big)^{-1}\Delta_\varphi u\Big\} \le 0,
$$
where
$$
\Omega_\eta = \{x\in M\, :\, u(x)>\eta\}
$$
and the inequality has to be intended in the following sense: if $u$ solves
\begin{equation}\label{def_weakWMP}
\Delta_\varphi u \ge K b(x)l(|\nabla u|) \qquad \text{on } \, \Omega_\eta,
\end{equation}
for some $K \in \R$, then necessarily $K \le 0$.

\item $(bl)^{-1}\Delta_\varphi$ satisfies the \emph{{strong maximum principle at infinity}}, shortly, $\smp$, if for each non-constant $u \in C^1(M)$ such that $u^* = \sup_M u < \infty$, and for each $\eta<u^*$, $\eps>0$,
\begin{equation}\label{def_omegaeta}
\Omega_{\eta,\eps} = \{x\in M : u(x)>\eta, \, \, |\nabla u(x)|<\eps\} \qquad \text{is non-empty,}
\end{equation}
and
$$
\inf_{\Omega_{\eta,\eps}} \Big\{\Big(b(x)l(|\nabla u|)\Big)^{-1}\Delta_\varphi u\Big\} \le 0,
$$
where, again, the inequality has to be intended in the way explained above.
\end{itemize}}
\end{definition}
%

%
%
%
%
%
%

\begin{remark}
\rm{Condition $\Omega_{\eta,\eps} \neq \emptyset$ in \eqref{def_omegaeta} is not automatic: for example, consider the function $u(x) = \exp(-|x|)$ on
$\R^m\backslash \{0\}$, for which $|\nabla u| \ra 1$ when $u \ra u^*=1$. However, it is easy to see that
$\Omega_{\eta,\eps}$ is always non-empty if $M$ is complete; this can be seen as a consequence of I. Ekeland quasiminimum principle\footnote{Here is a quick geometrical reasoning. By contradiction, suppose that $|\nabla u|\ge \eps$ on $\Omega_\eta$,
and take any maximal flow line $\gamma$ of $X=\nabla u/|\nabla u|$ starting from some $x \in \Omega_\eta$
(it might be locally non-unique since $X$ is just continuous, but it exists by Peano theorem and $\gamma$ is
defined on $\R^+$ since $X$ is bounded). From $\gamma(\R^+) \subset \Omega_\eta$, integrating along $\gamma$
 would imply $u^*=\infty$, against our assumptions.}.}
\end{remark}

As we shall see in a short while, $\wmp$ and $\smp$ hold under mild geometric assumptions, involving the Ricci curvature or the volume growth of geodesic balls. Moreover, $\wmp$ is equivalent to $\lio$ for each $f$ with $f(0)=0$ and $f>0$ on $\R^+$. Both principles relate to \eqref{KO_zero_intro} and \eqref{KO_infinity_intro} to guarantee, respectively, $\csp$ and $\slio$. Hereafter, we denote with $r(x)$ the distance of $x$ from a fixed subset $\mathcal{O} \subset M$ that we call an origin. The origin may be a point (in this case, we denote it by $o$) or a relatively compact, open set with smooth boundary. It is known that $r$ is smooth on an open, dense subset $\mathcal{D}_\mathcal{O} \subset M \backslash \mathcal{O}$, and we denote as usual $\cut(\mathcal{O}) = M \backslash ( \mathcal{D}_\mathcal{O} \cup \mathcal{O})$. To better describe our results and properly place them in the literature, we separately comment on each of the properties in Definitions \ref{def_properties} and \ref{def_maxprinc}.

\subsection{The finite maximum principle $\fmp$}
Beyond the basic requirements \eqref{assumptions} and \eqref{assumptions_bfl}, assume
\begin{equation}\label{assum_secODE_altreL_intro}
\left\{ \begin{array}{l}
\disp \mbox{$\varphi \in C^1(\R^+)$, $\qquad \varphi'>0$ on $\R^+$,}\\[0.3cm]
\disp \frac{t \varphi'(t)}{l(t)} \in L^1(0^+).
\end{array}\right.
\end{equation}
We construct $F$ and $K$ respectively as in \eqref{def_K} and \eqref{def_Fe_intro}. If $f > 0$ in a right neighbourhood of zero, the validity of $\fmp$ turns out to depend on the next non-integrability requirement:
\begin{equation}\label{KO_zero_intro_FMP}\tag{$\neg$KO$_0$}
\frac{1}{K^{-1}\circ F} \not\in L^1(0^+).
\end{equation}
If $l\equiv 1$, then $K$ coincides with the function
\begin{equation}\label{def_Hs}
H(t)=t\varphi (t)-\int_{0}^{t} \varphi (s)\di s, \qquad t\geq 0,
\end{equation}
and $f$ is non-decreasing and positive, in \cite{pucciserrin_revisited, pucciserrinzou} property $\fmp$ is shown to be equivalent to \eqref{KO_zero_intro_FMP}, see also Chapter 5 and Theorem 1.1.1 of \cite{pucciserrin}. We presently extend such a characterization to the case of a non-constant function $l$. The literature on the finite maximum principle for quasilinear inequalities is fairly intricate, with contributions from a number of different mathematicians. A detailed and commented account of previous works can be found in \cite[p. 125]{pucciserrin} and in \cite{pucciserrin_revisited}. To introduce our main result, we begin with 
\begin{definition}\label{def_Cincreasing}
We say that a function $h : \R \ra \R$ is $C$-increasing on $[a,b]$ for some constant $C \ge 1$ if
$$
\sup_{a \le s \le t} h(s) \le Ch(t).
$$
\end{definition}
Clearly, $1$-increasiness corresponds to $h$ being non-decreasing, but on the other hand $C$-increasiness allows oscillations of $h$.

\begin{theorem}\label{teo_FMP2_intro}
Let $M$ be a Riemannian manifold, and assume that $\varphi, b, f, l$ satisfy \eqref{assumptions}, \eqref{assumptions_bfl} and \eqref{assum_secODE_altreL_intro}. Suppose further that
\begin{itemize}
\item[-] $f(0)l(0)=0$;
\item[-] $f$ is non-negative and $C$-increasing on $(0,\eta_0)$, for some $\eta_0>0$;
\item[-] $l$ is $C$-increasing on $(0,\xi_0)$, for some $\xi_0>0$.
\end{itemize}
%
%
Then, $\fmp$ holds for non-negative solutions $u \in C^1(\Omega)$ of $(P_\le)$ on a domain $\Omega \subset M$ if and only if either
\begin{equation}\label{lasimple_2}
f\equiv 0 \qquad \text{on } \, [0,\eta_0),
\end{equation}
or
\begin{equation}\label{noninteinzero_4}
f>0 \quad \text{on } \, (0,\eta_0), \quad \text{and} \qquad \frac{1}{K^{-1}\circ F} \not \in L^1(0^+).
\end{equation}
\end{theorem}

\begin{remark}
\emph{For the sake of clarity, in \cite{pucciserrin_revisited} no differentiability of $\varphi$ is needed: indeed, $\varphi'$ does not appear in the definition of $H$, and the authors just require $\varphi$ to be strictly increasing. However, the presence of a possibly only continuous function $l$ forces us to increase the regularity of $\varphi$ to be able to define $K$.
}
\end{remark}
\begin{example}
\emph{Observe that Theorem \ref{teo_FMP2_intro} applies to the inequality
\begin{equation}\label{casogradiente}
\Delta_p u \le u^\omega |\nabla u|^q,
\end{equation}
with $\omega \ge 0$, $q \in [0,p)$, to guarantee that $\fmp$ holds if and only if
$$
\omega + q \ge p-1.
$$
}
\end{example}

The proof of Theorem \ref{teo_FMP2_intro} follows the standard method to show Hopf type lemmas, that is, it relies on the construction of suitable radial solutions of $(P_\ge)$ defined on annuli (see \cite{pucciserrin_revisited, pucciserrin}). However, the study of the related ODE is, for nontrivial gradient terms $l$, considerably more involved than that in \cite{pucciserrin}. This calls for a detailed investigation of singular Dirichlet and mixed Dirichlet-Neumann problems for quasilinear ODEs, accomplished in Section \ref{sec2.3}. The results therein have independent interest, and are central in many of the main theorems of the present paper.\par
For the prototype inequality $\Delta u \le f(u)$, very recently in \cite{pucciradulescu} the authors succeeded to prove Theorem \ref{teo_FMP2_intro} \emph{without requiring the $C$-monotonicity of $f$}. It is likely that the same is possible also for $(P_\le)$ in our generality, and so we propose the following 

\begin{problem}
Is it possible to prove Theorem \ref{teo_FMP2_intro} without requiring the $C$-monotonicity of $f$ and $l$? Or, at least, keeping the $C$-monotonicity of just one of them?
\end{problem}

\subsection{Strong and weak maximum principles at infinity}

We start describing the origin of properties $\smp$ and $\wmp$, and for simplicity we restrict to the case $b \equiv 1$, $l \equiv 1$ and $\Delta_\varphi = \Delta$, the Laplace-Beltrami operator. In this case, when $u \in C^2(M)$, $\smp$ and $\wmp$ can
equivalently be restated as the existence of a sequence of points $\{x_k\}_{k \in \mathbb{N}}\subset M$ such that
\begin{equation}\label{classical_SMP}
\disp u(x_k) > u^*-\frac{1}{k}, \qquad \Delta u(x_k) < \frac{1}{k}, \qquad |\nabla u|(x_k) < \frac{1}{k}
\end{equation}
for $\smp$ to hold, and
$$
\disp u(x_k) > u^*-\frac{1}{k}, \qquad \Delta u(x_k) < \frac{1}{k}
$$
for $\wmp$ to hold. In this case, $\smp$ is called the Omori-Yau principle, in view of the pioneering papers by H. Omori \cite{omori} and S.T. Yau (cf. \cite{yau2} and \cite{chengyau}, the second with S.Y. Cheng). It proved to be a fundamental tool in investigating geometric problems (see \cite{AMR_book} and the references therein). Omori in \cite{omori} realized that the validity of $\smp$ is not granted on a generic Riemannian manifold, although it is sufficient that $M$ enjoys very mild requirements. For example, by combining works of \cite{chenxin, rattorigolisetti, prsmemoirs, borbely}, see \cite[Thm. 2.4]{AMR_book}, $\Delta$ satisfies $\smp$ whenever
\begin{equation}\label{threshold_SMP}
\Ricc(\nabla r, \nabla r) \ge - G(r) \qquad \text{on } \, \mathcal{D}_o,
\end{equation}
$r(x)$ being the distance function from some origin $o$, and $G \in C^1(\R)$ has the following properties:
\begin{equation}\label{threshold_SMP_2}
G \ge 0, \quad G' \ge 0, \quad \frac{1}{\sqrt{G}} \not \in L^1(\infty).
\end{equation}
Clearly, in \eqref{threshold_SMP} and \eqref{threshold_SMP_2} what really matters is the growth of $G$ at infinity. A borderline example is given, for instance, by $G(t) \asymp 1+ t^2$. This is a particular case of a general criterion discovered in \cite{rattorigolisetti, prsmemoirs}, granting the validity of $\smp$ provided that $M$ supports a function satisfying
\begin{equation}\label{strongKhasminskii}
\begin{array}{l}
w \in C^2(M\backslash K) \ \text{ for some compact $K$}, \\[0.2cm]
w(x) \ra + \infty \ \text{ as } \, x \ \text{ diverges in $M$}, \\[0.2cm]
|\nabla w| \le \sqrt{G(w)}, \qquad \Delta w \le \sqrt{G(w)} \qquad \text{on } \, M \backslash K,
\end{array}
\end{equation}
where $G$ meets the requirements in \eqref{threshold_SMP_2}. For reasons that will be soon justified, we call $w$ a \emph{strong Khas'minskii} potential. To the best of our knowledge, this is essentially the only effective known condition, and $w$ is often explicitly given not exclusively via curvature bounds like \eqref{threshold_SMP}, but also by the geometrical nature of the problem at hand. This is the case, for instance, of immersed submanifolds, where $w$ depends on extrinsic data, and of generic Ricci soliton structures, see \cite{AMR_book}. \par
When the operator is nonlinear and non-homogeneous, instead of a single function $w$ we need a family of Khas'minskii potentials, see Section \ref{sec_SMP} below and Chapter 3 in \cite{AMR_book}. For $b^{-1}\Delta_\varphi$, $\smp$ has been studied in \cite[Sec. 6]{prsmemoirs} and \cite[Thm 1.1]{puccirigoli}, and again a family of potentials of strong Khas'minskii type is exhibited to ensure $\smp$ under an appropriate Ricci curvature bound. The construction of the potential in these papers is hand-made and appears not easily generalizable to cover the case of a non-constant $l$. Therefore, although our present strategy to prove $\smp$ for $(bl)^{-1}\Delta_\varphi$ is still based on finding a strong Khas'minskii potential family, the construction of the latter relies on a different approach involving the study of the maximal domain of existence and the asymptotic behaviour of solutions of a singular two-points boundary ODE problem, see Sections \ref{sec2.3} and \ref{sec_ricciandSMP}.\par
For the convenience of the reader, we summarize the assumptions of $\varphi,l$ in the following:
\begin{equation}\label{assum_SMP_intro}
\left\{
\begin{array}{l}
\disp \varphi \in C(\R^+_0)\cap C^1(\R^ +), \qquad \varphi(0)=0, \qquad \varphi' > 0 \ \text{ on } \, \R^+; \\[0.3cm]
l \in C(\R_0^+), \qquad l>0 \ \text{on } \, \R^+; \\[0.3cm]
\disp \frac{t \varphi'(t)}{l(t)} \in L^1(0^+).
\end{array} \right.
\end{equation}

We shall also require the growth conditions
\begin{equation}\label{assu_SMP_intro_growth}
\left\{\begin{array}{ll}
\disp l(t) \ge C_1 \frac{\varphi(t)}{t^\chi} & \qquad \text{on $(0,1]$, for some $C_1>0$, $\chi \ge 0$}; \\[0.4cm]
\disp \varphi(t) \le C_2t^{p-1} & \qquad \text{on $[0,1]$, for some $C_2>0$, $p>1$.}
\end{array}\right.
\end{equation}

\begin{remark}\label{rem_pel}
\emph{Since $l$ is continuous up to zero, if $\varphi(t) \asymp t^{p-1}$ near $t=0$ the first condition in \eqref{assu_SMP_intro_growth} forces the upper bound $\chi \le p-1$. For example, in the $p$-Laplacian case where $\varphi(t) = t^{p-1}$, chosen $l(t) = t^q$, the first in \eqref{assu_SMP_intro_growth} holds if and only if $q \in [0, p-1]$. Furthermore, to recover the case $l$ constant the best choice of $\chi$ is
$$
\chi = p-1;
$$
the choice $\chi=0$ represents the borderline case of strong gradient dependence $l(t) \asymp \varphi(t)$ near $t=0$. The latter often needs a special care to be treated.
}
\end{remark}

We express our main result in terms of a sharp condition on the Ricci tensor.

\begin{theorem}\label{teo_SMP_intro}
Let $M$ be a complete $m$-dimensional manifold such that, for some fixed origin $o \in M$, the distance $r(x)$ from $o$ satisfies
\begin{equation}\label{ricciassu_intro}
\Ricc (\nabla r, \nabla r) \ge -(m-1)\kappa^2\big( 1+r^2\big)^{\alpha/2} \qquad \text{on } \, \mathcal{D}_o,
\end{equation}
for some $\kappa \ge 0$, $\alpha \ge -2$. Let $l$ and $\varphi$ satisfy \eqref{assum_SMP_intro} and \eqref{assu_SMP_intro_growth}. Consider $0< b \in C(M)$ such that
$$
b(x) \ge C\big(1+ r(x)\big)^{-\mu} \qquad \text{on } \, M,
$$
for some constants $C>0$, $\mu \in \R$. If
\begin{equation}\label{condi_volumericci_intro}
\mu \le \chi - \frac{\alpha}{2} \quad \text{and either} \quad \left\{ \begin{array}{ll}
\alpha \ge -2 \quad \text{and} \quad \chi>0, \quad \text{or} \\[0.2cm]
\alpha = -2, \quad \chi = 0 \quad \text{and} \quad \bar\kappa  \le \frac{p-1}{m-1},
\end{array}\right.
\end{equation}
with $\bar\kappa  = \frac{1}{2}\big( 1 + \sqrt{1+4\kappa^2}\big)$, then $(bl)^{-1}\Delta_\varphi$ satisfies $\smp$.
\end{theorem}
In particular, the Euclidean space $M=\R^m$ is recovered by choosing $\kappa = 0$, $\alpha = -2$, while, to deal with the hyperbolic space $\HH^m$ of constant sectional curvature $-1$ we choose $\kappa = 1$, $\alpha = 0$. Even for these model manifolds, Theorem \ref{teo_SMP_intro}, in the above generality on $b$ and $l$, is new. As an example, Corollary \ref{cor_SMP_MCO} in Section \ref{sec_ricciandSMP} expresses the result for the mean curvature operator both in $\R^m$ and in $\HH^m$.

%

Often in geometry, $\smp$ is used to infer a-priori estimates, and consequently Liouville type theorems, for solutions of PDEs. As an example of this type of application, in Corollary \ref{cor_SMP} we establish sharp estimates for bounded above solutions of differential inequalities of the form
\begin{equation}\label{Pge_general_intro}
\Delta_p u \ge b(x)f(u) |\nabla u|^{q} - \bar b(x) \bar f(u) |\nabla u|^{\bar q}.
\end{equation}
A prototype example is given by
$$
\Delta_p u \ge f(u) - c|\nabla u|^q, \qquad \text{with $q>0$, $c>0$,}
$$
that in the semilinear case appear as value functions of stochastic control problems (cf. \cite{lasrylions, radulescu}). Existence and nonexistence of entire solutions have been investigated in a series of papers, notably \cite{lairwood, ghernicrad, FPREnt, felmerquaassirakov, maririgolisetti}.

Next, we turn our attention to $\wmp$, which has been introduced in \cite{prs_proceeding} following the authors' observation that, in many geometric applications, the gradient condition in \eqref{classical_SMP} was unnecessary. It has various advantages with respect to $\smp$: first, it can be stated for $u \in W^{1,p}_\loc(M)$, $p \ge 1$, which is a natural regularity class for solutions of $(P_\ge)$; second, the absence of the gradient bound allows to directly use the weak formulation together with refined integral estimates, to obtain sharp criteria for $\wmp$ that just depend on the volume growth of geodesic balls $B_r$, a requirement implied, but not equivalent, to \eqref{ricciassu_intro}. This approach will be described in more detail below.
\begin{remark}
\emph{It is important to observe that there exist manifolds satisfying $\wmp$ but \emph{not} $\smp$, hence the two principle are different. Counterexamples are very easy to construct in the setting of incomplete manifolds (indeed, $\R^m \backslash \{0\}$ satisfies $\wmp$ but not $\smp$, see \cite{prsmemoirs}), and a nice example in the complete case appeared very recently in \cite{borbely_counter}.
}
\end{remark}

First, we introduce the following characterizazion improving on \cite{prs_proceeding, prsnonlinear, marivaltorta}. Despite the simplicity of the proof, the equivalences stressed below are particularly useful in geometric applications.

\begin{proposition}\label{prop_equivalence}
Let $\varphi$ and $b,f,l$ satisfy respectively \eqref{assumptions} and \eqref{assumptions_bfl}. Then, the following properties are equivalent:
\begin{itemize}
\item[(i)] $(bl)^{-1}\Delta_\varphi$ satisfies $\wmp$;
\item[(ii)] $\lio$ holds for some (equivalently, every) $f$ with
$$
f(0)=0, \qquad f>0 \qquad \text{on } \, \R^+;
$$
\item[(iii)] each $u \in \lip_\loc(M)$ solving $(P_\ge)$ on $M$ and bounded above satisfies $f(u^*) \le 0$.
\end{itemize}
\end{proposition}

It should be stressed that, by a generalization of work of R.Z. Khas'minskii \cite{kh} (see \cite{grigoryan} for a nice exposition), $\lio$ with the choice $f(t) = \lambda t$ and $\lambda >0$ is related to the theory of the (minimal) Brownian motion on $M$, and indeed equivalent to the stochastic completeness of $M$, that is, the infinite lifetime of a.e. Brownian path. Exploiting this last equivalence, A. GrigorYan in \cite[Thm. 9.1]{grigoryan} found the weakest known geometric condition on a complete $M$ for $\Delta$ to satisfy $\lio$ with $f(t) = \lambda t$ and $\lambda>0$, that is,
$$
\frac{r}{\log\vol(B_r)} \not\in L^1(\infty).
$$
However, the beautiful method of proof in \cite{grigoryan} relies on the linearity of the Laplace-Beltrami operator. Hence, the search for similar volume conditions for general $\Delta_\varphi$ calls for different ideas, developed in a series of works \cite{karp, prs_gafa, rigolisalvatorivignati_3, rigolisalvatorivignati_4} and refined in \cite{prsmemoirs, maririgolisetti}. Our contributions are contained in Theorems \ref{teo_main_2} and \ref{teo_maximum} below.


\begin{remark}
\emph{A characterization similar to that of Proposition \ref{prop_equivalence} holds for $\lio$ when $f\equiv 0$ in a right neighbourhood of zero. In fact, by  \cite[Thm. A]{prsnonlinear} (for $\Delta_p$) and \cite{prsmemoirs} (general $\Delta_\varphi$), for these $f$'s property $\lio$ is equivalent to the parabolicity of $\Delta_\varphi$, see also \cite{marivaltorta}. For the $p$-Laplace operator, parabolicity is more often introduced via capacity estimates. In this respect, we refer to \cite{grigoryan} for $p=2$, and \cite{troyanov, troyanov2, HKM} for a general $p$.
}
\end{remark}

\begin{remark}
\emph{Khas'minskii introduced a sufficient condition for $M$ to be stochastically complete in terms of the existence of $w$ satisfying all of the properties in \eqref{strongKhasminskii} but that on the gradient, with $G(t) = \lambda t$, $\lambda>0$, see \cite{kh, grigoryan}. This justifies the name strong Khas'minskii condition given to \eqref{strongKhasminskii}. It should be observed that, for a large class of operators including some geometrically relevant fully nonlinear ones, appropriate Khas'minskii conditions turn out to be equivalent to suitably defined maximum principles at infinity, see \cite{marivaltorta} and the recent \cite{maripessoa}.
}
\end{remark}

To introduce a special case of our main Theorem \ref{teo_main_2}, observe that \eqref{ricciassu_intro} implies, via the Bishop-Gromov comparison theorem, the following estimates
\begin{equation}\label{riccivolume_intro}
\begin{array}{ll}
\disp \limsup_{r \ra \infty} \frac{\log \vol(B_r)}{r^{1+ \alpha/2}} < \infty & \quad \text{if } \, \alpha> -2, \\[0.5cm]
\disp \limsup_{r \ra \infty} \frac{\log \vol(B_r)}{\log r} \le (m-1) \bar \kappa + 1 & \quad \text{if } \, \alpha> -2,
\end{array}
\end{equation}
with
$$
\bar k = \frac{1}{2}\big( 1 + \sqrt{1+4\kappa^2}\big).
$$

Regarding our assumptions on $\varphi$ and $l$, differently from \eqref{assum_SMP_intro} we now require the milder

\begin{equation}\label{assum_WMP_intro}
\left\{\begin{array}{l}
\disp \varphi \in C(\R^+_0), \qquad \varphi(0)=0, \qquad \varphi > 0 \ \text{ on } \, \R^+; \\[0.3cm]
l \in C(\R_0^+), \qquad l>0 \ \text{on } \, \R^+.
\end{array}\right.
\end{equation}

We also need the next growth conditions, to be compared with those in \eqref{assu_SMP_intro_growth}.

\begin{equation}\label{assu_WMP_intro_growth}
\left\{\begin{array}{ll}
\disp l(t) \ge C_1 \frac{\varphi(t)}{t^\chi} & \qquad \text{on $\R^ +$, for some $C_1>0$, $\chi \ge 0$}; \\[0.4cm]
\disp \varphi(t) \le C_2t^{p-1} & \qquad \text{on $[0,1]$, for some $C_2>0$, $p>1$;} \\[0.3cm]
\disp \varphi(t) \le \bar C_2t^{\bar p-1} & \qquad \text{on $[1,\infty)$, for some $\bar C_2>0$, $\bar p>1$.}
\end{array}\right.
\end{equation}

The use of different upper bounds for $\varphi(t)$ related to its behaviour near $t=0$ and $t=\infty$ is crucial to obtain sharp results in the setting, for instance, of the mean curvature operator. We are now ready to state

\begin{theorem}\label{teo_WMP_intro}
Let $M$ be a complete $m$-dimensional manifold. Fix $\alpha \ge -2$ and suppose that
\begin{equation}\label{volumeassu_intro}
\begin{array}{ll}
\disp \liminf_{r \ra \infty} \frac{\log\vol(B_r)}{r^{1+ \alpha /2} } = V_\infty < \infty & \quad \text{if } \, \alpha > -2; \\[0.4cm]
\disp \liminf_{r \ra \infty} \frac{\log\vol(B_r)}{\log r} = V_\infty < \infty & \quad \text{if } \, \alpha = -2. \\[0.4cm]
\end{array}
\end{equation}
Let $\varphi$ and $l$ satisfy \eqref{assum_WMP_intro} and \eqref{assu_WMP_intro_growth}, and consider $0< b \in C(M)$ such that
$$
b(x) \ge C\big(1+ r(x)\big)^{-\mu} \qquad \text{on } \, M,
$$
for some constants $C>0$, $\mu \in \R$. Suppose:
\begin{equation}\label{condi_volumeWMP_intro}
\mu \le \chi - \frac{\alpha}{2} \quad \text{and either} \quad \left\{ \begin{array}{ll}
\alpha \ge -2, \quad \chi>0, \quad \text{or} \\[0.2cm]
\alpha \ge -2, \quad \chi = 0, \quad \mu <  - \frac{\alpha}{2}, \quad \text{or } \\[0.2cm]
\alpha > -2, \quad \chi = 0, \quad \mu = - \frac{\alpha}{2}, \quad V_\infty = 0, \quad \text{or } \\[0.2cm]
\alpha = -2, \quad \chi = 0, \quad \mu = - \frac{\alpha}{2}, \quad V_\infty \le p.
\end{array}\right.
\end{equation}
Then, $(bl)^{-1}\Delta_\varphi$ satisfies $\wmp$.
\end{theorem}

\begin{remark}
\emph{As underlined in Remark \ref{rem_pel}, $p$ and $\bar p$ are implicitly related to bounds on $\chi$ via \eqref{assu_WMP_intro_growth}. However, we feel remarkable that $p, \bar p$ do not appear in conditions \eqref{condi_volumeWMP_intro}, apart from the last borderline case. A detailed discussion follows the statement of Theorem \ref{teo_main_2} in Section \ref{sec_WMP}.
}
\end{remark}

Suitable counterexamples show the sharpness of Theorem \ref{teo_main_2}, and consequently of Theorem \ref{teo_WMP_intro}, with respect to each parameter involved. In particular, the restrictions in \eqref{condi_volumeWMP_intro} are sharp. 

\begin{conjecture}
In the setting of Theorem \ref{teo_SMP_intro}, the full $\smp$ holds if the range \eqref{condi_volumericci_intro} is replaced by \eqref{condi_volumeWMP_intro}.
\end{conjecture}

To better appreciate the range of applicability of Theorem \ref{teo_WMP_intro}, we state as a direct corollary the following extension of Do Carmo-Lawson's Theorem \ref{teo_docarmolawson} in the minimal setting. The result is a particular case of Theorems  \ref{teo_bern_minimal_hyper} and \ref{teo_bern_minimal_horo} below.

\begin{theorem}\label{teo_bern_minimal_intro}
Let $M$ be a complete manifold, and consider the warped product $\bar M = \R \times_h M$, with warping function $h$ satisfying either
\begin{itemize}
\item[$(i)$] $h$ is strictly convex and $h^{-1} \in L^1(-\infty) \cap L^1(+\infty)$, or
\item[$(ii)$] $h'>0$ on $\R$, $h'(s) \ge C$ for $s >>1$ and $h^{-1} \in L^1(+\infty)$.
\end{itemize}
If 
$$
\liminf_{r \ra \infty} \frac{\log\vol(B_r)}{r^2} < \infty, 
$$
then
\begin{itemize}
\item[] under $(i)$, the only entire minimal graph is the constant $u = s_0$, with $s_0$ the unique minimum of $h$.
\item[] under $(ii)$, there exists no entire minimal graph.
\end{itemize}
\end{theorem}

The corresponding statement for variable mean curvature will be given in Theorem \ref{teo_presc_nonconst}, under the validity of $\smp$. It should be noted that, besides bounded solutions, in Theorem \ref{teo_main_2} we can also consider solutions $u$ of
$$
\Delta_\varphi u \ge b(x) f(u) l(|\nabla u|) \qquad \text{on } \, \Omega_\eta = \{ u >\eta\}
$$
\emph{with a controlled growth at infinity.} Indeed, under appropriate assumptions, the theorem  guarantees both $u^*< \infty$ and $f(u^*) \le 0$. The result is a significant improvement of \cite[Thm 5.1]{maririgolisetti} and \cite[Thm. 2.1]{AMR}; it applies, for instance, to differential inequalities with borderline gradient dependence of the type
$$
\Delta_\varphi u \ge b(x) f(u) \varphi(|\nabla u|) \qquad \text{on } \, M,
$$
to ensure that, under mild assumptions, any solution that grows polynomially is bounded from above and satisfies $f(u^*) \le 0$. Recall that, by definition, $u$ grows polynomially if there exists $\sigma \ge 0$ such that
$$
|u(x)| = O\big(r(x)^\sigma \big) \qquad \text{as } \, r(x) \ra \infty.
$$
The reader that is interested in such borderline examples can see Corollary \ref{cor_strano}, as well as Theorem \ref{teo_meancurv_intro} in  the particular setting of the mean curvature operator. We report here the following application to entire vertical self-translators of the mean curvature flow:

\begin{theorem}\label{teo_soliton_intro}
Let $(M^m, \metric)$ be a complete manifold and consider the product $\bar M^{m+1} = \R \times M$. Fix $0 \le \sigma \le 2$ and suppose that either
\begin{equation}\label{ipo_volume_soliton_intro}
\begin{array}{ll}
\sigma<2 & \quad \disp \text{and} \qquad \liminf_{r \ra \infty} \frac{\log\vol(B_r)}{r^{2-\sigma}} < \infty, \qquad \text{or} \\[0.5cm]
\sigma = 2 & \quad \disp \text{and} \qquad \liminf_{r \ra \infty} \frac{\log\vol(B_r)}{\log r} < \infty.
\end{array}
\end{equation}
Then, there exist no entire graph $\Sigma \subset \bar M$ of $v : M \ra \R$ which is a self-translator for the MCF with respect to the vertical direction $\partial_s$ and satisfies
\begin{equation}\label{crescitav_soliton_intro}
|v(x)| = o\big( r(x)^\sigma \big) \qquad \text{as } \, r(x) \ra \infty.
\end{equation}
\end{theorem}

\begin{remark}
\emph{Specializing Theorem \ref{teo_soliton_intro} in Euclidean space $\R^{m+1}$, there is no entire graph $v : \R^m \ra \R$ over the horizontal $\R^m$ which is a self-translator in the vertical direction and satisfies $v = o(r^2)$ as $r \ra \infty$. The result is sharp, since the bowl soliton in $\R^{m+1}$ (cf. \cite{altwu} and \cite[Lem. 2.2]{cluschsch}) and the non-rotational manifolds in \cite{wang_annals} for $m \ge 3$ are examples of rotationally symmetric, entire (convex) graphs which translate vertically by MCF and have order of growth $r^2$.
}
\end{remark}

Other applications for entire self-translators in $\R^{m+1}$ (not necessarily vertical) and for entire self-expanders will be given in Theorems \ref{teo_soliton_Rm} and \ref{teo_soliton_expa}, respectively.


Results in the spirit of Corollary \ref{cor_strano} below, with a dependence on the gradient, appear in \cite{farinaserrin1,farinaserrin2} on Euclidean space $\R^m$, and will be compared with ours in Section \ref{sec_WMP}. In a manifold setting, due to the possible lack of a polynomial bound for the growth of the volume of geodesic balls, the integral methods in \cite{farinaserrin1, farinaserrin2, dambrosiomitidieri_2, DAmbrMit, brezis} are, in most of the cases, not sufficient to get sharp conclusions. Indeed, even in the polynomial setting of Euclidean space, Theorem \ref{teo_main_2} complements and in some cases improves on the existing literature.\par

We conclude with the next observations: in view of the volume estimates \eqref{riccivolume_intro} that follow from the Ricci bound \eqref{ricciassu_intro}, Theorems \ref{teo_SMP_intro} and \ref{teo_WMP_intro} hold precisely for the same range of $\alpha, \mu, \chi$ (aside from some borderline cases covered by \eqref{condi_volumeWMP_intro} but not by \eqref{condi_volumericci_intro}). In view of this, the following problem seems interesting to us:

\begin{problem}\label{prob_SMPeWMP}
Prove or disprove by exhibiting a counterexample (in a complete manifold), the validity of $\smp$ in the assumptions of Theorem \ref{teo_WMP_intro}.
\end{problem}
%
%

\subsection*{The strong Liouville property $\slio$}

In the literature, the validity of $\slio$ has been mainly investigated
by means of two different approaches: radialization techniques and refined comparison theorems
\cite{maririgolisetti, pucciserrin_2, PuRS, bordofilipucci}, or integral estimates, in the
spirit of the work of E. Mitidieri and S.I. Pohozaev \cite{mitpoho}, see \cite{DAmbrMit, farinaserrin1, farinaserrin2, dambrosiomitidieri_2, Serrin_5}.
Assume the validity of \eqref{assum_secODE_altreL_intro} and \eqref{assum_KO_nonl1}, in order for  the function $K$ in \eqref{def_K} to realize a homeomorphism of $\R^+_0$ onto itself. Suppose that
$$
f>0 \qquad \text{on } \, (T, \infty), \quad \text{for some } \, T \ge 0,
$$
and set
\begin{equation}\label{def_Fe_2}
F(t) = \int_T^t f(s) \di s.
\end{equation}
Under these assumptions, the proof of the validity of $\slio$ via radialization techniques relies on the construction of suitable blowing-up radial supersolutions, explicitly related to the Keller-Osserman condition
\begin{equation}\label{KO}\tag{$\mathrm{KO}_\infty$}
\frac{1}{K^{-1}\circ F} \in L^1(\infty).
\end{equation}
As far as we know, \eqref{KO} first appeared for nontrivial $l$ in the work of R. Redheffer \cite{redheffer} (Corollary 1 therein) for the inequality $\Delta u \ge f(u)l(|\nabla u|)$. Since then, it has been systematically studied by various authors. Among them, for nontrivial $l$ we quote

\begin{itemize}
\item[-] \cite{caristimiti} (for the 1-dimensional problem), \cite{bandlegrecoporru, greco, filippucci} (when $\Delta_\varphi$ is the mean curvature operator) and \cite{MartioPorru, fprgrad, filippucci} (when $\Delta_\varphi$ is the $p$-Laplacian);
\item[-] in a sub-Riemannian setting, \cite{mmmr,bm2,bordofilipucci,bm,AMR}.
\end{itemize}
For further generalizations to quasilinear inequalities, possibly with singular or  degenerate weights, we refer to \cite{damfamise,dAM3, fprarch, mitpoho}.
%

We first discuss the necessity of \eqref{KO_infinity_intro} for $\slio$, and recall the following 

\begin{definition}
A point $o \in M$ is said to be a \emph{pole} if the exponential map $\exp_o :T_oM \thickapprox  \R^m \ra M$ is a diffeomorphism. 
\end{definition}
It can be proved that $o$ is a pole for $M$ if and only if the distance function $r(x) = \mathrm{dist}(x,o)$ is smooth outside of $o$, see \cite{petersen} and the references therein. The radial sectional curvature $K_\rad$ is, by definition, the sectional curvature restricted to planes containing $\nabla r$. As a particular case of Theorem \ref{teo_SL_necessary} below, we obtain

\begin{theorem}[\textbf{Necessity of \eqref{KO_infinity_intro}}]\label{teo_SL_necessary_intro}
Let $M^m$ be a complete Riemannian manifold with a pole $o \in M$, and assume that
\begin{equation}\label{condi_kdasopra_intr}
K_\rad(x) \le - \frac{1}{4r(x)^2} \qquad \text{for } \, x \in M \backslash \{o\}.
\end{equation}
Let $\varphi, b, f, l$ satisfy \eqref{assum_SMP_intro}, \eqref{assum_KO_nonl1}, \eqref{assumptions_bfl} and 
$$
f(0)=0, \qquad f>0 \ \text{ and $C$-increasing on } \, \R^+.
$$
Then, \eqref{KO_infinity_intro} is necessary for the validity of $\slio$.
\end{theorem} 

\begin{remark}
\emph{Note that, with the notation $K_\rad(x) \le G(r(x))$ for some $G \in C(\R^+_0)$, we mean that the sectional curvature satisfies
\begin{equation}
K(X \wedge \nabla r)(x) \le - G\big(r(x)\big)
\end{equation}
for each $x \in M\backslash \{o\}$ and $X \perp \nabla r(x)$, $|X|=1$, where $X \wedge \nabla r$ is the $2$-plane spanned by $X$ and $\nabla r$ and $K$ is the sectional curvature.
}
\end{remark}

Inequality \eqref{condi_kdasopra_intr} is a mild requirement just needed to ensure that the model to be compared to $M$ is complete and the volume of its geodesic spheres increases. It allows to apply Theorem \ref{teo_SL_necessary_intro} to all \emph{Cartan-Hadamard manifolds}, namely complete, simply-connected manifolds with non-positive sectional curvature, a class that includes both the Euclidean and the hyperbolic space. This might suggest that \eqref{condi_kdasopra_intr} be just a technical assumption  (although not easy to remove) and thus, loosely speaking, that geometry does not to affect implication $\slio$ $\Rightarrow$ \eqref{KO_infinity_intro}.\par
On the contrary, the sufficiency of \eqref{KO_infinity_intro} heavily depends on the validity of maximum principles at infinity. To investigate the interplay, it is worth to consider $\slio$ as the combination of two properties:
\begin{itemize}
\item[-] an \emph{$L^\infty$-estimate} for non-negative solutions of $(P_\ge)$;
\item[-] property $\lio$ for bounded, non-negative solutions of $(P_\ge)$.
\end{itemize}
Note that \eqref{KO} plays a role in the first property, while, by Proposition \eqref{prop_equivalence}, the second property is equivalent to $\wmp$ provided that $f(0)=0$ and $f>0$ on $\R^+$. As a first result, Theorem \ref{teo_SMPeKO} below relates directly $\smp$ to $\slio$, by showing that for some classes of operators, notably including the $p$-Laplacian with constant $b$, and for general $f$ with $f > 0$ on $\R^+$,
\begin{equation}\label{scheme_SL}
\text{\eqref{KO}} \quad + \quad \text{$\smp$} \quad \Longrightarrow \quad \text{$\slio$}.
\end{equation}
However, for more general operators, in particular for non-homogeneous ones, such a simple relation is currently unknown. Nevertheless, for large classes of functions $\varphi,b,f,l$, we can guarantee $\slio$ by coupling \eqref{KO_infinity_intro} with the lower Ricci curvature bounds considered in Theorem \ref{teo_SMP_intro}, the latter being sharp for the validity of $\smp$. This is the content of Theorems \ref{teo_SMP_SL} and \ref{teo_SMP_SL_border} below, dealing respectively with the case $\chi>0$ and $\chi=0$, that should be considered the main results of Subsection \ref{subsec_SL_suffi}. We refer therein for the statements in full generality, and quote the following corollary for the mean curvature operator:

\begin{theorem}\label{teo_SL_ricci_intro}
Let $M^m$ be complete and assume that
\begin{equation}\label{ricciassu_SL_mc}
\Ricc (\nabla r, \nabla r) \ge -(m-1)\kappa^2\big( 1+r^2\big)^{\alpha/2} \qquad \text{on } \, \mathcal{D}_o,
\end{equation}
for some $\kappa \ge 0$, $\alpha \ge -2$ and some origin $o$. Let $b,f,l$ satisfying \eqref{assumptions_bfl} and 
$$
\begin{array}{l}
\disp b(x) \ge C_1\big(1+ r(x)\big)^{-\mu} \qquad \text{on } \, M, \\[0.3cm]
f(0)=0, \quad \text{$f>0 \ \ $ and $C$-increasing on } \, \R^+, \\[0.3cm]
\disp l(t) \ge C_1 \frac{t^{1-\chi}}{\sqrt{1+t^2}} \qquad \text{on } \, \R^+,
\end{array}
$$
for some constants $C, C_1>0$, $\mu \in \R$, $\chi \in (0,1]$ with
$$
\mu \le \chi - \frac{\alpha}{2}.
$$
Then, under the validity of the Keller-Osserman condition 
\begin{equation}\label{KO_mc_intro}
F^{-\frac{1}{\chi+1}} \in L^1(\infty)
\end{equation}
with $F$ as in \eqref{def_Fe_2}, $\slio$ holds for $C^1$ solutions of
\begin{equation}\label{rste}
\diver \left( \frac{\nabla u}{\sqrt{1+|\nabla u|^2}} \right) \ge b(x) f(u) l(|\nabla u|) \qquad \text{on } \, M. 
\end{equation}
\end{theorem}

Suitable counterexamples will show the sharpness of \eqref{KO_mc_intro}, that seems to be new even in the Euclidean and hyperbolic space setting, corresponding, respectively, to $\alpha = -2$ and $\alpha=0$. 
\begin{remark}
\emph{If $\chi=0$ and no Keller-Osserman condition is assumed, a Liouville theorem that well matches with the above result can be found in Theorem \ref{teo_meancurv_intro} below.
}
\end{remark}
When $f(t)$ is a power of $t$, say $f(t) \asymp t^\omega$, \eqref{KO_mc_intro} becomes $\omega> \chi$. In this case, we will prove $\slio$ under a mere volume growth requirement. More precisely, we have the following

\begin{theorem}\label{teo_main_intro}
Assume that the conditions in Theorem \ref{teo_WMP_intro} are satisfied, with the second and third of \eqref{assu_WMP_intro_growth} replaced by
\begin{equation}\label{eq_uppervarphi_intro_WMPeSL}
\varphi(t) \le C_2t^{p-1} \quad \text{for } \, t \in \R^+,
\end{equation}
for some $p>1$, $C>0$. Let $f \in C(\R)$ satisfy
$$
f(t) \ge C_2 t^\omega \qquad \text{for some $C>0$ and each } \, t \ge 1.
$$
If $\omega > \chi$, then any non-constant $u \in \lip_\loc(M)$ solution of $(P_\ge)$ on $M$ is bounded from above and satisfies $f(u^*) \le 0$. In particular, if $f > 0$ on $\R^ +$, $\slio$ holds for $\lip_\loc$ solutions.
\end{theorem}
\begin{remark}
\emph{If $\varphi \in C^1(\R^+)$ and $t \varphi'(t) \asymp \varphi(t)$ as $t \ra \infty$, like for instance in the case of the $p$-Laplacian, a direct check shows that $\omega>\chi$ in the above theorem imply \eqref{KO}, and it is equivalent to it whenever, instead of \eqref{assum_WMP_intro}, $l$ satisfies
$$
l(t) t^\chi \asymp \varphi(t) \qquad \text{as } \, t \ra \infty.
$$
}
\end{remark}

The argument of the proof of Theorem \ref{teo_main_intro}, though close in spirit to that  of Theorem \ref{teo_WMP_intro}, uses a different combination of integral estimates. Nevertheless, unlike \cite{mitpoho, DAmbrMit, farinaserrin1, farinaserrin2, dambrosiomitidieri_2} which treat similar results in $\R^m$, our method has again the advantage to work in settings where the volume growth of geodesic balls is not polynomial. This requires an iterative procedure originally due to \cite{prs_gafa, prsmemoirs}, of independent interest, but the appearance of the function $l$ also requires a careful mixing with techniques in \cite{farinaserrin2}.\par
To describe the range of applicability of Theorem \ref{teo_main_intro}, we consider the capillarity equation
\begin{equation}\label{capillarity}
\diver \left( \frac{\nabla u}{\sqrt{1+|\nabla u|^2}}\right) = \kappa(x) u \qquad \text{on } \, M, 
\end{equation}
modelling a graphical interface in $M \times \R$ whose mean curvature is proportional to the height of the graph via the non-homogeneous coefficient $\kappa(x)$. Then, we have
\begin{theorem}\label{teo_capillarity}
Suppose that $M$ is complete and that 
\begin{equation}\label{ipo_k_capillarity}
\kappa(x) \ge C\big( 1+ r(x)\big)^{-\mu} \qquad \text{on } \, M, 
\end{equation}
for some constants $C >0$ and $\mu < 2$. If there exists $\eps>0$ such that 
\begin{equation}\label{ipo_volumecapillarity}
\liminf_{r \ra \infty} \frac{\log\vol B_r}{r^{2-\eps-\mu}} < \infty, 
\end{equation}
then the only solution of the capillarity equation \eqref{capillarity} on $M$ is $u \equiv 0$.
\end{theorem}
To the best of our knowledge, Theorem \ref{teo_capillarity_2} seems to us to be the first result considering entire solutions of \eqref{capillarity_2} in a manifold setting, in particular allowing that the  volume of geodesic balls grows faster than polynomially. Nevertheless, it is of interest even for Euclidean space, guaranteeing $u \equiv 0$ whenever $\mu<2$. To our knowledge, when $\kappa>0$ is constant the vanishing of $u$ on $\R^m$ solving \eqref{capillarity} was first obtained in \cite{tkachev,nu} with no growth assumptions on $u$ (cf. also \cite[Thm. 8.1.3]{pucciserrin} for $u$ growing polynomially). The methods in \cite{tkachev, nu} are different from one another; in particular, the one in \cite{tkachev} has later been extended in \cite{Serrin_4} to more general inequalities, and Theorem \ref{teo_capillarity} is shown to hold on $\R^m$ but only for $\mu<1$. Recently, in \cite{farinaserrin1} the authors were able to achieve the sharp bound $\mu<2$ for solutions on $\R^m$. In Section \ref{subsub_capillarity}, we will describe in more detail the relationship between our result and the ones in \cite{farinaserrin1}, and we will improve on  \cite{tkachev,nu,pucciserrin,Serrin_4,farinaserrin1} for a class of equations including \eqref{capillarity}. We stress that none of the methods therein easily adapt to manifolds just satisfying \eqref{ipo_volumecapillarity}.

\begin{remark}\label{rem_unexpected}
\emph{Theorem \ref{teo_capillarity} has a curious and unexpected feature: although \eqref{capillarity} does not contain a gradient term, the ``artificial" inclusion of a suitable  $l(|\nabla u|)$ in the right-hand side of \eqref{capillarity} is the key to prove the corollary as a consequence of the Keller-Osserman condition $\omega>\chi$ in Theorem \ref{teo_main_intro} (note that $\omega=1$ for \eqref{capillarity}). This is in striking contrast with previous results for equation 
\begin{equation}\label{senzagradiente}
\diver \left( \frac{\nabla u}{\sqrt{1+|\nabla u|^2}} \right) = b(x) f(u), 
\end{equation}
in a manifold setting: for instance, to obtain the vanishing of $u$ in \eqref{senzagradiente} when $f(t) t \ge C|t|^{\omega+1}$ on $\R$, Theorem 4.8 in \cite{prsmemoirs} needs inequality $\omega>1$, which does not hold for \eqref{capillarity}. Loosely speaking, inserting a suitable gradient term enables us to weaken the requirement in the Keller-Osserman condition up to include the capillarity equation.
}
\end{remark}

Observing that the volume growth conditions in Theorem \ref{teo_main_intro} coincide  with those in Theorem \ref{teo_WMP_intro} for the validity of $\wmp$, one might wonder whether \eqref{scheme_SL} can be weakened to
\begin{equation}\label{scheme_WMPeSL}
\text{\eqref{KO}} \quad + \quad \wmp \quad \Longrightarrow \quad \slio.
\end{equation}

In Example \ref{ex_importante} below, we will show that $\wmp$ is \emph{not} sufficient, and the full strength of $\smp$ is needed. The counterexample is, however, on an incomplete manifold, and suggests the following

\begin{problem}
Prove or disprove: if $M$ is a complete manifold, then the implication \eqref{scheme_WMPeSL} holds at least for a subclass of operators $\Delta_\varphi$ and $b,f,l$.
\end{problem}

\subsection{The compact support principle $\csp$}

Unlike that on $\lio$ and $\slio$, the literature on $\csp$ is not so extensive. The subject initiated with the seminal paper by R. Redheffer \cite{redheffer_csp}, and received a renewed interest in the last 15 years starting from \cite{pucciserrinzou}, see also \cite{haitao, PuGaHuMaSerrin, felmermontenegroquaas} and the monograph \cite{pucciserrin}. However, all of these works consider the problem in the setting of Euclidean space, and to our knowledge just \cite{PuRS, rigolisalvatorivignati, rigolisalvatorivignati_5} analyze the role played by the geometry of the manifold. As we shall see, the link between geometry and $\csp$ does not depend on the validity of a maximum principles at infinity: to explain which geometric conditions are to be expected, we first comment on the following result in \cite[Thm. 1.1]{PuRS}:

\begin{theorem}[\cite{PuRS}]\label{teo_CSP}
Let $M$ be a complete manifold, and let $r$ be the distance from a fixed origin $o$. Assume \eqref{assumptions} and that $\varphi$ is strictly increasing on $\R^+$. Let $f \in C(\R)$ satisfy
\begin{equation}\label{ipo_CSP_original}
\quad f(0)=0, \qquad f>0 \quad \text{and non-decreasing on some } \, [0, \eta_0), \ \ \eta_0>0.
\end{equation}
Then, in order for $\csp$ to hold for $(P_\ge)$ with $b(x) =1$, $l(t) = 1$ it is necessary that
\begin{equation}\label{inteinzero_intro}
\frac{1}{H^{-1}\circ F} \in L^1(0^+),
\end{equation}
where $F$ and $H$ are defined in \eqref{def_Fe_intro} and \eqref{def_Hs}. Viceversa, \eqref{inteinzero_intro} is also
sufficient for $\csp$ provided that
\begin{equation}\label{M1}
\inf_M \Delta r > -\infty
\end{equation}
holds in the weak sense.
\end{theorem}

As a matter of fact, as we prove in Proposition \ref{prop_pole} in Appendix A,  \eqref{M1} forces the origin $o$ to be a pole for $M$, in particular $r$ shall be smooth outside of $M \backslash \{o\}$. With the aid of the Hessian comparison theorem, \eqref{M1} holds for a large class of manifolds including Cartan-Hadamard ones, such as the Euclidean and hyperbolic spaces. On the other hand, the topological restriction imposed  by the existence of a pole is binding, and it would be desirable to remove it. However,  already in \cite{PuRS} the authors realized that a condition like \eqref{M1} or some other extra assumption needs necessarily to be included. Their example, reported below, is illustrative.

\begin{example}\label{ex_CSP_intro}
\emph{Consider the radially symmetric model
$$
M = (\R^m, \di s_g^2), \qquad \di s_g^2 = \di r^2 + g(r)^2 \metricN_1,
$$
where $\metricN_1$ is the standard round metric on the unit sphere and $0<g \in C^\infty(\R^+)$ satisfy $g(r) = r$ for $r \in [0,1]$ and $g(r) = \exp\{-r^\alpha\}$ for $r \ge 2$, for some $\alpha >2$. Clearly $M$ is a manifold  with pole $o \in \R^m$.  Then, for each $\omega \in (0,1)$ the function
$$
u(r) = r^{-\beta}, \qquad \beta \in \left( 0, \frac{\alpha-2}{1-\omega}\right]
$$
solves $\Delta u \ge C u^\omega$ on the end $\Omega = M \backslash B_R$ for $R$ large enough. Although \eqref{inteinzero_intro} holds, $u$ clearly contradicts $\csp$. Note that in this case $\Delta r = -(m-1)\alpha r^{\alpha-1}$, hence \eqref{M1} is violated.
}
\end{example}
A more elaborated example along these lines will be given in Section 7 below. The construction shows that, in sharp contrast with $\slio$ and $\lio$, what matters in this case is that $M$ should not possess ends shrinking too rapidly at infinity. This is the content of our first contribution to $\csp$. We begin defining a weaker notion of a pole, to allow a nontrivial topology of $M$.

\begin{definition}
Let $\mathcal{O} \subset M$ be a relatively compact, open set with smooth boundary in $M$. We say that $\mathcal{O}$ is \emph{a pole of $M$} if the normal exponential map $\exp^\perp : T\mathcal{O}^\perp \ra M \backslash \mathcal{O}$ realizes a diffeomorphism.
\end{definition}
Here, $T\mathcal{O}^\perp$ is the subset of the normal bundle of $\partial \mathcal{O}$ consisting of vectors pointing outward from $\mathcal{O}$. The case of a point $o$ being a pole can easily be recovered by choosing $\mathcal{O} = B_\eps(o)$ for $\eps$ small enough. Let $r$ be the distance function from $\cal O$, which is therefore smooth on $M \backslash \mathcal{O}$. Denote with $\II_{-\nabla r}$ the second fundamental form of $\partial \cal O$ with respect to the inward unit normal $-\nabla r$, and let $B_R(\mathcal{O}) = \{ x \in M : 0 < r(x) < R\}$.\\
We assume that the radial sectional curvature $K_\rad$ satisfies
\begin{equation}\label{ipo_Krad_intro}
K_\rad \le - \kappa^2(1+r)^\alpha \qquad \text{on } \, M \backslash \cal O,
\end{equation}
for some $\kappa \ge 0$, $\alpha \ge -2$.\par
Beyond the standard requirements \eqref{assumptions} and \eqref{assumptions_bfl} we assume \eqref{assum_secODE_altreL_intro}, in order for $K$ to be defined, and the following condition corresponding to \eqref{ipo_CSP_original}:
\begin{equation}\label{Cincre_fl_CSP_intro}
\left\{\begin{array}{l}
f \ \text{ is positive and $C$-increasing on $(0,\eta_0)$, for some $\eta_0 >0$.} \\[0.2cm]
l \ \text{ is $C$-increasing and locally Lipschitz on $(0,\xi_0)$, for some $\xi_0 >0$.} \\[0.2cm]
f(0)l(0)=0.
\end{array} \right.
\end{equation}
Set $F$ as in \eqref{def_Fe_intro}. We first address the necessity of the Keller-Osserman condition
\begin{equation}\label{KO_zero_intro_2}\tag{$\mathrm{KO}_0$}
\frac{1}{K^{-1} \circ F} \in L^1(0^+).
\end{equation}
In analogy with Theorem \ref{teo_CSP}, we see that there is no geometric obstruction, at least on manifolds with a pole.
\begin{theorem}[\textbf{Necessity of \eqref{KO_zero}}]\label{teo_necessity_CSP_intro}
Let $M$ be a manifold with a pole $\mathcal{O}$. Assume \eqref{assumptions}, \eqref{assumptions_bfl}, \eqref{assum_secODE_altreL_intro} and \eqref{Cincre_fl_CSP_intro}. Then, \eqref{KO_zero} is necessary for the validity of $\csp$ for $(P_\ge)$.
\end{theorem}

The proof relies on the construction of a suitable radial solution of a Dirichlet problem at infinity for singular ODEs, of independent interest. As for the sufficiency part, geometry enters into play, and the statement is considerably more elaborated. We state the following corollary of our main result, Theorem \ref{teo_CSP_nuovo} in Section \ref{sec_CSP}, that considers $\varphi,l$ of polynomial type near $t=0$.

\begin{theorem}\label{cor_csp_specialized_intro}
Let $(M, \metric)$ be a manifold with a pole $\mathcal{O}$ such that \eqref{ipo_Krad_intro} holds for some $\alpha \ge -2$, $\kappa \ge 0$. Suppose
\begin{equation}\label{arra_CSP_assu_spec_intro}
\II_{-\nabla r} \ge - C_{\alpha,\kappa} \metric \qquad \text{on } \, T\partial \cal O,
\end{equation}
with
\begin{equation}
\begin{array}{l}
C_{\alpha, \kappa} = \left\{ \begin{array}{ll}
\kappa & \quad \text{if } \, \alpha \ge 0 \ \text{or } \ \kappa = 0, \\[0.2cm]
\left[\frac{\alpha + \sqrt{\alpha^2+ 16\kappa^2}}{4}\right] & \quad \text{otherwise}.
\end{array}\right.
\end{array}
\end{equation}
Consider $\varphi, b, f,l$ satisfying \eqref{assumptions}, \eqref{assumptions_bfl}, \eqref{assum_secODE_altreL_intro} and \eqref{Cincre_fl_CSP_intro}. Fix $\chi, \mu \in \R$ with
\begin{equation}\label{condi_chimuomega_CSP_intro}
\chi>0, \qquad  \mu \le \chi - \frac{\alpha}{2}
\end{equation}
and assume that
\begin{equation}\label{assu_poli_csp_intro}
\begin{array}{ll}
l(t) \asymp t^{1-\chi}\varphi'(t) & \quad \text{for } \, t \in (0,1), \\[0.2cm]
b(x) \ge C_1\big( 1+r(x)\big)^{-\mu} & \quad \text{for } \, r(x) \ge r_0,
\end{array}
\end{equation}
for some constant $C_1>0$. If there exists a constant $c_F \ge 1$ such that
\begin{equation}\label{condi_F_intro}
F(t)^{\frac{\chi}{\chi+1}} \le c_F f(t) \qquad \text{for each } \, t \in (0, \eta_0), 
\end{equation}
then, 
\begin{equation}\label{laequi_KO_intro}
\text{$\csp$ holds for $(P_\ge)$} \qquad \Longleftrightarrow \qquad \eqref{KO_zero}.
\end{equation}
\end{theorem}
\begin{remark}
\emph{By \eqref{assu_poli_csp_intro}, \eqref{KO_zero} is equivalent to 
$$
F^{-\frac{1}{\chi+1}} \in L^1(0^+).
$$
Note also that the $C$-increasing property of $l$ implies that $t^{1-\chi}l(t)$ is $C$-increasing too, possibly with a different $C$, and this forces a bound on the vanishing of $l$ near $t=0$. For instance, if 
$$
\varphi'(t) \asymp t^{p-2}, \qquad f(t) \asymp t^\omega \qquad \text{as } \, t \ra 0, 
$$
for some $p>1$, $\omega >0$, then \eqref{assu_poli_csp_intro} holds with $l$ $C$-increasing if and only if $\chi \le p-1$, \eqref{condi_F_intro} is satisfied whenever $\omega \le \chi$, and \eqref{KO_zero} is equivalent to $\omega<\chi$.
}
\end{remark}

The bound \eqref{arra_CSP_assu_spec_intro} means, roughly speaking, that $\partial \mathcal{O}$ should not be too concave in the inward direction. In particular, if $\kappa = 0$, \eqref{arra_CSP_assu_spec_intro} requires $\mathcal{O}$ to have a convex boundary. By choosing $\mathcal{O} = B_\eps(o)$, the theorem applies to Cartan-Hadamard manifolds.

\begin{example}
\emph{Another relevant example to which Theorem \ref{cor_csp_specialized_intro} applies is that of hyperbolic manifolds with finite volume. It is known by the thick-thin decomposition (see Theorem D.3.3 and Proposition D.3.12 in  \cite{benedettipetronio}) that a manifold $M^m$ with sectional curvature $-1$ and finite volume decomposes as the disjoint union $\mathcal{O} \cup \Omega_1 \cup \ldots \cup \Omega_s$, where $\mathcal{O}$ is a smooth, relatively compact open set and, for each $j$, $\Omega_j$ is a non-compact cusp end isometric to the warped product $\R^+_0 \times_{e^{-r}} N^{m-1}$, for some compact flat manifold $(N, g_N)$, with metric $\di r^2 + e^{-2r}g_N$. Therefore, $\mathcal{O}$ is a pole of $M$, $r$ is the distance from $\mathcal{O}$ and a direct computation gives
$$
\II_{-\nabla r} = - \metric \qquad \text{on } \, T\partial \cal O,
$$
precisely the borderline case in \eqref{arra_CSP_assu_spec_intro}.
}
\end{example}

\begin{remark}
\emph{Observe that the bound $\mu \le \chi - \frac{\alpha}{2}$ in \eqref{condi_chimuomega_CSP_intro} is the same as that in \eqref{condi_volumericci_intro} and \eqref{condi_volumeWMP_intro} for the validity of, respectively, $\smp$ and $\wmp$. Example \ref{ex_countCSP} below shows that it cannot be weakened (see the range \eqref{controesempio_csp_belle} of the parameter there).
}
\end{remark}

Because of the presence of nontrivial $b,l$, the proof of Theorem \ref{cor_csp_specialized_intro} is technically considerably more demanding than that of Theorem \ref{teo_CSP}, and calls for a few extra-assumptions guaranteeing certain mild ``homogeneity properties", which accounts for conditions \eqref{assu_poli_csp_intro}. However, the underlying principle is the same and relies on the explicit construction of a radial, compactly supported, $C^1$ solution of
\begin{equation}\label{supersol_CSP_intro}
\left\{ \begin{array}{l}
\Delta_\varphi w \le b(x) f(w) l(|\nabla w|) \qquad \text{on } \, M \backslash B_R(\mathcal{O}), \\[0.2cm]
w \ge 0 \qquad \text{on } \, M \backslash B_R(\mathcal{O}), \\[0.2cm]
w \equiv 0 \qquad \text{on } \, M \backslash B_{R_1}(\mathcal{O}),
\end{array} \right.
\end{equation}
for some $R_1>R$, via a direct use of \eqref{KO_zero_intro}. We provide two variants of the construction, that work under a mildly different set of assumptions: one of them is quite involved and closely related to that in \cite{rigolisalvatorivignati} (which seems to be the only reference investigating $\csp$ with a gradient term $l$), while the other one is new and considerably simpler.\par
Athough sharp, Theorem \ref{cor_csp_specialized_intro} still requires the presence of a pole in order to apply the Laplacian comparison theorem from below and deduce, from the combination of \eqref{ipo_Krad_intro} and \eqref{arra_CSP_assu_spec_intro}, the lower bound
$$
\Delta r \ge - C r^{\frac{\alpha}{2}} \qquad \text{for } \, r \ge 1,
$$
and some constant $C>0$, which is the weighted version of \eqref{M1}. For the relevant case of the $p$-Laplace operator, we introduce a different radialization method, that we believe to be of independent interest. It is based on smooth functions replacing the distance from $\mathcal{O}$, and called for this reason \emph{fake distances}. The method does not require the existence of a pole. Suppose that $M$ is a complete manifold satisfying
\begin{equation}\label{ipo_ricci_CSP_intro}
\Ricc \ge -(m-1)\kappa^2 \metric \qquad \text{on } \, M,
\end{equation}
for some $\kappa \ge 0$, and let $v_\kappa(r)$ be the volume of a geodesic sphere of radius $\kappa$ in the space form of constant sectional curvature $-\kappa^2$ (i.e. $\R^m$ for $\kappa=0$ or the hyperbolic space $\HH^m_\kappa$ of curvature $-\kappa^2$ for $\kappa>0$). We restrict here to the case
$$
p \in (1,m],
$$
the complementary case $p>m$ being slightly different and discussed in Section \ref{sec_exhaustion}. Assume
$$
v_\kappa^{-\frac{1}{p-1}} \in L^1(\infty)
$$
(which always holds if $\kappa>0$, while it requires $p< m$ if $\kappa=0$). Assume that $\Delta_p$ is not $p$-parabolic on $M$, equivalently, that for each fixed origin $o$ there exists a minimal positive Green kernel $\gr_p(x,o)$ for $\Delta_p$ with pole at $o$. Define the fake distance $\varrho : M \ra \R$ implicitly via the equation
$$
\gr_p(x,o) = \int_{\varrho(x)}^{\infty} \frac{\di s}{v_\kappa(s)^{\frac{1}{p-1}}}.
$$
In the literature, the fake distance modelled on the Green kernel of the Laplace Beltrami operator has been used with great success to study the geometry in the large of manifolds with non-negative Ricci curvature, see for instance the works of Cheeger-Colding \cite{cheegercolding}, Colding-Minicozzi \cite{coldingmini} and Colding \cite{colding}, together with the references therein. On the contrary, in a quasilinear setting and especially for the purposes of the present paper, its use seems to be new. In Proposition \ref{prop_gradienbounded} we prove that, if $\kappa>0$, then $|\nabla \varrho| \le C$ outside of a neighbourhood of $o$, and by direct computation
$$
\Delta_p \varrho = \frac{v_\kappa'(\varrho)}{v_\kappa(\varrho)}|\nabla \varrho|^p.
$$
Consequently, for each diffeomorphism $\psi : \R \ra \R$ we obtain
\begin{equation}\label{radializiamo}
\Delta_p \psi(\varrho) = v_\kappa^{-1}\big( v_\kappa |\psi'|^{p-2}\psi' \big)' |\nabla \varrho|^p.
\end{equation}
Since $v_\kappa^{-1}\big( v_\kappa |\psi'|^{p-2}\psi' \big)'$ is the expression of the $p$-Laplacian of a radial function in the model of curvature $-\kappa^2$, \eqref{radializiamo} enables us to construct solutions of \eqref{supersol_CSP_intro} by radializing with respect to $\varrho$ instead of $r$. However, to be able to conclude the validity of this property $\csp$ we shall need $\varrho$ to be an exhaustion function, equivalently, we need the vanishing of $\gr_p$ as $x$ diverges. Sufficient conditions for the validity  of this property follow, for instance, from the validity of Sobolev inequalities on $M$, and  will be investigated in Section \ref{sec_exhaustion}. The striking advantage with respect to $r$ is that $\varrho$ is smooth and $p$-subharmonic, hence a bound of the type in \eqref{M1} is automatically satisfied. We apply these ideas to obtain a characterization of $\csp$ on a class of manifolds satisfying \eqref{ipo_ricci_CSP_intro}, which is a particular case of Theorem \ref{teo_CSP_plapla} and is somehow complementary to Theorem \ref{cor_csp_specialized_intro}. Here is the statement of the result.
\begin{theorem}\label{teo_CSP_plapla_intro}
Let $(M, \metric)$ be a complete $m$-dimensional Riemannian manifold with
\begin{equation}\label{iporicci_csp_intro}
\Ricc \ge -(m-1)\kappa^2 \metric \qquad \text{on } \, M,
\end{equation}
and assume that, for some $p \in (1,2]$ and $\nu > p$, the Sobolev inequality
\begin{equation}\label{sobolev_intro}
\left( \int |\psi|^{\frac{\nu p}{\nu-p}} \right)^{\frac{\nu-p}{\nu}} \le S_{p,\nu} \int |\nabla \psi|^p
\end{equation}
holds for each $\psi \in \lip_c(M)$. Fix $\chi \in (0, p-1]$. Let $f \in C(\R)$ satisfy
$$
f(0)=0, \qquad f> 0 \ \ \text{ on } \, (0, \eta_0), \qquad \text{$f$ is $C$-increasing on } \, (0, \eta_0),
$$
for some $\eta_0>0$. Define $F$ as in \eqref{def_Fe_intro}, and furthermore suppose
\begin{equation}\label{mild_technical}
F(t)^{\frac{\chi}{\chi+1}} \le c_F f(t) \quad \text{on } \, (0,\eta_0),
\end{equation}
for some $c_F >0$. Let $\Omega$ an end of $M$. Then, $\csp$ holds for solutions of \begin{equation}\label{CSP_plaplacian_intro}
\left\{ \begin{array}{l}
\Delta_p u \ge f(u)|\nabla u|^{p-1-\chi} \qquad \text{on } \, \Omega \,  \\[0.2cm]
\disp u \ge 0, \qquad \lim_{x \in \Omega, \, x \ra \infty} u(x) = 0
\end{array}\right.
\end{equation}
if and only if
\begin{equation}\label{inteinzero_plap}
F^{-\frac{1}{\chi+1}} \in L^1(0^+).
\end{equation}
\end{theorem}

\begin{example}
\emph{Conditions for the validity of \eqref{sobolev_intro} will be given in Section \ref{sec_exhaustion}, see Examples \ref{ex_minimal}, \ref{prop_hebey} and \ref{ex_roughisometric}. In particular, we stress that \eqref{sobolev_intro} holds with $\nu = m$ if $M$ is (complete and) minimally immersed in a Cartan-Hadamard ambient space $N$. By Gauss equations, in this setting Theorem \ref{teo_CSP_plapla_intro} can be applied provided that the minimal immersion $M \ra N$ has bounded second fundamental tensor in order to satisfy \eqref{iporicci_csp_intro}.
}
\end{example}

We stress that Theorem \ref{teo_CSP_plapla} below allows  a dependence on the function $b$, which is required to be bounded from below in terms of a decaying function of $\varrho$. However, it is difficult to bound $\varrho$ from \emph{below} with the more manageable $r$. A particular case when we have effective estimates is for manifolds with $\Ricc \ge 0$ in the semilinear case $p=2$ (although this last restriction seems to be merely technical), where there is also no need for the Sobolev inequality \eqref{sobolev_intro}. The reader is referred to Theorem \ref{teo_CSP_plapla_riccimagzero} below for the precise statement.\par
Moreover, observe that in Theorem \ref{teo_CSP_plapla_intro} we require $p \in (1,2]$, that is, that  the $p$-Laplacian be non-degenerate. This condition is technical, and is necessary to apply the comparison theorems that are currently available in the literature. For this reason, we feel interesting to investigate the following

\begin{problem}
Prove or disprove the validity of Theorem \ref{teo_CSP_plapla_intro} (or, more generally, Theorems \ref{teo_CSP_plapla} and \ref{teo_CSP_plapla_riccimagzero} below) in the full range $p \in (1,\infty)$.
\end{problem}

When $l$ is constant a further fake distance, recently constructed in \cite{bianchisetti}, allows to improve on Theorems \ref{teo_CSP_plapla_intro} and \ref{teo_CSP_plapla_riccimagzero} when $u$ solves 
$$
\Delta u \ge (1+r)^{-\mu}f(u),
$$
by reducing the geometric conditions to the only
$$
\Ricc \ge -(m-1) \kappa^2\big( 1+r^2\big)^{\alpha/2} \metric,
$$
for some $\kappa>0$ and $\alpha \in (-2,2]$. The threshold $\alpha=2$ is sharp and related to a probabilistic requirement called the Feller property, cf. \cite{azencott, pigolasetti_feller}. More details on this issue are given in Subsection \ref{sec_feller}.

\subsection{More general inequalities}
The techniques discussed in the present paper are also effective to investigate more general inequalities of the type
\begin{equation}\label{P_lapiugeneral_intro}
\Delta_\varphi u \ge b(x)f(u) l(|\nabla u|) - \bar b(x) \bar f(u) \bar l(|\nabla u|).
\end{equation}
This class includes relevant geometric examples, such as \eqref{prescribed_geodesic} when the graph is neither minimal nor a MCF soliton. Another important prototype inequality is
$$
\Delta_p u \ge f(u) - c|\nabla u|^q, \qquad \text{with $q>0$, $c>0$,}
$$
that in the semilinear case appears (with the equality sign) as value functions of stochastic control problems (cf. \cite{lasrylions, radulescu}). Existence and nonexistence of entire solutions have been investigated in a series of papers, notably
\begin{itemize}
\item[-] \cite{lairwood, ghernicrad}, where the authors consider solutions of 
\begin{equation}\label{semilin}
\Delta u \pm \bar b(x)|\nabla u|^q = b(x)u^\gamma \qquad \text{on } \, \R^m,
\end{equation}
for $q,\gamma \in \R^+$ and non-negative $b, \bar b$ that are allowed to vanish in a controlled (yet very general) way.
\item[-] \cite{FPREnt, maririgolisetti}, that concern quasilinear analogues of \eqref{semilin} when the driving operator is, respectively, a weighted $p$-Laplacian and a general $\varphi$-Laplacian in a manifold setting. Although related, their techniques differ from those in \cite{lairwood, ghernicrad}.
\item[-] \cite{felmerquaassirakov}, that considers a fully nonlinear version of \eqref{semilin} of the type
$$
\mathcal{M}[u] \ge f(u) \pm g(|\nabla u|),
$$
where the driving operator $\mathcal{M}$ is uniformly elliptic. Their results are relevant also in the semilinear setting.
\end{itemize}
Apart from some special cases, the existence-nonexistence problem for such inequalities is far from being completely understood, and many interesting questions are still unanswered even in the Euclidean setting. For instance, none of the references considers the interplay of the weights and nonlinearities in the generality of \eqref{P_lapiugeneral_intro}. To keep the paper at a reasonable length, if still possible, we decided not to make attempts to deal with \eqref{P_lapiugeneral_intro} and we leave it to future research.

\section{Radialization and fake distances}\label{sec_exhaustion}
The proof of some of our main results, for instance the $\csp$, relies on the construction of a suitable radial solution of $(P_\ge)$ or $(P_\le)$ to be compared with a  given one. For convenience, hereafter we extend $\varphi$ to an odd function on all of $\R$ by setting
\begin{equation}\label{varphi}
\varphi(s) = - \varphi(-s) \qquad \text{for each } \, s<0.
\end{equation}
Suppose that $w \in C^1(\R^+_0)$ satisfies $\varphi(w') \in C^1(\R^+_0)$. If $u(x) = w(r(x))$, where $r(x)$ is the distance from a fixed origin $\mathcal{O}$ (a point, or a relatively compact open set with smooth boundary), then
$$
\Delta_\varphi u = \diver\left( \frac{\varphi(|w'|)}{|w'|} w' \nabla r\right) = \diver\big( \varphi(w') \nabla r\big) = \big(\varphi(w')\big)' + \varphi(w')\Delta r,
$$
therefore $u$ solves, say, $(P_\ge)$ if and only if
\begin{equation}\label{sol_PDE_22}
\big(\varphi(w')\big)' + \varphi(w')\Delta r \, \ge \, b(x)f(w)l(|w'|).
\end{equation}
Take $0< g \in C^2(\R^+_0)$ and a model manifold $M_g = \R^+ \times \Sph^{m-1}$ with polar  coordinates $(r,\theta)$ and metric
$$
\di s_g^2 = \di r^2 + g(r)^2 \metricN_1,
$$
where $\metricN_1$ is the standard round metric on the unit sphere. Let
$$
v_g(s) = \vol(\Sph^{m-1})g(s)^{m-1}
$$
be the volume growth of the set $\{r = s\}$. By comparison theory for $\Delta r$ (see Appendix A), given a radial bound for the function $b$ of the type $b(x) \ge \beta(r(x))$, for some positive $\beta \in C(\R_0^+)$, to find solutions of \eqref{sol_PDE_22} one is first lead to solve
\begin{equation}\label{sol_ODE_22_prima}
\big[v_g\varphi(w')\big]' = v_g \beta f(w) l(|w'|)
\end{equation}
on an interval of $\R^+$. Furthermore, a solution of \eqref{sol_ODE_22_prima} gives rise to a solution of \eqref{sol_PDE_22} provided the sign of $w'$ matches appropriately with the inequalities coming from the comparison theorems for $\Delta r$, and with the sign of $f(w)$. Therefore, the monotonicity of $w$ becomes relevant. We will devote the next section to the study of the ODE \eqref{sol_ODE_22_prima}.\par
The investigation of the compact support principle along these lines needs the use of comparison theorems from below, that requires $r$ to be smooth, equivalently, $\mathcal{O}$ to be a pole. As outlined in the Introduction, in this section we develop a different radialization procedure that uses a ``fake distance" $\varrho$ modelled on the operator $\Delta_\varphi$, in the particular case of the $p$-Laplacian.



Let $p \in (1,\infty)$, and suppose that $\Delta_p$ is non-parabolic on $M^m$, that is, that there exists a non-constant positive solution $w \in W^{1,p}_\loc(M)$ of $\Delta_p w \le 0$. From \cite{HKM, troyanov,troyanov2,pigolasettitroyanov}, we know this to be equivalent to the fact that the $p$-capacity of some (equivalently, each) compact set $K \subset M$, defined as
$$
\capac_p(K) = \inf \left\{ \int_M |\nabla \psi|^p \ \ : \ \ \psi \in \lip_c(M), \ \psi\ge 1 \ \text{ on } \, K\right\},
$$
is positive. Furthermore, by \cite{holopainen, holopainen3} the non-parabolicity is also equivalent to the existence for each fixed $o \in M$ of a positive Green kernel $\gr_p(x,o)$, that is, a distributional solution of $\Delta_p \gr_p(\cdot, o) = -\delta_o$. In other words, $\gr_p(x,o)$ satisfies
$$
\int_M |\nabla_x \gr_p(x,o)|^{p-2} \langle \nabla_x \gr_p(x,o), \nabla \psi(x) \rangle \di x = \psi(o) \qquad \forall \,  u \in C^\infty_c(M).
$$
The kernel $\gr_p$ was constructed in \cite{holopainen, holopainen3} starting with a smooth increasing exhaustion $\{\Omega_j\}$ of $M$, that is, a family of smooth open sets such that
$$
\disp o \in \Omega_j \Subset \Omega_{j+1} \Subset M \quad \text{for each } \, j \ge 1, \qquad \bigcup_{j=1}^{\infty} \Omega_j = M,
$$
and related Green kernels $\gr_{p,j}(\cdot,o)$ with Dirichlet boundary conditions on $\partial \Omega_j$.
The existence of $\gr_{p,j}$ was shown in \cite[Thm. 3.19]{holopainen} for $p \in (1,m]$, and\footnote{In \cite{holopainen3}, p. 656 the author observes that $p \le m$ in \cite{holopainen} has to be assumed because of Definition 3.9 therein, where $\gr_{p,j}$ is required to diverge as $x \ra o$.} in \cite{holopainen3} for $p>m$.

\begin{remark}\label{rem_localgreen}
\emph{The local behavior of a positive solution $\bar \gr_p$ of $\Delta_p \bar \gr_p = -\delta_o$ on a domain $\Omega \subset M$ has been investigated in \cite[Thm 12]{Serrin_1} for $p \le m$:
\begin{equation}\label{boundserrin}
\bar \gr_p(x,o) \asymp \int_{r(x)}^1 \frac{\di s}{[\omega_{m-1}s^{m-1}]^{\frac{1}{p-1}}} \sim \left\{ \begin{array}{ll} \omega_{m-1}^{-\frac{1}{p-1}}\left(\frac{p-1}{m-p}\right) r(x)^{ -\frac{m-p}{p-1}} & \qquad \text{if } \, p<m, \\[0.3cm]
\omega_{m-1}^{-\frac{1}{m-1}} |\log r| & \qquad \text{if } \, p=m
\end{array}\right.
\end{equation}
as $r(x)= \mathrm{dist}(x,o) \ra 0$. For the Euclidean space, in \cite{kichenveron, veron} the authors proved a stronger estimate with the asymptotic $\sim$ in place of $\asymp$ in \eqref{boundserrin}, and their argument is extended to a manifold setting in the forthcoming \cite{maririgolisetti_mono}. On the other hand, when $p>m$, $\bar \gr_p$ admits a continuous extension to $x=o$ with a positive, finite value, since in this case the point $o$ has positive capacity (see \cite[Thm. 6.33]{HKM} and \cite{holopainen3}). By the maximum principle, $\bar \gr_p(o,o) = \|\gr_p(\cdot,o)\|_\infty$.
}
\end{remark}
The uniqueness of $\gr_{p,j}$ was shown in \cite[Thm. 3.22]{holopainen} for $p=m$, while if $p> m$ it is a consequence of the standard comparison for $W^{1,p}$ solutions, since in this range of $p$ it is known that $\gr_{p,j} \in W^{1,p}(\Omega_j)$ (see also the discussion at p. 656 of \cite{holopainen3}). If $p \in (1,m]$, one can easily deduce the uniqueness of $\gr_{p,j}$ again by comparison, taking into account the refinements of the asymptotic behaviour near $o$ described in Remark \ref{rem_localgreen}.\par
By \cite[Thm. 3.25]{holopainen}, the sequence $\{\gr_{p,j}\}$ can be arranged to be increasing\footnote{As for the uniqueness of $\gr_{p,j}$, the increasiness of the sequence $\{\gr_{p,j}\}$ easily follows by comparison in view of the refinements of the asymptotic behaviour in Remark \ref{rem_localgreen}.}, and by \cite[Thm. 3.27]{holopainen} it converges to a finite limit if and only if $\Delta_p$ is non-parabolic on $M$.\par

Fix a model manifold $M_g$ such that
$$
g \in C^2(\R^+_0), \quad g>0 \quad \text{on } \, \R^+, \quad g(0)=0, \quad g'(0)=1,
$$
and assume that $\Delta_p$ is non-parabolic on $M_g$; it is well-known that this is equivalent to require
\begin{equation}\label{mod_nonparab}
v_g^{- \frac{1}{p-1}}\in L^1(\infty),
\end{equation}
and in fact
\begin{equation}\label{grmodel}
\grg_p(r) = \int_r^{\infty} \frac{\di s}{v_g(s)^{\frac{1}{p-1}}},\qquad r\in\mathbb R^+,
\end{equation}
when the integral exists, is the minimal positive Green kernel of~$\Delta_p$ on $M_g$ with pole at $\{r=0\}$. Fix an origin $o \in M$, and define implicitly $\varrho(x)$ on $M \backslash \{o\}$ as follows:

\begin{itemize}
\item[-] If $p \in (1,m]$ set
\begin{equation}\label{def_bxy}
\gr_p(x,o) = \gr^{(g)}_p \big(\varrho(x)\big) \qquad\mbox{on } \, M \backslash \{o\},
\end{equation}
and note that, because of \eqref{boundserrin} and $\gr^{(g)}_p(0^+) = +\infty$, $\varrho$ is well defined and can be extended continuously on $M$ by setting $\varrho(o)=0$.
\item[-] If $p >m$ set
\begin{equation}\label{def_bxy_pmagm}
L \gr_p(x,o) = \gr^{(g)}_p \big(\varrho(x)\big) \qquad\mbox{on } \, M \backslash \{o\},
\end{equation}
with
$$
L = \frac{\gr^{(g)}_p(0)}{\lim_{x \ra o} \gr_p(x,o)} >0.
$$
In this way, by the maximum principle $\varrho(x)$ is still well defined on the whole of $M \backslash \{o\}$, positive therein and has a continuous extension to $M$ with $\varrho(o)=0$.
\end{itemize}

Up to replacing $\gr_p$ with $L \gr_p$ when $p>m$, we can consider $\gr_p$ as being defined by \eqref{def_bxy} for each $p \in (1,\infty)$. Indeed, in what follows we do not need to use the property $\Delta_p \gr_p = -\delta_o$, but only  that $\Delta_p \gr_p =0$ outside of $o$.\par
If $M = M_g$ and $o$ is the reference point $\{r=0\}$ of $M_g$, $\varrho=r$ is the distance function on $M$ from $o$. In general, the relation between $\varrho$ and $r$ depends on that between $M$ and $M_g$, which we now investigate (see also \cite{maririgolisetti_mono}). By elliptic regularity, $\varrho \in C^{1,\alpha}_\loc(M \backslash \{o\})$ since so is $\gr_p$, see \cite{tolksdorf}. Differentiating \eqref{def_bxy} we obtain
\begin{equation}\label{gradb}
|\nabla \varrho| = v_g(\varrho)^{\frac{1}{p-1}} |\nabla \gr_p| = \left[v_g^{\frac{1}{p-1}} \int_{\varrho}^{\infty}\frac{\di s}{v_g(s)^{\frac{1}{p-1}}}\right] |\nabla \log \gr_p|.
\end{equation}
According to \cite{bmr2}, we define the critical curve $\chi_g$ for all $t\in\mathbb R^+$ by setting
\begin{equation}\label{beaz}
\chi_g(t) = \left(\frac{p-1}{p}\right)^p \left[v_g(t)^{\frac{1}{p-1}} \int_{t}^{\infty}\frac{\di s}{v_g(s)^{\frac{1}{p-1}}}\right]^{-p} = \left[\left(-\frac{p-1}{p}\log \int_{t}^{\infty} \frac{\di s}{v_g(s)^{\frac{1}{p-1}}}\right)'\right]^{p},
\end{equation}
and in this way \eqref{gradb} can be rewritten in the form
\begin{equation}\label{gradientb}
|\nabla \varrho| = \frac{p-1}{p}\chi_g(\varrho)^{-1/p} |\nabla \log \gr_p|.
\end{equation}

\begin{remark}
\emph{The critical curve in \eqref{beaz} is related, via comparison theory, to weights in Hardy-type inequalities on $M$, and it thus appears in geometrical problems where stationary Schr\"odinger type operators (linear or nonlinear) are considered. A systematic study with various applications has been given in \cite{bmr2, bmr3, bmr4} for $p=2$, and in \cite{bmr5} for general $p$.
}
\end{remark}

A second differentiation gives
\begin{equation}\label{deri_bG}
\Delta_p \varrho = \frac{v_g'(\varrho)}{v_g(\varrho)}|\nabla \varrho|^p \qquad \text{weakly on } M \backslash \{o\},
\end{equation}
hence, for each $\psi \in C^2(\R)$ with $\psi' \neq 0$ everywhere we obtain
\begin{equation}\label{bella!!!}
\begin{array}{lcl}
\disp \Delta_p \big[\psi(\varrho)\big] &= & \big|\psi'(\varrho)\big|^{p-2}\psi'(\varrho) \left[ (p-1)\psi''(\varrho) + \frac{v_g'(\varrho)}{v_g(\varrho)} \psi'(\varrho) \right] |\nabla \varrho|^p \\[0.4cm]
&=& \disp \left[v_g^{-1}\left( v_g |\psi'|^{p-2}\psi'\right)'\right](\varrho) |\nabla \varrho|^p.
\end{array}
\end{equation}
As we have already observed, $v_g^{-1}\left( v_g |\psi'|^{p-2}\psi'\right)'$ is the expression of the $p$-Laplacian of the radial function $\psi$ in the model $M_g$, making it possible to radialize with respect to $\varrho$. However, in order for this procedure to be effective we need to control the $L^\infty$-norm of $|\nabla \varrho|$ and to guarantee properness of $\varrho$. The latter, by \eqref{def_bxy}, is equivalent to the property $\gr_p(x,o) \ra 0$ as $r(x) \ra \infty$.



First we focus our attention on $|\nabla \varrho|$. While bounds from below for $|\nabla \varrho|$ seem difficult to obtain, bounds from above are simpler, as shown by the next

\begin{proposition}\label{prop_gradienbounded}
Let $(M, \metric)$ be a complete $m$-dimensional manifold satisfying
\begin{equation}\label{lowerRicci_fake}
\Ricc \ge -(m-1) \kappa^2 \metric
\end{equation}
for some $\kappa>0$. Let $p \in (1,\infty)$ and assume that $\Delta_p$ is non-parabolic. Define $\varrho$ as in \eqref{def_bxy}, with $g(r) = \kappa^{-1} \sinh(\kappa r)$. Then, for each $\varrho_0>0$ there exists a constant $C_\varrho>0$ depending on $(m,p,\varrho_0,\kappa)$ and on the geometry of $M$ such that
\begin{equation}\label{uppreuniform_gradb}
|\nabla \varrho| \le C_\varrho \qquad \text{on }  M_{\varrho_0} = \big\{ x \in M : \varrho(x) \ge \varrho_0 \big \}.
\end{equation}
\end{proposition}

\begin{proof}
For notational convenience we omit the dependence of $\gr_p$ on $o$. Since $\varrho$ is positive on $M \backslash \{o\}$ and $\varrho(o)=0$, given $\varrho_0>0$ there exists $R \in (0, 1)$ small enough that $M_{\varrho_0}$ is disjoint from $B_R(o)$. The Harnack inequality in \cite{wangzhang} implies, that for each $x \in M_o$,
\begin{equation}\label{estiyau_0}
|\nabla \log \gr_p(x)| \le C(m,p,R) \left(\frac{1+ \kappa r(x)}{r(x)}\right) \le C_1(m,p,R)\frac{1+\kappa r(x)}{1 + r(x)},
\end{equation}
and using \eqref{gradientb} with $g(r) = g_\kappa(r) = \kappa^{-1}\sinh(\kappa r)$ we get
\begin{equation}\label{gradientb_simpler}
|\nabla \varrho(x)| \le C_2(m,p,R)\frac{1+\kappa r(x)}{1+r(x)}\chi_{g_\kappa}\big(\varrho(x)\big)^{-1/p}
\end{equation}
for each $x\in M_{\varrho_0}$. By Proposition 3.12 in \cite{bmr2} and Example 5.3 in \cite{bmr5}, if $g/h$ is non-decreasing on $\R^+$ then $\chi_g \ge \chi_h$ on $\R^+$. Applying the result with $g = g_\kappa$ and $h(r)= \exp(\kappa r)$ we obtain
$$
\chi_{g_\kappa}(t) \ge \chi_h(t) \equiv \left(\frac{m-1}{p}\right)^p \kappa^p,
$$
hence inserting into \eqref{gradientb_simpler} and maximizing over $r \in [R, \infty)$ we deduce \eqref{uppreuniform_gradb}.
\end{proof}

\begin{remark}
\emph{Although Proposition \ref{prop_gradienbounded} is enough for our purposes, the upper bound in \eqref{uppreuniform_gradb} could be considerably refined: in view of the fact that $\varrho=r$ if $M=M_g$, it is reasonable to claim that the sharp bound is $|\nabla \varrho| \le 1$ on the entire $M\backslash \{o\}$. This has been proved by T. Colding in \cite{colding} for $p=2$ and $\Ricc \ge 0$ on $M$, using the Euclidean space as model $M_g$, with the equality $|\nabla \varrho|=1$ holding at some point if and only if $M=\R^m$. In \cite{maririgolisetti_mono}, the authors extend the sharp bound, with rigidity, to each $p >1$ and lower bound \eqref{lowerRicci_fake}, for $\kappa \ge 0$.
}
\end{remark}

The properness of $\varrho$, that is, the property that $\gr_p(x,o) \ra 0$ as $r(x) = \mathrm{dist}(x,o) \ra \infty$, is a non-trivial fact intimately related to the geometry of~$M$ at infinity. Regarding the case when $\Ricc \ge 0$, I. Holopainen has proved the following

\begin{theorem}[Prop. 5.10 of \cite{holopainen2}]\label{teo_holopainen}
Let $M$ be complete, and suppose that $\Ricc \ge 0$. Denote by $V(r)= \vol(B_r)$ the volume of a geodesic ball of radius $r$ centered at some fixed origin $o$, and with $\partial M(r)$ the portion of $\partial B_r$ which is the boundary of an unbounded connected component. Fix $p \in (1,\infty)$. Then, $\Delta_p$ is non-parabolic on $M$ if and only if
\begin{equation}\label{V}
\left(\frac{s}{V(s)}\right)^{\frac{1}{p-1}}\in L^1(\infty).
\end{equation}
In this case there exists a constant $C \ge 1$ independent of $r$ such that
\begin{equation}\label{bellastima}
\frac 1C\int_{2r}^{\infty} \left(\frac{s}{V(s)}\right)^{\frac{1}{p-1}} \di s \le \gr_p(x,o) \le C\int_{2r}^{\infty} \left(\frac{s}{V(s)}\right)^{\frac{1}{p-1}}\di s
\end{equation}
for each $x \in \partial M(r)$.
\end{theorem}

\begin{corollary}\label{cor_holopainen}
Let $p \in (1,\infty)$. If $M$ is complete with $\Ricc \ge 0$, and $\Delta_p$ is non-parabolic, then for each fixed origin $o\in M$ the Green kernel $\gr_p(x,o) \ra 0$ as $r(x) \ra \infty$.
\end{corollary}

\begin{proof}
For the ease of notation we omit the dependence on $o$ of $\gr_p$. Suppose that there exists $c>0$ and a sequence $\{x_j\}$ with $r_j = r(x_j)\ra \infty$ such that $\gr_p(x_j) \ge c$ for each $j$. By \eqref{bellastima}, up to removing a finite number of $x_j$ and relabelling, $x_j$ necessarily belongs to the boundary of a compact connected component of $M\backslash B_{r_j}$. For $r \le r_j$, let $U_r$ be the connected component of $M \backslash B_r$ containing $x_j$, and define
$$
I_j = \Big\{ r \in (0,r_j] \ : \ \text{$U_r$ is compact}\Big\}, \qquad \bar r_j = \inf(I_j).
$$
Note that $I_j \neq \emptyset$. For $r \in I_j$, since $\gr_p$ is $p$-harmonic on $U_r$ the maximum principle gives $\gr_p(x_j) \le \max_{\partial U_r} \gr_p$, and by continuity $\gr_p(x_j) \le \max_{\partial U_{\bar r_j}}\gr_p$. Choose $\hat r_j  \in (\bar r_j-1, \bar r_j)$ in order to satisfy $\gr_p(x_j) \le \max_{\partial U_{\hat r_j}}\gr_p+ c/2$. Applying \eqref{bellastima} to points of $\partial U_{\hat r_j}$ we have
$$
c \le \gr_p(x_j) \le \max_{\partial U_{\hat r_j}}\gr_p+ \frac{c}{2} \le \frac c 2 + C\int_{2\hat r_j}^{\infty} \left(\frac{s}{V(s)}\right)^{\frac{1}{p-1}}\di s.
$$
This implies that $\{\hat r_j\}$, and hence $\{\bar r_j\}$, is bounded. Fix $R> \sup_j \bar r_j$. Since $r(x_j) \ra \infty$ and using that, by construction, each $x_j$ belongs to a bounded connected component of $M \backslash B_R$, we deduce that $M \backslash B_R$ should necessarily have infinitely many connected components, a contradiction.
\end{proof}

If the Ricci tensor is negative somewhere, due to the possible presence of $p$-parabolic ends some additional condition must be placed on $M$ in order to guarantee that $\gr_p(x,o) \ra 0$ as $r(x) \ra \infty$. We recall that given $K \subset U \subset M$ with $K$ compact, $U$ open, the $p$-capacity of the condenser $(K,U)$ is
$$
\capac_p(K,U) = \inf \left\{ \int_M |\nabla \psi|^p \ \ : \ \ \psi \in \lip_c(U), \ \psi\ge 1 \ \text{ on } \, K\right\}.
$$
Note that $\capac_p(K,U)$ increases by enlarging $K$ or reducing $U$. If $K$ is the closure of a smooth open set and $U$ is smooth and relatively compact, the infimum is realized by the unique solution of $\Delta_p u =0$ on $U \backslash K$, $u=1$ on $\partial K$ and $u=0$ on $\partial U$, extended with $u=1$ on $K$. We prove the following

\begin{theorem}\label{teo_sobolev}
Suppose that $M$ is complete and satisfies a Sobolev inequality
\begin{equation}\label{sobolev}
\left( \int |\psi|^{ \frac{\nu p}{\nu-p}}\right)^{\frac{\nu-p}{\nu}} \le S_{p,\nu} \int |\nabla \psi|^p \qquad \forall \, \psi \in \lip_c(M),
\end{equation}
for some $p \in (1,\infty)$, $\nu>p$ and $S_{p,\nu}>0$. Then, $\Delta_p$ is non-parabolic and there exists a constant $C$ depending on $p, \nu$ and on the geometry of $M$ such that
$$
\gr_p(x,o) \le C r(x)^{- \frac{\nu-p}{p}} \qquad \text{for } \, r(x) \ge 2.
$$
In particular, $\gr_p(x,o) \ra 0$ as $r(x) \ra \infty$.
\end{theorem}


\begin{proof}
It is well-known that \eqref{sobolev} implies the non-parabolicity of $\Delta_p$ on $M$ (see for instance \cite{pigolasettitroyanov} for general $p$, and \cite[Lemma 7.13]{prs} for $p=2$). By \cite{holopainen}, $\gr_p$ is the locally uniform limit of an increasing sequence $\{\gr_{p,j}\}_{j \in \mathbb{N}}$ of Green kernels associated to a smooth increasing exhaustion $\{\Omega_j\}$ of $M$, with zero boundary conditions. Without loss of generality, assume that $\Omega_1 \Subset B_1(o) \Subset \Omega_2$. Hereafter, geodesic balls will always be centered at $o$. Extend each $\gr_{p,j}$ with zero on $M \backslash \Omega_j$, and observe that $\Delta_p \gr_{p,j} \ge 0$ on $M \backslash \{o\}$. By the maximum principle,
$$
\sup_{M \backslash \Omega_1}\gr_{p,j} = \max_{\partial \Omega_1} \gr_{p,j} \le \max_{\partial \Omega_1} \gr_p = \frac{c_1}{2},
$$
and passing to the limit, $\sup_{M \backslash \Omega_1}\gr_p \le c_1/2$. Again by the maximum principle and the monotonicity of the sequence $\{\gr_{p,j}\}$, $\{\gr_{p,j} \ge c_1 \} \subset \{\gr_p \ge c_1 \}\subset \Omega_1$ for each $j$. Moreover, $\gr_{p,j}/c_1$ is the capacitor of the condenser $( \{\gr_{p,j} \ge c_1\}, \Omega_{j})$, thus
$$
\int_{\Omega_j \backslash \{\gr_{p,j} \ge c_1\}} |\nabla \gr_{p,j}|^p = c_1^p \capac_p \big( \{\gr_{p,j} \ge c_1\}, \Omega_j\big) \le c_1^p\capac_p \big( \overline{\Omega}_1, \Omega_2\big)
$$
for each $j \ge 2$. Plugging in the Sobolev inequality \eqref{sobolev} the text function $\psi_j \in \lip_c(M)$ given by $\psi_j = \gr_{p,j}$ on $M \backslash \{\gr_{p,j} \ge c_1\}$, $\psi_j = c_1$ on $\{\gr_{p,j} \ge c_1\}$, we deduce
\begin{equation}\label{boundLp}
\begin{array}{lcl}
\disp \left( \int_{M\backslash \Omega_1} \gr_{p,j}^{ \frac{\nu p}{\nu-p}}\right)^{\frac{\nu-p}{\nu}} & \le & \disp \left( \int \psi_j^{ \frac{\nu p}{\nu-p}}\right)^{\frac{\nu-p}{\nu}} \le S_{p,\nu} \int |\nabla \psi_j|^p \\[0.4cm]
& = & \disp S_{p,\nu}\int_{\Omega_j \backslash \{\gr_{p,j} \ge c_1\}} |\nabla \gr_{p,j}|^p \le S_{p,\nu}c_1^p\capac_p \big( \overline{\Omega}_1, \Omega_2\big).
\end{array}
\end{equation}
Thus, $\{\gr_{p,j}\}$ has uniformly bounded $L^{\nu p/(\nu-p)}$-norm on $M \backslash \Omega_1$. We now perform a standard Moser's iteration on annuli to deduce the uniform decay estimate. In what follows, for convenience we write $\gr$ instead of $\gr_{p,j}$. Fix $\gamma \ge 1$, and let $\phi \in \lip_c(M\backslash \Omega_1)$ to be chosen later. We compute
\begin{equation}\label{gradestigreen}
\begin{array}{lcl}
\disp \int \big| \nabla (\phi \gr^\gamma)\big|^p & \le & \disp \disp \int \big| \gr^\gamma|\nabla \phi| + \gamma \phi \gr^{\gamma-1}|\nabla \gr|\big|^p \\[0.4cm]
& \le & \disp 2^p \int \gr^{\gamma p}|\nabla \phi|^p + 2^p\gamma^p \int \gr^{p(\gamma-1)} \phi^p|\nabla \gr|^p.
\end{array}
\end{equation}
To estimate the second integral on the right-hand, we plug the test function $\phi^p \gr^{p(\gamma-1)+1}$ in the weak definition of $\Delta_p\gr \ge 0$, and use Schwarz and Young inequalities with parameter $\eps>0$, to obtain
$$
\begin{array}{lcl}
\disp \int \phi^p \gr^{p(\gamma-1)}|\nabla \gr|^p & \le & \disp \frac{-p}{p(\gamma-1)+1} \int |\nabla \gr|^{p-2}\phi^{p-1}\gr^{p(\gamma-1)+1}\langle \nabla \gr, \nabla \phi \rangle \\[0.4cm]
& \le & \disp \frac{p\eps^{p'}}{p(\gamma-1)+1} \int \phi^p\gr^{p(\gamma-1)}|\nabla \gr|^p + \frac{p\eps^ {-p}}{p(\gamma-1)+1} \int \gr^ {\gamma p}|\nabla \phi|^p
\end{array}
$$
Choosing $\eps$ such that
$$
\frac{p\eps^{p'}}{p(\gamma-1)+1} = \frac{1}{2}
$$
we get
$$
\disp \int \phi^p \gr^{p(\gamma-1)}|\nabla \gr|^p \le \disp \left(\frac{2p}{p(\gamma-1)+1}\right)^p \int \gr^ {\gamma p}|\nabla \phi|^p.
$$
Inserting into \eqref{gradestigreen} we infer
$$
\disp \int \big| \nabla (\phi \gr^\gamma)\big|^p \le 2^p\left(1+ \left[\frac{2p\gamma}{p(\gamma-1)+1}\right]^p\right) \int \gr^{\gamma p}|\nabla \phi|^p \le C_p \int \gr^{\gamma p}|\nabla \phi|^p,
$$
for some $C_p>0$ independent of $\gamma \ge 1$. Using the Sobolev inequality \eqref{sobolev} we then obtain
\begin{equation}\label{firstmoser}
\disp \left( \int (\phi \gr^\gamma)^{\frac{\nu p}{\nu-p}}\right)^{\frac{\nu-p}{\nu}} \le S_{p,\nu}C_{p} \int \gr^{\gamma p}|\nabla \phi|^p.
\end{equation}
Fix $0 < T < R_0 < R_1 < \infty$ such that $R_0-T >1$, and for $i \ge 0$ set
$$
\rho_i = R_0 - T\left(2-\sum_{k=0}^i2^{-k}\right), \quad \rho_\infty = R_0, \quad r_i = R_1 + T\left(2-\sum_{k=0}^i2^{-k}\right), \quad r_\infty = R_1.
$$
Define $A_i = B_{r_i}\backslash B_{\rho_i}$, and note that $A_i \subset A_0 \subset M \backslash B_1 \subset M \backslash \Omega_1$. Let $\eta_i \in \lip(\R)$ which is $1$ on $[\rho_{i+1},r_{i+1}]$, $0$ outside $[\rho_i,r_i]$ and linear in between, and set $\phi_i = \eta_i(r)$. Inserting $\phi_i$ in \eqref{firstmoser} and using $|\nabla \phi_i| \le T^{-1}2^{i+1}$ we deduce
$$
\disp \left( \int_{A_{i+1}} \gr^{\gamma\frac{\nu p}{\nu-p}}\right)^{\frac{\nu-p}{\nu}} \le S_{p,\nu}C_{p} (2^{i+1}/T)^p\int_{A_i} \gr^{\gamma p}.
$$
We use the inequality with
$$
\gamma = \gamma_i = \left(\frac{\nu}{\nu-p}\right)^{i+1}>1
$$
to obtain
$$
\disp \|\gr^p\|_{L^{\gamma_{i+1}}(A_{i+1})} \le (S_{p,\nu}C_{p})^{\gamma_i^{-1}} (2^{i+1}/T)^{p\gamma_i^{-1}}\|\gr^ p\|_{L^{\gamma_i}(A_i)}.
$$
Iterating and computing explicitly the sums,
$$
\begin{array}{lcl}
\disp \|\gr\|^p_{L^{\infty}(A_{\infty})} & \le & \disp (S_{p,\nu}C_{p}T^{-p})^{\sum_{i=0}^{\infty}\gamma_i^{-1}} 2^{p\sum_{i=0}^{\infty}(i+1)\gamma_i^{-1}}\|\gr^ p\|_{L^{\gamma_0}(A_0)} \\[0.4cm]
& \le & \disp (S_{p,\nu}C_{p})^{\frac{\nu-p}{p}}
T^{p-\nu}2^{\frac{\nu(\nu-p)}{p}}\|\gr^ p\|_{L^{\gamma_0}(A_0)}.
\end{array}
$$
From \eqref{boundLp}, $A_0 \subset \Omega \backslash \Omega_1$ and our definition of $\gamma_0$ we deduce
$$
\|\gr^p\|_{L^{\gamma_0}(A_0)} \le S_{p,\nu} c_1^p \capac_p(\overline{\Omega}_1, \Omega_2),
$$
hence
$$
\disp \|\gr\|^p_{L^{\infty}(A_{\infty})} \le \disp S_{p,\nu}^{\frac{\nu}{p}}C_{p,\nu}
T^{p-\nu} c_1^p \capac_p(\overline{\Omega}_1, \Omega_2).\\[0.4cm]
$$
Fix $R>2$ and choose $T=R/2$, $R_0 = R$. Letting $R_1 \ra \infty$ and using again the maximum principle we conclude
$$
\disp \|\gr_{p,j}\|_{L^{\infty}(M \backslash B_{R})} \le \disp \|\gr_{p,j}\|_{L^{\infty}(A_\infty)} \le \disp S_{p,\nu}^{\nu/p^2} \hat C_{p,\nu} R^{- \frac{\nu-p}{p}} c_1 \capac_p(\overline{\Omega}_1, \Omega_2)^{1/p}.
$$
and letting then $j \ra \infty$ we obtain the desired decay for $\gr_p$.
\end{proof}

We conclude this section by describing three relevant examples where \eqref{sobolev} holds with $\nu = m$, and consequently for $p \in (1,m)$
\begin{equation}\label{decaygreen_examples}
\gr_p(x,o) \le C r(x)^{- \frac{m-p}{p}} \qquad \text{when } \, r(x) \ge 2.
\end{equation}
We stress that, by \cite[Prop. 2.5]{carron} and \cite{pigolasettitroyanov}, \eqref{sobolev} holds on $M$ possibly with a different constant $S_{p,\nu}$ if and only if it holds outside some compact set of $M$.
\begin{example}\label{ex_minimal}
\emph{We recall that a Cartan-Hadamard space is a complete, simply-connected manifold with non-positive sectional curvature. Let $M^m \ra N^n$ be a complete, minimal immersion into a Cartan-Hadamard space. By \cite{hoffmanspruck}, the $L^1$-Sobolev inequality
\begin{equation}\label{isoperimetric}
\left(\int |\psi|^{\frac{m}{m-1}}\right)^{\frac{m-1}{m}} \le S_1 \int |\nabla \psi| \qquad \forall \, \psi \in \lip_c(M)
\end{equation}
holds for some $S_1(m)>0$. Let $p \in (1,m)$. Plugging as test function $|\psi|^\alpha$, for $u \in C^\infty_c(M)$ and $\alpha = \frac{p(m-1)}{m-p}>1$, using H\"older inequality and rearranging we get \eqref{sobolev} with $\nu = m$. Consequently, by Theorem \ref{teo_sobolev}, \eqref{decaygreen_examples} holds with the constant $C$ depending on $(p,m)$ and the geometry of $M$.
}
\end{example}

\begin{example}\label{prop_hebey}
\emph{Let $M^m$ be a complete manifold satisfying
\begin{itemize}
\item[$(i)$] $\Ricc \ge -(m-1) \kappa^ 2 \metric$ for some $\kappa>0$, and
\begin{equation}\label{ipo_persobolev}
\inf_{x \in M} \vol\big(B_1(x)\big) = \upsilon > 0;
\end{equation}
\item[$(ii)$] for some $p \in (1,m)$ and $C_P>0$, we have the validity of the Poincar\'e inequality
\begin{equation}\label{poincare}
\int |\psi|^p \le C_P \int |\nabla \psi|^p \qquad \forall \, \psi \in \lip_c(M).
\end{equation}
\end{itemize}
By work of N. Varopoulos (see \cite{hebey_book}, Thm. 3.2), because of $(i)$ $M$ enjoys the $L^1$-Sobolev inequality
\begin{equation}\label{sobolevL1}
\left(\int |\psi|^{\frac{m}{m-1}}\right)^{\frac{m-1}{m}} \le S_1 \int \big[|\nabla \psi| + |\psi|\big] \qquad \forall \, \psi \in \lip_c(M),
\end{equation}
for some $S_1(m,\kappa,\upsilon) >0$. Using again as a test function $|\psi|^\alpha$, for $u \in C^\infty_c(M)$ and $\alpha = \frac{p(m-1)}{m-p}$, by H\"older inequality and rearranging we get (see \cite[Lem. 2.1]{hebey})
\begin{equation}\label{sobolevL1}
\left(\int |\psi|^{\frac{mp}{m-p}}\right)^{\frac{m-p}{m}} \le S_p \int \big[|\nabla \psi|^p + |\psi|^p\big] \qquad \forall \, \psi \in \lip_c(M),
\end{equation}
for some $S_p(m,\kappa, \upsilon,p)>0$. Assumption $(ii)$ then guarantees \eqref{sobolev} with $\nu = m$, and \eqref{decaygreen_examples} follows from Theorem \ref{teo_sobolev}, with $C$ depending on $(m,\kappa,\upsilon, p)$ and the geometry of $M$.
}
\end{example}

\begin{example}\label{ex_roughisometric}
\emph{Two manifolds $M,N$ of the same dimension, with metrics $\di_M, \di_N$, are said to be \emph{roughly isometric} if there exists $\varphi : M \ra N$ such that
\begin{itemize}
\item[-] $B_\eps \big( \varphi(M) \big) = N$ for some $\eps >0$;
\item[-] there exist constant $C_1 \ge 1, C_2 \ge 0$ such that
$$
C_1^{-1} \di_M(x,y) -C_2 \le \di_N  \big( \varphi(x), \varphi(y) \big) \le C_1 \di_M(x,y) + C_2
$$
for each $x,y \in M$.
\end{itemize}
Note that in fact $M,N$ need not have the same dimension, but in what follows we are not interested in this more general case. M. Kanai in \cite{kanai} proved that if $M$ and $N$ are roughly isometric manifolds both satisfying the uniform condition
\begin{itemize}
\item[$(iii)$] $\Ricc \ge -(m-1) \kappa^ 2 \metric$, for some $\kappa>0$, and $\mathrm{inj}(M) >0$,
\end{itemize}
with $\mathrm{inj}(M)$ the injectivity radius of $M$, then \eqref{isoperimetric} holds on $M$ if and only if it holds on $N$. In particular, a manifold $M$ satisfying $(iii)$ and roughly isometric to $\R^m$ enjoys \eqref{isoperimetric}, and therefore \eqref{sobolev} with $\nu = m$. In the same assumptions note also that $\Delta_p$ is parabolic for each $p \ge m$, see \cite[Thm. 3.16]{holopainen3}, and thus \eqref{sobolev} is false for any $m \le p < \nu$. We remark in passing that $\mathrm{inj}(M)>0$ implies the lower bound of the volume in $(ii)$ of Example \ref{prop_hebey}, see \cite[Prop. 14]{croke}.
}
\end{example}

%
%

\begin{remark}
\emph{For $p=2$, Theorem \ref{teo_sobolev} and a weaker version of the result in Example \ref{prop_hebey} have been obtained in \cite{ni} by integrating the corresponding decay estimate for the heat kernel.
}
\end{remark}



\section{Boundary value problems for nonlinear ODEs}\label{sec2.3}

At the beginning of Section \ref{sec_exhaustion}, we observed that to find radial solutions of $(P_\ge)$ and $(P_\le)$ one is lead to solve the following ODE:
\begin{equation}\label{sol_ODE_22}
\big[v_g\varphi(w')\big]' = v_g \beta f(w) l(|w'|)
\end{equation}
on an interval of $\R^+_0$, where we have extended $\varphi$ to an odd function on $\R$. The functions $v_g$ and $\beta$ are bounds, respectively, for the volume of geodesic spheres of $M$ and for $b$. We devote this section to the study of \eqref{sol_ODE_22}. Regarding $\varphi, f,l$ we assume the following:

\begin{equation}\label{assum_secODE}
\left\{ \begin{array}{l}
\disp \mbox{$\varphi \in C(\R^+_0)$, $\ \varphi(0)=0$, $\ \varphi(t)>0$ for all $t\in\mathbb R^+$,}\\[0.2cm]
\disp \text{$\varphi$ is strictly increasing on $\R^+$,} \\[0.2cm]
\disp f \in C(\R), \\[0.2cm]
f \ge 0 \quad \text{ in $(0, \eta_0)$, for some $\eta_0 \in (0, \infty)$,} \\[0.2cm]
\disp l\in C(\mathbb R^+_0), \quad l\ge0 \ \text{ in } \, \mathbb R^+_0. \\[0.2cm]
\end{array}\right.
\end{equation}

We point out that no monotonicity is needed neither on $f$ nor on $l$. In some results, we also require the validity of the next conditions:
\begin{equation}\label{assum_secODE_altreL}
\left\{ \begin{array}{l}
\disp \mbox{$\varphi \in C^1(\R^+)$, $\ \varphi'>0$ on $\R^+$,}\\[0.3cm]
\disp \frac{t \varphi'(t)}{l(t)} \in L^1(0^+).
\end{array}\right.
\end{equation}

Set
$$
\varphi(\infty) = \lim_{t \ra \infty} \varphi(t) \in (0,\infty],
$$
and define $K$ and $F$ respectively as in \eqref{def_K} and \eqref{def_Fe_intro}.

In what follows, we find convenient to normalize the interval where we study \eqref{sol_ODE_22}, say on $[0, T_0]$, and for this reason we introduce two functions $a, \wp$ that will be related to, respectively, $\beta$ and $v_g$,  in a way that may depend from the geometrical problem at hand. We require

\begin{equation}\label{aepi}
\left\{ \begin{array}{ll}
\disp \wp\in C^1([0, T_0]), & \qquad \wp>0 \ \text{ on }   [0, T_0], \\[0.1cm]
a \in C([0,T_0]), & \qquad a>0 \ \text{ on } \, [0, T_0].
\end{array}\right.
\end{equation}

\subsection{The Dirichlet problem}

We first investigate the existence and the qualitative properties of $C^1$ weak solutions of the singular boundary value problem
\begin{equation}\label{twoboundary}
\begin{cases}
[\wp\, \varphi(w')]'=\wp a f(w)l(|w'|)\quad\mbox{on } \, (0,T),\\[0.1cm]
w(0)=0, \quad w(T)=\eta, \\[0.1cm]
0 \le w\le \eta, \quad w'\ge 0 \ \text{on } \, (0,T)
\end{cases}
\end{equation}
where $\eta>0$, $T \in (0, T_0)$ are given.
%
%
The results of this section are inspired by Chapters 4 and 8 of \cite{pucciserrin}, and we also borrow some of the main ideas of the proof of Proposition 3.1 in \cite{fprgrad} and the appendix of Chapter 4 in \cite{pucciserrin},  but with several improvements in the spirit of  \cite[Thm. 4.1]{bordofilipucci}.\par

One of the main points in our investigation is to determine under which assumptions on $f$, $l$, $\varphi$ and $a$, solutions of \eqref{twoboundary} satisfy $w'(0)=0$ or $w'(0)>0$, that is, whether or not $w$ can be pasted to the zero function on $(-\infty, 0)$ in a $C^1$ way.  As we shall see, such assumptions will be substantially given by the integrability condition \eqref{KO_zero}.

For $\eta, \xi >0$ set
\begin{equation}\label{227}
\begin{array}{ll}
a_0= \min_{[0,T_0]} a, & \qquad a_1= \max_{[0,T_0]} a;\\[0.2cm]
\wp_0 = \min_{[0,T_0]} \wp, & \qquad \wp_1 = \max_{[0,T_0]}\wp;\\[0.2cm]
f_\eta = \max_{[0,\eta]} f, & \qquad l_\xi = \max_{[0,\xi]} l; \\[0.2cm]
\disp \Theta(T) = \sup_{[0,T]} \frac{1}{\wp(t)} \int_0^t \wp(s)a(s) \di s.
\end{array}
\end{equation}
\noindent Note that $\Theta(T) \ra 0$ as $T \ra 0$.\par
We aim to prove the following existence


\begin{theorem}\label{exi2}
Assume \eqref{assum_secODE}, \eqref{aepi} and
\begin{equation}
f(0)l(0)=0.
\end{equation}
Fix $\xi>0$, let $T \in (0,T_0)$ and $\eta \in (0, \eta_0)$, with $\eta_0$ as in \eqref{assum_secODE}, satisfying
\begin{equation}\label{restrict}
\frac{\wp_1}{\wp_0}\varphi\left(\frac{\eta}{T}\right) + 2\Theta(T)f_\eta l_\xi < \varphi(\xi).
\end{equation}
Then, problem \eqref{twoboundary} admits a weak solution $w \in C^1([0,T])$ such that
\begin{equation}\label{supnorm_twobound}
0 \le w' \le \varphi^{-1} \left( \frac{\wp_1}{\wp_0}\varphi\left(\frac{\eta}{T}\right) + 2\Theta(T)f_\eta l_\xi \right)
\end{equation}
In particular, $0\le w' < \xi$.
\end{theorem}

We begin with the following auxiliary result, see  Lemma 4.1.3 of \cite{pucciserrin}.

\begin{lemma}\label{lem_ODE}
Under assumptions \eqref{assum_secODE} and \eqref{aepi}, suppose that $\varphi$ is extended on all of $\R$ in such a way that $t\varphi(t)>0$ on $\R \backslash \{0\}$. Then, any weak solution $w \in C^1([0,T])$ of
\begin{equation}\label{equa_w}
\left\{ \begin{array}{l}
\mathrm{sign}(w)\cdot\big[ \wp \varphi(w')\big]' \ge 0 \qquad \text{in } \, (0,T), \\[0.2cm]
w(0)=0, \quad w(T) = \eta>0
\end{array}\right.
\end{equation}
is such that
\begin{equation}\label{propert_w}
w \ge 0, \qquad w' \ge 0 \qquad \text{in }  [0,T].
\end{equation}
Moreover, there exists $t_0 \in [0,T)$ such that
\begin{equation}\label{propert_w2}
w \equiv 0 \ \ \text{ in } \, [0, t_0]; \qquad w>0, \ w'>0 \ \ \text{ in } (t_0, T],
\end{equation}
\end{lemma}

\begin{proof}
We first claim that $w \ge 0$ on $[0,T]$. Otherwise, by contradiction there exist $t_0$, $t_1$, with $0 \le t_0 < t_1 <T$, such that $w(t_0)=w(t_1)=0$ and $w<0$ on $(t_0,t_1)$. Using the non-negative, Lipschitz test function $\psi = -w$ on $[t_0,t_1]$, $\psi=0$ otherwise, we get
$$
\int_{t_0}^{t_1} \wp \varphi(w')w' \le 0.
$$
Since $s\varphi(s)>0$ in $\R\backslash \{0\}$ by assumption, the integrand is strictly positive. This gives the desired contradiction.

Let $\mathcal J = \{t\in (0,T) : w'(t)>0\}$. Since $w \in C^1([0,T])$ and $w(T)>w(0)$, $\mathcal J \neq \emptyset$ and $\mathcal J$ is open in $(0,T)$. Let $t_0 = \inf \mathcal J \in [0,T)$, so that $w\equiv 0$ on $[0,t_0]$. For each fixed $t \in (t_0, T)$, there necessarily exists $\bar t \in (t_0,t)$ with $w'(\bar t)>0$. Integrating \eqref{equa_w} on $[\bar t,t]$ we deduce that $\wp(t) \varphi(w'(t)) \ge \wp(\bar t) \varphi(w'(\bar t))>0$, which imply that $w'(t)>0$. Hence, $w'>0$ on $(t_0,T]$. Integrating again we obtain $w>0$ on $(t_0,T]$, concluding the proof.
\end{proof}

\begin{remark}\label{rem_afterlemmaODE}
\emph{Note that the a-priori knowledge of $w\ge0$ in $[0,T]$ allows us to directly apply the second part of the proof of Lemma~\ref{lem_ODE} and conclude that $w'\ge0$ on $(0,T)$, and so $0 \le w \le \eta$.
}
\end{remark}

We are now ready to solve the singular two-points boundary value problem \eqref{twoboundary}.
%

\begin{proof}[Proof of Theorem \ref{exi2}] Redefine $f$ and $l$ on the complementary of, respectively, $[0, \eta]$ and $[0, \xi]$, in such a way that
\begin{equation}\label{def_modifications}
\begin{array}{l}
\disp 0 \le f(s) \le f_\eta \quad \text{ for } \, s \ge \eta, \qquad f(s)=0 \quad \text{ for } \, s < 0, \\[0.2cm]
0<l(s) \le l_\xi \quad  \text{ for } \, s \ge \xi
\end{array}
\end{equation}
Note that we can change $l$ as above still keeping the validity of $l \in C(\R^+_0)$, while, with this procedure, we can only ensure that $f \in C(\R\backslash \{0\})$, since $f$ has a jump discontinuity at $s=0$ when $f(0)>0$. The modifications will not affect the conclusions of the theorem since any ultimate solution with $w'\ge 0$ satisfies $0 \le w \le \eta$ and $|w'| \le \xi$. However, the region $\{w=0\}$ needs a special care. We extend $\varphi$ to a continuous function defined on all of $\R$ in such a way that
\begin{equation}\label{extension_varphi}
\begin{array}{l}
\disp \varphi< 0 \quad \text{on } \, (-\infty, 0), \qquad \varphi(t) = -\varphi(-t) \qquad \text{if } \, t \in \left[ - \varphi^{-1}\left(\frac{\eta}{T}\right), 0\right], \\[0.2cm]
\varphi \qquad \text{is strictly increasing on } \, \R; \\[0.2cm]
\lim_{t \ra -\infty} \varphi(t) = -\infty.
\end{array}
\end{equation}
Let
\begin{equation}\label{mu}
\mu_1 = \frac{\wp_1}{\wp_0}\varphi\left(\frac{\eta}{T}\right) + \Theta(T)f_\eta l_\xi < \varphi(\xi)
\end{equation}
by \eqref{restrict}, and set
$$
I = [-\wp_1\mu_1, \wp_0\mu_1].
$$
To show the existence of solutions of \eqref{twoboundary}, following the approach in \cite{pucciserrin} we use Browder's version of the Leray-Schauder theorem\footnote{The idea is attributed by the authors in \cite{pucciserrin} to M. Montenegro.} (see Theorem~11.6 of \cite{gilbargtrudinger}), for the parametric family of boundary value problems
\begin{equation}\label{twoboundary_sigma}
\left\{ \begin{array}{l}
\big[ \wp \varphi(w')\big]' = \sigma \wp af(w)l(|w'|) \qquad \text{on } (0,T), \\[0.2cm]
w(0)=0, \qquad w(T)=\sigma \eta \ge 0,
\end{array}\right.
\end{equation}
for $\sigma \in [0,1]$. In our case, however, the presence of a nonconstant $l$ makes things more subtle. To tackle the problem we let $X$ be the Banach space $X=(C^1([0,T]), \|\cdot\|)$, where
$\|w\|=\|w\|_\infty+\|w'\|_\infty$ for $w\in X$. Define $\HHH : X \times [0,1] \ra X$ as follows:
\begin{equation}\label{defFF}
\HHH(w,\sigma)(t) = \sigma \eta - \int_t^T \varphi^{-1}\left(\frac{1}{\wp(s)}\left[\delta + \sigma \int_0^s\wp(\tau)a(\tau)f(w(\tau))l(|w'(\tau)|)\di \tau\right]\right)\di s,
\end{equation}
where $\delta=\delta(w,\sigma) \in I$ and $\delta$ is chosen in such a way that
\begin{equation}\label{inicond}
\HHH(w,\sigma)(0)=0.
\end{equation}
We claim that such a choice of~$\delta$ is possible, and in fact it is unique. First, we check that $\HHH(w,\sigma)$ is well-defined for each fixed $(w,\sigma) \in X \times [0,1]$ and $\delta\in I$. This follows from the next chain of inequalities, where we use the definition of $\Theta(T)$ and \eqref{restrict}:
\begin{equation}\label{double_est}
\begin{array}{l}
\disp \frac{1}{\wp(s)} \left[ \delta + \sigma \int_0^s\wp(\tau)a(\tau)f(w(\tau))l(|w'(\tau)|)\di \tau\right] \le \frac{\wp_0\mu_1}{\wp_0} + \Theta(T)f_\eta l_\xi \\[0.4cm]
\qquad \disp = \mu_1 + \Theta(T)f_\eta l_\xi < \varphi(\xi).
\end{array}
\end{equation}
We remark that, by construction, $\varphi$ is a homeomorphism of $(-\infty,0)$ onto itself. Moreover, if $\delta=-\wp_1\mu_1$, recalling that $\varphi^{-1}$ is increasing on $\R$ we have
\begin{align*}
\HHH(w,\sigma)(0) & =  \disp \sigma \eta - \int_0^T \varphi^{-1}\left(\frac{1}{\wp(s)}\left[-\wp_1\mu_1 + \sigma \int_0^s\wp(\tau)a(\tau)f(w(\tau))l(|w'(\tau)|)\di \tau\right]\right)\di s \\
& \ge  \disp \sigma \eta - \int_0^T \varphi^{-1}\left( - \frac{\wp_1\mu_1}{\wp(s)} + \Theta(T)f_\eta l_\xi\right)\di s \\
& \ge \sigma \eta - \int_0^T \varphi^{-1}\left( - \mu_1 + \Theta(T)f_\eta l_\xi\right)\di s \ge \sigma \eta,
\end{align*}
where the last inequality follows from \eqref{extension_varphi} and since $\mu_1 \ge \Theta(T)f_\eta l_\xi$. On the other hand, for $\delta= \wp_0\mu_1$, for all $(w, \sigma) \in X \times [0,1]$ we find
\begin{align*}
\HHH(w,\sigma)(0) & =  \disp \sigma \eta - \int_0^T \varphi^{-1}\left(\frac{1}{\wp(s)}\left[\wp_0\mu_1 + \sigma \int_0^s\wp(\tau)a(\tau)f(w(\tau))l(|w'(\tau)|)\di \tau\right]\right)\di s \\
& \le  \disp \eta - \int_0^T \varphi^{-1}\left( \frac{\wp_0\mu_1}{\wp(s)}\right) \le \eta-\int_0^T \varphi^{-1}\left( \frac{\wp_1}{\wp(s)}\varphi\left( \frac\eta T\right)\right) \\
&  \le \eta - \int_0^T \varphi^{-1}\left( \varphi\left(\frac{\eta}{T}\right)\right)\di s = 0.
\end{align*}
Now, the integrand in the RHS of \eqref{defFF} is a strictly increasing function of~$\delta$ for $(w,\sigma)$ fixed. It is therefore clear that there exists a unique $\delta=\delta(w,\sigma) \in I$ such that $\HHH(w,\sigma)(0)=0$.

%

It remains to show that a fixed point of~$w=\HHH(w,1)$ exists. To apply Browder's version of the Leray--Schauder theorem, we shall check that:
\begin{itemize}
\item[$(i)$] $\HHH(w,0) = 0$, the zero function of $X$,
\item[$(ii)$] $\HHH: X \times [0,1] \ra X$ is continuous and compact,
\item[$(iii)$] There exists a constant $\Lambda>0$  such that  $\|w\|\le \Lambda$
for all $(w,\sigma)\in X\times [0,1]$, with $w = \HHH(w,\sigma)$.
\end{itemize}
Property $(i)$ is immediate by the definition of~$\HHH$, since $\delta(w,0)=0$ for all $w\in X$. Regarding $(iii)$, by construction each solution of~$w = \HHH(w, \sigma)$ is of class $C^1([0,T])$ and has the property that $\varphi(w') \in \lip([0,T])$.
We claim that $w \ge 0$ on $[0,T]$. Indeed, if $w< 0$ somewhere, fix an interval $(t_1,t_2) \subset (0,T)$ such that $w<0$ on $(t_1,t_2)$, $w(t_1)=w(t_2)=0$. From $f=0$ on $(-\infty,0)$ we deduce $\big(\varphi(w')\big)' =0$ on $(t_1,t_2)$; thus, integrating against the test function $(w+\eps)_-$ and letting $\eps \ra 0$ we get
$$
0 = \int_{t_1}^{t_2} \wp \varphi(w')w',
$$
and because of the positivity of $s \varphi(s)$ on $\R\backslash\{0\}$ we deduce that $w'$, and consequently $w$, vanishes identically on $(t_1,t_2)$, contradiction. By Remark \ref{rem_afterlemmaODE}, from $w \ge 0$ in $[0,T]$ we infer $w'\ge0$ and $0\le w\le\eta$ in $[0,T]$, and
the identity
\begin{equation}\label{propriewprimo}
w'(t) = [\HHH(w,\sigma)]'(t)=\varphi^{-1}\left(\frac{1}{\wp(t)}\left[\delta + \sigma \int_0^t\wp(\tau)a(\tau)f(w(\tau))l(|w'(\tau)|)\di \tau\right]\right)
\end{equation}
implies 
\begin{equation}\label{bound_deriw}
\begin{array}{lcl}
0 \le w'(t) & \le & \disp \varphi^{-1}\left( \frac{\delta}{\wp(t)} + \Theta(T) f_\eta l_\xi\right) \le \varphi^{-1}\Big( \mu_1 + \Theta(T) f_\eta l_\xi\Big) < \xi.
\end{array}
\end{equation}
Hence, each solution of $w = \HHH(w,\sigma)$ enjoys the a-priori estimate $\|w\| \le \eta + \xi$, as required. \par
We are left to prove $(ii)$. Let $\{(w_k, \sigma_k)\}_k$ be a bounded sequence in $X \times [0,1]$, say $\|w_k\|\le L$ for all $k$. Using that $\delta_k = \delta(w_k, \sigma_k) \in I$ and $0 \le f(t) \le f_\eta$ for all $t\in\mathbb R$, together with \eqref{double_est}, we deduce that
\begin{equation}\label{ionew}
\left\| \HHH(w_k, \sigma_k)' \right\|_\infty \le \max\left\{ \left|\varphi^{-1}\left(- \frac{\wp_1\mu_1}{\wp_0}\right)\right|, \varphi^{-1} \Big( \mu_1 + \Theta(T)f_\eta l_\xi\Big) \right\},
\end{equation}
thus $\big\{\HHH(w_k,\sigma_k)\big\}_k$ is equi-bounded in $X$ and equi-continuous in $[0,T]\times[0,1]$. To show the equicontinuity of $\big\{\HHH(w_k,\sigma_k)\big\}_k$ in $C^1$, we shall estimate the difference
$$
\big| \HHH(w_k,\sigma_k)'(t) - \HHH(w_k,\sigma_k)'(s)\big|
$$
for $0 \le s<t \le T$. Set for convenience
\begin{align*}
x_k=&\frac{1}{\wp(t)}\left(\delta_k + \sigma_k \int_0^t\wp(\tau)a(\tau)f(w_k(\tau))l(|w'_k(\tau)|)\di \tau\right) \\
y_k=&\frac{1}{\wp(s)}\left(\delta_k + \sigma_k \int_0^s\wp(\tau)a(\tau)f(w_k(\tau))l(|w'_k(\tau)|)\di \tau\right),
\end{align*}
and note that
$$
\big| \HHH(w_k,\sigma_k)'(t) - \HHH(w_k,\sigma_k)'(s)\big| = \big| \varphi^{-1}(x_k) - \varphi^{-1}(y_k) \big|.
$$
Fix $\varepsilon>0$ and let $\varrho=\varrho(\varphi^{-1},\varepsilon)>0$ be the corresponding number of the uniform continuity of $\varphi^{-1}$ in $[- \frac{\wp_1}{\wp_0}\mu_1,\mu_1]$. Set $c = \wp_1 a_1 f_\eta l_\xi$, and suppose that $|t-s|<\varrho/C$, where
\begin{equation}\label{ioonew}
C=\frac{\wp_1\mu_1}{\wp_0^2}\max_{\tau\in[0,T]}|\wp'(\tau)|+\kappa,\quad \kappa=\frac{c}{\wp_0}\left(\frac{T}{\wp_0}\,\max_{\tau\in[0,T]}|\wp'(\tau)|+1\right).
\end{equation}
This is possible by  \eqref{aepi}, since $\wp \ge \wp_0 >0$ on $[0,T]$ and $\wp\in C^1(\mathbb R^+_0)$. Define
$$
I_k(t)=\int_0^t\wp(\tau)a(\tau)f(w_k(\tau))l(|w'_k(\tau)|)\di \tau,\qquad\mathcal I_k(t)=\frac{I_k(t)}{\wp(t)},
$$
and note that for each $k$
$$
0 \le I_k(t)-I_k(s) \le c(t-s) \qquad\mbox{and}\quad\lim_{t\to 0^+}\mathcal I_k(t)=0.
$$
Using $\wp>0$ in $\mathbb R^+_0$ and $\wp\in C^1(\mathbb R^+_0)$ we get
\begin{align*}
\big|\sigma_k\mathcal I_k(s)-\sigma_k\mathcal I_k(t)\big|&\le\big|\mathcal I_k(s)-\mathcal I_k(t)\big|=\left|\frac{\wp(t)I_k(s)-\wp(s)I_k(t)}{\wp(s)\wp(t)}\right|\\
&\le \frac{|\wp(t)-\wp(s)|}{\wp(s)\wp(t)}\,I_k(s)+\frac{|I_k(s)-I_k(t)|}{\wp(t)}\\
&\le\frac{c}{\wp_0}\left(\frac{T}{\wp_0}\,\max_{\tau\in[0,T]}|\wp'(\tau)|+1\right)|t-s|=\kappa|t-s|,
\end{align*}
Since $\delta_k\in I$, by \eqref{ioonew} we estimate
$$
\begin{array}{lcl}
|x_k-y_k| & \le & \disp |\delta_k|\frac{|\wp(t)-\wp(s)|}{\wp(s)\wp(t)}+\big|\sigma_k\mathcal I_k(s)-\sigma_k\mathcal I_k(t)\big|\\[0.4cm]
&\le & \disp \left(\frac{\wp_1\mu_1}{\wp_0^2}\max_{\tau\in[0,T]}|\wp'(\tau)|+\kappa\right)|t-s|=C|t-s|<\varrho.\end{array}
$$

%
In conclusion,
$$
\big|\HHH(w_k,\sigma_k)'(t)- \HHH(w_k,\sigma_k)'(s)\big|=\big|\varphi^{-1}(x_k)-\varphi^{-1}(y_k)\big|<\varepsilon
$$
provided that $|t-s|< \varrho/C$, independently of $k$. As an immediate consequence of the Ascoli--Arzel\`{a} theorem $\HHH$ maps bounded sequences of $X\times[0,1]$ into relatively compact
sequences of $X$. Finally, if $(w_k,\sigma_k)$ converges to some $(w,\sigma)$ in $X\times[0,1]$, then
$\delta_k \ra \delta=\delta(w,\sigma)$ as $k\to\infty$. Hence, $\HHH$ is continuous in $X\times[0,1]$. The proof of this fact is similar to that given in Chapter~4 of~\cite{pucciserrin}. The Leray-Schauder fixed point theorem can therefore be applied and the mapping $\HHH(w,1)$ has a fixed point $w$, which by \eqref{bound_deriw} satisfies inequality \eqref{supnorm_twobound}.\par
To conclude, we prove that $w$ solves the ODE in \eqref{twoboundary} with the original $f,l$. Differentiating \eqref{propriewprimo} with $\sigma=1$, we see that $w$ is a weak solution of
\begin{equation}\label{equafbar}
\big( \wp \varphi(w')\big)' = \wp a \bar f(w) \bar l(|w'|) \qquad \text{on } \, (0,T),
\end{equation}
where now we have denoted with $\bar f, \bar l$ the modifications of $f,l$ in \eqref{def_modifications}.  Clearly, $\bar l(w') = l(w')$ on $(0,T)$, since $0 \le w' < \xi$. From $0 \le w \le \eta$ and Lemma \ref{lem_ODE}, there exists $t_0 \in [0,T)$ such that $w=0$ on $[0,t_0]$ and $w>0$, $w'>0$ on $(t_0,T]$, and since $\bar f = f$ on $(0,\eta)$ we deduce
\begin{equation}\label{equaf}
\big( \wp \varphi(w')\big)' = \wp a f(w) l(|w'|)
\end{equation}
on $(t_0,T)$. On the other hand, because of our assumption $f(0)l(0)=0$ the function $w=0$ solves \eqref{equaf} on $(0, t_0)$. Since $w \in C^1([0,T])$, \eqref{equaf} holds weakly on all of $(0,T)$, as claimed.
\end{proof}

\begin{remark}
\emph{The requirement $f(0)l(0)=0$ is crucial for the validity of the above theorem, because otherwise $w$ might be negative somewhere.
}
\end{remark}

Given $f, \varphi, l$ satisfying \eqref{assum_secODE}, define $K,F$ as in \eqref{def_K}, \eqref{def_Fe_intro}. We next present two auxiliary calculus results. The first is similar to Lemma~4.4.1--$(i)$ in \cite{pucciserrin}. We recall that the notion of a $C$-increasing function is given in Definition \ref{def_Cincreasing}.

\begin{lemma}\label{lem_auxil2}
Assume that $f$ is $C$-increasing in $(0,\eta_0)$, for some $\eta_0 >0$. Then, for each $\sigma \in [0,1]$ we have $F(\sigma t)\le C\,\sigma F(t)$ for all $t\in [0, \eta_0)$.
\end{lemma}

\begin{proof}
Fix $\sigma \in [0,1]$. Since $f$ is $C$-increasing in $(0,\eta_0)$, we have $\sigma f(\sigma t) \le C\,\sigma f(t)$ for each $t\in[0, \eta_0)$, and thus
$$
F(\sigma t) = \int_0^{\sigma t} f(s) \di s = \int_0^t \sigma f(\sigma \tau)\di \tau \le C\,\sigma\int_0^t f(\tau) \di \tau = C\sigma F(t),
$$
as claimed.
\end{proof}

The second lemma concerns the preservation of the validity of \eqref{KO_zero} and \eqref{KO} when we replace $f$ with $\sigma f$, $\sigma \in \R^ +$. Similar results have been proved in \cite[Lem. 4.1.2]{pucciserrin}, \cite{fprgrad} (remark on page 523) and \cite{maririgolisetti}.

\begin{lemma}\label{lem_mettimaosigmatau_novo}
Let $f,l$ satisfy \eqref{assum_secODE},  \eqref{assum_secODE_altreL} and suppose $l>0$ on $\R^+$.
\begin{itemize}
\item[(i)] Assume that $f$ is positive and $C$-increasing in $(0,\eta_0)$, for $\eta_0 >0$ in \eqref{assum_secODE}. Then
$$
\frac{1}{K^{-1}(F(s))} \in L^1(0^+)  \qquad \Longleftrightarrow \qquad \frac{1}{K^{-1}(\sigma F(s))} \in L^1(0^+)
$$
for some (equivalently, any) $\sigma \in \R^+$.
\item[(ii)] Assume that $f$ is positive and $C$-increasing on $(\bar \eta_0, \infty)$, for some $\bar \eta_0>0$. Having defined $F(t) = \int_{\bar \eta_0}^t f$, it holds
$$
\frac{1}{K^{-1}(F(s))} \in L^1(\infty)  \qquad \Longleftrightarrow \qquad \frac{1}{K^{-1}(\sigma F(s))} \in L^1(\infty)
$$
for some (equivalently, any) $\sigma \in \R^+$.
\end{itemize}
\end{lemma}

\begin{proof}
For the ease of notation we denote with $\KOzsigma$ and $\KOisigma$, respectively, the integrability conditions
$$
\int_{0^+} \frac{\di s}{K^{-1}(\sigma F(s))} < \infty, \qquad \int^{\infty} \frac{\di s}{K^{-1}(\sigma F(s))} < \infty.
$$
We prove $(i)$, beginning with implication
\begin{equation}\label{sigmaeno}
\eqref{KO_zero} \Rightarrow \KOzsigma.
\end{equation}
If $\sigma \ge 1$, \eqref{sigmaeno} is immediate from the monotonicity of $F$ and $K$.\\
If $\sigma \in (0,1)$, we apply Lemma \ref{lem_auxil2} with $\sigma/C$ replacing $\sigma$ (note that $C \ge 1$) to deduce $F( \sigma t/C) \le \sigma F(t)$ for each $t \in (0, \eta_0)$. Integrating and changing variables,
$$
\int_{0^+} \frac{\di s}{K^{-1}(\sigma F(s))} \le \int_{0^+} \frac{\di s}{K^{-1}(F(\sigma s/C))} = \frac{C}{\sigma} \int_{0^+} \frac{\di \tau}{K^{-1}(F(\tau))},
$$
which proves \eqref{sigmaeno}. To show the reverse implication in \eqref{sigmaeno}, it is enough to observe that $\KOzsigma$ is condition \eqref{KO_zero} for the function $\bar f = \sigma f$, and to apply the previous estimates with $\bar f$ replacing $f$ and $\sigma^{-1}$ replacing   $\sigma$.\\
The proof of $(ii)$ is analogous.
\end{proof}

We are now ready to obtain further information on the solution of problem \eqref{twoboundary} given in Theorem \ref{exi2}. First, we investigate sufficient conditions to ensure $w'(0)>0$. 

\begin{proposition}\label{prop_twobound_refined}
Assume \eqref{assum_secODE}, \eqref{assum_secODE_altreL} and \eqref{aepi}. Suppose that
\begin{itemize}
\item[] $f(0)l(0)=0$;
\item[] $f$ is $C$-increasing on $(0, \eta_0)$, for $\eta_0>0$ as in \eqref{assum_secODE},
\end{itemize}
\noindent and that one of the following two sets of conditions is met:
\begin{itemize}
\item[(i)] $f \equiv 0$ on $(0, \eta_0)$;
\item[(ii)] $f>0$ on $(0, \eta_0)$, and also
\begin{itemize}
\item[-] $l$ is $C$-increasing on $(0, \xi_0)$, for some $\xi_0 >0$.
\item[-] $\wp$ is monotone on $[0,T_0]$, either increasing or decreasing ($T_0$ as in \eqref{aepi}),
\end{itemize}
and
\begin{equation}\label{noninteinzero_2}\tag{$\neg$KO$_0$}
\frac{1}{K^{-1}\circ F} \not \in L^1(0^+).
\end{equation}
\end{itemize}
Then, the solution $w$ of problem \eqref{twoboundary}, with $\eta \in (0, \eta_0)$,
constructed in Theorem \ref{exi2}, has the further properties
\begin{equation}\label{goodfor_FMP}
w>0 \quad \text{on }(0,T], \qquad w'>0 \quad \text{on } [0,T].
\end{equation}
\end{proposition}

\begin{proof} First, observe that if we prove that $w'(0)>0$, then \eqref{goodfor_FMP} follows by a direct application of Lemma \ref{lem_ODE}. We prove $w'(0)>0$ for cases $(i)$ and $(ii)$ separately. \\[0.2cm]

\noindent
{\em Case $(i)$}.\\
Since $f\equiv 0$ on $(0, \eta_0)$, then $(\wp\varphi(w'))'=0$ in $(0,T)$ by \eqref{twoboundary} and the choice $\eta< \eta_0$. Hence, integrating
\begin{equation}\label{intt}
\wp(t) \varphi(w'(t)) = \wp(0) \varphi(w'(0))\quad \text{for all  }t\in(0,T].
\end{equation}
Suppose by contradiction that $w'(0)=0$. From \eqref{assum_secODE} we have $\varphi>0$ on $\R^+$, and also $\wp>0$ on $[0,T]$. Thus, \eqref{intt} would imply that $w'$, hence $w$, is identically zero in $[0,T]$, contradicting $w(T)=\eta$. \smallskip

\noindent
{\em Case $(ii)$}. \\
Theorem \ref{exi2} guarantees that $\varphi(w') \in C^1([0,T])$, $w'(0) \ge 0$ and $0 \le w \le \eta$. Let us reason by contradiction and suppose that $w'(0)=0$. We shall then prove that
$$
\frac{1}{K^{-1}\circ F} \in L^1(0^+),
$$
which contradicts \eqref{noninteinzero_2}, completing the proof.

First, by Lemma~\ref{lem_ODE} there exists $t_0 \in [0,T)$ such that $w(t) \equiv 0$ on $[0,t_0]$ while $w>0$ on $(t_0,T]$. If $t_0=0$ then $w'(t_0)=0$ by our assumption, while if $t_0>0$ then $w(t_0)=w'(t_0)=0$ since $w$ is $C^1([0,T])$. From \eqref{propriewprimo} and since $\varphi'>0$ on $\R^+$ we infer the existence of $w''$ in $(t_0,T)$. Thus, $w$ satisfies
\begin{equation}\label{stepimpo}
\wp \varphi'(w')w'' + \wp' \varphi(w') = \wp a f(w)l(w')\qquad\mbox{in }(t_0,T).
\end{equation}
We first suppose that $\wp' \ge 0$ on $[0, T_0]$. By \eqref{propert_w2} and \eqref{aepi}, $w$ is a solution of the  inequality
$$
w'\varphi'(w')w'' \le a f(w)w'l(w')\qquad\mbox{in }(t_0,T).
$$
Integrating on $[t_0,t)$, with $t\in (t_0,T]$ we have
$$
\int_{t_0}^t \frac{w'\varphi'(w')w''}{l(w')}\di \tau \le \int_{t_0}^t af(w)w' \di \tau\le a_1 \int_{t_0}^t f(w)w'\di \tau.
$$
Changing variables and using $w'(t_0)=w(t_0)=0$ we deduce
\begin{equation}\label{io}
K(w'(t)) = \int_{0}^{w'(t)} \frac{s\varphi'(s)}{l(s)}\di s \le a_1\int_{0}^{w(t)} f(s)\di s = a_1 F(w(t)).
\end{equation}
Assumption $f>0$ on $(0, \eta_0)$ implies that $F>0$ on $(0, \eta_0)$. Having chosen $T_1 \in (t_0,T]$ in such a way that $a_1 F(w(T_1)) < K_\infty$, we apply $K^{-1}$, rearrange and integrate to obtain
\begin{equation}\label{ioo}
\int_{t_0}^t \frac{w'(s)\di s}{K^{-1}\big(a_1F(w(s))\big)} \le (t-t_0) \qquad \forall \,   t\in(t_0,T_1].
\end{equation}
Changing variables,
\begin{equation}\label{ioo}
\int_0^{w(t)} \frac{\di \tau}{K^{-1}(a_1F(\tau))} \le (t-t_0) \qquad \forall \, t\in(t_0,T_1].
\end{equation}
By Lemma \ref{lem_mettimaosigmatau_novo} property \eqref{KO_zero} holds, as claimed.

We are left to consider the case $\wp' \le 0$. Then, by \eqref{stepimpo} we deduce that $w'' \ge 0$ on $(t_0,T)$, hence $w'$ is increasing there. Integrating \eqref{stepimpo} on $(t_0,t)$ and using the $C$-monotonicity of $f,l$ together with $w'(t_0)=0$, we get
\begin{equation}\label{uffi}
\begin{array}{lcl}
\disp \varphi\big(w'(t)\big) & = & \disp \frac{1}{\wp(t)}\int_{t_0}^t \wp(s)a(s)f(w(s))l(w'(s))\di s \\[0.4cm]
& \le & \disp C^2 a_1f(w(t)) l(w'(t)) \left[ \frac{1}{\wp(t)}\int_0^t \wp(s)\di s\right].
\end{array}
\end{equation}
Now, consider the energy $E(t) = K\big(w'(t)\big) - a_1F\big(w(t)\big)$. Differentiating and using \eqref{uffi} and the definition of $a_1$, we obtain
$$
\begin{array}{lcl}
E'(t) & = & \disp \frac{w'\varphi'(w')w''}{l(w')} - a_1f(w)w' = \frac{w'}{l(w')}\left[- \frac{\wp'}{\wp}\varphi(w') + af(w)l(w') - a_1 f(w)l(w')\right] \\[0.4cm]
& \le & \disp - \frac{\wp'}{\wp} \frac{w'\varphi(w')}{l(w')} \le \left|\frac{\wp'(t)}{\wp(t)^2}\int_0^t\wp(s) \di s\right|C^2 a_1 f(w)w' \\[0.4cm]
& \le & \disp \frac{\wp_1 T_0\|\wp'\|_{L^\infty([0,T_0])}}{\wp_0^2}C^2a_1 f(w)w' = \bar c f(w)w'.
\end{array}
$$
Integrating and using $w(t_0)=w'(t_0)=0$,
$$
K\big(w'(t)\big) \le a_1 F\big(w(t)\big) + \bar c F\big(w(t)\big) = (a_1+ \bar c)F\big(w(t)\big).
$$
Having obtained again an inequality like \eqref{io}, to achieve the desired contradiction it is sufficient to repeat verbatim the last steps of the proof for $\wp'\ge 0$.
\end{proof}

\begin{remark}\label{rem_noserveCincreasing}
\emph{When $\wp' \ge 0$, to reach the desired conclusion in $(ii)$ we do not use the assumption that $l$ is $C$-increasing.
}
\end{remark}

Proposition \ref{prop_twobound_refined} has a converse, at least if the threshold $\eta$ in \eqref{twoboundary} is sufficiently small, namely \eqref{KO_zero} implies that $w'(0)=0$. To reach the goal, the idea is to compare $w$ with an explicit supersolution of \eqref{twoboundary}, whose construction  generalizes the End Point Lemma in \cite[Lem. 4.4.1]{pucciserrin}.

\begin{proposition}\label{prop_wprimougualezero}
Assume \eqref{assum_secODE}, \eqref{assum_secODE_altreL} and \eqref{aepi}. Suppose that
\begin{itemize}
\item[] $f(0)l(0)=0$;
\item[] $f$ is $C$-increasing on $(0, \eta_0)$, for $\eta_0>0$ as in \eqref{assum_secODE};
\item[] $l$ is $C$-increasing on $(0, \xi_0)$, for some $\xi_0 >0$.
\end{itemize}
If $f>0$ on $(0, \eta_0)$ and 
\begin{equation}\label{inteinzero_22}\tag{KO$_0$}
\frac{1}{K^{-1}\circ F} \in L^1(0^+),
\end{equation}
then there exists $\eta_1$ sufficiently small that, for each $\eta \in (0, \eta_1)$, the solution $w$ of problem \eqref{twoboundary} constructed in Theorem \ref{exi2} satisfies
\begin{equation}\label{goodfor_CSP}
w'(0)=0. 
\end{equation}
\end{proposition}

\begin{proof}
For $\sigma \in (0,1)$ to be determined, using \eqref{KO_zero_section} and Lemma \ref{lem_mettimaosigmatau_novo} we implicitly define $z(t)$ by setting
$$
t = \int_0^{z(t)} \frac{\di s}{K^{-1}(\sigma F(s))} \qquad \text{for } \, t \in [0,T).
$$
Note that $z$ is positive on $(0,T)$. Differentiating,
\begin{equation}\label{aaa}
z' = K^{-1}(\sigma F(z)\big),
\end{equation}
whence $z'>0$ on $(0,T)$ and $z'(0)=0$. Evaluating $K$ on both sides of \eqref{aaa} and differentiating once more we get
$$
\frac{z' \varphi'(z') z''}{l(z')} = \sigma f(z)z' \qquad \text{on } \, (0,T).
$$
Since $\varphi',l,z'$ are positive, we deduce $z'' \ge 0$, and simplifying by $z'$ we infer
$$
\big( \varphi(z') \big)' = \sigma f(z) l(z').
$$
Integrating, using $z'(0)=0$ together with the $C$-increasing property of $f$ and $l$ (note that $z'$ is increasing) we obtain
$$
\varphi\big(z'(t)\big) = \disp \sigma \int_0^t f\big(z(s)\big)l\big(z'(s)\big)\di s \le \sigma C^2T f\big(z(t)\big)l\big(z'(t)\big).
$$
Summarizing, so far we have obtained,
$$
\begin{array}{lcl}
\disp \big[ \wp \varphi(z')\big]' = \wp \big(\varphi(z')\big)' + \wp' \varphi(z') & \le & \disp \frac{\sigma}{a_0} \left[ 1+ C^2 T \frac{\|\wp'\|_{L^\infty([0,T_0])}}{\wp_0} \right] \wp a f(z) l(z') \\[0.4cm]
\disp & = & \disp \frac{1}{2C} \wp a f(z) l(z'),
\end{array}
$$
where we have defined $\sigma$ in order to satisfy the last equality. Next, we fix $\eta_1 \le z(T)$ small enough  in such a way that each $\eta \in (0,\eta_1)$ meets the requirements in Theorem \ref{exi2}, to guarantee the existence of $w$. We claim that, for $\eta \in (0, \eta_1)$, the solution $w$ in Theorem \ref{exi2} satisfies $w \le z$ on $(0,T)$. This, together with the already established $z'(0)=0$, forces $w'(0) =0$ and concludes our proof. By contradiction, suppose that $c = \max_{[0,T]}(w-z)>0$ and let $\Gamma = \{ w-z = c\}$. By construction, $\Gamma \Subset (0,T)$, and since $w,z$ are $C^1$ we deduce $w'=z'$ on $\Gamma$. By continuity and since $w'>0$ on $(0,T)$, we can choose $\delta \in (0,c)$ close enough to $c$ in such a way that $l(z') \le 2l(w')$ on the set $I_\delta = \{w-z > \delta\}$. On $I_\delta$, using the $C$-increasing property we therefore have
$$
\big[ \wp \varphi(w')\big]' \ge a \wp f(w)l(w') \ge \frac{1}{2C} a \wp f(z)l(z') \ge \big[ \wp \varphi(z')\big]' = \big[ \wp \varphi((z+c)')\big]'
$$
and $w = z+c$ on $\partial I_\delta$. By standard comparison (one can apply, for instance Proposition \ref{prop_comparison} below to an appropriate radial model), $w \le z+c$ on $U_\delta$, contradiction. 
\end{proof}

\subsection{The mixed Dirichlet-Neumann problem}

We next move to investigate the problem

\begin{equation}\label{eq_mixed}
\begin{cases}
\big[\wp\, \varphi(w')\big]'=\wp a f(w)l(|w'|) \qquad \text{on } \, (0,T),\\[0.2cm]
w'(0)=0, \qquad w(T)=\eta, \\[0.2cm]
0 \le w \le \eta, \qquad w' \ge 0 \ \text{ on } \, (0,T),
\end{cases}
\end{equation}
for given $\eta >0$, $T \in (0, T_0)$. We here extend and generalize in several directions the core of Corollary~1.4 of~\cite{fprgrad}, \emph{without requiring any monotonicity on $l$}, as well as the results of Section~4 of~\cite{bordofilipucci}. We assume \eqref{assum_secODE} and \eqref{aepi}, and we define $a_0,a_1,\wp_0, \wp_1, f_\eta, l_\xi$ and $\Theta(T)$ as in \eqref{227}. 
\begin{theorem}\label{exi2_neumann}
Assume \eqref{assum_secODE} and \eqref{aepi}, and that 
\begin{equation}\label{assu_ancora_fl}
\begin{array}{l}
f> 0 \quad \text{on } \, \R^+, \qquad f(0)=0; \\[0.2cm] 
l>0 \quad \text{on } \, \R^ +_0. \\[0.2cm]
\end{array}
\end{equation}
Then, for each $\eta, \xi>0$ and $T \in (0, T_0)$ satisfying
\begin{equation}\label{bound_etaxi_neu}
\Theta(T) f_\eta l_\xi < \varphi(\xi),
\end{equation}
the problem
\begin{equation}\label{2.1core}\begin{cases}
[\wp\, \varphi(w')]'=\wp a f(w)l(|w'|) \qquad \text{on } \, (0, T) \\[0.2cm]
w'(0)=0, \qquad w(T)=\eta, \\[0.2cm]
0 \le w \le \eta, \qquad 0 \le w' < \xi \ \text{ on } \, (0,T),
\end{cases}
\end{equation}
admits solution $w \in C^ 1([0,T])$, and there exists $t_0 \in [0,T)$ such that 
\begin{equation}\label{proprie_neu_w}
w(t)\equiv w(t_0)\ge 0 \ \  \text{ on } \, [0, t_0], \qquad w'>0 \ \  \text{ on } \, (t_0,T].
\end{equation}
Moreover, if
$$
\varphi \in C^ 1(\R^+), \qquad \varphi'>0 \quad \text{ on } \, \R^ +,
$$
then $w \in C^2\big((t_0, T]\big)$ and satisfies
\begin{equation}\label{wder}
\frac{\varphi' (w')}{l(w')}\,w''= a\,f(w)-\frac{\wp'}{\wp}\cdot\frac{\varphi(w')}{l(w')} \qquad \text{on } \, (t_0,T).
\end{equation}
All of the above conclusions still hold if condition $\wp>0$ on $[0,T]$, in \eqref{aepi}, is replaced by
\begin{equation}\label{wpneum}
\begin{array}{l}
\disp \wp > 0 \quad \text{on }\, (0, T], \qquad \wp(0) = 0 \\[0.2cm]
\wp' \ge 0 \quad \text{on } \, [0,\delta), \ \, \text{ for some } \, \delta>0.
\end{array}
\end{equation}
\end{theorem}


\begin{remark}
\emph{Differently from the Dirichlet problem, if we allow $l$ to vanish at $t=0$ in the Neumann case we cannot  guarantee that the solution of \eqref{2.1core} be non-constant. This motivates the necessity to require $l>0$ on $\R^+_0$. 
}
\end{remark}

\begin{proof}
The strategy goes along the same lines as that for the Dirichlet problem. First, we   redefine $f$ outside of $[0,\eta]$ and $l$ outside of $[0,\xi]$ in such a way that 
\begin{equation}\label{def_modi_neu}
\begin{array}{l}
\disp f \in C(\R), \qquad \disp 0 \le f(s) \le f_\eta \quad \text{ for } \, s \ge \eta, \qquad f(s)=0 \quad \text{ for } \, s < 0, \\[0.2cm]
l \in C(\R^ +_0), \qquad 0<l(s) \le l_\xi \quad  \text{ for } \, s \ge \xi
\end{array}
\end{equation}
This will not affect the conclusion of the proposition, since any ultimate solution $w$ of \eqref{2.1core}, with $w \ge
0$, $w' \ge 0$ in $[0,T]$, satisfies $0 \le w \le \eta$ and $|w'|< \xi$. 

Denote with $X$ the Banach space $X=C^1\big([0,T]\big)$, endowed with the usual norm
$\|w\|=\|w\|_\infty+\|w'\|_\infty$. Define the homotopy $\mathcal H:X\times[0,1]\to X$ by
\begin{equation}\label{2.6core}
{\mathcal H}[w,\sigma](t)= \sigma
\eta- \int^{T}_t\varphi ^{-1} \left(\frac\sigma{\wp(s)}\int^s_{0} \wp(\tau)a(\tau) f(w(\tau))l(|w'(\tau)|)\di \tau\right) \di s.
\end{equation}
We claim that $\haus$ is well defined and valued in $X$. Indeed, in our assumptions  
\begin{equation}\label{bound_grad_neu}
0 \le \frac{\sigma}{\wp(t)}\int^t_{0} \wp(\tau)a(\tau)f(w(\tau))l(|w'(\tau)|)\di \tau  \le \Theta(T)f_\eta l_\xi < \varphi(\xi), 
\end{equation}
hence the term in round brackets in \eqref{2.6core} lies in the domain of $\varphi^{-1}$ and 
\begin{equation}\label{rap'} 
\haus[w,\sigma]'(t) = \varphi ^{-1} \left(\frac{\sigma}{\wp(t)}\int^t_{0} \wp(\tau)a(\tau)f(w(\tau))l(|w'(\tau)|)\di \tau\right) \in [0, \xi).
\end{equation}
Furthermore, $\haus[w,\sigma]'$ is continuous on $[0,T]$, hence $\haus$ is valued in $X$. By construction, $\haus[w,\sigma](T) = \sigma \eta$ and $\haus[w,0] = 0$. From 
$$
0\le\frac 1{\wp(t)}\int^t_{0}\wp(\tau)a(\tau) f(w(\tau))l(|w'(\tau)|)\di \tau \le f_\eta l_\xi \frac{1}{\wp(t)}\int^t_{0}\wp(\tau)a(\tau)\di \tau,
$$
we deduce that $\haus[w,\sigma]'(0)=0$.  Fix $\sigma \in (0,1]$, and let $w$ be a solution of $w = \haus[w,\sigma]$. We claim that $w(0)\ge 0$: otherwise, since $w(T)= \sigma \eta>0$ there would exist a
first point $t_1 \in (0,T)$ such that $w<0$ on $[0, t_1)$ and
$w(t_1)=0$, and therefore $f(w(t))=0$ on $[0, t_1]$. Thus, $w'\equiv
0$ on $[0,t_1]$ by \eqref{rap'}, that is, $w$ would be constant on $[0, t_1)$, contradicting $w(0) < 0 = w(t_1)$. From $w \ge 0$ we also deduce $w' \ge 0$ on $[0,T]$ by \eqref{rap'}. Also, \eqref{rap'} implies that $\varphi (w')$ is of class $C^{1}\big([0,T]\big)$, and then from \eqref{2.6core} that $w$ is a classical weak solution of the problem
\begin{equation}\label{2.7core} \begin{cases}
[\wp\varphi(w')]'=\sigma\wp a f(w)l(|w'|)\qquad \text{on } \, (0,T),\\[0.2cm]
w'(0)=0,\quad w(T)=\sigma \eta \\[0.2cm]
0 \le w \le \sigma \eta, \qquad 0 \le w' < \xi \qquad \text{on } \, (0,T). 
\end{cases}\end{equation}
In particular, for $\sigma =1$, $w$ is the desired solution of \eqref{2.1core}. To prove \eqref{proprie_neu_w}, let $\sigma=1$. From $w(T)= \eta$ and $w' \ge 0$ we infer the existence of a minimal $t_0 \in (0, T)$ such that $w>0$ on $(t_0,T]$. Since $f>0$ on $\R^+$, $l>0$ on $\R^+_0$ and $a>0$ on $[0,T]$, a solution of $w= \haus[w,1]$ satisfies 
$$
w'(t) = \haus[w,1]'(t) = \varphi ^{-1} \left(\frac{1}{\wp(t)}\int^t_{0} \wp(\tau)a(\tau)f(w(\tau))l(|w'(\tau)|)\di \tau\right) > 0 \qquad \forall \, t \in (t_0,T].
$$
If $t_0 \neq 0$, by the monotonicity and non-negativity of $w$ we get $w=0$ on $[0,t_0]$. To show \eqref{wder}, using $\varphi(w') \in C^1([0,T])$, $\varphi'>0$ on $\R^+$ and $w'>0$ on $(t_0,T]$ in \eqref{rap'} we deduce $w' \in C^1((t_0,T])$. Identity \eqref{wder} immediately follows by expanding the derivative in \eqref{2.1core}.\par
We assert that a solution of $w= \haus[w,1]$ exists, using again the Browder version of the Leray-Schauder theorem (see \cite[Thm 11.6]{gilbargtrudinger}).

To begin with, as already observed $\haus [w,0]\equiv 0$ for all $w\in X$. We next show  that $\mathcal H$ is
continuous on $X\times[0,1]$. Indeed, consider a sequence $\{(w_k,\sigma_k\} \subset X\times[0,1]$, with
$w_k\to w$ in $X$ and $\sigma_k\to\sigma$ as $k \ra \infty$. By continuity, $\sigma_k f(w_k)l(|w_{k}'|)\to\sigma f(w)l(|w'|)$, and so $\mathcal H[w_k,\sigma_k]\to\mathcal H[w,\sigma]$ by \eqref{2.6core} and Lebesgue convergence theorem, as required. Next we show that $\haus $ is compact. To this aim, let $\{(w_k,\sigma_k)\}$ be a bounded sequence in $X\times[0,1]$. From \eqref{rap'},
\begin{equation}\label{C'core}  
\|\mathcal H[w_k,\sigma_k]'\|_{\infty} < \xi, 
\end{equation}
and thus, since $\haus[w_k, \sigma_k](T)= \sigma \eta \in [0,\eta]$, $\{\mathcal H[w_k,\sigma_k]\}$ is equi-bounded in $X$. We shall prove that $\{\mathcal H[w_k,\sigma_k]'\}$ is equicontinuous. Set
\begin{equation}\label{def_xkyk_neu}
\begin{array}{l}
I_k(t) = \disp \int_0^t \wp(\tau)a(\tau) f(w_k(\tau)) l(|w'(\tau)|) \di \tau, \\[0.4cm]
x_k = \disp \frac{\sigma_kI_k(t)}{\wp(t)}, \qquad y_k = \frac{\sigma_k I_k(s)}{\wp(s)},
\end{array}
\end{equation}
and note that 
\begin{equation}\label{unicont}
\big|\haus[w_k,\sigma_k]'(t) - \haus[w_k,\sigma_k]'(s)\big| = \big| \varphi^{-1}(x_k)-\varphi^{-1}(y_k)\big|, 
\end{equation}
and that, by \eqref{bound_grad_neu}, 
\begin{equation}\label{ninini}
0 \le x_k, y_k < \varphi(\xi), \qquad |I_k(s)-I_k(t)| \le c|s-t| \quad \text{with } \, c = \wp_1 a_1 f_\eta l_\xi. 
\end{equation}
Since $\wp>0$ on $[0,T]$ and there it is $C^1$, we deduce 
\begin{equation}\label{catenaine}
\begin{array}{lcl}
\disp |x_k-y_k| & = & \disp \sigma_k \left|\frac{\wp(t)I_k(s)- \wp(s) I_k(t)}{\wp(t)\wp(s)}\right| \le \left|\frac{\wp(t)I_k(s)- \wp(s) I_k(t)}{\wp(t)\wp(s)}\right| \\[0.5cm]
& \le & \disp \frac{| \wp(t)-\wp(s)|}{\wp(s)\wp(t)} I_k(s) + \frac{|I_k(s)-I_k(t)|}{\wp(t)} \\[0.5cm]
& \le & \disp \left[\frac{cT}{\wp_0^2}\left(\max_{[0,T]} \wp'\right) + \frac{c}{\wp_0}\right]|s-t| = \kappa |s-t|.
\end{array}
\end{equation}
Given $\eps>0$, let $\varrho=\varrho(\varphi^{-1},\varepsilon)>0$ be given by the uniform continuity of $\varphi^{-1}$ on $[0, \varphi(\xi)]$. If $|s-t|< \varrho/\kappa$, then $|x_k-y_k|< \varrho$ and thus, by \eqref{unicont}, 
\begin{equation}\label{unicont_haus}
\big|\haus[w_k,\sigma_k]'(t) - \haus[w_k,\sigma_k]'(s)\big| = \big| \varphi^{-1}(x_k)-\varphi^{-1}(y_k)\big| < \eps, 
\end{equation}
proving the (uniform) equicontinuity of $\{\haus[w_k, \sigma_k]'\}$. The compactness of $\haus$ then follows from Ascoli-Arzel\'a theorem.\par
To apply the Leray-Schauder theorem it remains to check the existence of a constant $L>0$ such that $\|w\| \le L$ for each solution of $\haus[w,\sigma]= w$. But this immediately follows from properties \eqref{2.7core}, and indeed $\|w\| \le \eta + \xi$.\par 
%
%
It remains to consider the case when $\wp>0$ on $[0,T]$ is replaced by \eqref{wpneum}. From \eqref{rap'} and the monotonicity of $\wp$ on $[0, \delta)$, we deduce that for $t \in (0, \delta)$
$$
\begin{array}{lcl}
\haus[w,\sigma]'(t) & = & \disp \varphi^{-1} \left( \frac{\sigma}{\wp(t)}\int^t_{0}\wp(\tau)a(\tau) f(w(\tau))l(|w'(\tau)|)\di \tau\right) \\[0.5cm]
& \le & \disp \varphi^{-1} \left( \frac{f_\eta l_\xi a_1}{\wp(t)}\int^t_{0}\wp(\tau)\di \tau \right) \le  \varphi^{-1} \left(f_\eta l_\xi a_1 t \right). 
\end{array}
$$
Hence, $\haus[w,\sigma]'$ is continuous up to $t=0$ (i.e. $\haus$ is valued in $X$) and $w'(0)=0$ for each solution of $w = \haus[w,\sigma]$. The rest of the proof follows verbatim, except for the equicontinuity of $\{\haus[w_k,\sigma_k]'\}$ that we now consider. By the monotonicity of $\wp$, 
$$
|x_k| \le \frac{|I_k(t)|}{\wp(t)} \le a_1 f_\eta l_\xi t,
$$
independently of $k$. Thus, given $\eps>0$ and the corresponding $\varrho= \varrho(\varphi^{-1}, \eps)$ of the uniform continuity of $\varphi^{-1}$ on $[0, \varphi(\xi)]$, we can choose $\vartheta \in (0,T/2)$ independent of $k$ such that $|x_k| < \varrho/2$ if $t< 2\vartheta$ and $|y_k| < \varrho/2$ if $s< 2\vartheta$. Set
$$
\hat \wp_0 = \inf_{[\vartheta,T]} \wp>0, \qquad \hat \kappa = \frac{cT}{\hat \wp_0^2}\left(\max_{[0,T]} \wp'\right) + \frac{c}{\hat\wp_0}, 
$$
with $c$ as in \eqref{ninini}. Define 
$$
\bar \vartheta = \min \left\{ \frac{\varrho}{\hat \kappa}, \vartheta\right\}.
$$
Let $s,t \in [0,T]$ with $|s-t|< \bar \vartheta$. If $s,t \ge \vartheta$, then the chain of inequalities \eqref{catenaine} holds verbatim with $\hat \kappa, \hat \wp_0$ replacing $\kappa, \wp_0$, respectively, and we deduce
$$
|x_k-y_k| \le \hat \kappa|s-t| < \varrho.
$$
On the other hand, if one between $s,t$, say $t$, is less than $\bar \vartheta$, then from $\bar \vartheta \le \vartheta$ we deduce $s < t + |s-t| < 2\vartheta$. Hence, $|x_k| < \varrho/2$ and $|y_k|< \varrho/2$, and thus 
$$
|x_k-y_k| \le |x_k|+|y_k| < \varrho.
$$
In both the cases, $|x_k-y_k|< \varrho$ and therefore \eqref{unicont_haus} holds, proving the (uniform) equicontinuity of $\{\haus[w_k,\sigma_k]'\}$.
\end{proof}

The next result relates condition $w(0)=0$ to \eqref{KO_zero}. In order to do so, we shall further require \eqref{assum_secODE_altreL} in order to define $K$, $F$ as in \eqref{def_K}, \eqref{def_Fe_intro}.

\begin{proposition}\label{prop_wprimozero}
In the assumptions of Theorem \ref{exi2_neumann}, suppose further \eqref{assum_secODE_altreL} and that 
$$
\begin{array}{l}
\wp' \ge 0 \qquad \text{on } \, (0,T), \\[0.2cm]
\text{$f$ is $C$-increasing on $(0,\eta_0)$, for some constant $\eta_0> 0$.}
\end{array}
$$
If $w(0)=0$, then \eqref{KO_zero} holds.
\end{proposition}

\begin{proof}
Because of the monotonicity of $\wp'$ and $w$, from \eqref{2.1core} we deduce
$$
\big( \varphi(w')\big)' \le a f(w) l(w') \le a_1 f(w)l(w') \qquad \text{for } \, t \in (0,T).
$$
We now follow the steps in Proposition \ref{prop_twobound_refined}: differentiating on $(t_0,T)$, we deduce
$$
K'(w')w'' \le a_1 f(w)w',
$$
and integrating on $(t_0,t)$ with the aid of $w(t_0)=w'(t_0)=0$ we infer $K(w') \le a_1 F(w)$. Let $t_1 \in (t_0,t)$ be such that $a_1 F(w) < K_\infty$ for $t \in (t_0,t_1)$. This is possible, by continuity, since $F(w(t_0))=0$. Applying $K^{-1}$, integrating and changing variables we get
$$
\int_0^{w(t)} \frac{\di s}{K^{-1}(a_1 F(s))} \le (t-t_0) \qquad \text{on } \, (t_0,t_1), 
$$
and \eqref{KO_zero} follows from Lemma \ref{lem_mettimaosigmatau_novo}.
\end{proof}

We next investigate the maximal interval of definition of $w$. Assume the validity of \eqref{assum_secODE} and \eqref{assu_ancora_fl}. To tie $w$ with \eqref{KO_infinity_intro}, we assume \eqref{assum_secODE_altreL} and
\begin{equation}\label{perKO}
\frac{t \varphi'(t)}{l(t)} \not \in L^1(\infty),
\end{equation}
in order for $K$ to be a homeomorphism of $\R^+_0$ onto itself. We further replace \eqref{aepi} and \eqref{wpneum} with 
\begin{equation}\label{aepi_infty}
\left\{ \begin{array}{ll}
\disp \wp\in C^1(\R^+_0), & \qquad \wp>0, \ \ \wp' \ge 0 \quad \text{on }  \R^+, \\[0.1cm]
a \in C(\R^+_0), & \qquad a>0 \quad \text{on } \, \R^+_0.
\end{array}\right.
\end{equation}

Fix $T>0$. Applying Theorem \ref{exi2_neumann} we infer the existence of $w$ solving \eqref{2.1core} for each $\eta>0$ sufficiently small (inequality \eqref{bound_etaxi_neu} is always satisfied for small $\eta$ since, by \eqref{assu_ancora_fl}, $f_\eta \ra 0$ as $\eta \ra 0$). From $w'(T)>0$, we conclude that $w$ can be extended on a maximal interval $[0, R)$. Our next task is to prove that, if the Keller-Osserman condition \eqref{KO} is violated, then $R=\infty$.

%

\begin{proposition}\label{prop_Rinfty}
Assume \eqref{assum_secODE}, \eqref{assum_secODE_altreL}, \eqref{assu_ancora_fl} and \eqref{perKO}. Let $a, \wp$ satisfy \eqref{aepi_infty}. For a fixed $T>0$, consider the solution $w$ of \eqref{2.1core} for small positive $\eta$ and let $[0,R)$ be the maximal interval where $w$ is defined. If $t_0$ is as in Theorem \ref{exi2_neumann}, then
\begin{equation}\label{proprie_w_new}
w = w(t_0) \quad \text{on } \, [0,t_0], \qquad w>0, \quad w'>0 \quad \text{on } \, (t_0, R). 
\end{equation}
Furthermore, suppose that $f$ is $C$-increasing on $(\bar\eta_0, \infty)$ for some $\bar\eta_0 \ge 0$. If
\begin{equation}\label{notKO}\tag{$\neg$KO$_\infty$}
\frac{1}{K^{-1} \circ F} \not \in L^1(\infty),
\end{equation}
then $R = \infty$.
\end{proposition}

\begin{proof}
Taking into account the sign of $f,l,a$, by \eqref{2.1core} $\wp\varphi(w')$ is $C^1$ and strictly increasing where $w$ is positive, and from \eqref{proprie_neu_w} it readily follows that $\wp\varphi(w')>0$ on $(t_0, R)$. Properties \eqref{proprie_w_new} are then immediate from \eqref{proprie_neu_w}, $w'(T)>0$ and our assumptions on $\varphi$. Next, suppose by contradiction that $R<\infty$. We first claim that necessarily
\begin{equation}\label{eq:43}
w^* = \lim_{t \to R^-}w(t)=\infty,
\end{equation}
where the existence of the limit is guaranteed by the monotonicity of $w$. To prove \eqref{eq:43}, assume by contradiction that $w^* <\infty$. Because of \eqref{wder} and since $\wp' \ge 0$,
$$
\frac{\varphi'(w') w''}{l(w')} \le af(w) \le a_1 f(w) \qquad \text{on } \, (t_0,R),
$$
where we set $a_1 = \|a\|_{L^ \infty([0,R])}$. Multiplying by $w'$, integrating on $(t_0, t)$ and changing variables we deduce $K(w') \le a_1F(w)$ (we recall that $w'(t_0)=0$ by \eqref{2.1core}). Thus, $w'$ is bounded in $(R/2,R)$, namely $\|w'\|_\infty \le K^{-1}(a_1 F(w^*)) = L$. For $t,s \in (R/2,R)$, define $I_k, x_k, y_k$ as in \eqref{def_xkyk_neu}, and note that 
$$
|x_k| + |y_k| \le 2 a_1 f_{w^*} l_L R = \bar C
$$
Given $\eps>0$ let $\varrho=\varrho(\varphi^{-1},\eps)$ be given by the uniform continuity of $\varphi^{-1}$ on $[0, \bar C]$. Proceeding as in \eqref{catenaine} we deduce the existence of $\kappa>0$ such that
$$
|x_k-y_k| \le \kappa |s-t| \qquad \text{for each } \, s,t \in \left( \frac{R}{2},R\right).
$$
If $|s-t|< \varrho/\kappa$, then $|x_k-y_k|<\varrho$ and so
$$
|w'(t)-w'(s)| = \big| \varphi^{-1}(x_k)-\varphi^{-1}(y_k)\big| < \eps.
$$
In conclusion, $w'$ is uniformly continuous on $(R/2,R)$, and can be therefore extended by continuity at $t=R$. By the existence theory for ODEs, $w$ would be further extendible beyond $R$, contradiction. This proves that $w(R^-) = \infty$. Now, fix $T_1  \in (0,R)$ large enough that $w(T_1)> \bar \eta_0$. Applying $K^{-1}$ to inequality $K(w') \le a_1 F(w)$ on $(T_1,R)$, rearranging, integrating on $[T_1,t)$ and changing variables we get
\begin{equation}\label{ide_neu}
\int_{w(T_1)}^{w(t)} \frac{\di s}{K^{-1}(a_1 F(s))} \le t-T_1.
\end{equation}
Since \eqref{notKO} holds and, by assumption, $f$ is $C$-increasing on $(\bar \eta_0, \infty)$, applying Lemma \ref{lem_mettimaosigmatau_novo} we deduce that the left-hand side of \eqref{ide_neu} is unbounded as $t \ra R^-$ while the right-hand side is not, contradiction.
\end{proof}

\section{Comparison results and the finite maximum principle}\label{sec2.4}

\subsection{Basic comparisons and a pasting lemma}

In this subsection, we collect two comparison theorems and a ``pasting lemma" for $\lip_\loc$ solutions that will be repeatedly used in the sequel. Throughout the section we assume
\begin{equation}\label{assu_percompaepasting}
\varphi \in C(\R^ +_0), \qquad \varphi(0) = 0, \qquad \varphi>0 \ \text{ on } \, \R^+. \\[0.2cm]
\end{equation}


The first comparison is Proposition~6.1 of~\cite{prsmemoirs}, see also Theorem 2.4.1 of \cite{pucciserrin}.

\begin{proposition}\label{prop_comparison}
Assume \eqref{assu_percompaepasting} and that $\varphi$ is strictly increasing on $\R^+$. Let $\Omega \subset M$ be open, and suppose that $u$, $v \in \lip_\loc(\Omega) \cap C(\overline \Omega)$ solve
\begin{equation}\label{prob_simple}
\left\{\begin{array}{l}
\Delta_\varphi u \ge \Delta_\varphi v \qquad \text{weakly in }  \Omega, \\[0.3cm]
\disp u \le v \quad \text{on }  \partial \Omega,
\end{array}\right.
\end{equation}
and $\disp\limsup_{x \in \Omega, \, x \ra \infty}\big( u(x) -v(x)\big) \le 0$ if $\Omega$ has non-compact closure. Then $u \le v$ in $\Omega$.
\end{proposition}

Our second comparison result is a special case of \cite[Thm. 5.6]{antoninimugnaipucci}, see also \cite[Thm. 3.6.5]{pucciserrin}.
\begin{proposition}\label{prop_serrin}
Let $\varphi, f, l$ satisfy
$$
\begin{array}{l}
\varphi \in C(\R^+_0) \cap C^1(\R^ +), \qquad \varphi(0)=0, \qquad \varphi'>0 \quad \text{on } \, \R^+ \\[0.2cm]
f \in C(\R), \qquad \text{$f$ is non-decreasing on $\R$,} \\[0.2cm]
l \in C(\R^+_0) \cap \lip_\loc(\R^+), \qquad l>0 \ \text{ on } \, \R^+.
\end{array}
$$
Let $\Omega \subset M$ be an open subset of a Riemannian manifold $M$, and fix $0<b(x) \in C(\overline\Omega)$. If $u,v \in \lip_\loc(\Omega)\cap C(\overline{\Omega})$ solve
\begin{equation}\label{confronto_conbfl}
\left\{\begin{array}{l}
\Delta_\varphi u \ge b(x) f(u) l(|\nabla u|) \qquad \text{on } \, \Omega, \\[0.2cm]
\Delta_\varphi v \le b(x) f(v) l(|\nabla v|) \qquad \text{on } \, \Omega, \\[0.2cm]
u \le v \qquad \text{on } \, \partial \Omega \\[0.2cm]
\ess\inf_K \Big\{|\nabla v|+|\nabla u|\Big\} >0  \qquad \text{for each } \, K \Subset \Omega.
\end{array}\right.
\end{equation}
Then, $u \le v$ on $\Omega$.
\end{proposition}
\begin{remark}
\emph{Condition $\ess\inf_K \big\{|\nabla v|+|\nabla u|\big\} > 0$ for each $K \Subset \Omega$ cannot be avoided, as the counterexample in Remark 1, p. 79 of \cite{pucciserrin} shows. However, the restriction can be removed if $\Delta_\varphi$ is strictly elliptic, see Section 3.5 of \cite{pucciserrin} for definitions and relevant results.
}
\end{remark}

\begin{remark}\label{rem_regularitymetric}
\emph{The underlying metric is not required to be smooth, and indeed a metric whose local matrix $g_{ij}$ is continuous is sufficient.
}
\end{remark}

\begin{remark}
\emph{It is worth to comment on \cite[Thm. 5.6]{antoninimugnaipucci}. There, the authors consider solutions of more general quasilinear inequalities of the form
$$
\diver \mathbf{A}(x,\nabla u) \ge \mathcal{B}(x,u,\nabla u) \qquad \text{and}\qquad \diver \mathbf{A}(x,\nabla v) \le \mathcal{B}(x,v,\nabla v),
$$
for suitable Carathe\'odory maps $\mathbf{A}, \mathcal{B}$. In our setting,
$$
\mathbf{A}(x,\xi) = \frac{\varphi(|\xi|)}{|\xi|} \xi,
$$
thus the positivity of $\varphi(s)/s$ and $\varphi'(s)$ on $\R^+$ imply that the tangent map $\mathbf{A}_*$ of $\mathbf{A}$ is uniformly positive definite on compacta of fibers of $T\Omega \backslash\{\mathbf{0}\} \ra \Omega$, a condition needed to apply Lemma 5.7 in \cite{antoninimugnaipucci}. The regularity of $\mathcal{B}$, defined at the end of p.592 therein, is equivalent to condition $l \in \lip_\loc(\R^+)$.
}
\end{remark}

To conclude, we discuss the pasting lemma. It is well-known that the maximum of two subharmonic functions is still subharmonic. For subsolutions of more general operators the situation is more delicate and we refer to \cite{le} for a very general result, and to  \cite[Appendix]{bmr5} for an alternative approach via obstacle problems, in the setting of homogeneous operators. For our purposes, it is sufficient to consider the case in which one of the solutions is constant. The technique goes back to T. Kato in \cite{kato}, and has been generalized to a large class of quasilinear operators in \cite{dambrosiomitidieri_2}. The next result is special case of \cite[Thm. 2.1]{dambrosiomitidieri_2}.

\begin{lemma}\label{lem_pasting}
Assume \eqref{assu_percompaepasting}, and let $f\in C(\mathbb R)$, $l \in C(\R^+_0)$ with $f(0)l(0)=0$. Suppose furthermore than $b\in C(M)$. If $u \in \lip_\loc(M)$ is a nontrivial solution of
\begin{equation}\label{ineqb}
\Delta_\varphi u \ge b(x)f(u)l(|\nabla u|) \qquad \text{on } \, M,
\end{equation}
then $u_+= \max\{u,0\}$ is a $\lip_\loc(M)$, non-negative solution of
$$
\Delta_\varphi u_+ \ge b(x)f(u_+)l(|\nabla u_+|) \qquad \text{on } \, M.
$$
\end{lemma}

\begin{remark}
\emph{Note that $u=0$ is a solution of \eqref{ineqb} since $f(0)l(0)=0$.
}
\end{remark}

\subsection{The finite maximum principle}
We now prove Theorem \ref{teo_FMP2_intro} in the Introduction. For the convenience of the reader, we rewrite the assumptions and restate the result. We require
\begin{equation}\label{assumptions_FMP}
\left\{\begin{array}{l}
\disp \varphi \in C(\R^+_0)\cap C^1(\R^ +), \qquad \varphi(0)=0, \qquad \varphi' > 0 \ \text{ on } \, \R^+; \\[0.2cm]
\disp \frac{t \varphi'(t)}{l(t)} \in L^1(0^+) \\[0.4cm]
f \in C(\R), \qquad f \ge 0 \quad \text{on } \, (0,\eta_0);\\[0.3cm]
l \in C(\R_0^+), \qquad l>0 \ \text{on } \, \R^+.
\end{array}\right.
\end{equation}

\begin{theorem}\label{teo_FMP2}
Let $M$ be a Riemannian manifold, and assume that $\varphi, f$ and $l$ satisfy \eqref{assumptions_FMP} and moreover
\begin{itemize}
\item[-] $f(0)l(0)=0$;
\item[-] $f$ is $C$-increasing on $(0,\eta_0)$, with $\eta_0$ as in \eqref{assumptions_FMP};
\item[-] $l$ is $C$-increasing on $(0,\xi_0)$, for some $\xi_0>0$;
\end{itemize}
Fix a domain $\Omega \subset M$ and let $0< b \in C(\Omega)$. Then, $\fmp$ holds on $\Omega$ if and only if either
\begin{equation}\label{lasimple_2}
f\equiv 0 \qquad \text{on } \, (0,\eta_0),
\end{equation}
for some $\eta_0>0$, or
\begin{equation}\label{noninteinzero_4}
f>0 \quad \text{on } \, (0,\eta_0), \quad \text{and} \qquad \frac{1}{K^{-1}\circ F} \not \in L^1(0^+).
\end{equation}
\end{theorem}

\begin{proof} We recall that the validity of $\fmp$ means that for any solution $u \in C^ 1(\Omega)$ of
\begin{equation}\label{Ple2}
\left\{ \begin{array}{l}
\Delta_\varphi u \le b(x)f(u)l(|\nabla u|) \qquad \text{on } \, \Omega\\[0.2cm]
u \ge 0 \qquad \text{on } \, \Omega,
\end{array}\right.
\end{equation}
if $u(x_0)=0$ for some $x_0 \in \Omega$ then $u \equiv 0$. The argument follows the lines of the proof in \cite{pucciserrin}, with the help of a trick from \cite{maririgolisetti}. We prove separately the sufficiency and necessity of \eqref{lasimple_2},\eqref{noninteinzero_4}. First, having fixed a point $o \in \Omega$ to be specified later, we choose $R,\kappa>0$ such that $\overline{B_{2R}(o)}$ does not intersect $\cut(o)$ and
$$
\Ricc(\nabla r, \nabla r) \ge -(m-1) \kappa^2 \qquad \text{on } \, \overline{B_{2R}(o)} \backslash \{o\},
$$
where $r(x) = \dist(x,o)$. By the Laplacian comparison theorem (see Thm. \ref{teo_laplaciancomp} in Appendix A), denoting with $v_{g_\kappa}(r)$ the volume of a geodesic sphere of radius $r$ in a model manifold of sectional curvature $-\kappa^2$
\begin{equation}\label{deltar_baixo2}
\Delta r \le \frac{v_{g_\kappa}'(r)}{v_{g_\kappa}(r)} = (m-1)\kappa \coth(\kappa r) \qquad \text{on } \, \overline{B_{2R}(o)} \backslash \{o\}. \\[0.2cm]
\end{equation}
\emph{Sufficiency of conditions \eqref{lasimple_2} and \eqref{noninteinzero_4}}.\\[0.2cm]
For $o \in \Omega$ and having set $R,\kappa$ as above, fix $a_1 \in \R^+$ such that $b(x) \le a_1$ on $B_{2R}(o)$. Let $C \ge 1$ be the constant defining the $C$-increasing property of $f$. Define $\wp(t) = v_{g_\kappa}(2R-t)$, $a(t) = a_1$, $T_0 = 3R/2$ and $T=R$. We claim that we can suitably reduce $R$, and choose $\eta \in (0, \eta_0)$ small enough, in such a way that
\begin{equation}\label{ipoTeta}
\frac{\wp_1}{\wp_0}\varphi\left(\frac{\eta}{R}\right) + 4C\Theta(R)f_\eta l_{\xi} < \varphi(\xi),
\end{equation}
where $f_\eta$, $l_{\xi}$, $\wp_0$, $\wp_1$ and $\Theta$ are defined as in \eqref{227}. Indeed, since $v_{g_\kappa}' \ge 0$, by definition $\wp_1 = v_{g_\kappa}(2R)$, $\wp_0 = v_{g_\kappa}(R/2)$, and thus $\wp_1/\wp_0 \ra 4^m$ and $\Theta(R) \ra 0$ as $R \ra 0$. Hence, we can first reduce $R$ to guarantee
$$
4C\Theta(R)f_{\eta_0} l_{\xi} < \frac{\varphi(\xi)}{2},
$$
and then choose $\eta \in (0, \eta_0)$ small enough to satisfy \eqref{ipoTeta}.
Applying Theorem \ref{exi2}, there exists a solution $z(t)$ of
\begin{equation}\label{twoboundary_simpler_2}
\left\{ \begin{array}{l}
\big(\wp(t) \varphi(z_t)\big)_t = 2C \wp a_1 f(z)l(z_t) \qquad \text{on } \, (0, R), \\[0.2cm]
z(0)=0, \qquad z(R)= \eta,
\end{array}\right.
\end{equation}
where the subscript $t$ indicates differentiation with respect to $t$. Furthermore, combining Theorem \ref{exi2} and Proposition \ref{prop_twobound_refined} (note that here $\wp' < 0$), the solution $z$ satisfies the following properties:
\begin{equation}\label{property_z}\begin{gathered}
z>0 \ \text{ on } \, (0,R], \qquad  z_t>0 \ \text{ on } \, [0,R], \qquad \|z_t\|_\infty < \xi.
 \end{gathered}\end{equation}
Taking into account that we have extended $\varphi$ on $\R$ in such a way that $\varphi(-s) = -\varphi(s)$, the function $w(r) = z(2R-r)$ satisfies
\begin{equation}\label{twoboundary_simpler_2}
\left\{ \begin{array}{l}
\big[v_g\varphi(w')\big]' = 2C v_g a_1f(w)l(|w'|) \qquad \text{on } \, \left(R,2R\right), \\[0.2cm]
w(2R)=0, \quad w(R)= \eta, \quad w'<0 \ \text{ on } \, \left[R,2R\right].
\end{array}\right.
\end{equation}
Define $v(x) = w(r(x))$. Using \eqref{deltar_baixo2}, together with $\varphi(w') < 0$ and \eqref{twoboundary_simpler_2}, we deduce that $v$ solves
\begin{equation}\label{lasub}\begin{aligned}
\disp  \Delta_\varphi v & = \disp \big(\varphi(w')\big)' + \varphi(w') \Delta r \ge v_g^{-1}(r)\big[ v_g(r) \varphi(w')\big]' \\[0.3cm]
& =  v_g(r)^{-1}\big(2C v_g(r) a_1f(w)l(|w'|) \ge 2C b(x)f(v) l(|\nabla v|)
 \end{aligned}
\end{equation}
on the annulus $E_R(o) = B_{2R}(o)\backslash \overline{B_R(o)}$. Moreover, denoting with $\nu$ the outward pointing unit normal from $\partial B_{2R}(o)$,
\begin{equation}\label{deripositive}
\langle \nabla v, \nu \rangle = w'(2R) < 0 \qquad \text{on } \, \partial B_{2R}(o).
\end{equation}
Following E. Hopf's argument, we now prove that $u$ solving \eqref{Ple2} is identically zero provided that $u(x_0)=0$ at some~$x_0$. Suppose by contradiction that this is not the case, and let $\Omega_+ = \{x \in \Omega : u(x)>0\}$. Choose $x_1\in \Omega_+$ in such a way that $\dist(x_1, \partial \Omega_+) > \dist(x_1, \partial \Omega)$, and let $B(x_1) \subset \Omega_+$ be the largest ball contained in $\Omega_+$. Then, $u>0$ in $B(x_1)$, while $u(\bar x) =0$ for some $\bar x \in \partial B(x_1) \cap \Omega$. Clearly $\nabla u(\bar x) = 0$, since $\bar x$ is an absolute minimum for $u$. Take a unit speed minimizing geodesic $\gamma : [0, \mathrm{dist}(x_1,\bar x)] \ra \Omega$ from $x_1$ to $\bar x$. Up to choosing the arclength parameter $s$ sufficiently close to $\dist(x_1,\bar x)$ and setting $2R= \dist (x_1,\bar x)-s$, the closure of the ball $B_{2R}(o)$ centered at $o=\gamma(s)$ does not intersect $\cut(o)$, and $B_{2R}(o) \subset B(x_1)$, with $\bar x \in \partial B_{2R}(o)$. We consider the function $v$ constructed above on $E_{R}(o) \subset B_{2R}(o)$, with $\eta$ small enough to satisfy \eqref{ipoTeta} and also
\begin{equation}\label{lowercircle}
\eta < \inf_{\partial B_{R}(o)} u.
\end{equation}
We claim that $u \ge v$ on $E_R(o)$. Otherwise, suppose that
$$
\disp \max_{E_R(o)} (v-u) = \bar\delta >0,
$$
and let $\Gamma$ be the set of maximum points of $v-u$. Note that $\Gamma \Subset E_R(o)$ and that $\nabla u= \nabla v$ for each $x \in \Gamma$. For $\delta \in (0, \bar \delta)$, set $U_{\delta} = \{v-u > \delta\}$. By construction, there exists $\eps>0$ such that $\eps \le |\nabla v| \le 1$ on $E_R(o)$, and since $l>0$ on $\R^+$ we deduce that the quotient $l(|\nabla u|)/l(|\nabla v|)$ is continuous on $E_R(o)$ and equal to $1$ on $\Gamma$. A compactness argument shows that, for $\delta$ sufficiently close to $\bar \delta$, $l(|\nabla u|)/l(|\nabla v|) \le 2$ on $U_\delta$. Taking into account that the $C$-increasing property of $f$ implies $f(v) \ge C^{-1}f(u)$ on $U_\delta$, we deduce
$$
\Delta_\varphi u \le b(x)f(u)l(|\nabla u|) \le 2C b(x)f(v)l(|\nabla v|) \le \Delta_\varphi v \qquad \text{on } \, U_\delta.
$$
From $v = u + \delta$ on $\partial U_\delta$, by the comparison Proposition \ref{prop_comparison} we get $v \le u+ \delta$ on $U_\delta$, contradicting the very definition of $U_\delta$. Hence, $v \le u$ on $E_R(o)$, and in particular $\langle \nabla (u-v), \nu \rangle \le 0$ at $\bar x$, which is impossible by \eqref{deripositive} and by the fact that $\nabla u(\bar x) =0$. This contradiction concludes the proof of the sufficiency part.\\[0.2cm]
\emph{Necessity of conditions \eqref{lasimple_2} and \eqref{noninteinzero_4}}.\\[0.2cm]
Suppose the failure of both \eqref{lasimple_2} and \eqref{noninteinzero_4}, or equivalently (recall that $f$ is $C$-increasing) the validity of
\begin{equation}\label{KO_zero_section}\tag{KO$_0$}
f>0 \quad \text{on } \, (0, \eta_0), \qquad \frac{1}{K^{-1} \circ F} \in L^ 1(0^+).
\end{equation}
For each $o \in M$, we shall now construct on $B_{2R}(o)$ (with $R$ as in the beginning of the proof) a $C^1$ non-negative, nonzero solution $u$ of $(P_\le)$ with $u=0$ on $\overline{B_R(o)}$, contradicting $\fmp$. Set $T_0=2R$, $T=R$, 
$$
\wp(t) = v_{g_\kappa}(t+R), \quad a(t) = \inf_{B_{2R}(o) \backslash B_R(o)} b, 
$$
where $v_{g_\kappa}(r)$ is the volume of a geodesic sphere of radius $r$ in a model of curvature $-\kappa^2$, as defined at the beginning of the proof. Since $f(0)l(0)=0$, we can choose $\eta$ small enough to satisfy  \eqref{restrict} in Theorem \ref{exi2}, whence there exists $w \in C^1([0,T])$ non-decreasing and solving
\begin{equation}\label{twobound_FMP}
\left\{ \begin{array}{l}
\disp \big[ \wp \varphi(w')\big]' = \wp a f(w)l(w') \qquad \text{on } \, (0,T), \\[0.2cm]
0 \le w \le \eta, \quad w' \ge 0 \qquad \text{on } \, [0,T], \\[0.2cm]
w(0)=0, \qquad w(T)=\eta.
\end{array}\right.
\end{equation}
Up to reducing $\eta$ further, we can apply Proposition \ref{prop_wprimougualezero} to deduce that $w'(0)=0$. Set $u(x) = w\big(r(x)-R\big)$. By the Laplacian comparison theorem and since $w' \ge 0$, on $B_{2R}(o) \backslash B_R(o)$ it holds
$$
\begin{array}{lcl}
\Delta_\varphi u & = & \disp \big(\varphi(w')\big)' + \varphi(w')\Delta r \le \disp \big(\varphi(w')\big)' + \varphi(w')\frac{v_{g_\kappa}'(r)}{v_{g_\kappa}(r)} \\[0.3cm]
& \le & \disp \Big(\wp^{-1} \big[ \wp \varphi(w')\big]'\Big)_{r(x)-R} \le a f(w)l(w') \\[0.3cm]
& \le & b(x) f(u)l(|\nabla u|).
\end{array}
$$
Extending $u$ to be zero on $B_R(o)$ defines a nonzero, $C^1$-solution of $(P_\le)$ on all of $B_{2R}(o)$, which clearly violates $\fmp$.
\end{proof}

\begin{remark}
\emph{The function $u$ in the proof of the necessity part, defined on $B_{2R}(o)$ and vanishing identically on a smaller ball $B_R(o)$, is an example of a \emph{dead core} (super)solution. For a thorough investigation of dead core problems, we refer the reader to \cite{pucciserrin} and the references therein. We mention that $u$ can even be constructed to be positive on $B_{2R}(o)\backslash \overline{B_R(o)}$. Indeed, it is enough to replace the solution $w$ of \eqref{twobound_FMP} with the supersolution $z$ defined in the proof of Proposition \ref{prop_wprimougualezero}, which is known to be positive on $(0,T]$.
}
\end{remark}

\begin{remark}
\emph{With a similar technique, one could consider the $\fmp$ for more general equations of the type
$$
\Delta_\varphi u \le b(x)f(u)l(|\nabla u|) + \bar b(x) \bar f(u) \bar l(|\nabla u|),
$$
for suitable $\bar b, \bar f, \bar l$. The prototype case 
$$
\Delta_p u \le f(u) + |\nabla u|^q, \qquad \text{with } \, p>1
$$
has been considered in \cite[Thm. 5.4.1]{pucciserrin} and \cite{pucciserrinzou} when $q \ge p-1$, and in  \cite{felmermontenegroquaas} for $q < p-1$.
}
\end{remark}

\section{Weak maximum principle, a-priori estimates and Liouville's property}\label{sec_WMP}

\subsection{The equivalence between $\wmp$ and $\lio$}

In this section, we prove a more general version of Proposition \ref{prop_equivalence}, that describes the relationship between $\wmp$ and the Liouville property $\lio$. We begin by introducing another form of the weak maximum principle.

\begin{definition}
\emph{Assume \eqref{assumptions} and fix $b,l$ satisfying \eqref{assumptions_bfl}. We say that
\begin{itemize}
\item the \emph{open weak maximum principle} $\owmp$ holds for $(bl)^{-1}\Delta_\varphi$ if, for each $f \in C(\R)$, for each open set $\Omega \subset M$ with $\partial \Omega \neq \emptyset$ and for each $u \in C(\overline \Omega) \cap \lip_\loc (\Omega)$ solving
\begin{equation}\label{problem_OWMP}
\left\{ \begin{array}{l}
\Delta_\varphi u \ge b(x) f(u) l(|\nabla u|) \qquad \text{on } \, \Omega, \\[0.2cm]
\sup_\Omega u < \infty,
\end{array}\right.
\end{equation}
we have that
\begin{equation}\label{sup_attainedboundary}
\text{either} \quad \sup_{\Omega} u = \sup_{\partial \Omega} u, \qquad \text{or} \quad f\big( \sup_\Omega u \big) \le 0.
\end{equation}
\end{itemize}
}
\end{definition}

The open weak maximum principle at infinity has been introduced in \cite{aliasmirandarigoli} in the study of immersed submanifolds of warped product ambient spaces, and parallels Ahlfors' definition of parabolicity. For a detailed investigation and an extensive bibliography, we refer to Chapters 3 and 4 of \cite{AMR}.

\begin{remark}
\emph{The recent \cite{maripessoa} contains a different approach to maximum principles at infinity in the spirit of $\owmp$ (called there \emph{the Ahlfors property}), which is based on viscosity theory and enables to consider classes of fully nonlinear operators of geometric interest. A systematic approach via viscosity theory is very well suited to treat weak and strong maximum principles in a unified way, and especially to investigate principles like the classical Ekeland or $\smp$,  where a gradient condition on $u$ appears.
}
\end{remark}

%

%

\begin{proposition}\label{prop_equivalence_2}
Let $\varphi$ and $b,f,l$ satisfy respectively the assumptions in \eqref{assumptions_bfl}, \eqref{assumptions}. Then, the following properties are equivalent:
\begin{itemize}
\item[(i)] $(bl)^{-1}\Delta_\varphi$ satisfies $\wmp$;
\item[(ii)] $\lio$ holds for some (equivalently, every) $f$ satisfying
$$
f(0)=0, \qquad f>0 \quad \text{on } \, \R^+;
$$
\item[(iii)] each non-constant $u \in \lip_\loc(M)$ solving $(P_\ge)$ on $M$ and bounded above satisfies $f(u^*) \le 0$, with $u^* = \sup_M u$;
\item[(iv)] $(bl)^{-1}\Delta_\varphi$ satisfies $\owmp$.
\end{itemize}
\end{proposition}

\begin{proof}
We prove the chain of implications $(i) \Rightarrow (iii) \Rightarrow (ii) \Rightarrow (i)$ and then $(iv) \Leftrightarrow (i)$. When the ``some-every" alternative occurs, we always assume the weaker property and prove the stronger.\\
$(i) \Rightarrow (iii)$.\\
Let $u \in \lip_\loc(M)$ be a non-constant solution of $(P_\ge)$ that is bounded from above, and assume by contradiction that $f(u^*) = 2K >0$. By continuity, there exists $\eta < u^*$ sufficiently close to $u^*$ such that $f(u) \ge K$ on $\Omega_\eta = \{u> \eta\}$, thus
$$
\Delta_\varphi u \ge K b(x)l(|\nabla u|) \qquad \text{on } \, \Omega_\eta.
$$
The definition of $\wmp$ then implies $K \le 0$, contradiction.\\
$(iii) \Rightarrow (ii)$.\\
Let $u \in \lip_\loc(M)$ be a non-constant, bounded, non-negative solution of~$(P_\ge)$. Then $f(u^*)\le 0$ by $(iii)$. However, from $f>0$ on $\R^+$ we get $u^* \equiv 0$, that is, $u \equiv 0$ is constant, contradiction.\\
$(ii) \Rightarrow (i)$.\\
Let us consider a problem $(P_\ge)$ with $f>0$ on $\R^+$, $f(0)=0$ for which $\lio$ holds. Suppose by contradiction that $(i)$ is not satisfied, that is, there exists a non-constant $u \in \lip_\loc(M)$ with
$$
\Delta_\varphi u \ge K b(x) l(|\nabla u|) \qquad \text{on some } \, \Omega_{\bar \eta},
$$
for some $K>0$ and $\bar \eta <u^*$. Since $f(0)=0$, we can choose $\eta \in
(\bar \eta,u^*)$ in such a way that $f<K$ on $(0,u^*-\eta)$. Hence, $u_\eta = u-\eta$ solves
$$
\Delta_\varphi u_\eta \ge K b(x)l(|\nabla u_\eta|) \ge b(x)f(u_\eta)l(|\nabla u_\eta|) \qquad \text{on }
\Omega_\eta.
$$
Thanks to Lemma \ref{lem_pasting}, $w = \max\{u_\eta,0\}$ is a
non-constant, non-negative, bounded solution of $\Delta_\varphi w \ge b(x)f(w)l(|\nabla w|)$ on $M$, contradicting property $\lio$ for such an $f$.\\
$(i) \Rightarrow (iv)$.\\
If $u$ solves \eqref{problem_OWMP} but none of the properties in \eqref{sup_attainedboundary} holds, then by continuity we can choose  $\eta \in (\sup_{\partial \Omega} u, \sup_\Omega u)$ such that $f(u) \ge K>0$ on $\{u>\eta\}$. Thus, $u$ solves $\Delta_\varphi u \ge Kb(x)l(|\nabla u|)$ on $\{u>\eta\} \subset \Omega$, contradicting $\wmp$.\\
$(iv) \Rightarrow (i)$.\\
Assuming that $\wmp$ does not hold, take a non-constant $u$ which is bounded above and solves $\Delta_\varphi u \ge K b(x)l(|\nabla u|)$ on some set $\Omega_\eta = \{u>\eta\}$, for some $K>0$. If $\eta$ is close enough to $u^*$, $\partial \Omega_\eta \neq \emptyset$, thus  clearly $u$ contradicts $\owmp$ on $\Omega_\eta$ with the choice $f=K$.
%
%
%

\end{proof}

\subsection{Volume growth and $\wmp$}

In this section, we explore geometric conditions that ensure the validity of $\wmp$, in the form given by $(iii)$ of Proposition \ref{prop_equivalence}, that is, each non-constant $u \in \lip_\loc(M)$ solving $(P_\ge)$ and bounded above satisfies $f(u^*) \le 0$. Here, as usual, $u^*= \sup_M u$, and similarly we will use the notation $u_* = \inf_M u$.\\
When $l \equiv 1$, the problem has been tackled in a series of papers \cite{karp,rigolisalvatorivignati_3,prs_gafa,prsmemoirs} by means of integral methods, and in particular we refer to \cite{prsmemoirs} for a thorough discussion. Since then, in a manifolds setting, the first results that we are aware of allowing a nontrivial $l$ appeared in \cite{maririgolisetti}, where $l$ is assumed to be $C$-increasing (in fact, polynomial in Thm. 5.1 therein). In particular, the relevant case when $l$ can vanish both as $t \ra 0^+$ and as $t \ra \infty$ seems to be still open even in the Euclidean space, although it has been recently considered for Carnot groups in \cite{AMR}. It is a remarkable feature that the results quoted above ensure $f(u^*) \le 0$ not only when $u$ is a-priori bounded above, but also when $u$ does not grow too fast at infinity, in the sense that in this case $u^*$ is also shown to be finite. This is a natural condition, and its origin is related to the growth of an explicit Khas'minskii-type potential for the operator considered, see Section 4 of \cite{prsmemoirs} for more information. Related interesting results on $\R^m$ can be found in \cite{farinaserrin1, farinaserrin2, pucciserrin_2, DAmbrMit} and will be described in more detail later.

\par
The next theorem improves on \cite[Thm 5.1]{maririgolisetti} and \cite[Thm. 2.1]{AMR}. Throughout this section, we assume the following growth conditions:
\begin{equation}\label{assu_WPC}
\begin{array}{ll}
\text{there exist constants $p, \bar p > 1$, $C, \bar C>0$ such that } \\[0.2cm]
\varphi(t) \le Ct^{p-1} \ \text{ on } \, [0,1], \qquad \varphi(t) \le \bar C t^{\bar p-1} \quad \text{on } \, [1,\infty).
\end{array}
\end{equation}
If $p= \bar p$, \eqref{assu_WPC} is called the \emph{weak $p$-coercivity} of $\Delta_\varphi$ in \cite{DAmbrMit}. We will explain in depth the different role played by $p$ and $\bar p$ for the validity of $\wmp$.

%
%
\begin{theorem}\label{teo_main_2}
Let $M$ be a complete Riemannian manifold, and consider $\varphi, b,f$ and $l$ meeting the requirements in  \eqref{assumptions_bfl}, \eqref{assumptions} and the bounds  \eqref{assu_WPC} for some $p, \bar p>1$. Assume that, for some $\mu,\chi \in \R$ satisfying
\begin{equation}\label{pararange_2}
\chi \ge 0, \qquad \mu \le \chi+1,
\end{equation}
the following inequalities hold:
\begin{equation}\label{assum_main_2'}
\begin{array}{ll}
b(x) \ge C\big(1+r(x)\big)^{-\mu} & \quad \text{on } \, M, \\[0.2cm]
f(t) \ge C & \quad \text{for } \, t \, \gg 1 \\[0.2cm]
\disp l(t) \ge C\frac{\varphi(t)}{t^{\chi}} & \quad \text{on } \, \R^+,
\end{array}
\end{equation}
for some constant $C>0$. Let $u \in \lip_\loc(M)$ be a non-constant, weak solution of $(P_\ge)$ such that either
\begin{itemize}
\item[(i)] $u$ is bounded above and one of the following properties hold:
\begin{equation}\label{volgrowth_sigmazero}
\begin{array}{lll}
\mu < \chi+1 & \text{and} & \disp \qquad \liminf_{r \ra \infty} \frac{\log\vol(B_r)}{r^{\chi+1-\mu}} < \infty \quad (=0 \, \text{ if } \, \chi=0); \\[0.4cm]
\mu = \chi+1 & \text{and} & \disp \qquad \liminf_{r \ra \infty} \frac{\log\vol(B_r)}{\log r} < \infty \quad (\le p \, \text{ if } \, \chi=0).
\end{array}
\end{equation}
\item[(ii)] $u$ satisfies
\begin{equation}\label{opequeno}
u_+(x) = o\big(r(x)^\sigma\big) \qquad \text{as } \, r(x) \ra \infty,
\end{equation}
for some $\sigma>0$ such that
\begin{equation}\label{ipo_sigma_teomain2}
\chi \sigma \le \chi+1-\mu,
\end{equation}
and one of the following properties hold:
\begin{equation}\label{volgrowth_sigmamagzero}
\begin{array}{ll}
\chi \sigma < \chi+1-\mu, & \quad \disp \liminf_{r \ra \infty} \frac{\log\vol(B_r)}{r^{\chi+1-\mu -\chi\sigma}} < \infty \quad (=0 \, \text{ if } \, \chi=0);\\[0.4cm]
\chi \sigma = \chi+1-\mu, \ \ \ \chi>0, & \quad \disp \liminf_{r \ra \infty} \frac{\log\vol(B_r)}{\log r} < \infty; \\[0.4cm]
\chi\sigma = \chi+1-\mu, & \quad \disp  \liminf_{r \ra \infty} \frac{\log\vol(B_r)}{\log r} \le \left\{ \begin{array}{ll}
p-\sigma(p-1) & \text{if } \, \sigma \le 1, \\[0.1cm]
\bar p - \sigma(\bar p -1) & \text{if } \, \sigma >1.
\end{array}\right.
\end{array}
\end{equation}
\end{itemize}
Then, $u$ is bounded above on $M$ and $f(u^*) \le 0$. If moreover $u$ satisfies $(P_=)$,
\begin{equation}\label{ipo_f_dinuovo}
tf(t) \ge C|t| \qquad \text{for } \, |t|>>1
\end{equation}
and either $(i)$ or $(ii)$ holds both for $u_+$ and for $u_-$, then
\begin{equation}\label{linda}
u \in L^\infty(M) \qquad \text{and} \qquad f(u^*) \le 0 \le f(u_*).
\end{equation}
\end{theorem}

\begin{remark}\label{rem_borderlinepol}
\emph{In the third case of \eqref{volgrowth_sigmamagzero}, that is, when $\chi\sigma = \chi+1-\mu$ and the volume growth of $B_r$ is suitably small with respect to $p,\bar p,\sigma$, the result still holds under the weaker assumption
\begin{equation}\label{ogrande}
u_+(x) = O\big(r(x)^\sigma\big) \qquad \text{as } \, r(x) \ra \infty.
\end{equation}
}
\end{remark}

As a consequence of Theorem \ref{teo_main_2}, we deduce Theorem \ref{teo_WMP_intro} of the Introduction.

\begin{proof}[Proof of Theorem \ref{teo_WMP_intro}]
Consider a non-constant solution $u \in \lip_\loc(M)$ with $u^* < \infty$ of $\Delta_\varphi u \ge K b(x)l(|\nabla u|)$ on some upper level  set $\Omega_\eta = \{x \in M \ : \ u(x) >\eta\}$. We shall prove that $K \le 0$. Suppose that this is not the case. By adding a constant to $u$, we can suppose that $\eta <0$ but sufficiently near to $0$ so that $\Omega_0 = \{x \in M \ : \ u(x)>0 \} \not= \emptyset$. Choose $f \in C(\R)$ such that $0 \le f \le K$, $f(0)=0$ and $f(t) = K$ for $t > u^*/2$. Then, $u$ solves
\begin{equation}\label{sempre_WMP}
\Delta_\varphi u \ge b(x)f(u)l(|\nabla u|)
\end{equation}
on $\Omega_\eta \supset  \Omega_0 $, and with the aid of Lemma \ref{lem_pasting} we can assume that $u \ge 0$ solves \eqref{sempre_WMP} on the whole of $M$. Moreover, $f(u^*)=K$. To reach a  contradiction, we just need to check the requirements to apply Theorem \ref{teo_main_2}, case $(i)$ and conclude $f(u^*)\le 0$. This, by Proposition \ref{prop_equivalence}, implies the $\wmp$.\\
First, observe that $\mu \le \chi - \alpha/2$ in \eqref{condi_volumeWMP_intro} can be rewritten as
\begin{equation}\label{ipp}
1+ \frac{\alpha}{2} \le \chi + 1 - \mu,
\end{equation}
and from $\alpha \ge -2$ it implies $\mu \le \chi+1$. We exhamine the validity of \eqref{volgrowth_sigmazero}. If $\alpha > -2$ and $\chi>0$, then $\mu < \chi+1$ by \eqref{ipp} and the volume assumption \eqref{volumeassu_intro} imply the first in \eqref{volgrowth_sigmazero}. If $\alpha = -2$ and $\chi>0$, then again by \eqref{volumeassu_intro} both of \eqref{volgrowth_sigmazero} are met, respectively, when $\mu<\chi+1$ and $\mu=\chi+1$. Suppose now that $\chi = 0$ and $\mu < \chi - \alpha/2$. Then, $\mu < \chi +1$ and, for each $\alpha \ge -2$, the strict inequality in \eqref{ipp} coupled with \eqref{volumeassu_intro} guarantees the first in \eqref{volgrowth_sigmazero}. If $\chi = 0$ and $\mu = \chi - \alpha/2$, according to whether $\alpha > -2$ or $\alpha = -2$ the requirement $V_\infty = 0$, respectively $V_\infty \le p$, in \eqref{condi_volumeWMP_intro} is precisely what is needed to deduce the validity of  \eqref{volgrowth_sigmazero}, concluding the proof.
\end{proof}

It is worth to postpone the proof of Theorem \ref{teo_main_2} and comment on various aspects of its statement.

\begin{itemize}
\item[-] $p,\bar p$ play no explicit role in \eqref{pararange_2}. However, a bound on $\chi$ in terms of $p$ alone is hidden, in some cases, in the requirement $l(t) \ge \varphi(t)/t^\chi$ on $\R^+$: indeed, if $\varphi(t) \asymp t^{p-1}$ near $t=0$, the continuity of $l$ at zero forces
$$
\chi \le p-1.
$$
\item[-] For $l \equiv 1$, that is, when no gradient appears, the third in \eqref{assum_main_2'} forces $\varphi(t) \le C^{-1}t^\chi$ on $\R^+$. If we suppose that $\varphi(t) \asymp t^{p-1}$ on $[0,1]$ and $\varphi(t) \asymp t^{\bar p-1}$ on $[1, \infty)$, the above theorem can be applied provided
$$
\bar p-1 \le \chi \le p-1.
$$
Therefore, when the operator is the $p$-Laplacian operator, the gradientless case $l\equiv 1$ is recovered with the choice $\chi = p-1$. On the other hand, for the mean curvature operator, $p \le 2$ and $\bar p>1$ can be chosen arbitrarily close to $1$, and the gradientless case $l \equiv 1$ is recovered for any choice of $\chi \in [0,1]$. For a fixed $\mu$, in case $(i)$ or in $(ii)$ with $\sigma \le 1$, it is evident that the choice $\chi=1$, $p =2$ gives the best result, while in case $(ii)$ for $\sigma>1$ the best choice is $\chi=0$ and $\bar p$ approaching $1$.
\item[-] The second in \eqref{volgrowth_sigmamagzero} is the only place where $p$ and $\bar p$ appear. Its validity forces an upper bound for $\sigma$, since the right-hand side in the second of \eqref{volgrowth_sigmamagzero} is non-negative if and only if
$$
\sigma \le \frac{\bar p}{\bar p-1}.
$$
\item[-] The third in \eqref{volgrowth_sigmamagzero} supports and makes rigorous the next idea: in the sublinear range $\sigma < 1$, the region where $|\nabla u|$ is close to zero should be, somehow, larger than the one where $|\nabla u|$ is big, and consequently the growth of $\varphi$ on $[1,\infty)$ (if still polynomial) might be neglectable with respect to the behaviour of $\varphi$ on $[0,1]$. If $\sigma>1$ the situation reverses, and now the main contribution should be given by $\varphi$ on $[1,\infty)$.\par 
We feel interesting and a bit surprising that the method to prove Theorem \ref{teo_main_2} is able to detect, in some sense, the size of the regions where $|\nabla u|$ is small or large. In particular, if $u$ is bounded above this might suggest that the above proof could be refined to show that, under the same assumptions, the \emph{strong} maximum principle $\smp$ is true, see Problem \ref{prob_SMPeWMP} in the Introduction.
\end{itemize}

As we will show at the end of the section, Theorem \ref{teo_main_2} is sharp in the following sense: under the validity of the range of the parameters $\chi,\mu$ in \eqref{pararange_2}, for almost each condition on $\sigma$ and $\vol(B_r)$ we are able to find a non-constant solution of $(P_\ge)$ with $f \equiv 1$ satisfying all the remaining assumptions but the chosen one.

If $f$ has a unique zero, from \eqref{linda} we deduce a Liouville type theorem for slowly growing solutions $u$ of $(P_=)$, that fits very well with some results obtained, in the Euclidean setting, by Farina and Serrin in \cite{farinaserrin1, farinaserrin2}. For the sake of comparison, we state their theorems renaming their parameters to agree with our notation:

\begin{theorem}\cite[Thm. 11 and 12]{farinaserrin2}\label{teo_farinaserrin}
On Euclidean space, consider $\varphi, b,f$ and $l$ meeting the requirements in  \eqref{assumptions_bfl}, \eqref{assumptions} and the bounds \eqref{assu_WPC} for some $p= \bar p >1$. Assume that, for some $\mu,\chi, \omega \in \R$ satisfying
$$
0 < \omega \le \chi \le p-1,
$$
the following inequalities hold:
\begin{equation}\label{assum_main_2_farina}
\begin{array}{ll}
b(x) \ge C\big(1+r(x)\big)^{-\mu} & \quad \text{on } \, \R^m, \\[0.2cm]
tf(t) \ge C|t|^{\omega+1} & \quad \text{for } \, t \in \R, \\[0.2cm]
\disp l(t) \ge Ct^{p-1-\chi} & \quad \text{on } \, \R^+,
\end{array}
\end{equation}
for some constant $C>0$. Then, a solution $u \in \lip_\loc(\R^m)$ of $(P_=)$ satisfying
$$
|u(x)| = O\big(r(x)^\sigma\big) \qquad \text{as } \, r(x) \ra \infty,
$$
for some $\sigma>0$, is constant provided one of the following cases occur:
\begin{itemize}
\item[(i)] $\omega < \chi$, $\mu< \chi+1$,
\begin{equation}\label{condi_case1farina}
(p-1)(1+\chi-\mu) \ge (p-m)(\chi-\omega) \qquad \text{and} \qquad \sigma \in \left(0, \frac{\chi+1-\mu}{\chi-\omega}\right);
\end{equation}
\item[(ii)] $\omega<\chi$, $m<p$,
\begin{equation}\label{condi_case2farina}
(p-1)(1+\chi-\mu) \le (p-m)(\chi-\omega) \qquad \text{and} \qquad \sigma \in \left(0, \frac{p-m}{p-1}\right);
\end{equation}
\item[(iii)] $\omega = \chi$, $\mu < \chi+1$, independently of $\sigma>0$.
\item[(iv)] $\omega=\chi$, $m<p$,
$$
\mu \ge \chi+1 \qquad \text{and} \qquad \sigma \in \left(0, \frac{p-m}{p-1}\right).
$$
\end{itemize}
\end{theorem}

\begin{remark}
\emph{Note that the bounds in $(i)$ and $(ii)$ well match when equality holds in the first of \eqref{condi_case1farina} and \eqref{condi_case2farina}. If $l \equiv 1$, that is, $\chi=p-1$, $(i)$ and $(iii)$ have been proved, respectively, in Theorems B and A in \cite{farinaserrin1}, and their sharpness is discussed in Examples 5 and 4, Section 11 therein. It is interesting to observe that the first in \eqref{condi_case1farina} is not required in \cite[Thm. B]{farinaserrin1}, but appears in discussing the sharpness of $(i)$. More precisely, Example 5 in \cite{farinaserrin1} stresses  that the conclusion of Theorem \ref{teo_farinaserrin} in case $(i)$ does not hold when $\sigma = \frac{\chi+1-\mu}{\chi-\omega}$ in \eqref{condi_case1farina}, provided that the first in \eqref{condi_case1farina} is strengthened to 
\begin{equation}\label{notfarina}
(p-1)(1+\chi-\mu) > (p-m)(\chi-\omega) \qquad \text{with } \, \chi = p-1.
\end{equation}
On the contrary, perhaps surprisingly, the conclusion of $(ii)$ still holds for $\sigma = \frac{p-m}{p-1}$, see the next Corollary \ref{cor_strano}. This is, clearly, not in contradiction with \cite{farinaserrin1} in view of the incompatibility of \eqref{notfarina} with \eqref{condi_case2farina}, see also the discussion \cite[p. 677]{pucciserrin_2}. Further results for solutions of $(P_\ge)$ which are a-priori bounded or vanishing at infinity are given in Theorems D,E,F in \cite{farinaserrin1} and Theorems 1 and 2 in \cite{Serrin_4}. Inspection shows that they fit very well with the case when $u$ is bounded above in Theorem \ref{teo_main_2}.
}
\end{remark}

First, we compare Theorem \ref{teo_main_2} with case $(i)$ in Theorem \ref{teo_farinaserrin}, and we therefore assume $M = \R^m$, $p = \bar p$ in \eqref{assu_WPC}, $0< \chi \le p-1$ and $\mu< \chi+1$. It is apparent that, for each $\omega>0$, condition $tf(t) \ge C|t|^{\omega+1}$ implies both \eqref{ipo_f_dinuovo} and $tf(t)>0$ on $\R^+$. Theorem \ref{teo_main_2} then gives the constancy of solutions of $(P_=)$ on $\R^m$ under the assumption
\begin{equation}\label{condi_nostra}
u(x) = o\big(r(x)^\sigma\big) \qquad \text{as } \, r(x) \ra \infty \qquad \text{and} \qquad \sigma \in \left[ 0, \frac{\chi+1-\mu}{\chi}\right),
\end{equation}
for any dimension $m$ and any $p>1$. The upper bound for $\sigma$ is smaller than the one in \eqref{condi_case1farina}, as a counterpart of the stronger requirement $tf(t) \ge C|t|^{\omega+1}$, but \eqref{condi_case1farina} converges to \eqref{condi_nostra} as $\omega \ra 0^+$. Moreover, the first in \eqref{condi_case1farina} is not needed, in accordance with \cite[Thm. B]{farinaserrin1} and \cite[Thm 2 $(i)$]{Serrin_4}. Next, we investigate the relationship with $(iii),(iv)$ in Theorem \ref{teo_farinaserrin}, where $\omega = \chi$ is assumed. Since in our case $\omega=0$, $(iii)$ and $(iv)$ should be compared with case $\chi=0$ of Theorem \ref{teo_main_2}. Observe that $\chi=0$ includes an interesting class of borderline inequalities such as
\begin{equation}\label{bellaplaplacian}
\Delta_p u \ge b(x) f(u)|\nabla u|^{p-1},
\end{equation}
for which we have the next corollary; since, for $\chi=0$, any $\sigma>0$ satisfies \eqref{ipo_sigma_teomain2}, when $\mu<1$ we obtain a Liouville theorem for solutions $u$ with polynomial growth (i.e. satisfying $|u| =O(r^\sigma)$ as $r \ra \infty$ for some $\sigma>0$).


\begin{corollary}\label{cor_strano}
Let $M$ be a complete Riemannian manifold, and consider $\varphi, b,f$ meeting the requirements in  \eqref{assumptions_bfl}, \eqref{assumptions} and the bounds \eqref{assu_WPC} for some $p, \bar p>1$. Assume that, for some $\mu \le 1$, the following inequalities hold:
\begin{equation}\label{assum_main_2}
\begin{array}{ll}
b(x) \ge C\big(1+r(x)\big)^{-\mu} & \quad \text{on } \, M, \\[0.2cm]
f(t) \ge C & \quad \text{for } \, t \gg 1
\end{array}
\end{equation}
and for some constant $C>0$. Let $u \in \lip_\loc(M)$ be a non-constant, weak solution of
\begin{equation}\label{casostrano!}
\Delta_\varphi u \ge b(x)f(u) \varphi(|\nabla u|) \qquad \text{on } \, M.
\end{equation}
Suppose that $u_+(x) = O\big(r(x)^\sigma\big)$, for some $\sigma>0$, and that
\begin{equation}\label{volgrowth_sigmazero_strano}
\begin{array}{rlll}
(1) & \quad \mu<1, & \disp \qquad \liminf_{r \ra \infty} \frac{\log\vol(B_r)}{r^{1-\mu}} = 0, & \quad \text{ and } \, \sigma>0, \ \text{ or}  \\[0.4cm]
(2) & \quad \mu=1, & \disp \qquad \liminf_{r \ra \infty} \frac{\log\vol(B_r)}{\log r} = d_0, & \quad \text{ and } \, 0 < \sigma \le \left\{ \begin{array}{ll}
\frac{p-d_0}{p-1}, & \text{if } \, d_0 \ge 1, \\[0.3cm]
\frac{\bar p-d_0}{\bar p-1} & \text{if } \, d_0 < 1.
\end{array}\right.
\end{array}
\end{equation}
Then, $u$ is bounded above on $M$ and $f(u^*) \le 0$. If moreover $u$ satisfies $(P_=)$,
\begin{equation}\label{ipo_f_dinuovo_strano}
tf(t) \ge C|t| \qquad \text{for } \, |t|>>1
\end{equation}
and either $(i)$ or $(ii)$ holds for $u_+$ and for $u_-$, then
\begin{equation}\label{linda_strano}
u \in L^\infty(M) \qquad \text{and} \qquad f(u^*) \le 0 \le f(u_*).
\end{equation}
\end{corollary}

\begin{proof}[Proof of Corollary \ref{cor_strano}, assuming Theorem \ref{teo_main_2}] It is enough to choose $\chi=0$ and $\sigma>0$ in Theorem \ref{teo_main_2}. An algebraic manipulation shows that the condition on $\sigma$ appearing in $(2)$ of \eqref{volgrowth_sigmazero_strano} is equivalent to
$$
d_0 \le \left\{ \begin{array}{ll}
p - \sigma(p-1) & \text{if } \, \sigma \le 1; \\[0.2cm]
\bar p -\sigma(\bar p-1) & \text{if } \, \sigma >1.
\end{array}\right.
$$
While Theorem \ref{teo_main_2} requires $u_+ = o(r^\sigma)$, when $\mu>1$ no problem arise as we can enlarge $\sigma$ a bit to match this last requirement. On the other hand, if $\mu=1$, thanks to Remark \ref{rem_borderlinepol} we can still reach the conclusion in Theorem \ref{teo_main_2} when $u_+=O(r^\sigma)$. In particular, the upper bound for $\sigma$ in \eqref{volgrowth_sigmazero_strano} can be achieved.
\end{proof}

\begin{remark}
\emph{Comparing with $(iii),(iv)$ in Theorem \ref{teo_farinaserrin}, we readily see that \eqref{volgrowth_sigmazero_strano} fits very well with their assumptions, and we can also capture the case $\sigma = \frac{p-m}{p-1}$. On the other hand, it should be remarked that our result is restricted to $\mu \le 1$.
}
\end{remark}

\begin{remark}\label{rem_liou}
\emph{In the Euclidean setting $M = \R^m$ and for $p=\bar p \ge m$, other interesting Liouville theorems for slowly growing solutions of $(P_=)$ can be found in \cite[Thm. 1.1]{pucciserrin_2} and \cite[Thm. 10]{farinaserrin2}, where the case $tf(t) \ge 0$ ($\equiv 0$ in the second) is considered. There, the authors obtain the constancy of solutions of $(P_=)$ on $\R^m$ whenever $p> m$ and\footnote{It should be observed that assumption (1.3) in \cite{pucciserrin_2}, when rephrased for $(P_\ge)$, gives necessarily $\varphi(t) = Ct^{p-1}$. However, the above restriction does not appear in Theorem 10 of \cite{farinaserrin2}, which considers the case $f(t) \equiv 0$.}
\begin{equation}\label{opiccolo}
u(x) = o \left( r(x)^\frac{p-m}{p-1}\right) \qquad \text{as } \, r(x) \ra \infty.
\end{equation}
Condition \eqref{opiccolo} is sharp and related to the growth of the fundamental solution for the $p$-Laplacian (see \cite{pucciserrin_2} and Remark 10.3 in \cite{dambrosiomitidieri_2}). Further interesting results covering $p \ge m$ can be found in \cite{dambrosiomitidieri_2} (Theorems 10.1 and  10.4 therein).
}
\end{remark}

Next, let us show how Corollary \ref{cor_strano} applies in the setting of the mean curvature operator to give the next
\begin{theorem}\label{teo_meancurv_intro}
Let $M$ be a complete manifold satisfying
\begin{equation}\label{ipo_volumecurvaturamedia}
\liminf_{r \ra \infty} \frac{\log \vol(B_r)}{r^{1-\mu}} =0,
\end{equation}
for some $\mu<1$. Let $b,f,l$ satisfy \eqref{assumptions_bfl} and
\begin{equation}\label{assu_meancurv_intro}
\begin{array}{l}
\disp b(x) \ge C \big(1+r(x)\big)^{-\mu} \qquad \text{on } \, M, \\[0.2cm]
\disp l(t) \ge C \frac{t}{\sqrt{1+t^2}} \qquad \text{on } \, \R^+, \\[0.4cm]
\text{$f$ is non-decreasing on $\R$ and $f \not \equiv 0$} \\[0.2cm]
\end{array}
\end{equation}
for some $C>0$. If $u \in \lip_\loc(M)$ is a non-constant solution of
\begin{equation}\label{chebella_meancurv}
\diver\left( \frac{\nabla u}{\sqrt{1+|\nabla u|^2}}\right) = b(x)f(u)l(|\nabla u|) \qquad \text{on } \, M
\end{equation}
with a polynomial growth, then $u$ solves the minimal surface equation
$$
\diver\left( \frac{\nabla u}{\sqrt{1+|\nabla u|^2}}\right) = 0 \qquad \text{on } \, M
$$
and it is bounded on one side. In particular, if \eqref{ipo_volumecurvaturamedia} is strengthened to
\begin{equation}\label{assu_ricci_percurv_intro}
\Ricc \ge 0, \quad \qquad K \ge - \kappa^2 \qquad \text{on M,}
\end{equation}
for some constant $\kappa >0$ ($K$ the sectional curvature of $M$), then a solution of \eqref{chebella_meancurv} with a polynomial growth, if any, is constant for each $\mu<1$.
\end{theorem}
\begin{remark}
\emph{The assumption $f \not\equiv 0$ is necessary, as the example of affine functions solving the minimal graph equation on $\R^m$ shows.
}
\end{remark}

\begin{remark}
\emph{When $l \equiv 1$, $b \equiv 1$ and $M = \R^m$, V. Tkachev \cite{tkachev} proved Theorem \ref{teo_meancurv_intro} for $C^2$ solutions of \eqref{chebella_meancurv} without any growth restriction, see also Thm. 10.4 of A. Farina's survey \cite{farina_minimal}. The result has first been extended for mean curvature type operators by Y. Naito and Y. Usami \cite[Thm. 1]{nu}, with a different argument using radialization in a way related to the one in Section \ref{sec_SL}. An improvement of \cite{tkachev, nu} for \eqref{chebella_meancurv} with a nontrivial gradient dependence is shown in Theorem \ref{teo_tkachev} below, while generalizations to a larger class of quasilinear operators can be found in \cite[Thm. 3]{Serrin_4}, with the same method as in \cite{tkachev}, and in \cite[Thm. 8.1.3]{pucciserrin} with an approach, by means of Khas'minskii potentials, close to the one in \cite{nu}.
}
\end{remark}

\begin{proof}[Proof of Theorem \ref{teo_meancurv_intro}]
Since one of the sets $\{t \in \R \ : \ f(t)>0\}$ and $\{t \in \R \ : \ f(t)<0\}$ is non-empty, the monotonicity of $f$ implies that either $f(t) \ge C_1$ for $t >>1$, or $f(t) \le -C_1$ for $t << -1$, for some constant $C_1>0$. Without loss of generality, we can suppose $f(t) \ge C_1$ for $t>>1$. Using Corollary \ref{cor_strano}, since $u$ is non-constant we deduce $u^* < \infty$ and $f(u^*) \le 0$. Again by the monotonicity of $f$, either $f =0$ on $(-\infty, u^ *]$ or $f<0$ somewhere. In the second case, $f(t) \le -C$ for $t<<-1$ and thus, applying again Corollary \ref{cor_strano} with $-u$ replacing $u$ and $-f(-t)$ replacing $f(t)$ we obtain $f(u_*) \ge 0$. Therefore, $f = 0$ on $[u_*,u^*]$. Summarizing, in both cases $u$ solves the minimal surface equation
\begin{equation}\label{mse}
\diver \left( \frac{\nabla u}{\sqrt{1+|\nabla u|^2}} \right) = 0 \qquad \text{on } \, M
\end{equation}
(hence, $u$ is smooth), and $u$ is bounded on one side. If \eqref{ipo_volumecurvaturamedia} is replaced by \eqref{assu_ricci_percurv_intro}, the volume comparison implies $\vol(B_r) \le Cr^m$ and therefore \eqref{ipo_volumecurvaturamedia} holds for each $\mu<1$. The constancy of $u$ is then a consequence of Theorem 1.3 in \cite{rosenbergschulzespruck}, contradicting our assumption.
\end{proof}

\begin{remark}
\emph{It is an open (and, we believe, very interesting) problem to prove the constancy of positive solutions of 
$$
\diver \left( \frac{\nabla u}{\sqrt{1+|\nabla u|^2}} \right)=0
$$
under the only condition that $\Ricc \ge 0$, that is, removing the requirement on the sectional curvature in \eqref{assu_ricci_percurv_intro}. Indeed, the technique in \cite{rosenbergschulzespruck} relies on uniform gradient estimates for $u$ achieved via the Korevaar's method, for which an assumption on the sectional curvature seems challenging to remove.
}
\end{remark}

We now come to the proof of Theorem \ref{teo_main_2}. The result follows in a more or less direct way from the following refined maximum principle for slowly growing solutions of
\begin{equation}\label{ineq_superlevel}
\Delta_\varphi u \ge K(1+r)^{-\mu}\frac{\varphi(|\gru|)}{|\gru|^{\chi}} \qquad \text{on } \, \Omega_\gamma,
\end{equation}
where $\Omega_\gamma=\{x\in M\,:\,u(x)>\gamma\}$ is the superlevel set of $u$ at height $\gamma\in\mathbb R$.

\begin{theorem}\label{teo_maximum}
Let $M$ be a complete Riemannian manifold and let the growth \eqref{assu_WPC} be met for some $p, \bar p>1$. Let $\mu, \chi \in \R$ verify \eqref{pararange_2}
and let $u \in \lip_\loc(M)$ be a function for which
\begin{equation}\label{sup_u}
\hat u = \limsup_{r(x) \ra \infty} \frac{u_+(x)}{r(x)^\sigma} < \infty,
\end{equation}
for some $\sigma \in \R^+_0$ satisfying
\begin{equation}\label{assu_sosigma}
0 \le \chi \sigma \le \chi + 1 -\mu.
\end{equation}
Suppose that either one of the following assumptions is met:
\begin{equation}\label{volgrowth}
\begin{array}{lll}
\chi \sigma < \chi+1-\mu & \quad \text{and} & \disp \qquad \liminf_{r \ra \infty} \frac{\log \vol(B_r)}{r^{\chi+1-\mu -\chi\sigma}} = d_0 < \infty; \\[0.2cm]
\text{or} \\[0.2cm]
\chi \sigma = \chi+1-\mu & \quad \text{and} & \disp \qquad \liminf_{r \ra \infty} \frac{\log \vol(B_r)}{\log r} = d_0 < \infty;
\end{array}
\end{equation}
If, for some $\gamma \in \R$, the open set $\Omega_\gamma$ is non-empty and $u$ is a non-constant, weak solution of \eqref{ineq_superlevel} in $\Omega_\gamma$,
then
\begin{equation}
K \le H \cdot \hat u^{\chi},
\end{equation}
where, setting
\begin{equation}\label{def_zetadstar}
\zeta = \chi + 1 -\mu, \qquad {\rm and} \quad d^* = \min\Big\{p-\sigma(p-1), \ \bar p-\sigma (\bar p-1)\Big\},
\end{equation}
the constant $H = H(\sigma, \chi,p,\mu, d_0)$ is given by
\begin{equation}\label{casi_C}
\begin{array}{l}
H= \left\{ \begin{array}{rll}
(i) & \ 0 & \quad \text{if } \quad \chi>0, \quad \sigma=0; \\[0.2cm]
(ii) & \ d_0\big[\zeta -\chi\sigma\big]^{\chi+1} & \quad \disp \text{if } \quad \chi>0, \quad 0 < \chi\sigma < \frac{\chi}{\chi+1} \zeta \\[0.3cm]
(iii) & \ d_0 \sigma^\chi (\zeta -\chi\sigma) & \quad \disp \text{if } \quad \chi>0, \quad \frac{\chi}{\chi+1} \zeta \le \chi\sigma < \zeta \\[0.4cm]
(iv) & \ d_0(1-\mu) & \quad \text{if } \quad \chi=0, \quad \mu<1, \quad \sigma \ge 0\\[0.2cm]
(v) & \ \sigma^\chi (d_0-d^*)_+ & \quad \disp \text{if } \quad \chi\sigma = \zeta >0 \ \  \text{ or }  \ \ \chi \sigma = \zeta=0, \ \chi=0.
\end{array}\right.
\end{array}
\end{equation}
\end{theorem}

\begin{proof}
We can suppose $K>0$, otherwise the estimate is trivial. Note that \eqref{ineq_superlevel} is invariant with respect to translations $u \mapsto u_s = u+s$. Fix $\beta > \hat u$. We claim that a suitable translated $u_s$ satisfies
\begin{equation}
u_s \le \beta(1+r)^\sigma \qquad \text{on } \, M, \qquad u_s>0 \ \text{ somewhere.}
\end{equation}
Indeed, if $\sigma>0$, \eqref{sup_u} implies that $u < \beta(1+r)^\sigma$ outside a large compact set $\Omega$, and translating $u$ downwards we can achieve the same inequality also in $\Omega$, still keeping $u_s>0$ somewhere. On the other hand, the claim is obvious if $\sigma=0$. In this last case, note that here we do \emph{not} claim that $\hat u$ is not attained: this would follow from a strong maximum principle, that to the best of our knowledge is unknown under the sole assumption \eqref{assu_WPC}. Using that the resulting $u_s$ is positive somewhere, we can also assume $\gamma> 0$. Hereafter, computations will be performed with $u=u_s$. Choose $\alpha>\beta$ and define
$$
v(x) = \alpha\big(1+r(x)\big)^\sigma - u(x),
$$
so that
\begin{equation}\label{def_v}
(\alpha-\beta)(1+r)^\sigma \le v \le \alpha(1+r)^\sigma \qquad \text{on } \, \Omega_\gamma.
\end{equation}
Hereafter, with $C_1,C_2,C_3, \ldots$ we will denote positive absolute constants, that is, independent of $\sigma, \mu,\chi$. Fix a function $\lambda \in C^1(\R)$ such that
$$
0 \le \lambda \le 1, \quad \lambda \equiv 0 \ \  \text{ on } \, (-\infty, \gamma], \qquad \lambda>0 \ \ \text{on } \, (\gamma, \infty), \quad \lambda' \ge 0,
$$
and a cut-off function $\psi \in \lip_c(M)$, $0 \le \psi \le 1$, whose support has nontrivial intersection with $\Omega_\gamma$. Next, consider $F \in C^1(\R^2)$, $F=F(r,v)$, satisfying
\begin{equation}\label{propriety_F}
F(r,v)>0, \qquad F_v = \frac{\partial F}{\partial v}(r,v) < 0.
\end{equation}
Suppose first that $\bar p \ge p$. We insert the test function
\begin{equation}\label{test_maxprinc}
\psi^{\bar{p}} \lambda(u) F(v, r)  \in \lip_c(\overline{\Omega}_\gamma)
\end{equation}
in the weak definition of \eqref{ineq_superlevel}. Using $\lambda' \ge 0$ together with the Cauchy-Schwarz inequality we obtain
\begin{equation}
\begin{array}{l}
\disp K \int \psi^{\bar{p}} \lambda F (1+r)^{-\mu}\frac{\varphi(|\gru|)}{|\gru|^\chi} \le \disp \bar{p} \int \psi^{\bar{p}-1} \lambda F \varphi(|\gru|)|\grho| + \int \psi^{\bar{p}} \lambda F_v\varphi(|\gru|)|\gru| \\[0.5cm]
\disp \qquad \qquad + \int \psi^{\bar{p}}\lambda \varphi(|\gru|)\left|\alpha\sigma (1+r)^{\sigma-1}F_v + F_r\right|.  \end{array}
\end{equation}
Rearranging,
\begin{equation}\label{basepoint}
\begin{array}{l}
\disp \int \psi^{\bar{p}} \lambda \left|F_v\right| \frac{\varphi(|\gru|)}{|\gru|^\chi} B(x,u) \le \bar{p}\int \psi^{\bar{p}-1} \lambda F \varphi(|\gru|)|\grho|,
\end{array}
\end{equation}
with
\begin{equation}\label{def_B}
\disp B(x,u) = \disp K(1+r)^{-\mu} \frac{F}{|F_v|} + |\gru|^{\chi+1} - |\gru|^\chi \left|-\alpha\sigma (1+r)^{\sigma-1} + \frac{F_r}{|F_v|}\right|.
\end{equation}
Let us assume the validity of the following
\begin{equation}\label{lower_boundB}
\emph{claim: } \quad B(x,u) \ge \Lambda |\gru|^{\chi+1} \qquad \text{for some  $\Lambda \in (0,1]$ independent of $r$.}
\end{equation}
Plugging into \eqref{basepoint} gives
\begin{equation}\label{ihihih}
\begin{array}{l}
\disp \frac{\Lambda}{\bar{p}} \int \psi^{\bar{p}} \lambda \left|F_v\right| \varphi(|\gru|)|\gru| \le \int \psi^{\bar{p}-1} \lambda F \varphi(|\gru|)|\grho|,
\end{array}
\end{equation}
We now split the integrals on the subsets $\{|\nabla u| < 1\}$ and $\{|\nabla u| \ge 1\}$, where we apply different Young inequalities. Letting $p', \bar p'$ be the conjugate exponents to $p$ and $\bar p$, we can rewrite \eqref{ihihih} as follows:
\begin{equation}\label{ihihih}
\begin{array}{l}
\disp \frac{\Lambda}{\bar{p}} \int_{\{|\nabla u| <1\}} \psi^{\bar{p}} \lambda \left|F_v\right| \left[\frac{\varphi(|\gru|)}{|\gru|^{p-1}}\right]|\gru|^p + \disp \frac{\Lambda}{\bar{p}} \int_{\{|\nabla u|\ge 1\}} \psi^{\bar{p}} \lambda \left|F_v\right| \left[\frac{\varphi(|\gru|)}{|\gru|^{\bar p -1}}\right] |\gru|^{\bar p} \\[0.4cm]
\disp \qquad \le \int_{\{|\nabla u|<1\}} \left[\psi^{\bar p} \lambda \left|F_v\right|\left(\frac{\varphi(|\gru|)}{|\gru|^{p-1}}\right)|\gru|^{p}\right]^{\frac{1}{p'}} \left[\psi^{\bar{p}-1- \frac{\bar p}{p'}}\lambda^{\frac{1}{p}} F|F_v|^{-\frac{1}{p'}}\left(\frac{\varphi(|\gru|)}{|\gru|^{p-1}}\right)^{\frac{1}{p}}|\grho|\right] \\[0.4cm]
\disp \qquad \ \ + \int_{\{|\nabla u|\ge 1\}} \left[\psi^{\bar p} \lambda \left|F_v\right|\left(\frac{\varphi(|\gru|)}{|\gru|^{\bar{p}-1}}\right)|\gru|^{\bar{p}}\right]^{\frac{1}{\bar{p}'}} \left[\lambda^{\frac{1}{\bar p}} F|F_v|^{-\frac{1}{\bar{p}'}}\left(\frac{\varphi(|\gru|)}{|\gru|^{\bar{p}-1}}\right)^{\frac{1}{\bar p}}|\grho|\right]
\end{array}
\end{equation}
Observe that $\bar p \ge p$ implies the non-negativity of the exponent $\bar p-1 - \frac{\bar p}{p'} \ge 0$ for $\psi$ above. We apply Young inequality $ab \le (a\eps)^{p'}/p' + (b/\eps)^p/p$ to the first term in the right-hand side of \eqref{ihihih}, and an analogous one with $\bar \eps, \bar p, \bar p'$ to the second one. Rearranging, we obtain
\begin{equation}\label{ihihih}
\begin{array}{l}
\disp \left( \frac{\Lambda}{\bar{p}}-\frac{\eps^{p'}}{p'}\right) \int_{\{|\nabla u| <1\}} \psi^{\bar{p}} \lambda \left|F_v\right| \left[\frac{\varphi(|\gru|)}{|\gru|^{p-1}}\right]|\gru|^p \\[0.4cm]
+ \disp \left(\frac{\Lambda}{\bar{p}}-\frac{\bar\eps^{\bar p'}}{\bar p'}\right) \int_{\{|\nabla u|\ge 1\}} \psi^{\bar{p}} \lambda \left|F_v\right| \left[\frac{\varphi(|\gru|)}{|\gru|^{\bar p -1}}\right] |\gru|^{\bar p} \\[0.4cm]
\disp \qquad \disp \le \frac{\eps^{-p}}{p}\int_{\{|\nabla u|<1\}} \psi^{\bar{p}-p}\lambda F \left[\frac{F}{|F_v|}\right]^{p-1}\left(\frac{\varphi(|\gru|)}{|\gru|^{p-1}}\right)|\grho|^p \\[0.4cm]
\disp \qquad \ \ + \frac{\bar\eps^{-\bar p}}{ \bar p} \int_{\{|\nabla u|\ge 1\}} \lambda F \left[\frac{F}{|F_v|}\right]^{\bar p-1}\left(\frac{\varphi(|\gru|)}{|\gru|^{\bar{p}-1}}\right)|\grho|^{\bar p}
\end{array}
\end{equation}
Choose $\eps, \bar \eps$ in such a way that both the coefficients in round brackets in the left-hand side are $\Lambda/(2\bar p)$. From $\bar p \ge p$, $\psi \le 1$ and $\Lambda \le 1$, using \eqref{assu_WPC} we  infer the inequality
\begin{equation}\label{basestep}
\int \psi^{\bar p} \lambda \left|F_v\right| \varphi(|\nabla u|) |\gru| \le C_1\left[\int \lambda F \left( \frac{F}{|F_v|}\right)^{p-1}|\grho|^p + \int \lambda F \left(\frac{F}{|F_v|}\right)^{\bar p-1}|\grho|^{\bar p}\right],
\end{equation}
where $C_1$ is some constant depending on $\Lambda$ and on $p,\bar p, C, \bar C$ in \eqref{assu_WPC}. Fix $R_0 \ge 1$ large enough that $u$ is not constant on $\Omega_\gamma \cap B_{R_0}\neq \emptyset$. Then, clearly $\nabla  u$ is not identically zero on $\Omega_\gamma \cap B_{R_0}$, because otherwise $u$ would be constant on connected components of $\Omega_\gamma$, which would imply that $\Omega_\gamma \equiv M$ and $u$ be constant, contradiction. Fix $\delta \in (1/2,1)$. For $R > 2R_0 \ge 2$, we choose $\psi$ in such a way that
\begin{equation}\label{specipsi}
0 \le \psi \le 1, \quad \psi \equiv 1 \ \text{ on } B_{\delta R}, \quad \supp(\psi) \subset B_{R}, \quad |\grho| \le \frac{C_2}{(1-\delta)R},
\end{equation}
for some absolute constant $C_2$.
Inserting into \eqref{basestep} and recalling that $\lambda \le 1$ we obtain
\begin{equation}\label{quasifinale}
\begin{array}{l}
\disp \int_{\Omega_\gamma \cap B_{R_0}} \lambda \left|F_v\right| \varphi(|\gru|)|\gru| \\[0.5cm]
\qquad \le \quad \disp C_3\left[ \frac{1}{R^p}\int_{(B_{R}\backslash B_{\delta R})\cap \Omega_\gamma} F \left( \frac{F}{|F_v|}\right)^{p-1} + \frac{1}{R^{\bar p}}\int_{(B_{R}\backslash B_{\delta R})\cap \Omega_\gamma} F \left( \frac{F}{|F_v|}\right)^{\bar p-1}\right],
\end{array}
\end{equation}
for some $C_3= C_3(p,\bar p, C, \bar C, \Lambda, \delta)$. In the complementary case $\bar p \le p$, we  achieve \eqref{quasifinale} by simply exchanging the role of $p,\bar p$ and the related inequalities on $\{|\nabla u| <1\}$, $\{|\nabla u|  \ge 1\}$.\par
We now need to check the validity of the claim in \eqref{lower_boundB}, for a suitable choice of $F$. Observe that the expression of $B$ in \eqref{def_B} is a function of the type
$$
g(s)= P + s^{\chi+1} - Q s^{\chi},
$$
for $s = |\gru|$ and positive parameters
$$
P = K(1+r)^{-\mu} \frac{F}{|F_v|}, \qquad Q = \left|-\alpha\sigma (1+r)^{\sigma-1} + \frac{F_r}{|F_v|}\right|
$$
depending on $r$. It is a calculus exercise to check that $g(s) \ge \Lambda s^{\chi+1}$ on $\mathbb R^+_0$ when either
\begin{equation}\label{ine_lambda_0}
\left\{ \begin{array}{lll}
\chi=0, & \quad Q \le P & \quad \text{and} \qquad \Lambda \le 1, \quad \text{or}  \\[0.2cm]
\chi > 0, & \disp \quad \frac{Q^{\chi+1}}{P} \le \frac{(\chi+1)^{\chi+1}}{\chi^\chi} & \quad \text{and} \disp \qquad \Lambda \le 1 - \frac{\chi}{(\chi+1)^{\frac{\chi+1}{\chi}}}\left( \frac{Q^{\chi+1}}{P}\right)^{1/\chi}.
\end{array}\right.,
\end{equation}
Having fixed a parameter $\theta \in (0,1)$, we will choose $F$ in order to satisfy the next relations between $Q$ and $P$:
\begin{equation}\label{ine_lambda}
\left\{ \begin{array}{ll}
\text{if } \, \chi=0 \, \text{ we want } \ \ Q = P, & \quad \text{and in this case set} \quad \Lambda = 1; \\[0.2cm]
\text{if } \, \chi > 0 \, \text{ we want } \ \ \frac{Q^{\chi+1}}{P} = \theta^\chi \frac{(\chi+1)^{\chi+1}}{\chi^\chi} & \quad \text{and in this case set} \disp \quad \Lambda = 1 - \theta.
\end{array}\right.
\end{equation}
In this way, \eqref{ine_lambda_0} and thus \eqref{lower_boundB} is met. Observe that the first case in \eqref{ine_lambda} can be obtained by letting $\chi \ra 0$ and then $\theta \ra 0$ in the second one. To meet the identities in \eqref{ine_lambda}, we necessarily need an upper bound for $Q/P$ or $Q^{\chi+1}/P$ that does not depend on $r$, and this suggests our choice of $F$, that will be different from case to case. Set for convenience
\begin{equation}\label{def_eta}
\eta = \mu + (\sigma-1)(\chi+1) = (\chi+1)\sigma - \zeta,
\end{equation}
where $\zeta$ is as in \eqref{def_zetadstar}, and note that
\begin{equation}\label{ine_eta}
\disp \sigma -\eta = \zeta - \chi\sigma \ge 0.
\end{equation}
\vspace{0.2cm}
\noindent \textbf{Analysis of case $(ii)$, case $(i)$ for $\zeta>0$, and case $(iv)$ for $\sigma <  1-\mu$:}
$$
\chi>0, \quad 0 \le \chi\sigma < \frac{\chi}{\chi+1} \zeta, \qquad \text{or} \qquad \chi=0, \quad \mu<1, \quad 0 \le \sigma < 1-\mu.
$$
Using the definition of $\eta$, these cases correspond to
$$
\chi>0, \qquad \sigma \ge 0, \quad \eta < 0.
$$
Note that $\sigma > \eta$. We choose
\begin{equation}\label{primocaso_F}
F(v,r) = \exp\left\{ - \tau v(1+r)^{-\eta}\right\},
\end{equation}
for a real number $\tau>0$ that will be specified later in order to satisfy the identity for $P,Q$ in \eqref{ine_lambda}. Then, on $\Omega_\gamma$
$$
\frac{F}{|F_v|} = \frac{(1+r)^\eta}{\tau}, \qquad  \frac{F_r}{|F_v|} = \frac{v\eta}{(1+r)},
$$
and hence, by \eqref{def_v} and using $\sigma>\eta$, $\eta < 0$
$$
\disp -\alpha(\sigma-\eta)(1+r)^{\sigma-1} \le -\alpha\sigma (1+r)^{\sigma-1} + \frac{F_r}{|F_v|} \le 0
$$
Plugging into \eqref{def_B} we get
\begin{equation}\label{lowbounds_B}
B(x,u) \ge \disp \frac{K}{\tau}(1+r)^{\eta-\mu} + |\gru|^{\chi+1} - |\gru|^{\chi}\alpha(\sigma-\eta)(1+r)^{\sigma-1},
\end{equation}
In view of \eqref{def_eta}, to satisfy \eqref{ine_lambda} with
\begin{equation}\label{cosasonoPQ}
P = \frac{K}{\tau}(1+r)^{\eta-\mu}, \qquad Q = \alpha(\sigma-\eta)(1+r)^{\sigma-1}
\end{equation}
we need the identities
\begin{equation}\label{uff}
\begin{array}{ll}
\disp \frac{[\alpha(\sigma-\eta)]^{\chi+1}\tau}{K} = \theta^\chi \frac{(\chi+1)^{\chi+1}}{\chi^\chi} & \quad \text{if } \, \chi>0; \\[0.4cm]
\disp \frac{\alpha(\sigma-\eta)\tau}{K} = 1 & \quad \text{if } \, \chi=0.
\end{array}
\end{equation}
According to whether $\chi=0$ or $>0$, we then define $\tau$ as the value such that \eqref{uff} holds. With this choice, \eqref{lower_boundB} is satisfied with $\Lambda = 1-\theta$ (if $\chi>0$) or $\Lambda=1$ (if $\chi=0$), and in view of  our choice of $F$, \eqref{quasifinale} becomes
\begin{equation}\label{quasifinale_0}
\begin{array}{lcl}
\disp \int_{B_{R_0}\cap \Omega_\gamma} \lambda F(1+r)^{-\eta} \varphi(|\gru|)|\gru| & \le & \disp C_4 \left[ \frac{1}{R^p} \int_{(B_{R}\backslash B_{\delta R})\cap \Omega_\gamma} F(1+r)^{\eta(p-1)} \right. \\[0.5cm]
& & \disp \left. +  \frac{1}{R^{\bar p}} \int_{(B_{R}\backslash B_{\delta R})\cap \Omega_\gamma} F(1+r)^{\eta(\bar p-1)} \right]
\end{array}
\end{equation}
Where $C_4$ also depends on $\tau$. Up to increasing $\bar p$, a change that does not alter the validity of \eqref{assu_WPC}, we can suppose that $\bar p \ge p$. On $(B_{R}\backslash B_{\theta R})\cap \Omega_\gamma$, \eqref{def_v} and $\sigma-\eta >0$, $\eta < 0$ give
$$
F(v,r) \le \exp\left(-\tau(\alpha-\beta)(\delta R)^{\sigma-\eta}\right), \qquad (1+r)^{\eta(\bar p-1)} \le (1+r)^{\eta(p-1)} \le 1,
$$
Inserting into \eqref{quasifinale_0} and using $R^{\bar p} \ge R^p$, we eventually get
\begin{equation}\label{quasifinale_2}
0< \disp \int_{\Omega_\gamma \cap B_{R_0}}\lambda F(1+r)^{-\eta}\varphi(|\gru|)|\gru| \le  \frac{C_5}{R^p}\exp\left(-\tau(\alpha-\beta)(\delta R)^{\sigma-\eta}\right)\vol(B_{R}).
\end{equation}
Because of \eqref{volgrowth} and \eqref{ine_eta}, for each $d>d_0$ there exists a sequence $\{R_k\} \uparrow \infty$ such that
$$
\vol(B_{R_k}) \le \exp\left\{d R_k^{\sigma -\eta}\right\}.
$$
Substituting into \eqref{quasifinale_2} and letting $k \ra \infty$,
\begin{equation}\label{quasifinale_22}
\begin{array}{lcl}
0 & < & \disp \int_{\Omega_\gamma \cap B_{R_0}}\lambda F(1+r)^{-\eta}\varphi(|\gru|)|\gru| \\[0.5cm]
& \le & \disp C_5 \limsup_{k \ra \infty} \left(\frac{\exp\left\{-\tau(\alpha-\beta)\delta^ {\sigma-\eta}R_k^{\sigma-\eta}+ dR_k^ {\sigma-\eta}\right\}}{R_k^p}\right).
\end{array}
\end{equation}
Being the left-hand side of the above inequality strictly bigger than zero, we deduce that necessarily $d \ge \tau(\alpha-\beta)\delta^{\sigma-\eta}$, and letting $\delta \ra 1$ and $d \downarrow d_0$. we get
\begin{equation}
d_0 \ge \tau(\alpha-\beta)
\end{equation}
Substituting the expression of $\tau$ in \eqref{uff}, setting $\alpha =t\beta$ with $t>1$ and letting $\theta \ra 1$ we deduce
$$
\begin{array}{ll}
\disp K \le d_0\frac{\chi^ \chi}{(\chi+1)^{\chi+1}} (\sigma-\eta)^{\chi+1}\frac{t^{\chi+1}}{t-1} \beta^\chi & \quad \text{if } \, \chi>0; \\[0.4cm]
\disp K \le d_0(\sigma-\eta)\frac{t}{t-1} & \quad \text{if } \, \chi=0. \\[0.4cm]
\end{array}
$$
Minimizing with respect to $t \in (1,\infty)$ and letting $\beta \downarrow \hat u$, we eventually get
\begin{equation}\label{fineprimo}
\begin{array}{ll}
K \le d_0(\sigma-\eta)^{\chi+1} \hat u^\chi = d_0\big[\zeta - \chi \sigma\big]^{\chi+1} \hat u^\chi & \quad \text{if } \, \chi>0;\\[0.2cm]
K \le d_0(\sigma-\eta) = d_0(1-\mu) & \quad \text{if } \, \chi =0,\\[0.2cm]
\end{array}
\end{equation}
as claimed. We conclude by investigating part of case $(i)$, that is, when
$$
\chi>0, \quad \sigma=0< \zeta.
$$
Observe that a downward translation $u_s$ of $u$ still satisfies \eqref{ineq_superlevel} with the same constant $K$ (without loss of generality, we can suppose $\gamma=0$). Hence, by \eqref{fineprimo},
\begin{equation}\label{bonita}
K \le H \hat u_s^\chi \qquad \text{with } \quad H \le d_0\big[\zeta-\chi\sigma\big]^{\chi+1}.
\end{equation}
If $K>0$, since $u$ is bounded above and $\chi>0$ we can choose $u_s^\chi$ positive and small enough to contradict \eqref{bonita}. Hence, necessarily $K=0$, and a-posteriori we can choose $H=0$ in \eqref{casi_C} as required. At the end of the present proof, with the same trick we investigate the remaining case of $(i)$, that is, when $\sigma = \zeta = 0$. Note that the trick is not possible if $\chi=0$, being $u_s^\chi  \equiv 1$. \\[0.2cm]

\noindent \textbf{Analysis of case $(iii)$, and case $(iv)$ for $\sigma \ge 1-\mu$:}
$$
\chi>0, \quad \frac{\chi}{\chi+1} \zeta \le \chi\sigma < \zeta, \qquad \text{or} \qquad \chi=0, \quad \sigma \ge 1-\mu >0.
$$
Again from the definition of $\eta$, these cases correspond to the range
$$
0 \le \eta < \sigma,
$$
for which we choose
\begin{equation}\label{casomudado}
F(v,r) = \exp\left\{ - \tau v^{\frac{\sigma-\eta}{\sigma}}\right\}.
\end{equation}
Also in these cases, we increase $\bar p$ in order for $\bar p \ge p$ to hold, since ultimately the size of $\bar p$ will not affect the conclusion of the theorem. Performing computations analogous to the ones giving $(ii)$, we obtain the desired estimate
$$
\begin{array}{ll}
\disp K \le d_0\sigma^\chi(\sigma-\eta) \hat u^\chi = d_0\sigma^\chi\big[\zeta-\chi\sigma\big] \hat u^\chi & \quad \text{if } \, \chi>0; \\[0.3cm]
\disp K \le d_0(\sigma-\eta) = d_0(1-\mu) & \quad \text{if } \, \chi=0. \\[0.3cm]
\end{array}
$$
\vspace{0.2cm}

\noindent \textbf{Analysis of case $(v)$, and case $(i)$ for $\zeta = 0$:}
$$
\chi\sigma= \zeta > 0, \quad \text{or} \quad \chi\sigma = \zeta =0.
$$
In this case, by \eqref{def_eta} it holds $\sigma-\mu=(\sigma-1)(\chi+1)$. We choose
\begin{equation}
F(v,r) = v^{-\tau},
\end{equation}
$\tau>0$ to be determined. Then, using \eqref{def_v},
\begin{equation}\label{lowbounds_Bpol}
\begin{array}{lcl}
B(x,u) & \ge & \disp \frac{K}{\tau}(1+r)^{-\mu}v + |\gru|^{\chi+1} - |\gru|^{\chi}\alpha\sigma(1+r)^{\sigma-1} \\[0.5cm]
 & \ge & \disp \frac{K(\alpha-\beta)}{\tau}(1+r)^{\sigma-\mu} + |\gru|^{\chi+1} - |\gru|^{\chi}\alpha\sigma(1+r)^{\sigma-1},
\end{array}
\end{equation}
and \eqref{ine_lambda} applied with
$$
P = \frac{K(\alpha-\beta)}{\tau}(1+r)^{\sigma-\mu}, \qquad Q = \alpha\sigma(1+r)^{\sigma-1}
$$
implies the identities
\begin{equation}\label{uff_poli}
\begin{array}{ll}
\disp \frac{(\alpha\sigma)^{\chi+1}\tau}{K(\alpha-\beta)} = \theta^\chi \frac{(\chi+1)^{\chi+1}}{\chi^\chi} & \quad \text{if } \, \chi>0; \\[0.4cm]
\disp \frac{\alpha\sigma\tau}{K(\alpha-\beta)} = 1 & \quad \text{if } \, \chi=0.
\end{array}
\end{equation}
Having specified $\tau$ to satisfy \eqref{uff_poli}, \eqref{lower_boundB} is met and \eqref{quasifinale} with $\delta =3/4$ now reads
\begin{equation}\label{quasifinale_5}
0 < \disp \int_{\Omega_\gamma \cap B_{R_0}} \lambda \left|F_v\right| \varphi(|\gru|)|\gru| \le \disp C_4 \int_{(B_{R}\backslash B_{3R/4})\cap \Omega_\gamma} \frac{v^{p-1-\tau}}{R^p} + \frac{v^{\bar p-1-\tau}}{R^{\bar p}}.
\end{equation}
Estimating $v$ with the aid of \eqref{def_v}, choosing the upper or lower bound according to the sign of $p-1-\tau$ and $\bar p-1-\tau$, the right-hand side is bounded from above by
$$
C_5 \vol(B_R)\Big(R^{-p+ \sigma(p-1-\tau)}+ R^{-\bar p+ \sigma(\bar p-1-\tau)}\Big),
$$
for a suitable constant $C_5>0$. Because of $\eqref{volgrowth}$, for $d>d_0$ we consider a sequence $\{R_k\}$ for which $\vol(B_{R_k}) \le R_k^d$. Evaluating \eqref{quasifinale} on $R_k$, and letting $k \ra \infty$ in \eqref{quasifinale_5} and then $d \downarrow d_0$, we deduce that necessarily
$$
d_0 \ge \min \big\{p - \sigma(p-1-\tau), \bar p - \sigma(\bar p-1-\tau)\big\},
$$
that is, by \eqref{def_zetadstar},
\begin{equation}\label{finla}
\tau \sigma \le d_0- \min\Big\{p-\sigma(p-1), \bar p - \sigma(\bar p -1)\Big\} = d_0 - d^*.
\end{equation}

\begin{itemize}
\item[(i)] If $d_0 <d^*$ or $d_0 = d^*$ and $\sigma >0$, then there exists no $\tau>0$ satisfying \eqref{finla}. Thus, $K>0$ leads to a contradiction, and we can therefore choose $H=0$. This proves $(v)$ for $d_0< d_*$ and for $d_0 = d^*$, $\sigma >0$, as well as $(i)$ for $\chi>0$, $\sigma=\zeta=0$, $d_0<d^*$.
\item[(ii)] If $d_0 \ge d^*$ then, inserting the expression of $\tau$ obtained from \eqref{uff_poli} in \eqref{finla}, setting $\alpha = t\beta$ for $t > 1$, solving \eqref{finla} with respect to $K$, letting $\theta \uparrow 1$ and $\beta \downarrow \hat u$ we deduce
$$
\begin{array}{ll}
\disp K \le  \left[d_0-d^*\right] \frac{\chi^\chi}{(\chi+1)^{\chi+1}}\hat u^{\chi}\sigma^{\chi} \frac{t^{\chi+1}}{t-1} & \quad \text{if } \, \chi>0; \\[0.4cm]
\disp K \le  \left[d_0-d^*\right]\frac{t}{t-1} & \quad \text{if } \, \chi=0,
\end{array}
$$
and minimizing over $t \in (1,\infty)$ we get for both $\chi>0$ and $\chi=0$
\begin{equation}\label{allafine}
K \le  \left[d_0-d^*\right]\sigma^\chi\hat u^{\chi}.
\end{equation}
This concludes the cases $d_0> d^*$ and $\sigma>0$, and $d_0 \ge d^*$ and $\chi=0$. To deal with the remaining part of $(i)$, that is, $\chi>0$, $\sigma = \zeta = 0$ and $d_0 \ge d^*$, we can consider a downward translation $u_s$ of $u$ in place of $u$, and $\gamma = 0$. Then, $u_s$ satisfies \eqref{ineq_superlevel} with the same constant $K$, hence \eqref{allafine} holds for each $\hat u_s$. However, from $\chi>0$, $\hat u_s^\chi$ can be made as small as we wish, and since we have assumed $K>0$ this would contradict \eqref{allafine}. Concluding, necessarily $K \le 0$, and $H$ can be chosen to be zero, as required.
\end{itemize}
\end{proof}
We now prove Theorem \ref{teo_main_2}

\begin{proof}[Proof of Theorem \ref{teo_main_2}]
Suppose, by contradiction, that either $u^* = \infty$ or $f(u^*) >0$. Because of the second in \eqref{assum_main_2} and the continuity of $f$, in both of the cases there exists $\gamma < u^*$ sufficiently close to $u^*$ such that $f(t) \ge C>0$ on $(\gamma, \infty)$, for some constant $C>0$. By \eqref{assum_main_2}, $u$ would solve
$$
\Delta_\varphi u \ge K(1+r)^{-\mu}\frac{\varphi(|\gru|)}{|\gru|^{\chi}} \qquad \text{on } \, \Omega_\gamma,
$$
For some $K > 0$. To apply Theorem \ref{teo_maximum}, we shall consider
$$
d^* = \min \Big\{ p- \sigma(p-1), \bar p - \sigma(\bar p-1)\Big\}.
$$
As said before, $p$ (respectively, $\bar p$) in \eqref{assu_WPC} can be reduced (resp. increased) as much as we wish, still keeping the validity of \eqref{assu_WPC}. Therefore,
\begin{itemize}
\item[-] if $\sigma \le 1$, we can increase $\bar p$ up to satisfy $\bar p \ge p$, that gives $d^* = p -\sigma(p-1)$. In particular, $d^*=p$ when $\sigma=0$;
\item[-] if $\sigma>1$, we can reduce $p$ to satisfy $\bar p \ge p$, that now implies $d^* = \bar p -\sigma(\bar p-1)$.
\end{itemize}
The volume growth conditions in the second lines of \eqref{volgrowth_sigmazero} and \eqref{volgrowth_sigmamagzero} can therefore be rewritten as
$$
\liminf_{r \ra \infty} \frac{\log\vol(B_r)}{\log r} \le d^*.
$$
We are in the position to apply Theorem \ref{teo_maximum} and deduce $0 < K \le H \hat u^\chi$. We reach a contradiction by proving that either $H=0$ or $\hat u^\chi=0$. We split into cases:
\begin{itemize}
\item[-] If $\sigma=0$ and $\chi>0$, then we are in case $(i)$ of Theorem \ref{teo_maximum} and \eqref{volgrowth} is satisfied because of \eqref{volgrowth_sigmazero}, thus $H=0$, contradiction.
\item[-] If $\sigma \ge 0$ and $\chi=0$, then we are either in case $(iv)$ or in case $(v)$ of Theorem \ref{teo_maximum}, according to whether $\mu<1$ or $\mu=1$. In case $(iv)$, our growth requirements \eqref{volgrowth_sigmazero} and \eqref{volgrowth_sigmamagzero} for $\chi=0$, $\mu<1$ imply that $d_0=0$ in \eqref{volgrowth}. Applying Theorem \ref{teo_maximum} we get $H \le d_0(1-\mu)=0$, contradiction. In case $(v)$, as said conditions \eqref{volgrowth_sigmazero} and \eqref{volgrowth_sigmamagzero} for $\chi=0$, $\mu=1$ are equivalent to $d_0 \le d^*$. By Theorem \ref{teo_maximum}, $H \le \sigma^\chi(d_0-d^*)_+ = 0$, contradiction.
\item[-] If $\sigma >0$, $\chi>0$, then by \eqref{opequeno} we get $\hat u^\chi = 0$, contradiction.
\end{itemize}
We have thus proved $u^* < \infty$ and $f(u^*) \le 0$. If now $u$ satisfies $(P_=)$ and \eqref{ipo_f_dinuovo} is in force, we can apply the result both to $u$ and to $\bar u = -u$, noting that
$$
\Delta_\varphi \bar u = b(x)\bar f(\bar u)l(|\nabla \bar u|), \qquad \bar f(t) = -f(-t),
$$
and that $\bar f(t) \ge C$ for $t$ large enough. From $\bar f(\bar u^*) \le 0$ we get $f(u_*) \ge 0$,  and \eqref{linda} follows.
\end{proof}

\begin{remark}
\emph{We now check Remark \ref{rem_borderlinepol}. Since the third in \eqref{volgrowth_sigmamagzero} corresponds to case $(v)$ of Theorem \ref{teo_maximum}, if $d_0 \le d^*$ we achieve the contradiction $0< K \le H \hat u^\chi$ irrespectively to the vanishing of $\hat u$, as claimed.
}
\end{remark}

\subsubsection{Bernstein theorems for minimal and MCF soliton graphs}

We now apply Theorem \ref{teo_main_2} to deduce some Bernstein type results for prescribed mean curvature graphs. We consider an ambient space $(\bar M^{m+1}, \metricN)$ with the warped product structure  
\begin{equation}\label{warped_geodesic_sec}
\bar M = \R \times_h M, \qquad \metricN = \di s^2 + h(s)^2 \metric
\end{equation}
for some complete manifold $(M^ m, \metric)$ and some $0< h \in C^\infty(\R)$. Given $v : M \ra \R$, we consider the graph 
$$
\Sigma = \Big\{ (v(x), x) \in \bar M \ \ : \ \ x \in M\Big\}.
$$
In the next theorems we always consider \emph{entire} graphs, that is, graphs defined on all of $M$. As in Subsection \ref{subsec_bernstein} in the Introduction, we let $\Phi_t$ be the flow of the conformal field $X= h(s) \partial_s$, and we note that its flow parameter $t$ starting from the slice $\{s=0\}$ satisfies \eqref{bonitinho}:
\begin{equation}\label{bonitinho_theorem}
t = \int_{0}^s \frac{\di \sigma}{h(\sigma)}, \qquad t : \R \ra t(\R) = I.
\end{equation}
We define $\lambda(t) = h(s(t))$, $u(x) = t(v(x))$, and note that $u : M \ra I$. 

\begin{theorem}\label{teo_bern_minimal_hyper}
Let $(M^m, \metric)$ be a complete manifold, and consider the warped product $\bar M = \R \times_h M$. Suppose that $h$ satisfies:
\begin{equation}\label{ipo_h_bernstein}
\left\{\begin{array}{l}
h'(s)s \ge Cs \qquad \text{for } \, |s| \ge r_0; \\[0.3cm]
h^{-1} \in L^1(+\infty) \cap L^1(-\infty),
\end{array}\right.
\end{equation}
for some constants $C, s_0>0$. If the volume growth of geodesic balls in $M$ satisfies
\begin{equation}\label{ipovol_bern_1}
\liminf_{r \ra \infty} \frac{\log\vol(B_r)}{r^2} < \infty, 
\end{equation}
Then every entire, geodesic minimal graph $\Sigma$ over $M$ is bounded and, letting $v: M \ra \R$ be the graph function, 
\begin{equation}\label{bella_bernstein}
h'(v^*) \le 0 \le h'(v_*).
\end{equation}
In particular, if $h$ is strictly convex, $\Sigma$ is the totally geodesic slice $\{s=s_1\}$, where $s_1$ is the unique minimum of $h$.
\end{theorem} 

\begin{proof}
By \eqref{prescribed_geodesic} and the minimality of $\Sigma$, $u$ solves
$$
\diver \left( \frac{\nabla u}{\sqrt{1+|\nabla u|^2}}\right) = m \frac{\lambda_t(u)}{\lambda(u)} \frac{1}{\sqrt{1+|\nabla u|^2}}
$$
The second in \eqref{ipo_h_bernstein} implies that $I$ is a bounded interval, and thus $u : M \ra I$ is bounded. Taking into account that 
\begin{equation}\label{laf_bernstein}
f(t) = \frac{\lambda_t(t)}{\lambda(t)} = h'(s(t)) \ge C \qquad \text{for } \, t \ge t(s_0),
\end{equation}
we get
$$
\diver \left( \frac{\nabla u}{\sqrt{1+|\nabla u|^2}}\right) \ge \frac{mC}{\sqrt{1+|\nabla u|^2}} \qquad \text{on } \, \{u > t(s_0)\}.
$$
If the set $\{u> t(s_0)\}$ were non-empty, we can apply Theorem \ref{teo_maximum} with $\mu=0$, $\chi =1$, $p=\bar p = 2$ and $K = mC$ to obtain a contradiction. Therefore, $u$ is bounded from above and from Theorem \ref{teo_main_2} we deduce $f(u^*) \le 0$. Analogously, considering $-u$ we infer $f(u_*) \ge 0$, and thus \eqref{bella_bernstein} follows by changing variables. Note that, because of the first in  \eqref{ipo_h_bernstein}, \eqref{bella_bernstein} implies the boundedness of $v$. If $h$ is strictly convex, then by \eqref{ipo_h_bernstein} it has a unique stationary point (a minimum) $s_1$, and \eqref{bella_bernstein} gives $v^* \le s_1 \le v_*$, that is, $v \equiv s_1$. We remark that each slice $\{s = s_2\}$ is totally umbilical in $\bar M$, with second fundamental form $h'(s_2)/h(s_2) \metric$ in the direction $-\partial_s$. 
\end{proof}

\begin{theorem}\label{teo_bern_minimal_horo}
Let $\bar M = \R \times_h M$ be as above, and suppose that $h$ satisfies:
\begin{equation}\label{ipo_h_bernstein_2}
\left\{\begin{array}{l}
h'>0 \quad \text{on } \, \R, \\[0.3cm]
h'(s) \ge C \qquad \text{for } \, s \ge s_0; \\[0.3cm]
h^{-1} \in L^1(+\infty),
\end{array}\right.
\end{equation}
for some constants $C, s_0>0$. If $M$ is complete and the volume growth of geodesic balls in $M$ satisfies
\begin{equation}\label{ipovol_bern_1}
\liminf_{r \ra \infty} \frac{\log\vol(B_r)}{r^2} < \infty, 
\end{equation}
Then there are no entire, geodesic minimal graphs over $M$.
\end{theorem} 

\begin{proof}
The proof is analogous. By the third in \eqref{ipo_h_bernstein_2}, $u : M \ra I$ is bounded from above, hence applying Theorems \ref{teo_maximum} and \ref{teo_main_2} we get $h'(v^ *) \le 0$, contradicting the first in \eqref{ipo_h_bernstein_2}.
\end{proof}

\begin{proof}[Proof of Theorem \ref{teo_bern_minimal_intro}] It is immediate from Theorems \ref{teo_bern_minimal_hyper} and \ref{teo_bern_minimal_horo}, respectively for cases $(i)$ and $(ii)$. Note that, in $(i)$, the strict convexity of $h$ and $h^{-1} \in L^1(-\infty) \cap L^1(+\infty)$ imply that $sh'(s) \ge Cs$ for $|s|$ large enough.
\end{proof}

\begin{remark}
\emph{As a direct corollary we deduce Do Carmo-Lawson's result \cite{docarmolawson} in the case $H=0$, Theorem \ref{teo_docarmolawson} in the Introduction: there are no geodesic, entire minimal graphs over horospheres, and the only geodesic, entire minimal graph over a totally geodesic hypersphere $M$ is $M$ itself. To see the first claim, apply Theorem \ref{teo_bern_minimal_horo} to the warped product $\HH^{m+1} = \R \times_{e^{s}} \R^m$, and note that $h(s) = e^s$ satisfies all the requirements in \eqref{ipo_h_bernstein_2}. For the second claim, apply Theorem \ref{teo_bern_minimal_hyper} to $\HH^{m+1} = \R \times_{\cosh s} \HH^m$. 
}
\end{remark}

The case of non-constant mean curvature will be investigated in Section \ref{sec_SMP}. Now, we focus on MCF solitons, starting with product ambient manifolds. Our first result is for self-translators, Theorem \ref{teo_soliton_intro} in the Introduction:


\begin{proof}[Proof of Theorem \ref{teo_soliton_intro}]
In the product case, $h(s) =1$ and thus $t=s$, $\lambda(t) = 1$, $u(x) = t(v(x)) = v(x)$. By \eqref{soliton_geodesic}, a soliton for the field $\partial_s$ satisfies the equation 
$$
\diver\left( \frac{\nabla u}{\sqrt{1+|\nabla u|^2}} \right) = \frac{1}{\sqrt{1+|\nabla u|^2}} \qquad \text{on } \, M.
$$
We apply the first part of Theorem \ref{teo_main_2} with the choices $f(t) \equiv 1$, $b \equiv 1$, $\chi = 1$, $\mu = 0$, $p=\bar p=2$: since the first two in \eqref{volgrowth_sigmamagzero} correspond to \eqref{ipo_volume_soliton_intro}, we infer from Theorem \ref{teo_main_2} that $u$ is bounded from above and $f(u^*) \le 0$, contradiction. 
\end{proof}

We next examine more closely the case of self-translators in Euclidean space, in particular the case when the translation vector field $Y$ differs from the vertical field $\partial_s$. 

\begin{theorem}\label{teo_soliton_Rm}
Let $(\R^{m+1}, \metricN)  = \R \times \R^m$ with coordinates $(s,x)$, and let $\Sigma = \{ (v(x),x) : x \in \R^m\}$ be an entire graph. Assume that
\begin{equation}\label{ipo_soliton_rm}
\limsup_{r(x) \ra \infty} \frac{|v(x)|}{r(x)} = \hat v < \infty
\end{equation}
Then, $\Sigma$ cannot be a self-translator with respect to any vector $Y$ whose angle $\vartheta \in (0, \pi/2)$ with the horizontal hyperplane $\R^m$ satisfies
\begin{equation}\label{strano_soliton}
\tan \vartheta > \hat v.
\end{equation}
In particular, if $\hat v = 0$, $\Sigma$ cannot be a self-translator with respect to a vector $Y$ which is not tangent to the horizontal $\R^m$.
\end{theorem}

\begin{proof}
If we reflect $\Sigma$ with respect to the horizontal hyperplane, the reflected graph is a self-translator with respect to the reflection of $Y$. Therefore, without loss of generality we can assume that $(Y, \partial_t) >0$, and that $Y$ has unit norm, by time rescaling. Moreover, $Y \neq \partial_t$ since $\vartheta < \pi/2$. Therefore, up to a rotation of coordinates on $\R^m$, $Y = \cos \vartheta e_1 + \sin \vartheta \partial_t$, where $e_1$ is the gradient of the first coordinate function $x_1$ and $\sin \vartheta >0$. In view of \eqref{normal_geodesic} and \eqref{prescribed_geodesic}, and since $h=1$, $t=s$, the soliton equation $m H = (Y, \nu)$ satisfied by $u(x) = t(v(x)) = v(x)$ reads
$$
\diver \left( \frac{\nabla u}{\sqrt{1+|\nabla u|^2}} \right) = mH = \frac{1}{\sqrt{1+|\nabla u|^2}} \left[ \sin \vartheta - \cos \vartheta \langle \nabla u, e_1 \rangle \right], 
$$
Rearranging,  
$$
\diver \left( \frac{\nabla u}{\sqrt{1+|\nabla u|^2}} \right) + \left\langle \frac{\nabla u}{\sqrt{1+|\nabla u|^2}}, \nabla (x_1\cos \vartheta) \right\rangle = \frac{\sin \vartheta}{\sqrt{1+|\nabla u|^2}}, 
$$
that we rewrite as
\begin{equation}\label{eq_soli_Rm}
e^{-x_1 \cos \vartheta} \diver \left( e^ {x_1 \cos \vartheta} \frac{\nabla u}{\sqrt{1+|\nabla u|^2}} \right) = \frac{\sin \vartheta}{\sqrt{1+|\nabla u|^2}}.
\end{equation}
The operator in the left-hand side is in divergence form if we consider the weighted volume measure $e^{x_1 \cos \vartheta} \di x$, with $\di x$ the Euclidean volume. For these weighted operators, the proof of Theorems \ref{teo_maximum} and \ref{teo_main_2} follow verbatim by replacing the Euclidean volume with the weighted volume
$$
\vol_{x_1 \cos \vartheta}(B_r) = \int_{B_r} e^{x_1 \cos\vartheta} \di x.
$$
Explicit computation gives
\begin{equation}\label{weightedvol}
\liminf_{r \ra \infty} \frac{\log\vol_{x_1 \cos\theta}(B_r)}{r} = \cos \vartheta < \infty. 
\end{equation}
Suppose by contradiction that \eqref{strano_soliton} holds. In particular, $u$ is non-constant. We apply Theorem \ref{teo_maximum} to \eqref{eq_soli_Rm} with the choices 
$$
K = \sin \vartheta, \ \ \sigma = 1, \ \ \chi = 1, \ \ \mu = 0, \ \ d_0 = \cos \vartheta, \ \ p=\bar p = 2,
$$
to conclude from case $(iii)$ in \eqref{casi_C} that $\sin \vartheta \le (\cos \vartheta) \hat u$. Since $u=v$, we eventually contradict \eqref{strano_soliton}.
\end{proof}

\begin{remark}\label{rem_translator}
\emph{The requirement $\vartheta >0$, that is, $(Y, \partial_t) \neq 0$, is essential in the above theorem because otherwise the slices $\{s=\mathrm{const}\}$ are trivial self-translators. Furthermore, condition \eqref{strano_soliton} is sharp. Indeed, a totally geodesic hyperplane $\Sigma$ corresponding to an affine, non-constant solution $v$ is a self-translator with respect to each vector $Y \in T\Sigma$, and in this case $\tan \vartheta = \hat v$.
}
\end{remark}

\begin{remark}
\emph{A result related to Theorem \ref{teo_soliton_Rm} appears in \cite{baoshi}, where the authors proved that there exist no nontrivial complete self-translator (i.e. not a hyperplane) whose Gauss image lies in a geodesic ball of $\Sph^{m}$ of radius $< \pi/2$. Their assumption implies that $\Sigma$ is a graph with respect to the plane orthogonal to the center of the ball (not necessary the direction of translation), and the graph function $v$ has bounded gradient\footnote{In fact, they also require that $H$ be bounded, but this automatically follows from the self-translator equation $\overrightarrow{H}= Y^\perp$, being $Y$ constant.}. Our requirements \eqref{ipo_soliton_rm} and \eqref{strano_soliton} seem to be skew with their one. However, it might be possible that a suitable gradient estimate guarantees that a self-translator satisfying \eqref{ipo_soliton_rm} has automatically bounded gradient. If this is the case, the main result in \cite{baoshi} would imply that the graph $\Sigma$ of $\hat v$ be a hyperplane, which would prove Theorem \ref{teo_soliton_Rm} in view of Remark \ref{rem_translator}.
}
\end{remark}

Our last application is for entire graphs $\Sigma^m \ra \R^{m+1}$ which are self-expanders for the mean curvature flow, that is, they move under MCF along the integral curves of the position vector field 
$$
Y(\bar x) = (\bar x^j-q^j) \partial_j = \frac{1}{2} \bar \nabla |\bar x-q|^2 \qquad \forall \, \bar x \in \R^{m+1}, 
$$
for some fixed origin $q \in \R^{m+1}$.  

\begin{theorem}\label{teo_soliton_expa} 
Let $\Sigma = \{ (v(x),x) : x \in \R^m\}$ be an entire graph in $(\R^{m+1}, \metricN)$. If $\Sigma$ is a self-expander for the MCF and $v$ is bounded, then $\Sigma$ is a hyperplane.
\end{theorem}

\begin{proof}
Without loss of generality, we can assume that the center $q$ of the homothetic field $Y$ is placed at the origin. We let $(s, x) \in \R \times \R^m = \R^{m+1}$ be coordinates on $\R^{m+1}$. If $\rho(x) : \R^m \ra \R$ denotes the distance to the origin in $\R^m$, then
$$
Y\big(v(x),x\big) = x^j \partial_j + v \partial_s = \frac{1}{2} \nabla (\rho^2) + v \partial_s.
$$
In view of \eqref{normal_geodesic} and \eqref{prescribed_geodesic}, and since $h=1$, $t=s$, the soliton equation $m H = (Y, \nu)$ satisfied by $u(x) = t(v(x)) = v(x)$ reads
$$
\diver \left( \frac{\nabla u}{\sqrt{1+|\nabla u|^2}} \right) = mH = \frac{1}{\sqrt{1+|\nabla u|^2}} \left[ u - \langle \nabla u, \nabla \frac{\rho^2}{2} \rangle \right].
$$
Rearranging,  
$$
\diver \left( \frac{\nabla u}{\sqrt{1+|\nabla u|^2}} \right) + \left\langle \frac{\nabla u}{\sqrt{1+|\nabla u|^2}}, \nabla \frac{\rho^2}{2} \right\rangle = \frac{u}{\sqrt{1+|\nabla u|^2}}, 
$$
that we rewrite as 
\begin{equation}\label{eq_soli_Rm}
e^{-\rho^2/2} \diver \left( e^{\rho^2/2} \frac{\nabla u}{\sqrt{1+|\nabla u|^2}} \right) = \frac{u}{\sqrt{1+|\nabla u|^2}}.
\end{equation}
Suppose that $v$ is non-constant. An explicit computation shows that 
$$
\liminf_{r \ra \infty} \frac{\log \vol_{e^{\rho^2/2}}(B_r)}{r^2} < \infty.
$$
Since, by assumption, $v$ is bounded, we can apply Theorem \ref{teo_main_2} (adapted to weighted volumes) with the choices $\chi=1$, $\mu = 0$, $p=\bar p = 2$, $f(t) = t$ to deduce $f(u^*) \le 0$, that is, $u \le 0$. Applying the same theorem to $-u$ we infer $u \equiv 0$, that is, $\Sigma$ is a hyperplane containing the origin of $Y$.
\end{proof}

\begin{remark}
\emph{It is worth to stress that, differently from \cite{eckerhuisken, wang, baoshi}, Theorems \ref{teo_soliton_intro}, \ref{teo_soliton_Rm} and \ref{teo_soliton_expa} \emph{do not use the conformality of the soliton vector field $Y$.} In fact, what is needed to apply our results, possibly with a weighted volume, is just that the field $Y$ can be decomposed as 
$$
Y(s,x) = h^2(s)\nabla^{\Sigma_s} g + \beta(s,x) \partial_s
$$
where $\Sigma_{s_0} = \{s = s_0\}$, $g \in C^1(\R \times M)$ just depends on the second variable, and $\beta \in C(M \times \R)$. 
}
\end{remark}

\subsection{Counterexamples}\label{sec_counter}

To show the sharpness of Theorem \ref{teo_main_2}, we consider a model manifold $M_g^m$ with radial sectional curvature
$$
K_\rad = -\kappa^2\big(1+r^2\big)^{\alpha/2} \qquad \text{for } \, r(x) \ge 1,
$$
for some $\kappa>0$ and $\alpha \ge -2$. By Propositions 2.1 and 2.11 in \cite{prs}, as $r(x) \ra \infty$
\begin{equation}\label{esti_lapla_model}
\Delta r \ge \left\{ \begin{array}{ll}
\disp (m-1)\kappa r^{\alpha/2}(1+o(1)) & \quad \text{if } \, \alpha > -2; \\[0.3cm]
\disp \frac{(m-1)\bar \kappa}{r}(1+o(1)) & \quad \text{if } \, \alpha = -2,
\end{array}\right.
\end{equation}
where $\bar \kappa = (1+\sqrt{1+4\kappa^2})/2$, and thus
\begin{equation}\label{crescitevol_counter}
\log\vol B_r \sim \left\{ \begin{array}{ll} r^{1+ \frac \alpha 2} & \quad \text{if } \, \alpha>-2; \\[0.2cm]
\big[(m-1)\bar\kappa +1\big] \log r & \quad \text{if } \, \alpha = -2,
\end{array}\right.
\end{equation}
as $r \ra \infty$. If $\alpha=-2$, letting $\kappa \ra 0$ we deduce the classical expressions for $\Delta r$ and $\vol(B_r)$ for the Euclidean space. We shortly write the inequalities in \eqref{esti_lapla_model} as $\Delta r \ge \zeta r^{\alpha/2}$, where $\zeta$ tends to $(m-1) \kappa$ (if $\alpha>-2$) or to $(m-1)\bar\kappa $ (if $\alpha =-2$) as $r(x) \ra \infty$. We are going to find an operator $\Delta_\varphi$ meeting \eqref{assumptions} and \eqref{assu_WPC} with the following property: given $\mu \in \R$, $0 \le \chi \le p-1$ and $\sigma>0$, we will construct an unbounded, radial solution $u \in C^\infty(M_g\backslash \{o\})$, increasing as a function of $r$, solving
\begin{equation}\label{aim}
\Delta_\varphi u \ge K(1+r)^{-\mu}\frac{\varphi(|\nabla u|)}{|\nabla u|^\chi}
\end{equation}
if $r$ is large enough (equivalently, for a high enough upper level set). The solution is $u(x) = r(x)^\sigma$ whenever one of the following conditions hold:
\begin{equation}\label{ipo_volume_couter}
\begin{array}{ll}
1) & \quad \chi\sigma > \chi+1-\mu; \\[0.2cm]
2) & \quad \chi \sigma = \chi+1-\mu, \qquad \text{and} \quad \alpha > -2; \\[0.2cm]
3) & \quad \chi \sigma = \chi+1-\mu, \qquad \alpha = -2, \qquad \sigma \in (0,1] \qquad \text{and} \\[0.2cm]
& \disp \lim_{r \ra \infty} \frac{\log\vol(B_r)}{\log r} > p-\sigma(p-1); \\[0.4cm]
4) & \quad \chi \sigma < \chi+1-\mu, \qquad \alpha > -2, \qquad \text{and} \\[0.2cm]
& \disp \lim_{r \ra \infty} \frac{\log\vol(B_r)}{r^{\chi+1-\mu-\chi\sigma}} >0.
\end{array}
\end{equation}
Moreover, the solution is $u(x) = r(x)^\sigma/\log r(x)$ when either
\begin{equation}\label{ipo_volume_couter_2}
\begin{array}{ll}
5) & \quad \chi \sigma< \chi+1-\mu, \qquad \chi>0, \qquad \text{and} \\[0.2cm]
& \disp \lim_{r \ra \infty} \frac{\log\vol(B_r)}{r^{\chi+1-\mu-\chi\sigma}} = \infty, \qquad \text{or}\\[0.4cm]
6) & \quad \chi \sigma< \chi+1-\mu, \qquad \chi = 0, \qquad \text{and} \\[0.2cm]
& \disp \lim_{r \ra \infty} \frac{\log\vol(B_r)}{r^{\chi+1-\mu-\chi\sigma}}  \in (0, \infty).
\end{array}
\end{equation}
In any of $1), \ldots, 6)$, observe that the bound $\mu \le \chi+1$ s not needed. Once we establish \eqref{aim} on $\Omega_\gamma$ for large enough $\gamma>0$, we can choose $f \in C(\R)$ increasing and satisfying $f\equiv 0$ on $(0, \gamma)$, $f=K$ on $(2\gamma, \infty)$, $0 \le f \le K$ on $\R$. By the pasting Lemma, $\bar u = \max\{u,\gamma\}$ is a $\lip_\loc$ solution of
\begin{equation}\label{oookkk}
\Delta_\varphi \bar u \ge (1+r)^{-\mu}f(\bar u)\frac{\varphi(|\nabla \bar u|)}{|\nabla \bar u|^\chi} \qquad \text{on } \, M.
\end{equation}
(here we used $\chi \le p-1$, to guarantee that $\varphi(t)/t^\chi$ does not diverge as $t \ra 0$ and hence that the constant $\gamma$ solves \eqref{oookkk}). The existence of $\bar u$ under any of $1),\ldots, 6)$ above shows the sharpness of the parameter ranges \eqref{pararange_2} and \eqref{ipo_sigma_teomain2}, and of the growth conditions \eqref{opequeno} for $u$ and \eqref{volgrowth_sigmamagzero} for $\vol(B_r)$. In particular,
\begin{itemize}
\item[-] in $(2)$, all the assumptions are satisfied but the third in \eqref{volgrowth_sigmamagzero}, where the liminf is $\infty$, while in $(3)$, the liminf is finite but bigger than the threshold $p-\sigma(p-1)$ for $\sigma \le 1$;
\item[-] in $(4)$ the requirements in the first of \eqref{volgrowth_sigmamagzero} are all met, but $u_+ = O(r^\sigma)$ instead of $u_+ = o(r^\sigma)$.
\item[-] in $(5)$ and $(6)$, $u_+ = o(r^\sigma)$ but the requirements in the first of \eqref{volgrowth_sigmamagzero} barely fail.
\end{itemize}
To show \eqref{aim} first note that, because of \eqref{crescitevol_counter}, the volume growth conditions in $3)$ and $4)$ are equivalent, respectively, to
\begin{equation}\label{crescitevol_simple}
(m-1)\bar\kappa  +1 > p -\sigma(p-1), \qquad \text{and} \qquad \frac{\alpha}{2} + 1 \ge \chi+1-\mu-\chi\sigma.
\end{equation}
Consider
\begin{equation}\label{def_varphi_couterex}
\varphi(t) = \frac{t^{p-1}}{(1+t)^{q-1}}, \quad p>1, \ 1 \le q \le p,
\end{equation}
and we search for solutions of the form $u(x) = h\big(r(x)\big)$, for some increasing $0< h \in C^2(\R^+)$.  Then,
\begin{equation}\label{laequaacca}
\begin{array}{lcl}
\Delta_\varphi u & = & \disp \varphi'(h')h'' + \varphi(h')\Delta r \ge   \varphi'(h')h'' + \zeta \varphi(h')r^{\alpha/2} \\[0.4cm]
& = & \disp \frac{(h')^{p-2}}{(1+h')^q}\Big\{ \left[(p-1)+(p-q)h'\right]h'' + \zeta r^{\alpha/2}h'(1+h')\Big\}.
\end{array}
\end{equation}
Set $h(t) = t^\sigma$, $\sigma>0$. Then,
\begin{equation}\label{contacci}
\begin{array}{lcl}
\Delta_\varphi u & \ge & \disp C_1\frac{r^{(\sigma-1)(p-1)-1}}{(1 + \sigma r^{\sigma-1})^q} \cdot \\[0.4cm]
& & \disp \cdot \Big\{ \left[(p-1)+\sigma(p-q)r^{\sigma-1}\right](\sigma-1) + \zeta r^{\alpha/2+1}(1+ \sigma r^{\sigma-1})\Big\}
\end{array}
\end{equation}
for some $C_1(p,\sigma)>0$. If $\sigma \ge 1$, getting rid of the first term in brackets and using the definition of $\zeta$ we obtain
\begin{equation}\label{wwwww}
\Delta_\varphi u \ge C_2\frac{r^{(\sigma-1)(p-1)-1 + \left(\alpha /2 + 1\right)}}{(1 + \sigma r^{\sigma-1})^{q-1}}
\end{equation}
On the other hand, if $\sigma \in (0,1]$ and $\alpha > -2$, the term in between brackets in \eqref{contacci} is
\begin{equation}\label{caso11}
(p-1)(\sigma-1)(1+o(1)) + \zeta r^{\alpha/2+1}(1+o(1)) \ge C_3 \zeta r^{\alpha/2+1},
\end{equation}
for large enough $r$, and the last inequality of \eqref{wwwww} still holds. If $\sigma \in (0,1)$ and $\alpha = -2$, the term is
$$
\big[(p-1)(\sigma-1)+\zeta\big](1+o(1)),
$$
while if $\sigma =1$, $\alpha=-2$, the term is simply $2 \zeta$. If $r$ is large enough, the last expression is bounded from below by a positive constant whenever $\lim_{r \ra \infty} \zeta > -(\sigma-1)(p-1)$, that is, when the first in \eqref{crescitevol_simple} is met (i.e. the volume growth in $3)$ of \eqref{ipo_volume_couter}). Summarizing, for each $\sigma>0$ (if $\sigma  \le 1$ and $\alpha = -2$, under the growth condition in $3)$ of \eqref{ipo_volume_couter}), the function $u$ turns out to solve \eqref{wwwww} for suitable constant $C_2>0$. Since
$$
(1+r)^{-\mu} \frac{\varphi(|\nabla u|)}{|\nabla u|^\chi} \le C_4 \frac{r^{-\mu + (\sigma-1)(p-1-\chi)}}{(1+\sigma r^{\sigma-1})^{q-1}},
$$
\eqref{aim} holds for $r$ large enough provided that
$$
-\mu + (\sigma-1)(p-1-\chi) \le (\sigma-1)(p-1)-1 + \left(\alpha /2 + 1\right),
$$
that is, simplifying, if $\frac{\alpha}{2} +1 \ge \chi+1-\mu-\chi\sigma$. The relation is automatically satisfied both in cases $1)$, $2)$ and $3)$ of \eqref{ipo_volume_couter}, and in case $4)$ it is equivalent to the growth condition. We have thus shown \eqref{aim}, as required. \par

To prove $(5)$ and $(6)$, we use $h(t) = t^\sigma/\log t$ in \eqref{laequaacca} and we consider for convenience the $p$-Laplacian ($q=0$) to obtain
$$
\disp \disp\Delta_\varphi u \ge C_1 \frac{r^{(\sigma-1)(p-1)-1}}{\log^{p-1} t} \Big\{ (p-1)(\sigma-1)\big( 1+ o(1)\big) + \zeta r^{1+ \frac{\alpha}{2}} \big( 1+o(1)\big)\Big\},
$$
as $r \ra \infty$, for some constant $C_1(\sigma,p)>0$.

In order to meet the volume conditions in $(5),(6)$ and $\chi \sigma < \chi+1-\mu$, necessarily $\alpha> -2$ and thus 
$$
\disp \disp\Delta_\varphi u \ge C_2 \frac{r^{(\sigma-1)(p-1) + \frac{\alpha}{2}}}{\log^{p-1} t}\big( 1+o(1)\big)
$$
as $r \ra \infty$. On the other hand, by the definition of $\Delta_\varphi$ on $\Omega_\gamma$ we have
$$
(1+r)^{-\mu} \frac{\varphi (|\nabla u|)}{|\nabla u|^\chi} \le C_3 \frac{r^{-\mu +(\sigma-1)(p-1)-\chi (\sigma-1)}}{\log^{p-1} r}\log^\chi r
$$

Combining the last two inequalities, we examine two cases.
\begin{itemize}
\item[-] In case $(5)$, $\chi>0$ and $u$ solves \eqref{aim} whenever $\alpha$ is big enough to satisfy
$$
(\sigma-1)(p-1) + \frac{\alpha}{2} > -\mu +(\sigma-1)(p-1)-\chi(\sigma-1)
$$
that is,
\begin{equation}\label{quattro}
1+\frac{\alpha}{2}>\chi+1-\mu-\chi\sigma>0,
\end{equation}
that  in view of \eqref{crescitevol_counter} implies
$$
\lim_{r \ra \infty} \frac{\log\vol (B_r)}{r^{\chi+1-\mu-\chi\sigma}}= \infty,
$$
as required.
\item[-] In case $(6)$, $\chi=0$ and for $u$ to solve \eqref{aim} it is sufficient the weak inequality
$$
1+\frac{\alpha}{2} \ge \chi+1-\mu-\chi\sigma>0,
$$
and choosing $\alpha$ in order to meet the equality sign, we obtain
$$
\lim_{r \to \infty} \frac{\log\vol(B_r)}{r^{\chi+1-\mu-\chi\sigma}} \in (0, \infty),
$$
as required.
\end{itemize}

\begin{remark}
\emph{Example \ref{def_varphi_couterex} satisfies \eqref{assu_WPC} with $\bar p = p-q$. The case
$$
\begin{array}{ll}
3') & \quad \chi \sigma = \chi+1-\mu, \qquad \alpha = -2, \qquad \sigma >1 \qquad \text{and} \\[0.2cm]
& \disp \lim_{r \ra \infty} \frac{\log\vol(B_r)}{\log r} > \bar p-\sigma(\bar p-1)
\end{array}
$$
is not covered by our counterexamples.
}
\end{remark}

\section{Strong maximum principle and Khas'minskii potentials} \label{sec_SMP}

The aim of this section is to prove Theorem \ref{teo_SMP_intro} in the Introduction. We observe that the argument is based on the existence of what we call a ``Khas'minskii potential", according to the following

\begin{definition}
A Khas'minskii potential at $o \in M$ is a function $\bar{w}$ depending on the parameters $r_0$, $r_1$, $\eta$, $K$, $\eps >0$ satisfying the next requirements:
\begin{equation}\label{Khasm_family}
\left\{ \begin{array}{l}
\Delta_\varphi \bar w \le K b(x) l(|\nabla \bar w|) \qquad \text{on } \, M \backslash B_{r_0} \\[0.2cm]
\bar w > 0 \qquad \text{on } \, M \backslash B_{r_0}, \\[0.2cm]
\bar w \le \eta \qquad \text{ on } \, B_{r_1} \backslash B_{r_0}, \\[0.2cm]
\bar w(x) \ra +\infty \qquad \text{as } \, r(x) \ra \infty, \\[0.2cm]
|\nabla \bar w| \le \eps \quad \text{on } \, M \backslash B_{r_0}
\end{array}\right.
\end{equation}
\end{definition}

The strategy to prove Theorem \ref{teo_SMP_intro} is as follows: assuming, by contradiction, that $\smp$ does not hold, the Khas'minskii potential $w$ will be compared to a non-constant, bounded solution $u \in C^1(M)$ of 
$$
\Delta_\varphi u \ge K b(x)l(|\nabla u|)
$$
on an appropriate subset of $\Omega_{\eta,\eps}$ to reach the conclusion. This approach is very old and we can trace it back, for instance, to Phr\'agmen-Lindel\"off in the realm of classical complex analysis. In more recent times, it has been used by Redheffer \cite{redheffer_2, redheffer}, and later refined in \cite[Sect. 6]{prsmemoirs}, \cite[Thm. 18]{prsoverview} and \cite[Ch. 3]{AMR}. A similar approach, although with a somehow different point of view, was systematically used by Serrin \cite{Serrin_6} and recently by Pucci-Serrin \cite[Thm. 8.1.1]{pucciserrin}. In all of the quoted results, the Khas'minskii potential (or variants thereof) is either a data of the problem or it is constructed ``ad-hoc". These methods seem difficult to be applied when a nontrivial gradient term $l$ is present. Thus, we provide the existence of an appropriate potential via the solution of an associated ODE problem, see Lemma \ref{prop_exi2'} below, that will be coupled with a condition on the Ricci tensor to transplant the function on the original manifold. 

\subsection{Ricci curvature and $\smp$}\label{sec_ricciandSMP}

To investigate the ODE problem mentioned above, we assume the following:

\begin{equation}\label{assum_secODE_1}
\left\{ \begin{array}{l}
\disp \varphi \in C(\R^+_0) \cap C^1(\R^+), \qquad \varphi(0)=0, \qquad \varphi'>0 \ \text{ on } \, \R^+,\\[0.2cm]
\disp f \in C(\R), \qquad f(0)=0, \qquad f>0 \ \text{ on } \, \R^+, \\[0.2cm]
\disp l \in C(\R^+_0), \qquad l > 0 \ \text{ on } \, \R^+,\\[0.2cm]
\beta \in C(\R^+_0), \qquad  \beta >0 \ \text{ on } \, \R^+_0, \\[0.1cm]
\end{array}\right.
\end{equation}
and the next growth conditions:
\begin{equation}\label{assum_secODE_altreL_2}
\left\{ \begin{array}{ll}
\disp \frac{t\varphi'(t)}{l(t)} \in L^1(0^+), & \\[0.4cm]
\disp l(t)\ge C_1 \frac{\varphi(t)}{t^\chi} & \quad \text{on } \, (0,1] \ \text{ for some } \, C_1>0, \ \chi \ge 0 \\[0.4cm]
\disp \varphi(t) \le Ct^{p-1} & \quad \text{on } \, [0,1], \ \text{ for some } \, C>0, \ p>1. \end{array}\right.
\end{equation}
Fix a ``volume" function $v \in C^2(\R^+)$ such that
\begin{equation}\label{v_gbonito}
v>0 \ \text{ on }   \R^+, \qquad v' \ge 0 \qquad \text{on } \, \R^+.
\end{equation}

We are now ready to prove

\begin{lemma}\label{prop_exi2'}
Assume the validity of \eqref{assum_secODE_1}, \eqref{assum_secODE_altreL_2}, \eqref{v_gbonito} and furthermore suppose that
\begin{equation}\label{ipo_bvg}
\begin{array}{l}
\disp \beta(r) \ge C(1+r)^{-\mu} \qquad \text{on $ \, \R^+_0$, for some $ \, \mu \le \chi+1$,}\\[0.3cm]
\disp \limsup_{r \ra \infty} \frac{1}{v(r)} \int_{r_0}^{r} \frac{v(s)}{(1+s)^{\mu}} \di s  < \infty,
\end{array}
\end{equation}
for some (hence any) $r_0>0$, and that either
\begin{equation}\label{volume_ODE}
\begin{array}{l}
\disp \mu < \chi+1 \quad \text{ and }  \quad \disp \liminf_{t \ra \infty} \frac{\log \int_{r_0}^t v}{t^{\chi+1-\mu}} < \infty \quad (=0 \ \text{ if } \, \chi=0), \ \ \text{ or} \\[0.4cm]
\disp \mu = \chi+1 \quad \text{ and } \quad \disp \liminf_{t \ra \infty} \frac{\log \int_{t_0}^t v}{\log t} < \infty \quad ( \le p \ \text{ if } \, \chi=0). \\[0.4cm]
\end{array}
\end{equation}
Then, for each $0<r_0< r_1$, $\eta>0$ and $\eps>0$ there exists a function $w \in C^1([r_0, \infty))$ satisfying:
\begin{equation}
\left\{ \begin{array}{l}
\big[ v\varphi(w')\big]' \le v \beta f(w)l(|w'|) \qquad \text{on } \, [r_0, \infty) \\[0.2cm]
w > 0, \ w'>0 \qquad \text{on } \, [r_0, \infty), \\[0.2cm]
w \le \eta \qquad \text{ on } \, [r_0, r_1], \\[0.2cm]
w(t) \ra \infty \qquad \text{as } \, t \ra \infty, \\[0.2cm]
|w'| \le \eps \quad \text{on } \, [r_0,\infty).
\end{array}\right.
\end{equation}
\end{lemma}

\begin{proof}
Clearly, it is enough to prove the result for $\beta(r) = C(1+r)^{-\mu}$ and for $\eta$ and $\eps$ small enough. First of all, we modify $f,l, \varphi$ and choose a suitable $\eta$. These adjustments will be essential to prove the $L^\infty$-gradient bound for $w$. Fix $\xi \in (0,1]$ and choose $\bar l(t)$ satisfying
\begin{equation}\label{ipo_modil}
\begin{array}{l}
\disp \bar l \in C(\R^+_0), \qquad \bar l> 0 \ \text{ on } \, \R^+, \qquad \bar l \le 2\|l\|_{L^\infty([0,\xi])} \ \text{ on } \, \R^+, \\[0.3cm]
\bar l(t) = l(t) \ \text{ if } \, t \in [0,\xi), \qquad \bar l(t) t^\chi \ge \bar C_2 \ \text{ on } \, [\xi, \infty),
\end{array}
\end{equation}
for some constant $\bar C_2$ (here, we use $\chi \ge 0$). Regarding $\varphi$, we choose $\bar \varphi \in C^1(\R^ +_0)$ such that
$$
\bar \varphi'>0 \ \text{ on } \R^+, \quad \bar\varphi = \varphi \ \text{ on } \, (0, \xi), \quad \bar \varphi \le \bar C_1 t^\chi \bar l(t) \ \text{ on } \, [\xi, \infty)
$$
note that this is possible if $\bar C_1$ is sufficiently large, by the last of  \eqref{ipo_modil}. By construction, since $\varphi(0)=0$,
\begin{equation}\label{bellepropbarvarphi}
\bar \varphi(t) \le Ct^{p-1} \ \text{ on } \, [0,1], \quad \bar \varphi(t) \le C_3 t^\chi \ \text{ on } \, [\xi,\infty), \quad \bar l(t) \ge C_4 \frac{\bar \varphi(t)}{t^\chi} \qquad \text{on } \, \R^+,
\end{equation}
for some constants $C_3,C_4>0$. For $\eta>0$ and $\xi>0$, we introduce the notation
\begin{equation}\label{227_dinuovo}
\begin{array}{ll}
\beta_0= \min_{[r_0,r_1]} \beta, & \quad \beta_1= \max_{[r_0,r_1]} \beta;\\[0.2cm]
v_0 = \min_{[r_0,r_1]} v, & \quad v_1 = \max_{[r_0,r_1]}v;\\[0.2cm]
f_{2\eta} = \max_{[0,2\eta]} f, & \quad l_\xi = \max_{[0,\xi]} \bar l.
\end{array}
\end{equation}
Given $\sigma \in (0, \xi)$ to be specified later, we choose $\eta_\sigma \in [0,\sigma)$ small enough in order to satisfy
\begin{equation}\label{important_fordiri}
\frac{v_1}{v_0}\left[ (r_1-r_0)\beta_1 f_{2\eta_\sigma} l_\xi + \bar \varphi\left( \frac{\eta_\sigma}{r_1-r_0}\right) \right] < \bar \varphi(\sigma).
\end{equation}
This is possible because $\bar \varphi(0)=0$ and $f_{2\eta_\sigma} \ra 0$ as $\eta_\sigma \ra 0$ (since $f(0)=0$). We next choose $f_\sigma \in C(\R_0^+)$ satisfying
\begin{equation}\label{properties_barf}
\begin{array}{ll}
0 \le f_\sigma(t) \le \min\big\{f(t),1\big\}, & \qquad f_\sigma(t)=0 \ \ \text{ if } \, t \le \eta_\sigma, \\[0.2cm]
f_\sigma >0 \quad \text{if } \, t > \eta_\sigma, & \qquad f_\sigma(\eta_\sigma + t) \le K'(t) \ \text{ for } \, t \in [0, \xi],
\end{array}
\end{equation}
where $K(t)$ is the function defined in \eqref{def_K}. The last condition can be satisfied because of the positivity of $\varphi'$ and $l$ on $\R^+$, hence of $K'$. We consider the Dirichlet problem
\begin{equation}\label{diri_partial}
\begin{cases}
\big([v \bar \varphi(w_\sigma')\big)]'= \sigma v \beta f_\sigma(w_\sigma) \bar l(|w_\sigma'|)\quad \text{on } \, [r_0,r_1],\\[0.1cm]
w_\sigma(r_0)= \eta_\sigma, \qquad w_\sigma(r_1)= 2\eta_\sigma, \\[0.1cm]
\eta_\sigma \le w_\sigma \le 2\eta_\sigma, \qquad w_\sigma' > 0 \ \text{ on } \, [r_0,r_1],
\end{cases}
\end{equation}
We claim that a solution $w_\sigma \in C^1([r_0,r_1])$ exists if $\sigma>0$ is small enough. Indeed, one can apply Theorem \ref{exi2} with the following choices:
$$
\begin{array}{l}
t= r-r_0, \qquad T=r_1-r_0, \qquad \quad \wp(t) = v(t+r_0), \\[0.2cm]
w(t) = w_\sigma(r_0+t)-\eta_\sigma,\qquad a(t)=\beta(r_0+t)
\end{array}
$$
and $f(t), \varphi(t), l(t)$ in Theorem \ref{exi2} replaced, respectively, by $\sigma f_\sigma(\eta_\sigma + t), \bar \varphi(t)$ and $\bar l(t)$ (note that $f(0)=0$). It is easy to see that \eqref{important_fordiri} implies \eqref{restrict} for each $\sigma \in (0,\xi]$, hence Theorem \ref{exi2} can be applied to guarantee the existence of a solution $w_\sigma$ of \eqref{diri_partial} together with the bound $|w_\sigma'| \le \sigma$. We are left to prove that $w_\sigma'>0$ on $[r_0,r_1]$, provided that $\sigma$ is small enough. Because of Lemma \ref{lem_ODE}, it is enough to show $w_\sigma'(r_0)>0$ and to this aim we follow the argument in Proposition \ref{prop_twobound_refined}, see also Remark \ref{rem_noserveCincreasing}. Indeed, suppose by contradiction that $w_\sigma'(r_0)=0$. By Lemma \ref{lem_ODE}, there exists $\bar r_\sigma \in [r_0,r_1)$ such that $w_\sigma = \eta_\sigma$ for $r \le \bar r_\sigma$, $w_\sigma'(\bar r_\sigma) =0$ and $w_\sigma'>0$ on $(\bar r_\sigma, r_1]$. Expanding  \eqref{diri_partial} and using $v' \ge 0$ we deduce
$$
\frac{w_\sigma' \bar \varphi'(w_\sigma')}{\bar l(w_\sigma')}w_\sigma'' \le \sigma \beta_1 f_\sigma(w_\sigma)w_\sigma',
$$
and integrating on $(\bar r_\sigma, r)$ we get
\begin{equation}\label{nov}
\bar K(w_\sigma') \le \sigma \beta_1 F_\sigma(w_\sigma),
\end{equation}
where
$$
\bar K(t) = \int_0^t\frac{s \bar \varphi'(s)}{\bar l(s)}\di s, \qquad F_\sigma(t) = \int_{\eta_\sigma}^t f_\sigma(s) \di s.
$$
Because of \eqref{properties_barf}, for $t \in [0,\xi + \eta_\sigma]$
$$
F_\sigma(t) \le \int_{\eta_\sigma}^t K'(s-\eta_\sigma) \di s = K(t-\eta_\sigma) \equiv \bar K(t-\eta_\sigma),
$$
where in the last inequality we have  used $\bar \varphi=\varphi$ and $\bar l = l$ on $[0, \xi]$. Choosing $\sigma \le \beta_1^{-1}$ from \eqref{nov} we deduce  the inequality $\bar K(w_\sigma') \le \bar K(w_\sigma-\eta_\sigma)$, and therefore $w_\sigma' \le w_\sigma-\eta_\sigma$. By Gronwall's inequality and $w_\sigma'(\bar r_\sigma)=0$ we obtain $w_\sigma \equiv \eta_\sigma$ on $[\bar r_\sigma, r_1]$, contradicting $w_\sigma(r_1)=2\eta_\sigma$.\\
We now let $[r_0, R)$, $R > r_1$ be the maximal interval where $w_\sigma$ is defined. Integrating \eqref{diri_partial}, and using $|w_\sigma'| \le \sigma$ on $[r_0,r_1]$ and $\|f_\sigma\|_\infty \le 1$, for $r \in [r_0, R)$ we get
\begin{equation}\label{boundgradwsigma}
\begin{array}{lcl}
\disp \bar\varphi\big(w_\sigma'(r)\big) & = & \disp \frac{1}{v(r)} \left[\bar \varphi\big(w_\sigma'(r_0)\big)v(r_0) + \sigma \int_{r_0}^r v(s)\beta(s)f_\sigma(w_\sigma)\bar l(w_\sigma') \di s \right]\\[0.4cm]
& \le & \disp \frac{1}{v(r)} \left[\bar \varphi(\sigma)v(r_0) + \sigma \|f_\sigma\|_\infty \|\bar l\|_\infty \int_{r_0}^r v(s)\beta(s) \di s\right] \\[0.4cm]
& \le & \disp C(\bar \varphi(\sigma) + \sigma),
\end{array}
\end{equation}
where $C>0$ is a constant depending on $r_0$ but independent of $R, \sigma$, and where the last inequality follows from $v'\ge 0$, the second in \eqref{ipo_bvg} and $\beta(r)=(1+r)^{-\mu}$. In particular, if $\sigma$ is small enough then $w_\sigma$ has bounded gradient. Since $w_\sigma$ is also increasing, necessarily $R = \infty$, otherwise one could extend the solution beyond $R$. Up to a further reduction of $\sigma$, by \eqref{boundgradwsigma} we can guarantee $|w_\sigma'| \le \eps$ on $[r_0, \infty)$, $\eps$ being the parameter in the statement of the Lemma that we can assume to belong on $(0, \xi)$. On the other hand, from the first line in \eqref{boundgradwsigma} and the positivity of $w_\sigma'$ on $[r_0,r_1]$ we deduce $w_\sigma'>0$ on all of $[r_0, \infty)$. It is clear that, for each $\sigma < \xi$,  by construction $w_\sigma$ solves
$$
\big[ v\varphi(w_\sigma')\big]' = \big[ v\bar \varphi(w_\sigma')\big]' = \sigma v \beta f_\sigma(w_\sigma)\bar l(w_\sigma') \le v \beta f(w_\sigma)l(w_\sigma'),
$$
as required. We are left to prove that $w_\sigma(r) \ra \infty$ as $r \ra \infty$. Suppose, by contradiction, that $w_\sigma^* = \sup_{[r_0,\infty)} w_\sigma < \infty$, and consider the $C^2$-model manifold $M_g$ with metric given, in polar coordinates, by
$$
\di s_g^2 = \di r^2 + g(r)^2 \metricN_1, \qquad g(r) = \left\{ \begin{array}{ll}
r & \quad \text{for } \, r \in (0, r_0/2) \\[0.2cm]
v(r)^{\frac{1}{m-1}} & \quad \text{for } \, r \ge r_0.
\end{array}\right.
$$
By construction, the radial function $w_\sigma(r)$ on $M_g$ is a solution of
$$
\Delta_{\bar \varphi} w_\sigma = \sigma \beta(r) f_\sigma(w_\sigma)\bar l(|\nabla w_\sigma|) \qquad \text{on } \, M_g \backslash \overline{B}_{r_0},
$$
and from $w_\sigma'>0$ we get $w_\sigma^* > \eta_\sigma$. Consider the Lipschitz extension of $w_\sigma$ obtained by setting $w_\sigma = \eta_\sigma$ on $B_{r_0}$. An analogous reasoning as in Lemma \ref{lem_pasting} shows that, $f_\sigma(\eta_\sigma)=0$ and $w_\sigma'>0$ guarantee that the extended function solves $\Delta_{\bar \varphi} w_\sigma \ge \sigma \beta(r) f_\sigma(w_\sigma)\bar l(|\nabla w_\sigma|)$ on $M_g$. Combining properties \eqref{ipo_bvg} (for $b$) and \eqref{bellepropbarvarphi} (for $\bar \varphi, \bar l$) with the volume growth conditions  \eqref{volume_ODE}, and fixing any $\bar p>\chi+1$, we are in the position to apply case $(i)$ of Theorem \ref{teo_main_2} to deduce that necessarily $f_\sigma(w_\sigma^*) \le 0$, hence $w_\sigma \le \eta_\sigma$ because of \eqref{properties_barf}. This contradicts the previously established inequality $w_\sigma^*> \eta_\sigma$, and concludes the proof.
\end{proof}

%

\begin{remark}
\emph{We stress that no growth condition on $\varphi$ at infinity is required. Indeed, we applied the weak maximum principle (Theorem \ref{teo_main_2}) to the modification $\bar \varphi$, but a-posteriori just the value of $\varphi$ for sufficiently small $t$ is needed to produce the solution.
}
\end{remark}

\begin{remark}
\emph{The simultaneous validity of the second in \eqref{ipo_bvg} and of \eqref{volume_ODE} requires a delicate balancing between $\mu$ and $v(r)$, since \eqref{ipo_bvg} is easier to satisfy for $\mu$ large while \eqref{volume_ODE} forces an upper bound on $\mu$. For various examples of $v(r)$ with geometric interest, in particular for those appearing in the next theorem, there is a non-empty interval of $\mu$ for which both the conditions are met.
}
\end{remark}

We are now ready to prove Theorem \ref{teo_SMP_intro}. We report here the statement to facilitate the reading.

\begin{theorem}\label{teo_SMP}
Let $M^m$ be a complete Riemannian manifold of dimension $m \ge 2$ such that, for some origin $o \in M$, the distance function $r(x)$ from $o$ satisfies
\begin{equation}\label{ricciassu}
\Ricc (\nabla r, \nabla r) \ge -(m-1)\kappa^2\big( 1+r^2\big)^{\alpha/2},
\end{equation}
for some $\kappa \ge 0$ and $\alpha \ge -2$. Let $l,\varphi$ satisfy \eqref{assum_secODE_1} and \eqref{assum_secODE_altreL_2}, for some $\chi \ge 0$ and $p>1$. Consider $0< b \in C(M)$ such that
$$
b(x) \ge C\big(1+ r(x)\big)^{-\mu} \qquad \text{on } \, M,
$$
for some constants $C>0$, $\mu \in \R$. Assume
\begin{equation}\label{condi_volumericci}
\mu \le \chi - \frac{\alpha}{2} \quad \text{and either} \quad \left\{ \begin{array}{ll}
\alpha \ge -2 \quad \text{and} \quad \chi>0, \quad \text{or} \\[0.2cm]
\alpha = -2, \quad \chi = 0 \quad \text{and} \quad \bar\kappa  \le \frac{p-1}{m-1},
\end{array}\right.
\end{equation}
with $\bar\kappa  = \frac{1}{2}\big( 1 + \sqrt{1+4\kappa^2}\big)$.
Then, the operator $(bl)^{-1}\Delta_\varphi$ satisfies $\smp$.
\end{theorem}

\begin{proof}
Suppose, by contradiction, that there exists a non-constant $u \in C^1(M)$ with $u^* = \sup_M u < \infty$, and $\eta,\eps>0$ such that
$$
\Delta_\varphi u \ge 2K b(x) l(|\nabla u|) \qquad \text{on } \, \Omega_{\eta,\eps} = \big\{ x : u(x) >\eta, \ |\nabla u(x)|< 2\eps \big\},
$$
for some constant $K>0$. In particular, $\Delta_\varphi u \ge 0$ in $\Omega_{\eta,\eps}$, and thus, since maximal points of $u$ belong to $\Omega_{\eta,\eps}$, applying the finite maximum principle (Theorem \ref{teo_FMP2}) to $u^* -u$ with $f \equiv 0$ we deduce that $u^*$ is never attained. Consequently, since $u \in C^1(M)$ we infer that $\Omega_{\eta,\eps}$ is unbounded. In what follows, balls are always considered to be centered at $o$. Fix $r_0>0$ and choose $\gamma \in (0, \eta/2)$ in such a way that 
$$
u(x) < u^* - 4\gamma \qquad \text{for each } \, x \in B_{r_0}.
$$
Next, we choose $\bar x \in \Omega_{\eta, \eps}$ such that $u(\bar x) > u^*- \gamma$, and a large ball $B_{r_1}\Subset M$ containing $\overline{B}_{r_0} \cup \{ \bar x\}$. By \eqref{ricciassu} and \cite[Prop. 2.1]{prs},
\begin{equation}\label{compadelta_ricci_1}
\Delta r \le (m-1)\frac{g'(r)}{g(r)} \qquad \text{weakly on $M$,}
\end{equation}
for some $g(r) \in C^2(\R^+_0)$ increasing and satisfying \eqref{asintog_appe}, that is, 
\begin{equation}\label{asintog}
g(r) \asymp \left\{ \begin{array}{ll}
\exp \left\{ \frac{2\kappa}{2+\alpha} (1+r)^{ 1+ \frac{\alpha}{2}}\right\} & \qquad \text{if } \, \alpha \ge 0 \\[0.4cm]
r^{-\frac{\alpha}{4}} \exp \left\{ \frac{2\kappa}{2+\alpha} r^{ 1+ \frac{\alpha}{2}}\right\} & \qquad \text{if } \, \alpha \in (-2,0) \\[0.4cm]
r^{\bar\kappa }, \quad \bar\kappa  = \frac{1+ \sqrt{1+4\kappa^2}}{2} & \qquad \text{if } \, \alpha = -2
\end{array}
\right.
\end{equation}
as $r \ra \infty$. Define $v_g(r) = \vol(\Sph^{m-1})g(r)^{m-1}$, and note that
\begin{equation}\label{crescitevol_counter_dinuovo}
\log \int_{r_0}^r v_g \sim \left\{ \begin{array}{ll} \frac{2\kappa(m-1)}{2+\alpha} r^{1+ \frac \alpha 2} & \quad \text{if } \, \alpha > -2; \\[0.2cm]
\big[(m-1)\bar\kappa +1\big] \log r & \quad \text{if } \, \alpha = -2
\end{array}\right.
\end{equation}
as $r \ra \infty$. For each $\alpha \ge -2$, set
$$
\beta(r) = C(1+r)^{-\bar \mu}, \qquad \text{with} \qquad \bar \mu = \chi - \frac{\alpha}{2}.
$$
Because of \eqref{condi_volumericci}, $\mu \le \bar \mu$ and therefore
\begin{equation}\label{ipo_bbeta}
b(x) \ge \beta \big(r(x)\big)
\end{equation}
Furthermore, by \eqref{asintog} and the fact that $\bar \mu \ge -\alpha/2$ we get, for each $r_0 >0$,
\begin{equation}\label{vebeta}
\limsup_{r \ra \infty} \frac{1}{v_g(r)} \int_{r_0}^r \frac{v_g(s)}{(1+s)^{\bar \mu}} \di s <\infty
\end{equation}
while, in view of \eqref{crescitevol_counter_dinuovo} we obtain
$$
\begin{array}{ll}
\text{if $\alpha > -2$, $\quad$ then $\bar \mu < \chi +1$ and } & \quad \disp \lim_{r \ra \infty} \frac{\log \int_{r_0}^r v}{r^{\chi+1-\bar \mu}} < \infty; \\[0.5cm]
\text{if $\alpha = -2$, $\quad$ then $\bar \mu = \chi +1$ and } & \quad \disp \lim_{r \ra \infty} \frac{\log \int_{r_0}^r v}{\log r} < \infty \quad \text{($\le p$ if $\chi = 0$).}
\end{array}
$$
We are in the position to apply Lemma \ref{prop_exi2'}: for $\eta=\gamma_\eps \le \gamma$ small enough, there exists $w \in C^1([r_0, \infty))$ satisfying:
\begin{equation}
\left\{ \begin{array}{l}
\big[ v_g\varphi(w')\big]' \le v_g K \beta l(|w'|) \qquad \text{on } \, [r_0, \infty) \\[0.2cm]
w > 0, \ w'>0 \qquad \text{on } \, [r_0, \infty), \\[0.2cm]
w \le \gamma_\eps \qquad \text{ on } \, [r_0, r_1], \\[0.2cm]
w(r) \ra \infty \qquad \text{as } \, r \ra \infty, \\[0.2cm]
|w'| \le \eps \quad \text{on } \, [r_0,\infty)
\end{array}\right.
\end{equation}
We define the radial function $\bar w(x) = w(r(x))$, and we note that, because of \eqref{compadelta_ricci_1}, $w'>0$ and \eqref{ipo_bbeta}, $\bar w$ solves
\begin{equation}\label{def_potKhasminskii}
\left\{ \begin{array}{l}
\Delta_\varphi \bar w \le K \beta(r) l(w'(r)) \le K b(x) l(w'(r)) \qquad \text{on } \, M \backslash B_{r_0} \\[0.2cm]
\bar w > 0 \qquad \text{on } \, M \backslash B_{r_0}, \\[0.2cm]
\bar w \le \gamma_\eps \qquad \text{ on } \, B_{r_1} \backslash B_{r_0}, \\[0.2cm]
\bar w(x) \ra +\infty \qquad \text{as } \, r(x) \ra \infty, \\[0.2cm]
|\nabla \bar w| \le \eps \quad \text{on } \, M \backslash B_{r_0}
\end{array}\right.
\end{equation}
Let $\Gamma$ be the set of maxima of $u- \bar w$, which is non-empty and compact since $\bar w$ has compact sublevel sets. For each $x \in \Gamma$
\begin{equation}\label{chainimpl}
\begin{array}{lcl}
\disp u(x)-\bar w(x) & \ge & \disp u(\bar x)-\bar w(\bar x) > u^*-\gamma - \gamma_\eps \ge u^*- 2\gamma \\[0.2cm]
& > & \disp \max_{B_{r_0}} u \ge \max_{B_{r_0}} (u-\bar w),
\end{array}
\end{equation}
hence $\Gamma \subset M \backslash \overline{B}_{r_0}$. From the first line in \eqref{chainimpl}, we get
$$
u(x) \ge u^* - 2 \gamma + \bar w(x) \ge u^* - 2 \gamma > u^* -\eta,
$$
and thus $\Gamma \subset \big(M \backslash \overline{B}_{r_0}\big) \cap \{u> u^*-\eta\}$. Furthermore,  for each $x \in \Gamma \backslash \cut(o)$ it holds
$$
|\nabla u(x)|=|\nabla \bar w(x)|= w'\big(r(x)\big) \le \eps.
$$
We claim that the same relation holds even for $x \in \Gamma \cap \cut(o)$. Let $\sigma : [0, r(x)] \ra M$ be a unit speed minimizing geodesic from $o$ to $x$, and for $0< \tau <<1$ define $r_\tau(\cdot) = \tau + \dist(., \sigma(\tau))$. Then, $r_\tau \ge r$, with equality at $x$, and furthermore $r_\tau$ is smooth around $x$. Setting $\bar w_\tau(x) = w(r_\tau(x))$, $w'>0$ implies that $\bar w_\tau \ge \bar w$, with equality at $x$. Hence $x$ is a maximum for $u- \bar w_\tau$, which gives
$$
|\nabla u(x)| = |\nabla w_\tau(x)| = w'\big(r_\tau(x)\big) = w'\big(r(x)\big) \le \eps.
$$
We have therefore shown that $\Gamma \Subset \Omega_{\eta, \eps}$. Equality $|\nabla u|=w'(r)$, combined with $w \in C^1$ and $w'>0$ on $[r_0, \infty)$, guarantee the existence of $\delta >0$ such that $\delta \le |\nabla u| \le \eps$ on $\Gamma$. Using the continuity and positivity of $l$, we can fix a small open neighbourhood $V \Subset \Omega_{\eta, \eps}$ of $\Gamma$ of the form $\{u-\bar w > c\}$, $c$ close enough to $\max\{u-\bar w\}$, such that $l(|\nabla u|) \ge 2^{-1} l(w'(r))$ on $V$. Consequently,
$$
\left\{ \begin{array}{ll}
\disp \Delta_\varphi u \ge 2K b(x)l(|\nabla u|) \ge K b(x)l\big(w'(r(x)\big) \ge \Delta_\varphi \bar w = \Delta_\varphi(\bar w +c) \qquad \text{on } \, V \\[0.3cm]
u = \bar w + c \qquad \text{on } \, \partial V.
\end{array}
\right.
$$
By comparison, $u \le \bar w + c$ on $V$, contradicting the very definition of $V$.
\end{proof}

\begin{remark}
\emph{Unfortunately, Lemma \ref{prop_exi2'} cannot be applied as above to prove Theorem \ref{teo_SMP} also in the range $\alpha > -2$ and $\chi=0$. We recall that, for the $p$-Laplace operator, this corresponds to gradient terms with borderline growth $l(t) \asymp t^{p-1}$. In fact, for \eqref{vebeta} to hold it is necessary that $\bar \mu \ge -\alpha/2$, but on the other hand, because of  \eqref{crescitevol_counter_dinuovo},
$$
\begin{array}{l}
\disp \text{if $\bar \mu < \chi +1 =1$} \quad \text{then} \quad  \liminf_{r \ra \infty} \frac{\log \int_{r_0}^r v}{r^{\chi+1-\bar \mu}} = 0 \quad \text{iff} \quad \bar \mu < \chi -\frac{\alpha}{2} = -\frac{\alpha}{2} \\[0.5cm]
\disp \text{if $\bar \mu = \chi +1 =1$} \quad \text{then} \quad  \liminf_{r \ra \infty} \frac{\log \int_{r_0}^r v}{\log r} \le p \quad \text{does not hold for any $\alpha>-2$}.
\end{array}
$$
Therefore, no choice of $\bar \mu$ is admissible for Lemma \ref{prop_exi2'}.
}
\end{remark}

To better appreciate Theorem \ref{teo_SMP}, we express it for the mean curvature operator, both in the Euclidean space $\R^ m$ and in the hyperbolic space $\HH^m$.

\begin{corollary}\label{cor_SMP_MCO}
Let $l \in C(\R_0^+)$ satisfy
\begin{equation}\label{boundsl_MCO}
l(t) \ge C_1 \frac{t^{1-\chi}}{\sqrt{1+t^2}} \qquad \text{on } \, [0,1],
\end{equation}
for some $\chi \in [0,1]$. Then, $\smp$ holds for the operator
$$
\big(1+r(x)\big)^{\mu}l\big(|\nabla u|\big)^{-1} \diver \left( \frac{\nabla u}{\sqrt{1+|\nabla u|^2}}\right),
$$
\begin{itemize}
\item[(i)] in $\R^ m$, provided that $\mu \le \chi + 1$ and either $\chi>0$ or $\chi=0$ and $m = 2$;
\item[(ii)] in $\HH^m$, provided that $\mu \le \chi$ and $\chi>0$.
\end{itemize}
\end{corollary}

\begin{proof}
Let $\varphi(t) = t/\sqrt{1+t^2}$, and choose $p=2$ in \eqref{assum_secODE_altreL_2}. To recover the Euclidean space set $\kappa=0$, $\alpha = -2$, while for the hyperbolic space set $\kappa=1$, $\alpha =0$. The rest of the proof is a direct application of Theorem \ref{teo_SMP}.
\end{proof}

As a further application of Theorem \ref{teo_SMP}  in the next corollary we obtain a Liouville theorem for bounded solutions of
\begin{equation}\label{Pug_general1}
\Delta_p u \ge b(x) f(u) |\nabla u|^{q} - \bar b(x) \bar f(u) |\nabla u|^{\bar q}.
\end{equation}
\begin{corollary}\label{cor_SMP}
Let $(M,\metric)$ be a $m-$dimensional complete manifold satisfying
\begin{equation}\label{ricciassu_corollary}
\Ricc (\nabla r, \nabla r) \ge -(m-1)\kappa^2\big( 1+r^2\big)^{\alpha/2} \qquad \text{on } \, \mathcal{D}_o,
\end{equation}
for some $\kappa \ge 0$, $\alpha \ge -2$. Consider $0< b \in C(M)$ such that
$$
b(x) \ge C_1\big(1+ r(x)\big)^{-\mu} \qquad \text{on } \, M,
$$
for some $C_1>0$. Let $f,\bar f \in C(\R)$, $\bar b \in C(M)$ and $C>0$ such that
\begin{equation}\label{ipobarbbarf}
\bar b \le Cb \quad \text{on } \, M, \qquad \bar f \le Cf \quad \text{on } \, \R.
\end{equation}
Fix
$$
p \in (1,\infty), \qquad q \in [0, p-1), \qquad \bar q > q
$$
and consider a bounded above solution $u \in C^1(M)$ of
\begin{equation}\label{Pge_general}
\Delta_p u \ge b(x) f(u) |\nabla u|^{q} - \bar b(x) \bar f(u) |\nabla u|^{\bar q}.
\end{equation}
If
\begin{equation}\label{condi_volumericci_coro}
\mu \le p-1-q - \frac{\alpha}{2}
\end{equation}
and $u$ is non-constant, then $f(u^*) \le 0$. In particular, if $u \in C^1(M) \cap L^\infty(M)$ solves \eqref{Pug_general1} with the equality sign, and
\begin{equation}\label{ipobarbbarf_peruguale}
C^{-1} f \le \bar f \le Cf \qquad \text{on } \, \R,
\end{equation}
then, $u$ must be constant in each of the following cases:
\begin{itemize}
\item[(i)] $f<0$ on $(-\infty,t_0)$ and $f>0$ on $(t_0,\infty)$;
\item[(ii)] $f$ has no zeroes.
\end{itemize}
\end{corollary}

\begin{remark}
\emph{If $q=0$, under $(i)$ above the only constant solution of \eqref{Pug_general1} with the equality sign is $u \equiv t_0$, while, under $(ii)$, \eqref{Pug_general1} with the equality sign does not admit any constant solution.
}
\end{remark}

\begin{proof}
Suppose by contradiction that $f(u^*) = 4K>0$, and pick $\eta < u^*$ such that $f(t)>2K$ if $t>\eta$. Because of \eqref{Pge_general} and \eqref{ipobarbbarf}, on
$$
\Omega_{\eta,\eps} = \big\{ x \in M \ \ : \ \ u(x)>\eta, \ |\nabla u(x)| < \eps\big\}
$$
we have
$$
\Delta_p u \ge b(x)f(u)|\nabla u|^q \Big( 1 - C^2|\nabla u|^{\bar q-q}\Big) \ge 2K b(x)|\nabla u|^q \big(1-C^2\eps^{\bar q-q}\big),
$$
and since $\bar q>q$ we can choose $\eps>0$ small enough that
\begin{equation}\label{alSMP}
\Delta_p u \ge K b(x)|\nabla u|^q \qquad \text{on } \, \Omega_{\eta,\eps}.
\end{equation}
Set $\chi = p-1-q \in (0,p-1]$. Then, in our assumptions, we can apply Theorem \ref{teo_SMP} to deduce that $(bl)^{-1}\Delta_p$ satisfies $\smp$. Consequently, from \eqref{alSMP} we get $K \le 0$, contradiction.\\
Suppose now, by contradiction, that $u \in C^1(M) \cap L^\infty(M)$ is a non-constant solution of \eqref{Pug_general1} with the equality sign. By the first part of the proof we get $f(u^ *)\le 0$. Next, observe that $\bar u = -u$ solves
\begin{equation}\label{Pge_general_ubar}
\Delta_p \bar u = b(x) f_1(\bar u) |\nabla \bar u|^{q} - \bar b(x) \bar f_1(\bar u) |\nabla \bar u|^{\bar q},
\end{equation}
with $f_1(t) = -f(-t)$ and $\bar f_1(t) = -\bar f(-t)$. In view of \eqref{ipobarbbarf_peruguale}, applying again the first part we obtain $f_1(\bar u^*)\le 0$, that is, $f(u_*) \ge 0$ with $u_* = \inf_M u$. From $f(u^*) \le 0 \le f(u_*)$ and using $(i)$ or $(ii)$, we deduce that $u$ is necessarily constant, contradiction.
\end{proof}

\begin{remark}
\emph{Theorem \ref{teo_SMP} could be improved to include slowly growing solutions of $(P_\ge)$ as in $(ii)$ of Theorem \ref{teo_main_2}, provided that one is able to estimate from below the order of growth of a family of Khas'minskii potentials \eqref{Khasm_family} in a way independent of the origin $o$ and of $\eta, r_0,r_1,\eps$. If this holds, repeating the proof verbatim one shows that any solution $u$ of $(P_\ge)$ on $M$, or on some upper level set, is bounded from above and satifies $f(u^*) \le 0$ whenever
$$
u_+(x) = o\big( \bar w(x) \big) \qquad \text{as } \, r(x) \ra \infty,
$$
with $\bar w$ being any of such Khas'minskii potentials. Growth estimates are achieved provided that one can explicitly exhibit $\bar w$, and this is the case when $l(0)>0$. Indeed, when $l(0)>0$  the first two of \eqref{assum_secODE_altreL_2} are automatically satisfied, and to produce solutions of \eqref{Khasm_family} we can consider radial solutions of 
$$
\big( v_g\varphi(w')\big)' = \sigma v_g \beta \qquad \text{on } \, [r_0,\infty),
$$
for small enough $\sigma>0$. Explicit integration with $w(r_0) = w'(r_0)=0$ gives
$$
w(r) = \int_{r_0}^r \varphi^{-1} \left( \frac{\sigma}{v_g(t)}\int_{r_0}^t v_g(s)\beta(s) \di s\right) \di t.  
$$
This approach has been developed in Section 6 of \cite{prsmemoirs} and in \cite[Thm. 18]{prsoverview}, the latter dealing with inequality
$$
\Delta u \ge (1+r)^{-\mu}f(u)l(|\nabla u|)
$$
on complete manifolds satisfying $\Ricc \ge -(m-1)\kappa^2\metric$, for some $\kappa >0$, for increasing $f$ and for $\mu \in [0,1]$. The conclusion $f(u(x)) \le 0$ on $M$ is shown to hold provided that, as $r(x) \ra \infty$,
$$
u(x) = \left\{ \begin{array}{ll}
o\big( r(x)^{1-\mu}\big) & \quad \text{if } \, \mu \in [0,1), \\[0.2cm]
o\big( \log r(x)\big) & \quad \text{if } \, \mu=1.
\end{array}
\right.
$$
Inspection shows that the case $\mu<1$ well fits with $(ii)$ of Theorem \ref{teo_main_2} (apply with $\varphi(t)=t$, $\sigma = 1-\mu$, $\chi=1$ and use Bishop-Gromov comparison to check the first of \eqref{volgrowth_sigmamagzero}). Case $\mu=1$, on the other hand, has no analogue in Theorem \ref{teo_main_2}.
}
\end{remark}

\subsubsection{Bernstein theorems for prescribed mean curvature}

We now apply $\smp$ to entire graphs with prescribed mean curvature in a warped product $\bar M = \R \times_h M$. We recall that the mean curvature of the totally umbilic slice $\{s=s_0\}$ of $\bar M$ in the upward direction $\partial_s$ is
$$
H_{\partial_s}\big(\{s=s_0\}\big) = - \frac{h'(s_0)}{h(s_0)}.
$$
The next theorem gives an a-priori estimate for entire graphs with prescribed mean curvature, and in particular it characterizes all constant mean curvature entire graphs. For simplicity, we state the result for warped products with $h(s) = \cosh s$, a class including the fibration $\HH^{m+1} = \R \times_{\cosh s} \HH^m$ by hyperspheres $\{s=s_0\}$ of constant mean curvature $H = -\tanh s_0 \in (-1,1)$. 

\begin{theorem}\label{teo_presc_nonconst}
Let $\bar M = \R \times_{\cosh s} M$, for some complete manifold $(M^m, \metric)$ whose Ricci tensor satisfies
$$
\Ricc(\nabla r, \nabla r) \ge -(m-1) \kappa^2(1+r)^2 \qquad \text{on } \, \mathcal{D}_o,
$$
for some constant $\kappa > 0$. Fix a constant $H_0 \in (-1,1)$, and consider an entire geodesic graph of $v : M \ra \R$ with prescribed curvature $H(x) \ge -H_0$ in the upward direction. Then, $v$ is bounded from above and satisfies
\begin{equation}\label{bound_graph}
v^* \le \mathrm{arctanh}(H_0). 
\end{equation}
In particular, 
\begin{itemize}
\item[(i)] there is no entire graph with prescribed mean curvature satisfying $|H(x)| \ge 1$ on $M$;
\item[(ii)] the only entire graph with constant mean curvature $H_0 \in (-1,1)$ in the upward direction is the totally umbilic slice $\{s = \mathrm{arctanh}(H_0)\}$.
\end{itemize}
\end{theorem}

\begin{proof}
Define $t, \lambda(t)$ and $u(x)$ as in Subsection \ref{subsec_bernstein} in the Introduction with $h(s) = \cosh s$:
$$
t(s) = \int_0^s \frac{\di \sigma}{\cosh \sigma} = 2 \arctan(e^s) - \frac{\pi}{2}, \quad \lambda(t) = h(s(t)), \quad u(x) = t(v(x)).
$$
Note that $u : M \ra \left(-\frac{\pi}{2},\frac{\pi}{2}\right)$. Since $\lambda(u) = \cosh v$ and $\lambda_t(u)/\lambda(u) = \sinh v$, by \eqref{prescribed_geodesic} $u$ satisfies
\begin{equation}
\begin{array}{lcl}
\disp \diver\left( \frac{\nabla u}{\sqrt{1+|\nabla u|^2}} \right) & = & \disp m \cosh v \left[ H(x) + \tanh v \frac{1}{\sqrt{1+|\nabla u|^2}}\right] \\[0.5cm]
& \ge & \disp m \cosh v \left[- H_0 + \tanh v \frac{1}{\sqrt{1+|\nabla u|^2}}\right].
\end{array}
\end{equation}
Suppose, by contradiction, that the following upper level set of $v$ (hence, of $u$) is non-empty for some $\eta>0$: 
$$
\Omega_\eta =\{ \tanh v > H_0 + \eta\}.
$$
Then,
\begin{equation}
\disp \diver\left( \frac{\nabla u}{\sqrt{1+|\nabla u|^2}} \right) \ge \frac{m \cosh v}{\sqrt{1+|\nabla u|^2}} \left[\eta - H_0(\sqrt{1+|\nabla u|^2} -1)\right] \qquad \text{on } \, \Omega_\eta.
\end{equation} 
If $H_0<0$, from $\cosh v \ge 1$ we deduce 
\begin{equation}\label{eq_prescribed_nonconst}
\disp \diver\left( \frac{\nabla u}{\sqrt{1+|\nabla u|^2}} \right) \ge \frac{m\eta}{\sqrt{1+|\nabla u|^2}} \qquad \text{on } \, \Omega_\eta.
\end{equation} 
On the other hand, if $H_0>0$, for $\eps>0$ we consider the set $\Omega_{\eta,\eps} = \Omega_\eta \cap \{ |\nabla u|< \eps\}$. Note that $\Omega_{\eta,\eps}$ is non-empty by Ekeland's principle, since $M$ is complete. If $\eps$ is sufficiently small, the term in square brackets is less than $\eta/2$, and since $\cosh v \ge 1$ we deduce
\begin{equation}\label{eq_SMP_prescribed}
\disp \diver\left( \frac{\nabla u}{\sqrt{1+|\nabla u|^2}} \right) \ge \frac{m\eta}{2\sqrt{1+|\nabla u|^2}} \qquad \text{on } \Omega_{\eta,\eps}.
\end{equation} 
We now apply Theorem \ref{teo_SMP} with the choices $\alpha = 1$, $\mu = 0$, $\chi = 1$ to deduce the validity of $\smp$ for the operator $l^{-1} \Delta_\varphi$, with 
$$
\varphi = \frac{t}{\sqrt{1+t^2}}, \qquad l(t) = \frac{1}{\sqrt{1+t^2}}.
$$
Since $u$ is bounded from above, applying $\smp$ to \eqref{eq_SMP_prescribed} (for $H_0>0$) or to \eqref{eq_prescribed_nonconst} (for $H_0<0$) we reach the desired contradiction.\par
To prove $(i)$, suppose that $|H(x)| \ge 1$ on $M$. Since $\cosh s$ is even, the graph of $-v$ has curvature $-H(x)$ in the upward direction. Thus, up to replacing $v$ with $-v$ we can suppose that $H(x) \ge 1$. Applying the first part of the theorem to any $H_0 > -1$ we obtain $v^* \le \mathrm{arctanh}(H_0)$, and the non-existence of $v$ follows by letting $H_0 \ra -1$.\par
To prove $(ii)$, let $H(x) = -H_0 \in (-1,1)$ be the mean curvature of the graph of $v$ in the upward direction. Then, Theorem \ref{teo_presc_nonconst} gives $\tanh v^* \le H_0$. On the other hand, the graph of $-v$ has mean curvature $H(x) = H_0$ in the upward direction, and applying again Theorem \ref{teo_presc_nonconst} we deduce $\tanh [(-v)^*] \le -H_0$, that is, $\tanh v_* \ge H_0$. Combining the two estimates gives $v \equiv \mathrm{arctanh}(H_0)$, as required.   
\end{proof}

\begin{remark}
\emph{If $H_0<0$, to conclude from \eqref{eq_prescribed_nonconst} it is sufficient to require the validity of $\wmp$.
}
\end{remark}

\begin{remark}
\emph{Observe that $(ii)$ generalizes item $(ii)$ in Do Carmo-Lawson Theorem \ref{teo_docarmolawson}: it is sufficient to apply Theorem \ref{teo_presc_nonconst} to $\HH^{m+1}$ with the warped product structure $\R \times_{\cosh r} \HH^m$. 
}
\end{remark}

\begin{remark}
\emph{The above result can be generalized, with the same proof, to warped products $\R \times_h M$ for $h$ satisfying
$$
\left\{ \begin{array}{l}
h \ \ \text{ even},\\[0.2cm]
h^{-1} \in L^1(-\infty) \cap L^1(+\infty), \\[0.2cm]
(h'/h)' >0 \qquad \text{on } \, \R.
\end{array}\right.
$$
We leave the statement to the interested reader.
}
\end{remark}

\section{The compact support principle}\label{sec_CSP}

Consider the problem
\begin{equation}\label{csp_problem}
\left\{ \begin{array}{l}
\Delta_\varphi u \ge b(x)f(u)l(|\nabla u|) \qquad \text{on } \, \Omega \, \text{ end of $M$.} \\[0.2cm]
\disp u \ge 0, \qquad \lim_{x \in \Omega, \, x \ra \infty} u(x) = 0.
\end{array}\right.
\end{equation}
We recall that an end $\Omega \subset M$ is a connected component with non-compact closure of $M\backslash K$, for some compact set $K$. In this section, we investigate the necessity and sufficiency of condition
\begin{equation}\label{KO_zero}\tag{KO$_0$}
\frac{1}{K^{-1} \circ F} \in L^1(0^+)
\end{equation}
for the validity of the compact support principle $\csp$, that is, the statement that each $u$ solving \eqref{csp_problem} has compact support. We assume the following:

\begin{equation}\label{assumptions_CSP_necessity}
\left\{\begin{array}{l}
\varphi \in C(\R^+_0) \cap C^1(\R^+), \qquad \varphi(0)=0, \qquad \varphi'>0 \ \text{ on } \, \R^+, \\[0.3cm]
f \in C(\R), \qquad f \ge 0 \ \text{ in } \, (0, \eta_0), \ \text{ for some } \, \eta_0 \in (0, \infty), \\[0.3cm]
l \in C(\R^+_0), \qquad l>0 \ \text{on } \, \R^+, 
\end{array}\right.
\end{equation}
and moreover
\begin{equation}\label{phiel_solozero}
\frac{t \varphi'(t)}{l(t)} \in L^1(0^+).
\end{equation}
Having defined $F,K$ as in \eqref{def_Fe_intro} and \eqref{def_K}, that is,
\begin{equation}\label{def_K_CSP}
K(t) = \int_0^t \frac{s \varphi'(s)}{l(s)}\di s, \qquad F(t) = \int_0^t f(s) \di s,
\end{equation}
set $K_\infty = \lim_{t \ra \infty} K(t) \in (0, \infty]$; since $\varphi'>0$, the inverse $K^{-1} : [0, K_\infty) \ra \R^+$ exists, and \eqref{KO_zero} is meaningful. In most of the results, we also require
\begin{equation}\label{assu_perCSP}
\left\{\begin{array}{l}
\text{$f$ is $C$-increasing on $[0, \eta_0)$,} \\[0.2cm]
\text{$l$ is $C$-increasing on $[0, \xi)$, for some $\xi >0$.}
\end{array}\right.
\end{equation}
We underline that condition $f(0)l(0)=0$ does not appear in \eqref{assumptions_CSP_necessity}. In fact, some of the next results do not need it. As usual, having fixed a relatively compact, smooth open set $\mathcal{O} \subset M$ we denote with $r(x) = \mathrm{dist}(x, \mathcal{O})$.


\subsection{Necessity of \eqref{KO_zero_intro} for the compact support principle}
Suppose the failure \eqref{KO_zero_intro_FMP} of the Keller-Osserman condition. Because of Theorem \ref{teo_FMP2}, under assumptions \eqref{assumptions_CSP_necessity}, \eqref{phiel_solozero} and $f(0)l(0)=0$ each $C^1$ solution of \eqref{csp_problem} \emph{with the equality sign} must satisfy $\fmp$, and consequently it cannot be compactly supported. However, finding solutions with the equality sign for \eqref{csp_problem}, and especially proving their $C^1$-regularity, seems to be tricky in the generality of \eqref{assumptions_CSP_necessity} and \eqref{phiel_solozero}. For this reason, we follow a different path producing, on \emph{each} complete manifold, radial solutions of inequality \eqref{csp_problem} which are positive on $\Omega= M\backslash B_{r_0}(\mathcal{O})$. The $C^1$-regularity will be therefore a consequence of the assumption that the radial function be smooth, that is, that the origin $\mathcal{O}$ be a pole of $M$.\par
The key step is provided by the following theorem that considers the exterior Dirichlet problem. Fix $r_0>0$ and functions $v,\beta$ satisfying
\begin{equation}\label{ipo_betav_CSP}
\begin{array}{l}
v \in C^1\big([r_0, \infty)\big), \qquad v>0, \ v' \ge 0 \quad \text{on } \, [r_0, \infty); \\[0.2cm]
\beta \in C\big([r_0, \infty)\big), \qquad \beta>0 \quad \text{on } \, [r_0,\infty).
\end{array}
\end{equation}
For $\eta, \xi >0$, define $f_\eta, l_\xi$ as in \eqref{227}, that is,
$$
f_\eta = \max_{[0,\eta]}f, \qquad l_\xi = \max_{[0,\xi]} l.
$$
We are ready to state
\begin{theorem}\label{teo_exteriorDiri}
Let $\varphi, f, l$ satisfy \eqref{assumptions_CSP_necessity} and
\begin{equation}\label{assu_adicionais}
\left\{\begin{array}{lcl}
\text{$f$ is non-decreasing on $(0,\eta_0)$,}\\[0.2cm]
l \in \lip_\loc(\R^+), \\[0.2cm]
f(0)l(0)=0.
\end{array}\right.
\end{equation}
Fix $r_0>0$ and let $v,\beta$ as in \eqref{ipo_betav_CSP}.
Then, for each $R>0$, $\xi \in (0,1)$ and $\eta \in (0, \eta_0)$ (with $\eta_0$ as in \eqref{assum_secODE}) satisfying
\begin{equation}\label{theuniformbound}
\frac{v(r_0+R)}{v(r_0)} \varphi\left(\frac{\eta}{R}\right) + f_\eta l_{\xi} \left[\sup_{[r_0,r_0+R)} \frac{1}{v(r)} \int_{r_0}^r v(s)\beta(s) \di s\right]  < \varphi (\xi),
\end{equation}
there exists a solution $z \in C^1([r_0,\infty))$ of
\begin{equation}\label{eq_exteriorDiri}
\left\{\begin{array}{ll}
\big[v \varphi(z')\big]' = \beta v f(z) l(|z'|) & \quad \text{on } \, [r_0, \infty) \\[0.2cm]
z(r_0) = \eta, \qquad -\xi < z' \le 0 & \quad \text{on } \, [r_0, \infty).
\end{array}\right.
\end{equation}
Furthermore, if
\begin{equation}\label{varphinonparab_model}
\varphi^{-1} \left(\frac{c}{v(r)}\right) \in L^1(\infty), \qquad \text{for some constant } \, c>0, 
\end{equation}
there exists $\eta_1= \eta_1(v,c,\varphi)$ such that, for each $\eta \in (0,\min\{\eta_0,\eta_1\})$ satisfying \eqref{theuniformbound}, $z(r) \ra 0$ as $r \ra \infty$.
\end{theorem}
\begin{remark}\label{rem_varphinonparab}
\emph{Condition \eqref{varphinonparab_model}, to be meaningful, needs to be considered on the interval of integration of the type $[r_c, \infty)$ where the integrand is well defined, that is, because of the monotonicity of $v$, for $c < v(r_c)\varphi(\infty)$. The existence of such $r_c$ is implicit since the validity of \eqref{varphinonparab_model} and the monotonicity of $v$ force $\lim_{r \ra \infty} v(r) = \infty$. 
}
\end{remark}

\begin{proof}
Set 
$$
h = \max\left\{1, \, 2 \frac \eta R \right\}.
$$
We define $\bar \varphi, \bar l$ on $\R^+_0$ as follows:
$$
\begin{array}{l}
\disp \bar \varphi(t) = \varphi(t) \ \text{ on } \, [0, h], \qquad \bar \varphi(t) = \varphi(h) + (t-h) \ \text{ on } \, (h, \infty);\\[0.2cm]
\bar l \in C(\R^+_0), \qquad \bar l = l \quad \text{on } \, [0,\xi], \qquad 0< \bar l \le l_{\xi} \quad \text{on } \, \R^+,
\end{array}
$$
and we extend $\bar \varphi$ to an odd function on the entire $\R$. For each $j \in \mathbb{N}$, $j \ge 1$ set also
$$
\qquad \wp_j(t) = v(r_0+jR-t), \qquad a_j(t) = \beta(r_0+jR-t),
$$
and let $w_j$ be a solution of the Dirichlet problem
\begin{equation}\label{eq_dirineum_exterior}
\left\{\begin{array}{l}
\big[ \wp_j \bar \varphi(w_t) \big]_t = a_j \wp_j f(w)\bar l(|w_t|) \quad \text{on } \, [0, jR], \\[0.2cm]
w(0)=0, \qquad w(jR)= \eta, \\[0.2cm]
0 \le w \le \eta, \qquad w_t \ge 0 \quad \text{on } \, [0, jR], \\[0.2cm]
\end{array}\right.
\end{equation}
where the subscript $t$ denotes differentiation in the $t$ variable. We stress that $w_j$ exists for each $j$. Indeed, we shall apply Theorem \ref{exi2} with the parameter $\xi$ replaced by some suitably chosen $\bar \xi$. Note that \eqref{restrict} is satisfied up to choosing $\bar \xi$ sufficiently large, because $\bar\varphi(\infty) = \infty$ and $\bar l_{\bar \xi} \le l_{\xi}$. Observe that $\bar \xi$ might depend on $j$, but this does not affect the rest of the proof. From $\wp' \le 0$, again by Theorem \ref{exi2} we deduce
\begin{equation}\label{esti_deriwj}
0 \le (w_j)_t \le \bar \varphi^{-1} \left( \frac{\wp(0)}{\wp(jR)} \bar \varphi\left( \frac{\eta}{jR}\right) + f_\eta l_{\xi} \left[\sup_{[0,jR]} \frac{1}{\wp_j(t)} \int_0^t \wp_j(s) a_j(s) \di s\right] \right).
\end{equation}
Set $z_j(r) = w_j(r_0+jR-r)$, and note that $z_j$ solves
\begin{equation}\label{eq_dirineum_exterior_j}
\left\{\begin{array}{l}
\big[ v \bar \varphi(z_j') \big]' = \beta v f(z_j)\bar l(|z_j'|) \quad \text{on } \, (r_0, r_0+jR), \\[0.2cm]
z_j(r_0)=\eta, \qquad z_j(r_0+jR)=0 \\[0.2cm]
0 \le z_j \le \eta, \qquad z_j' \le 0 \quad \text{on } \, [r_0, r_0+jR]. \\[0.2cm]
\end{array}\right.
\end{equation}
Next, we estimate the derivative $z_j'$ uniformly in $j$. First, observe that integrating on $[t_1,t_2]$ the inequality $[v \bar \varphi(z')]' \ge 0$ that follows from \eqref{eq_dirineum_exterior_j} and \eqref{assum_secODE}, we deduce
$$
v(t_2) \big[\bar \varphi\big(z'(t_2)\big)- \bar \varphi\big(z'(t_1)\big)\big] \ge \big[v(t_1)-v(t_2)\big] \bar \varphi\big(z'(t_1)\big) \ge 0.
$$
Using $\bar \varphi(z') \le 0$ and $v'\ge 0$ we conclude that $\bar \varphi(z')$, hence $z'$, is increasing. In particular, since $z_j' \le 0$, we have $|z_j'| \le |z_j'(r_0)|$.\\[0.2cm]
\noindent \emph{Claim: } $\{z_j\}$ in an increasing sequence, \\
We show that $z_j \le z_{j+1}$ on $[r_0,r_0+jR]$. Applying Lemma \ref{lem_ODE} to $w_j$ and rephrasing for $z_j$, there exists $r_j \in (r_0,r_0+jR]$ such that $z_j >0$, $z_j'<0$ on $[r_0, r_j)$ while $z_j = 0$ on $[r_j,r_0+jR)$. On $(r_0, r_j)$ it holds
\begin{equation}\label{eq_dirineum_exterior_j0}
\left\{\begin{array}{l}
\big[ v \bar \varphi(z_j') \big]' = \beta v f(z_j)\bar l(|z_j'|) \quad \text{on } \, (r_0, r_j), \\[0.2cm]
\big[ v \bar \varphi(z_{j+1}') \big]' = \beta v f(z_{j+1})\bar l(|z_{j+1}'|) \quad \text{on } \, (r_0, r_j), \\[0.2cm]
z_j(r_0)=z_{j+1}(r_0)=\eta, \qquad z_j(r_j)=0 \le z_{j+1}(r_j).
\end{array}\right.
\end{equation}
The inequality $z_j \le z_{j+1}$ on $[r_0,r_j]$, hence on $[r_0, r_0+jR]$, is then a consequence of the comparison result in Proposition \ref{prop_serrin} applied to the model manifold $M_g = [r_0,\infty) \times \Sph^{m-1}$ with the radially symmetric $C^1$-metric
$$
\di r^2 + g(r)^2 ( \, , \, )_1, \qquad \text{with} \qquad g(r) = v(r)^{\frac{1}{m-1}},
$$
recall also Remark \ref{rem_regularitymetric}.\\[0.2cm]
The convexity and monotonicity of $z_j$, together with the above claim, imply the uniform estimate $|z_j'| \le |z_1'(r_0)|$. Changing variables in \eqref{esti_deriwj} and exploiting \eqref{theuniformbound},
$$
\begin{array}{lcl}
|z_j'| & \le & \disp |z_1'(r_0)| \\[0.5cm]
& \le & \disp \bar \varphi^{-1} \left( \frac{v(r_0+R)}{v(r_0)} \bar \varphi\left( \frac\eta R\right) + f_\eta l_{\xi}\left[ \sup_{[r_0,r_0+R]} \frac{1}{v(r)} \int_{r_0}^r v(s)\beta(s) \di s\right] \right)  < \xi,
\end{array}
$$
where we used again the identity $\bar \varphi = \varphi$ on $[0,h]$, the definition of $h$ and $\xi < 1$. Thererefore, by Ascoli-Arzel\'a the sequence $\{z_j\}$ converges locally uniformly to a solution $z \in C^1([r_0,\infty))$ of
\begin{equation}\label{eq_dirineum_exterior}
\left\{\begin{array}{l}
\big[ v \bar \varphi(z') \big]' = \beta v f(z)\bar l(|z'|) \quad \text{on } \, (r_0, \infty), \\[0.2cm]
z(r_0)=\eta, \\[0.2cm]
0 \le z \le \eta, \qquad -\xi < z' \le 0 \quad \text{on } \, [r_0, \infty),
\end{array}\right.
\end{equation}
Since $\bar \varphi = \varphi$ and $\bar l = l$ on $(0,\xi) \subset (0,1)$, $z$ is the desired solution of \eqref{eq_dirineum_exterior}. Eventually, we assume \eqref{varphinonparab_model}, which in particular implies that $\lim_{r \ra \infty} v(r) = \infty$, and we choose $r_c$ such that
$$
r_c \ge r_0, \quad \frac{c}{v(r_c)} \le \varphi(1).
$$
Define
$$
\bar z(r) = \int_r^{\infty} \varphi^{-1} \left( \frac{c}{v(s)} \right) \di s \qquad \text{on } \, [r_c, \infty).
$$
From $\bar \varphi = \varphi$ on $[0,h] \supset [0,1]$, $\bar z$ solves $0 = \big[v \varphi(\bar z')\big]' = \big[v\bar \varphi(\bar z')\big]' = 0$ on $[r_c,\infty)$. Choose now 
$$
\eta_1 = \int_{r_c}^\infty \varphi^{-1} \left( \frac{c}{v(s)} \right) \di s
$$
and consider $\eta$ satisfying the further restriction $\eta \in (0,\min\{\eta_0,\eta_1\})$. For $j$ large enough, since $z_j(r_c) \le z_j(r_0) = \eta < \eta_1 = \bar z(r_c)$, by the comparison Proposition \ref{prop_serrin} and the non-negativity of $\beta,f(z_j)$ and $l(|z_j'|)$ we get $\bar z \ge z_j$ on $[r_c, r_0+jR]$, thus $\bar z \ge z$ on $[r_c, \infty)$. The thesis follows since $\bar z \ra 0$ as $r \ra \infty$.
\end{proof}

We are ready to prove our main result, Theorem \ref{teo_necessity_CSP_intro} in the Introduction, in the following more general form: it says, loosely speaking, that there is no geometric obstruction for \eqref{KO_zero} to be necessary for the compact support principle.

\begin{theorem}\label{teo_necessitybello}
Let $(M^m, \metric)$ be a complete manifold, and let $\varphi, f,l$ satisfy \eqref{assumptions_CSP_necessity}, \eqref{phiel_solozero}, \eqref{assu_perCSP} and
$$
\begin{array}{l}
\disp l \in \lip_\loc\big((0,\xi_0)\big) \\[0.2cm]
f(0)l(0)=0.
\end{array}
$$
Then, for each
\begin{itemize}
\item[] origin $\mathcal{O} \subset M$ with associated distance $r(x) = \mathrm{dist}(x,\mathcal{O})$,
\item[] $r_0>0$, $\xi \in (0,\xi_0)$ (with $\xi_0$ as in \eqref{assu_perCSP}),
\item[] $0 < b \in C\big(M \backslash B_{r_0}(\mathcal{O})\big)$,
\end{itemize}
there exists $\eta \in (0,\eta_0)$ sufficiently small and a radial solution $u \in \lip(M \backslash B_{r_0}(\mathcal{O}))$ of
$$
\left\{\begin{array}{l}
\Delta_\varphi u \ge b(x) f(u) l(|\nabla u|) \qquad \text{weakly on } \, M \backslash B_{r_0}(\mathcal{O}), \\[0.2cm]
0 \le u \le \eta \quad \text{on } \, M \backslash B_{r_0}(\mathcal{O}),  \\[0.2cm]
u = \eta \quad \text{on } \, \partial B_{r_0}(\mathcal{O}), \qquad  u(x) \ra 0 \quad \text{as } \, r(x) \ra \infty, \\[0.2cm]
|\nabla u| < \xi \qquad \text{on } \, M \backslash B_{r_0}(\mathcal{O}).
\end{array} \right.
$$
Moreover, if  \eqref{KO_zero_intro_FMP} holds then $u>0$ on $M \backslash \mathcal{O}$. In particular, if $\mathcal{O}$ is a pole for $M$, $u \in C^1\big( M \backslash B_{r_0}(\mathcal{O})\big)$ and \eqref{KO_zero_intro} is necessary for the validity of the compact support principle $\csp$.
\end{theorem}

\begin{proof}
We choose $0<\bar g \in C^\infty(\R^+_0)$ enjoying the following properties:
\begin{itemize}
\item[$(i)$] if $H_{-\nabla r}$ is the mean curvature of $\partial \mathcal{O}$ with respect to the inward pointing unit normal $-\nabla r$,
\begin{equation}\label{ipo_barg_CSP}
(m-1) \frac{\bar g'(0)}{\bar g(0)} > \max\left\{ 0, \sup_{\partial \cal O} H_{-\nabla r}\right\};
\end{equation}
\item[$(ii)$] setting $\bar v(r) = \vol(\Sph^{m-1})\bar g(r)^{m-1}$,
\begin{equation}\label{gni}
\begin{array}{l}
\bar v' \ge 0 \qquad \text{on } \, \R^+, \\[0.2cm]
\disp \bar v(r) \ge \max\left\{ 1, \left[ \varphi\left( \frac{1}{r^2}\right) \right]^{-1} \right\} \quad \text{for } \, r \ge 1.
\end{array}
\end{equation}
\end{itemize}
Note that, for $r \ge 1$,
\begin{equation}\label{homovarphi}
\varphi^{-1} \left( \frac{1}{\bar v(r)}\right) \le \varphi^{-1} \left( \varphi \left( \frac{1}{r^2}\right)\right) = \frac{1}{r^2} \in L^1(\infty). 
\end{equation}
Next, choose $0 \le \bar G \in C(\R^+_0)$ in such a way that
$$
\Ricc \ge -(m-1)\bar G(r)\metric \qquad \text{on } \, M,
$$
and define $G(r) = \max\{ \bar G, \bar g''/\bar g\}$. Let $g \in C^2(\R^+_0)$ solve
\begin{equation}\label{ipo_g_CSP}
\left\{ \begin{array}{l}
g'' - Gg =0 \quad \text{on } \, \R^+ \\[0.2cm]
\disp \frac{g'(0)}{g(0)} = \frac{\bar g'(0)}{\bar g(0)},
\end{array}
\right.
\end{equation}
and let $M_g$ be the model associated to $g$. By construction and by \eqref{homovarphi}, setting $v_g(r) = \vol(\Sph^{m-1})g(r)^{m-1}$ it holds
\begin{equation}\label{ipo_compamodel}
\begin{array}{l}
\Ricc \ge -(m-1)G(r) \metric \qquad \text{on } \, M, \\[0.2cm]
v_g(r) \ge \bar v(r) \qquad \text{on } \, \R^+ \, \text{ by Sturm comparison,} \\[0.2cm]
\disp \varphi^{-1} \left(\frac{1}{v_g(t)} \right) \in L^1(\infty).
\end{array}
\end{equation}
We define
$$
\bar f(t) = \sup_{[0,t]} f(s), \qquad \beta(r) = \sup_{\partial B_r} b,
$$
and note that, since $f$ is $C$-increasing,
\begin{equation}\label{febarf_necessity}
f(t) \le \bar f(t) \le C f(t) \qquad \forall \, t \in [0,\eta_0).
\end{equation}
Choose $\bar l \in \lip_\loc(\R^+)$ such that
$$
\bar l \ge l \quad \text{on } \, \R^+_0, \qquad \bar l \equiv l \quad \text{on } \, [0, \xi].
$$
Eventually, fix $r_0 \ge 1$ and choose $R$ small enough that
$$
\bar f_{\eta_0} \bar l_1 \left[\sup_{[r_0,r_0+R)} \frac{1}{v_g(r)} \int_{r_0}^r v_g(s)\beta(s) \di s\right] < \frac{\varphi(\xi)}{2}.
$$
Then, pick $\eta \in (0,\eta_0)$ small enough to enjoy
$$
\frac{v_g(r_0+R)}{v_g(r_0)} \varphi\left(\frac{\eta}{R}\right) < \frac{\varphi(\xi)}{2}.
$$
Since $\bar f$ is increasing, $\eta \in (0, \eta_0)$ and $\xi \in (0, 1)$, the last two inequalities imply
\begin{equation}\label{theuniformbound_teo}
\frac{v_g(r_0+R)}{v_g(r_0)} \varphi\left(\frac{\eta}{R}\right) + \bar f_\eta \bar l_{\xi} \left[\sup_{[r_0,r_0+R)} \frac{1}{v_g(r)} \int_{r_0}^r v_g(s)\beta(s) \di s\right]  < \varphi(\xi).
\end{equation}
We are therefore in the position to apply Theorem \ref{teo_exteriorDiri} and infer the existence of $\eta_1$ such that, for each $\eta \in (0,\min\{\eta_0,\eta_1\})$ satisfying \eqref{theuniformbound_teo}, there exists $z$ solving
\begin{equation}\label{eq_exteriorDiri}
\left\{\begin{array}{ll}
\big[v_g \varphi(z')\big]' = \beta v_g \bar f(z) \bar l(|z'|) & \quad \text{on } \, [r_0, \infty) \\[0.2cm]
z(r_0) = \eta, \qquad -\xi < z' \le 0 & \quad \text{on } \, [r_0, \infty), \\[0.2cm]
z(r) \ra 0 \qquad \text{as } \, r \ra \infty.
\end{array}\right.
\end{equation}
We consider $u(x) = z(r(x))$ on $M \backslash B_{r_0}(\mathcal{O})$, with $r(x) = \mathrm{dist}(x, \mathcal{O})$. The first of \eqref{ipo_compamodel} together with \eqref{ipo_g_CSP} and \eqref{ipo_barg_CSP} imply, via the Laplacian comparison theorem (Thm. \ref{teo_laplaciancomp} in Appendix A), the inequality
$$
\Delta r \le \frac{v_g'(r)}{v_g(r)} \qquad \text{weakly on } \, M \backslash \mathcal{O},
$$
hence
$$
\begin{array}{lcl}
\Delta_\varphi u & \ge & \disp \varphi'(z')z'' + \varphi(z')\Delta r \ge \varphi'(z')z'' + \varphi(z')\frac{v_g'}{v_g} \\[0.3cm]
& = & v_g^{-1} \big[v_g \varphi(z')\big]' = \beta \bar f(z)\bar l(|z'|) \ge b(x)f(z)l(|z'|) \\[0.3cm]
& = & b(x) f(u) l(|\nabla u|)
\end{array}
$$
weakly on $M \backslash B_{r_0}(\mathcal{O})$. This concludes the first part of the theorem.\\
Next, we prove that if \eqref{KO_zero_intro_FMP} holds then $z>0$ on $[r_0,\infty)$. To see this, we consider the radial function $v(x) = z(r(x))$ on the model $M_g$, that satisfies
$$
\Delta_\varphi v = v_g^{-1} \big[ v_g \varphi(z')\big]' = \beta(r) \bar f(z)\bar l(|z'|) = \beta(r) \bar f(v)l(|\nabla v|)
$$
where we used $|z'| < \xi$ and $l = \bar l$ on $[0, \xi]$. Moreover, $v = \eta$ on $\{r=r_0\}$. Set
$$
\bar F(t) = \int_0^t \bar f(s) \di s.
$$
Because of \eqref{febarf_necessity}, Lemma \ref{lem_mettimaosigmatau_novo} guarantees that \eqref{KO_zero_intro_FMP} is equivalent to
$$
\frac{1}{K^{-1} \circ \bar F} \not \in L^1(0^+).
$$
Therefore, we can apply Theorem \ref{teo_FMP2} on the set $\Omega = \{ r>r_0\} \subset M_g$ to deduce that $v>0$ on $\Omega$, as claimed.
\end{proof}

\subsection{Sufficiency of \eqref{KO_zero_intro} for the compact support principle}

\subsubsection{General operators, and no cut-locus}

%
As remarked in the Introduction, \eqref{KO_zero} alone is not sufficient to prove $\csp$ on complete manifolds. Indeed, the influence of geometry is for this property particularly subtle, as confirmed by the next refinement of Example \ref{ex_CSP_intro}.
\begin{example}\label{ex_countCSP}
{\rm For $\delta \ge 0$, consider a model $M_\delta = (\R^m, \di s_\delta^ 2)$ (cf. also Example \ref{ex_modellishrinking} below) with
$$
\di s_\delta^2 = \di t^2+ g_\delta(t)^2 \metricN_1, \qquad \text{where} \quad \left\{ \begin{array}{ll}
g_\delta \in C^2(\R^+_0) & g_\delta>0 \quad \text{on } \, \R^+ \\[0.2cm]
g_\delta(t)=t & \text{if } \, t \le 1/4 \\[0.2cm]
g_\delta(t) = \exp\{-t^\delta\} & \text{if } \, t \ge 1.
\end{array}\right.
$$
Define $\mathcal{O}= B_1(o)$, and let $r=t-1$ be the distance from $\cal O$. We have
$$
\begin{array}{rcl}
\disp \II_{-\nabla r} & = & \disp -\delta \di s_\delta^2, \\[0.2cm]
K_{\rad}(r) & = & \disp -\delta \big[ -(\delta-1)(1+r)^{\delta-2}+\delta(1+r)^{2\delta-2}\big] \le -\delta^2(1+r)^{\delta-2}\left[ (1+r)^\delta -1\right], \\[0.2cm]
\Delta r & = & \disp -(m-1)\delta(1+r)^{\delta -1}.
\end{array}
$$
Define $\alpha = 2\delta -2 \ge -2$, and let $\mu, \chi, \omega \in \R$ satisfy
\begin{equation}\label{controesempio_csp_belle}
\mu > \chi - \frac{\alpha}{2}, \qquad \omega < \chi.
\end{equation}
It is easy to show that, for
$$
\sigma \in \left( 0, \frac{\mu - \chi +\alpha/2}{\chi - \omega} \right],
$$
and for each $p >1$, the function $v(x) = (1+r(x))^{-\sigma}$ is a bounded, positive solution of
$$
\Delta_p v \ge C (1+r)^{-\mu} v^\omega |\nabla v|^{p-1-\chi} \qquad \text{on } \, M_\delta \backslash \cal O,
$$
for some constant $C>0$, and furthermore $v(x) \ra 0^+$ as $x$ diverges. However, defining $f(t)=t^\omega$ and $l(t) = t^{p-1-\chi}$, inequality $\omega < \chi$ implies
$$
\frac{1}{K^{-1}\circ F} = C_1 s^{-\frac{\omega+1}{\chi+1}} \in L^1(0^+).
$$
Thus, condition \eqref{KO_zero} is met but $\csp$ fails on $M_\delta$.
}
\end{example}
%
%
%
As in \cite{pucciserrin, PuRS}, the proof of $\csp$ will be achieved via the construction of a compactly supported, $C^1$-supersolution $\bar w$ for \eqref{csp_problem}. Under the above assumptions on $\varphi,b,f,l$, to produce $\bar w$ one could try to use the solution $w$ of the related ODE \eqref{twoboundary} in a way analogous to the one in the proof of the finite maximum principle (Theorem \ref{teo_FMP2}). However, a direct use of $w$ seems difficult, also because of the delicate interplay between the threshold $\eta$ in \eqref{twoboundary} and the global behaviour of the constants $a_1,\wp_1,\wp_0,l_\xi$ in Theorem \ref{exi2}, depending on the interval $[0,T]$ under consideration. As we shall see, a certain independence between $\eta$ and $a_1,\wp_1,\wp_0,l_\xi$ is key to conclude $\csp$ from the existence of $w$. To overcome the problem, we will use a different technique: instead of solving a Dirichlet problem we will exhibit an explicit, compactly supported supersolution by a direct use of \eqref{KO_zero}, an idea that is closer to the one in \cite{pucciserrin, PuRS, rigolisalvatorivignati}, which in turn are improvements of \cite{pucciserrinzou}. However, extending the method therein to non-constant $b,l$ presents nontrivial hurdles and calls for new ideas. To this aim, we shall assume some further conditions, that although seemingly somewhat articifial are still quite general, and enable us to capture the right growths to get a sharp result.\par

To exhibit $\bar w$ we follow two slightly different constructions that need a (mildly) different set of assumptions. Besides \eqref{assumptions_CSP_necessity}, \eqref{phiel_solozero} and \eqref{assu_perCSP}, for the first construction we require


%
%

%
%
\begin{itemize}
\item[$(C_1)$] there exists a constant $k_1 \ge 1$ such that
$$
t K'(t) \le k_1 K(t) \qquad \text{for each } \, t \in (0,1];
$$
\item[$(C_2)$] there exists a constant $k_2 \ge 1$ such that
$$
K'(st) \le k_2 K'(s)K'(t) \qquad \text{for each } \, s,t \in (0,1];
$$
\item[$(C_3)$] there exists a constant $\bar C \ge 1$ such that
$$
\frac{t}{K^{-1}(t)} \quad \text{is $\bar C$-increasing,} \qquad \text{and} \qquad \frac{t}{K^{-1}(t)} \ra 0 \quad \text{as } \, t \ra 0.
$$
\item[$(C_4)$] there exists a constant $c_F \ge 1$ such that
$$
\frac{F(t)}{K^{-1}(F(t))} \le c_F f(t) \qquad \text{for each } \, t \in (0, \min\{1,\eta_0\}).
$$
\end{itemize}
Note that these requirements are all related to the behaviour of the various functions considered in a right neighbourhood of zero. 

\begin{example}\label{ex_CSPpoli}
\emph{Having fixed
$$
p>1, \qquad \chi \in [0,p-1], \qquad \omega > 0,
$$
the prototype example of $f,l, \varphi$ is given by 
$$
\varphi'(t) \asymp t^{p-2}, \qquad f(t) \asymp t^\omega, \qquad l(t) \asymp t^{p-1-\chi}
$$
as $t \ra 0^+$. Then, $K(t) \asymp t^{\chi+1}$ satisfies $(C_1),(C_2)$, while $(C_3)$ and $(C_4)$ are met if and only if, respectively, $\chi>0$ and $\omega \le \chi$.}
\end{example}
Suppose that
$$
b(x) \ge \beta\big(r(x)\big) \qquad \text{for } \, r(x) \ge r_0.
$$
We express the relation between $K$ and $\beta$ in terms of an auxiliary weight $\bar \beta$, that is tied with $K$ in the way expressed by the next two conditions. Later, \eqref{solu_radialCSP} in Proposition \ref{prop_CSP_radial} will relate $\beta$ to $\bar \beta$. 
\begin{itemize}
\item[$(\beta_1)$] $0<\beta \in C([r_0,\infty))$, $\quad \bar \beta \in C^1([r_0,\infty))$, $\quad 0< \bar \beta < K_\infty, \ \ \bar \beta' \le 0$ on $[r_0,\infty)$;
\item[$(\beta_2)$] there exists a constant $c_\beta \ge 1$ such that
$$
\frac{-\bar \beta'(t)}{K^{-1}(\bar \beta(t))} \le c_\beta \bar \beta(t) \qquad \text{for each } \, t \in [r_0, \infty).
$$
\end{itemize}

\begin{remark}
\emph{Note that $(\beta_2)$ is meaningful since $\bar \beta < K_\infty$ because of $(\beta_1)$.
}
\end{remark}

\begin{example}\label{ex_borderCSP}
\emph{Referring to Example \ref{ex_CSPpoli}, a borderline behaviour of $\bar \beta$ for the validity of $(\beta_2)$ is 
$$
\bar \beta(t) = (1+t)^{- \chi -1}.
$$
}
\end{example}

We first describe our main ODE result, that should be compared to Lemma 4.1 in \cite{rigolisalvatorivignati_5}. Here, we consider a different and (in some cases) weaker set of assumptions, and the proof that we present is considerably simpler. Below, we will describe in more detail the interplay between the two results.

\begin{proposition}\label{prop_CSP_radial}
Let $\varphi, l, f$ satisfy \eqref{assumptions_CSP_necessity}, \eqref{phiel_solozero} and \eqref{assu_perCSP}, and assume the validity of $(C_1), \ldots, (C_4)$ and $(\beta_1), (\beta_2)$. Having fixed a non-negative $\theta \in C([r_0,\infty))$, suppose that
\begin{equation}\label{solu_radialCSP}
\max\left\{ \frac{\bar \beta(s)}{\beta(s)}, \ \frac{\theta(s)\bar \beta(s)}{\beta(s) K^{-1}(\bar \beta(s))} \right\} \in L^ \infty\big([r_0,\infty)\big).
\end{equation}
If
\begin{equation}\label{KO_0CSP}\tag{$\mathrm{KO}_0$}
\frac{1}{K^{-1}\circ F} \in L^1(0^+),
\end{equation}
then for each $\eps>0$ there exists $\lambda \in (0,1)$ such that the following holds: for each $R \ge r_0$, there exists $R_1 > R$ and a function $w$ with the following properties:
\begin{equation}\label{proprie_w_CSP}
\left\{ \begin{array}{l}
w \in C^1([R,\infty)) \quad \text{and $C^2$ except possibly at $R_1$;} \\[0.2cm]
0 \le w \le \lambda, \qquad w(R) = \lambda, \qquad w \equiv 0 \quad \text{on } \, [R_1, \infty),  \\[0.2cm]
w' < 0 \quad \text{on } \, [R, R_1), \qquad |w'| \le \eps \quad \text{on } \, [R, \infty), \\[0.2cm]
\big( \varphi(w')\big)' - \theta(r)\varphi(w') \le \eps \beta(r)f(w)l(|w'|) \qquad \text{on } \, [R, \infty).
\end{array}
\right.
\end{equation}
\end{proposition}

\begin{remark}\label{rem_utile}
\emph{Inspecting the proof of Proposition \ref{prop_CSP_radial}, we can weaken the third in \eqref{assumptions_bfl} to $l \in L^\infty_\loc(\R^+_0) \cap C(\R^+)$ and $l>0$ on $\R^+$, that is, the continuity of $l$ at $t=0$ is not needed. This will be used in the proof of Theorem \ref{cor_csp_specialized_intro}.
}
\end{remark}

\begin{remark}
\emph{Condition \eqref{solu_radialCSP} relates $\beta$ to $\bar \beta$. For instance, in Example \ref{ex_CSPpoli}, we can set $\bar \beta(t) = c(1+t)^{-\chi-1}$ for a constant $c$ small enough to satisfy $(\beta_1)$, and choose $\beta(t) = (1+t)^{-\mu}$, for some $\mu \in \R$. Then, \eqref{solu_radialCSP} is met if and only if
$$
\mu \le \chi+1 \qquad \text{and} \qquad \theta(s)s^{\mu-\chi} \in L^\infty([r_0, \infty)).
$$
}
\end{remark}

To begin the proof, we need a simple technical lemma.

\begin{lemma}\label{lem_furtherprop}
Assume \eqref{assu_perCSP} and \eqref{phiel_solozero}, and also $(C_1), (C_2), (\beta_1), (\beta_2)$. Then, the following properties hold:
\begin{itemize}
\item[$(K_1)$] $K(st) \le k_1k_2K(s)K(t) \quad$ for each $\ \ s,t \in [0,1]$;
\item[$(K_2)$] $K^{-1}(\tau)K^{-1}(\rho) \le K^{-1}(k_1k_2\tau \rho)\quad $ for each $\ \ \tau, \rho \in \left(0, \min\left\{1, \frac{K_\infty}{k_1k_2}\right\}\right)$;
\item[$(K\beta)$] For each $\sigma \in (0,1)$,
$$
K^{-1}\big(\sigma \bar \beta\big) \not \in L^1(\infty);
$$
\end{itemize}
\end{lemma}

\begin{proof}
Property $(K_1)$ follows immediately from $(C_2)$ and $(C_1)$ by integration:
$$
K(st) = \int_0^{st} K'(\tau)\di \tau = s\int_0^t K'(s\zeta)\di \zeta \le k_2 sK'(s)K(t) \le k_2k_1 K(s)K(t).
$$
To show $(K_2)$, use $(K_1)$ with the choices $s= K^{-1}(\tau)$, $t=K^{-1}(\rho)$ and then apply $K^{-1}$. To prove $(K\beta)$ first observe that, because of $(K_2)$, $K^{-1}(\sigma \bar \beta) \ge C_\sigma K^{-1}(\bar \beta)$ for some $C_\sigma>0$, and thus it is sufficient to restrict to $\sigma=1$. Using $(\beta_2)$,
\begin{equation}\label{basicas}
c_\beta K^{-1}(\bar \beta) \ge -\frac{\bar \beta'}{\bar \beta}.
\end{equation}
If $\bar \beta$ is bounded from below by a positive constant, then $K^{-1}(\bar \beta)$ is not infinitesimal and clearly $(K\beta)$ is met. Otherwise, from $\bar \beta'\le 0$ we get $\bar\beta(r) \ra 0$ as $r \ra \infty$, and integrating \eqref{basicas} on $[r_1,r)$ we deduce
\begin{equation}\label{basicas_lowerint}
c_\beta \int_{r_1}^r K^{-1}(\bar \beta) \ge \log \bar\beta(r_1) - \log \bar\beta(r) \ra +\infty \quad \text{as } \, r \ra \infty,
\end{equation}
as claimed.
\end{proof}

\begin{proof}[Proof of Proposition \ref{prop_CSP_radial}]
Let $\sigma,\lambda \in (0, \eta_0)$ to be specified later. Because of $(K\beta)$ in Lemma \ref{lem_furtherprop} and \eqref{KO_0CSP}, there exists $R_\sigma = R_\sigma(\sigma, \lambda, R)>R$ such that
\begin{equation}\label{labase_CSP_ode}
\int_0^\lambda \frac{\di s}{K^{-1}(F(s))} = \int_R^{R_\sigma} K^{-1}\big(\sigma \bar \beta(s)\big)\di s.
\end{equation}
We define implicitly a function $\alpha$ via the identity
\begin{equation}\label{equa_CSP_ode}
\int_0^{\alpha(t)} \frac{\di s}{K^{-1}(F(s))} = \int_{R_\sigma-t}^{R_\sigma} K^{-1}\big(\sigma \bar \beta(s)\big)\di s.
\end{equation}
Clearly, $\alpha(0)=0$, $\alpha(R_\sigma-R) = \lambda$, $\alpha>0$ on $(0, R_\sigma-R)$ and, differentiating,
\begin{equation}\label{gradiente_ode}
\begin{array}{l}
\disp \alpha'(t) = K^{-1}\big(F(\alpha(t)\big)K^{-1}\big(\sigma \tilde \beta(t)\big) >0 \quad \text{on } \, (0, R_\sigma -R), \\[0.3cm]
\text{where} \quad \tilde \beta(t) = \bar \beta(R_\sigma-t), \qquad \text{and} \qquad \alpha'(0)=0.
\end{array}
\end{equation}
Note that $\tilde \beta$ is non-decreasing by $(\beta_1)$. By construction, $\alpha \in [0, \lambda] \subset [0, \eta_0)$, and
\begin{equation}\label{controllow'}
\begin{array}{lcl}
0 < \disp \alpha'(s) & \le & \disp K^{-1}\big(F(\alpha)\big)K^{-1}\big(\sigma\|\bar \beta\|_\infty\big) \\[0.2cm]
& \le & \disp  K^{-1}\big(F(\lambda)\big)K^{-1}\big(\|\bar \beta\|_\infty\big) \qquad \text{on } \, (0, R_\sigma -R).
\end{array}
\end{equation}
We can therefore reduce $\sigma$ and $\lambda$, independently, in such a way that
\begin{equation}\label{controllow'_1}
\quad K^{-1}(F(\alpha)) \le \min\{1,\xi_0\}, \qquad K^{-1}\big(\sigma \tilde{\beta}\big) \le 1
\end{equation}
on $[0, R_\sigma)$. For convenience, we define
$$
\rho = K^{-1}\big(F(\alpha(t)\big), \qquad \tau = K^{-1}\big(\sigma \tilde \beta(t)\big)
$$
applying $K$ to \eqref{gradiente_ode} and differentiating, we obtain
\begin{equation}\label{catena_ode_csp}
\begin{array}{lcl}
\disp \big(K(\alpha')\big)' & = & \disp K'(\rho\tau)[\rho'\tau+\rho \tau'] \stackrel{(C_2)}{\le} k_2 K'(\rho)K'(\tau)\left[\frac{f(\alpha)\alpha'\tau}{K'(\rho)} + \frac{\rho \sigma \tilde \beta'}{K'(\tau)}\right] \\[0.4cm]
& \stackrel{(C_1)}{\le} & k_2 k_1 \Big[ f(\alpha) \alpha' K(\tau) + \sigma \tilde \beta' K(\rho)\Big] \\[0.4cm]
& = & \disp k_2 k_1 \Big[ f(\alpha) \alpha' \sigma \tilde \beta  + \sigma \tilde \beta' F(\alpha) \Big]
\end{array}
\end{equation}
However, by $(C_4)$ and $(\beta_2)$, together with $(K_2)$ in Lemma \ref{lem_furtherprop} applied twice,
\begin{equation}\label{usefulbounds}
\begin{array}{rcl}
\disp F(\alpha) & \le & \disp c_F K^{-1}\big(F(\alpha)\big)f(\alpha), \\[0.3cm]
\disp \tilde \beta' & \le & \disp c_\beta K^{-1}\big(\tilde \beta\big) \tilde \beta \stackrel{(K_2)}{\le} \disp \frac{\bar c_\beta}{K^{-1}(\sigma)} K^{-1}\big(\sigma \tilde \beta\big) \tilde \beta
\end{array}
\end{equation}
for some $\bar c_\beta$ depending on $c_\beta, K, k_1,k_2$. Inserting into \eqref{catena_ode_csp} and recalling \eqref{gradiente_ode} we get
\begin{equation}\label{catena_ode_csp_2}
\begin{array}{lcl}
\disp \big(K(\alpha')\big)' & \le & \disp k_2 k_1 \sigma \Big[ f(\alpha) \alpha' \tilde \beta + \frac{c_F \bar c_\beta}{K^{-1}(\sigma)} \tilde \beta K^{-1}(F(\alpha))K^{-1}(\sigma \tilde \beta) f(\alpha) \Big] \\[0.4cm]
& = & \disp \disp k_2 k_1 \sigma \tilde \beta f(\alpha) \alpha' \Big[1 + \frac{c_F \bar c_\beta}{K^{-1}(\sigma)} \Big] \le \frac{C_1\sigma}{K^{-1}(\sigma)} \tilde \beta f(\alpha) \alpha', \\[0.4cm]
\end{array}
\end{equation}
for some constant $C_1$ depending on $k_1,k_2,\bar c_\beta, c_F$. Using the definition of $K'$ and $\alpha'>0$ we can simplify \eqref{catena_ode_csp_2} to deduce
\begin{equation}\label{primopezzo_csp}
\varphi'(\alpha') \alpha'' \le \frac{C_1\sigma}{K^{-1}(\sigma)} \tilde \beta f(\alpha) l(\alpha').
\end{equation}
Moreover, observe that the first in \eqref{catena_ode_csp} and the definition of $\rho, \tau$ imply $(K(\alpha'))' \ge 0$ and hence, expanding, $\alpha'' \ge 0$.\\
We now integrate \eqref{primopezzo_csp} on $[0, t)$ and use $\alpha'(0)=0$, $\varphi(0)=0$ to obtain
\begin{equation}\label{primopezzo_csp_3}
\varphi\big(\alpha'(t)\big) \le \frac{C_1\sigma}{K^{-1}(\sigma)} \int_{0}^t \tilde \beta  f(\alpha) l(\alpha').
\end{equation}
Apply the third in \eqref{assu_perCSP} (note that $\alpha'' \ge 0$) we get
\begin{equation}\label{primopezzo_csp_4}
\varphi\big(\alpha'(t)\big) \le \frac{C_2\sigma}{K^{-1}(\sigma)} f\big(\alpha(t)\big) l\big(\alpha'(t)\big) \int_{0}^t \tilde \beta
\end{equation}
for some constant $C_2>0$. Next, we exploit \eqref{labase_CSP_ode} to estimate from above the integral in \eqref{primopezzo_csp_4}. To this end, we use $(C_3)$, $\tilde \beta' \ge 0$ together with the second in \eqref{usefulbounds} to get
\begin{equation}\label{secondopezzo_csp}
\begin{array}{lcl}
\disp \int_{0}^t \tilde \beta & = & \disp \int_{0}^t \frac{\tilde \beta}{K^{-1}(\tilde\beta)} K^{-1}(\tilde \beta) \stackrel{(C_3)}{\le} \bar C \frac{\tilde \beta(t)}{K^{-1}(\tilde \beta(t))} \int_{0}^t K^{-1}(\tilde \beta) \\[0.4cm]
& \stackrel{(K_2)}{\le} & \disp \frac{\bar C \bar c_\beta}{K^{-1}(\sigma)} \frac{\tilde \beta(t)}{K^{-1}(\tilde \beta(t))} \int_{R}^{R_\sigma} K^{-1}(\sigma \tilde \beta) \\[0.4cm]
& = & \disp \frac{C_3}{K^{-1}(\sigma)} \frac{\tilde \beta(t)}{K^{-1}(\tilde \beta(t))} \int_0^\lambda \frac{\di s}{K^{-1}(F(s))}.
\end{array}
\end{equation}
Define $\tilde \theta(t) = \theta(R_\sigma-t)$. Plugging together \eqref{primopezzo_csp}, \eqref{primopezzo_csp_3} and \eqref{secondopezzo_csp}, because of \eqref{solu_radialCSP} we obtain
\begin{equation}\label{nuovaentrata}
\begin{array}{l}
\disp \big(\varphi\big(\alpha'\big)\big)' + \tilde \theta(t)\varphi\big(\alpha'\big) \\[0.3cm]
\disp \quad \le \ \ \beta(R_\sigma - t) f\big(\alpha\big) l\big(\alpha'\big) \left[ \frac{\tilde \beta(t)}{\beta(R_\sigma - t)} \frac{C_1\sigma}{K^{-1}(\sigma)} \right.\\[0.5cm]
\disp \qquad + \left.\frac{C_4\sigma}{[K^{-1}(\sigma)]^2} \frac{\tilde \theta(t) \tilde \beta(t)}{\beta(R_\sigma-t) K^{-1}(\tilde \beta(t))} \int_0^\lambda \frac{\di s}{K^{-1}(F(s))} \right] \\[0.5cm]
\disp \quad \le \ \ \beta(t) f\big(\alpha\big) l\big(\alpha'\big) \left[ \left\|\frac{\bar\beta}{\beta}\right\|_\infty \frac{C_1\sigma}{K^{-1}(\sigma)} + \frac{C_4\sigma}{[K^{-1}(\sigma)]^2} \left\|\frac{\theta \bar \beta}{\beta K^{-1}(\bar \beta)}\right\|_{\infty} \int_0^\lambda \frac{\di s}{K^{-1}(F(s))} \right] \\[0.5cm]
\end{array}
\end{equation}
Next, using the second in $(C_3)$ we can choose $\sigma$ sufficiently small to make the first term in square brackets smaller than $\eps/2$. We can then choose $\lambda>0$ small enough to make the second term smaller than $\eps/2$. Eventually, define
\begin{equation}\label{def_w_CSP}
w(r) = \alpha(R_\sigma -r) \qquad \text{and} \qquad R_1 = R_\sigma.
\end{equation}
Because of \eqref{controllow'}
$$
|w'| \le K^{-1}\big(F(\lambda)\big) K^{-1}\big(\|\bar \beta\|_\infty\big)
$$
Up to reducing $\lambda$ further, $|w'|\le \eps$ on $[R, R_1)$ (with $\eps$ as in the statement of the Proposition). Using the definition of $\tilde \theta, \tilde \beta$ and the fact that $\varphi$ is odd on $\R$, it is immediate to check that $w$ satisfies all the properties listed in \eqref{proprie_w_CSP}.
\end{proof}

\begin{remark}
\emph{A crucial feature of the above construction is that $\lambda$ is independent of $R \ge r_0$. This will allow to construct compactly supported supersolutions attaining value $\lambda$ on the boundary of any fixed geodesic sphere $\partial B_R(\cal O)$ with $R \ge r_0$.
}
\end{remark}

With this preparation, we can now prove our first main result for the compact support principle. We recall that a pole $\mathcal{O} \subset M$ is a smooth, relatively compact open set such that the normal exponential map realizes a diffeomorphism between $M \backslash \cal O$ and $\partial \cal O \times \R^+$. Let $\II_{-\nabla r}$ be the second fundamental form of $\partial \cal O$ in the direction pointing towards $\mathcal{O}$.

\begin{theorem}\label{teo_CSP_nuovo}
Let $(M, \metric)$ be a manifold possessing a pole $\mathcal{O}$, and let $r(x) = \dist(x, \mathcal{O})$. Suppose that
\begin{equation}\label{arra_CSP_assu}
\begin{array}{l}
\disp K_\rad \le -\kappa^2(1+r)^{\alpha} \qquad \text{on } \, M \backslash \mathcal{O}, \quad \text{for some } \, \kappa \ge 0, \ \alpha \ge -2;  \\[0.3cm]
\II_{-\nabla r} \ge \left\{ \begin{array}{ll}
- \kappa \metric & \quad \text{if } \, \alpha \ge 0 \ \text{or } \ \kappa = 0, \\[0.2cm]
- \left[\frac{\alpha + \sqrt{\alpha^2+ 16\kappa^2}}{4}\right] \metric & \quad \text{otherwise},
\end{array}\right.
\end{array}
\end{equation}
where $\II_{-\nabla r}$ denotes the second fundamental form of $\partial \mathcal{O}$ in the inward-pointing direction. Fix $0<b \in C(M)$, and let $0<\beta \in C([r_0,\infty))$ such that
$$
b(x) \ge \beta\big(r(x)\big) \qquad \text{ for } \, r(x) \ge r_0.
$$
Let $\varphi, f,l$ satisfy \eqref{assumptions_CSP_necessity}, \eqref{phiel_solozero}, \eqref{assu_perCSP} and $(C_1), \ldots, (C_4)$. Assume that, for some $\bar \beta$ matching $(\beta_1)$ and $(\beta_2)$, it holds
\begin{equation}\label{assu_subeta}
\max \left\{ \frac{\bar \beta(s)}{\beta(s)}, \ \frac{s^{\alpha/2} \bar \beta(s)}{\beta(s)K^{-1}(\bar \beta(s))}\right\} \in L^\infty\big([r_0, \infty)\big).
\end{equation}
Then, 
$$
\eqref{KO_zero} \qquad \Longrightarrow \qquad \text{$\csp$ holds for \eqref{csp_problem}.}
$$
If moreover 
$$
l \in \lip_\loc\big((0,\xi_0)\big), \qquad f(0)l(0)=0, 
$$
then
$$
\eqref{KO_zero} \qquad \Longleftrightarrow \qquad \text{$\csp$ holds for \eqref{csp_problem}.}
$$
\end{theorem}

\begin{proof}
We first prove implication $\eqref{KO_zero} \Rightarrow \csp$. Note that it is enough to consider solutions of \eqref{csp_problem} when $\Omega = \Omega_{r_0}$ is a connected component of $M \backslash B_{r_0}(\cal O)$. Let now $u$ be a $C^1$ solution of \eqref{csp_problem} on $\Omega_{r_0}$. By Proposition \ref{lem_ODE_infty_increasing} applied with, respectively,
$$
\begin{array}{l}
\disp G(r) = \kappa^2(1+r)^\alpha, \qquad \disp \left\{ \begin{array}{lll}
\theta_* = 0, & \quad D = D_-(\theta_*) = -1  & \quad \text{if } \, \alpha \ge 0; \\[0.2cm]
\theta_* = \frac{\alpha}{2\kappa}, & \quad D = D_-(\theta_*) = - \frac{\alpha + \sqrt{\alpha^2+ 16\kappa^2}}{4\kappa} & \quad \text{if } \, \alpha \in [-2,0)
\end{array}\right. \\[0.4cm]
C = 1, \qquad \lambda = D \kappa
\end{array}
$$
we deduce that
$$
g(t) = \exp \left\{ D \int_0^t \kappa (1+s)^{\alpha/2} \di s \right\}
$$
is a positive solution on $\R^+$ of $g''-Gg \le 0$, $g(0)=1$, $g'(0) = D \kappa$, and thus by the Laplacian comparison theorem from below (cf. Appendix A)
\begin{equation}\label{laplacomp_csp}
\Delta r \ge (m-1)\frac{g'(r)}{g(r)} = (m-1)D\kappa (1+r)^{\alpha/2} \qquad \text{for } \, r > 0.
\end{equation}
Because of  \eqref{assu_subeta}, we can apply Proposition \ref{prop_CSP_radial} with
$$
\theta(t) = (m-1)|D|\kappa (1+r)^{\alpha/2}, \qquad \eps = 1/2C
$$
($C$ the increasing constant in the third of \eqref{assu_perCSP}) to infer the existence of $\lambda$ sufficiently small such that, for each chosen $R \ge r_0$, there exists a solution $w$ of
\begin{equation}\label{proprie_w_CSP_mainteo}
\left\{ \begin{array}{l}
w \in C^1([R,\infty)) \quad \text{and $C^2$ except possibly at some $R_1>R$;} \\[0.2cm]
w \ge 0, \quad \text{on } \, [R, \infty), \qquad w \equiv 0 \quad \text{on } \, [R_1, \infty), \\[0.3cm]
w(R) = \lambda, \qquad w'<0 \quad \text{on } \, [R, R_1), \\[0.2cm]
\disp \big( \varphi(w')\big)' - \theta(r)\varphi(w') \le \frac{1}{2C} \beta(r)f(w)l(|w'|) \qquad \text{on } \, [R, \infty).
\end{array}
\right.
\end{equation}
We then specify $R \ge r_0$ large enough to satisfy
\begin{equation}\label{boundabovev}
u(x) < \lambda  \qquad \text{for } \, r(x) \ge R.
\end{equation}
Defining $\bar w(x) = w(r(x))$ and using that $\cal O$ is a pole, we obtain
\begin{equation}\label{proprie_w_CSP_mainteo_2}
\left\{ \begin{array}{l}
\bar w \in C^1(M \backslash B_{R}(\cal O)); \\[0.2cm]
\bar w \ge 0, \qquad \bar w\equiv 0 \quad \text{on } \, M \backslash B_{R_1}(\cal O), \\[0.3cm]
\bar w = \lambda \quad \text{on } \, \partial B_R(\mathcal{O}), \qquad |\nabla \bar w|>0 \quad \text{on } \, B_{R_1}(\mathcal{O}) \backslash B_R(\mathcal{O})
\end{array}
\right.
\end{equation}
and also, since $w' \le 0$ and $\varphi$ is odd, by \eqref{laplacomp_csp}
\begin{equation}\label{altranuovaentrata}
\begin{array}{lcl}
\disp \Delta_\varphi \bar w & = & \disp \big(\varphi(w')\big)' + \varphi(w')\Delta r \le \big(\varphi(w')\big)' - \theta(r)\varphi(w') \le  \disp \frac{1}{2C} \beta(r)f(w)l(|w'|) \\[0.3cm]
& = & \disp \disp \frac{1}{2C} \beta(r)f(\bar w)l(| \nabla \bar w|) \qquad \text{weakly on } \, M \backslash B_R(\cal O).
\end{array}
\end{equation}
Define $\Omega_R = \Omega \cap (M \backslash B_R(\cal O))$. By assumption, $u < \lambda = \bar w$ on $\partial \Omega_R$, and we are going to show that $u \le \bar w$ on $\Omega_R$. Once this is shown, then clearly $u$ has compact support since $\bar w$ does, concluding our proof. We reason by contradiction and we suppose that $c = \sup_{\Omega_R}(u-\bar w) >0$. For $\delta \in (0, c)$, set $U_\delta = \{u - \bar w > \delta\} \neq \emptyset$. Note that $U_\delta \cap \partial \Omega_R = \emptyset$, and moreover $U_\delta$ is relatively compact since $v$ vanishes at infinity and $\bar w \ge 0$. On the compact set $\Gamma = \{u - \bar w = c\}$ identity $|\nabla \bar w| = |\nabla u|$ holds. We claim that $\inf_{\Gamma}|\nabla \bar w| >0$. Suppose, by contradiction, that $|\nabla \bar w(x)|=0$ for some $x \in \Gamma$. By the construction of $\bar w$, $x \in M \backslash B_{R_1}(\cal O)$ and we examine two cases.
\begin{itemize}
\item[(i)] $x \in M \backslash \overline{B_{R_1}(\cal O)}$. In this case, $\bar w\equiv 0$ in a small neighbourhood $V$ of $x$, and thus, by the definition of $\Gamma$, $x$ is a local maximum of $u$ on $V$. Since
$$
\Delta_\varphi u \ge b(x) f(u) l(|\nabla u|)\ge 0,
$$
applying the finite maximum principle in Theorem \ref{teo_FMP2} to $c -u$ we deduce that $u \equiv c$ on $V$. Therefore, the set where $u = c$ is open, closed and non-empty in $\Omega_R\backslash B_{R_1}(\cal O)$, and we conclude $u \equiv c$ on the connected component of $\Omega_R \backslash \overline{B_{R_1}(\cal O)}$ containing $x$. Note then that $\Omega_R \backslash \overline{B_{R_1}(\cal O)}$ is connected and unbounded: in fact, since $\cal O$ is a pole of $M$ the normal exponential map realizes a diffeomorphism between $M \backslash \cal O$ and $\partial \cal O \times \R^+$. In particular, $\Omega$ is diffeomorphic to $K \times (r_0, \infty)$ for some connected component $K \subset \partial \cal O$, and $\Omega_R \backslash \overline{B_{R_1}(\cal O)}$ is diffeomorphic to the connected set $K \times (R_1,\infty)$. Concluding, $u \equiv c$ on $\Omega_R \backslash \overline{B_{R_1}(\cal O)}$, a contradiction since $u$ is assumed to vanish at infinity.
\item[(ii)] $x \in \partial B_{R_1}(\mathcal{O})$. In this case, $\nabla u(x) = \nabla \bar w(x) = 0$ and $u$ on $\Omega_R \backslash B_{R_1}(\cal O)$ has a boundary, global  maximum at $x$. Moreover, by $(i)$ the set $\Gamma$ does not intersect $\Omega_R \backslash \overline{B_{R_1}(\cal O)}$, hence $u<c$ on $\Omega_R \backslash \overline{B_{R_1}(\cal O)}$. Let $\gamma(t)$ be a ray from $\cal O$ with $\gamma(R_1)=x$. As in the proof of the finite maximum principle and the Hopf Lemma, (Thm. \ref{teo_FMP2}),  on a small enough annulus $E_\rho = B_{2\rho}\backslash B_\rho$ centered at $x_0 = \gamma(R_1 + 2\rho)$ we can construct a solution of
\begin{equation}\label{constru_w}
\begin{array}{l}
\Delta_\varphi v \ge 0, \qquad v=0 \ \text{ on } \, \partial B_{2\rho}, \quad v = \eta < c - \max_{\partial B_\rho} u \quad \text{on } \, \partial B_{\rho} \\[0.3cm]
\langle \nabla v, \nabla r_{x_0} \rangle <0 \quad \text{on } \, \partial B_{2\rho}, \qquad |\nabla v|>0 \quad \text{on } \, \overline{E}_\rho.
\end{array}
\end{equation}
where $r_{x_0}$ is the distance from $x_0$. With the aid of Proposition \ref{prop_comparison}, we can compare $u$ with $c-v$ on $E_\rho \subset \Omega \backslash B_{R_1}(\cal O)$ since
$$
\left\{ \begin{array}{ll}
\Delta_\varphi u  \ge b(x)f(u)l(|\nabla u|) \ge 0 \ge \Delta_\varphi(c-v), \\[0.3cm]
u- (c-v) = u-c \le 0 \quad \text{on } \, \partial B_{2\rho}, \\[0.2cm]
u-(c-v) \le \max_{\partial B_\rho} u - c + \eta < 0 \quad \text{on } \, \partial B_\rho
\end{array}\right.
$$
to deduce $u \le c-v$ on $E_\rho$. Since equality holds at $x \in \partial E_\rho$, we get
$$
0 \le \langle \nabla (u - c+v), \nabla r_{x_0}\rangle(x) = \langle \nabla v, \nabla r_{x_0} \rangle(x) < 0,
$$
contradiction.
\end{itemize}
We have therefore shown that $|\nabla \bar w|=|\nabla u|>0$ on $\Gamma$, and we are in the position to conclude as usual: from $l>0$ on $\R^+$ the quotient $l(|\nabla \bar w|)/l(|\nabla u|)$ is continuous and $\le 2$ on $U_\delta$, for $\delta$ sufficiently close to $c$. By the $C$-increasing property of $f$,
$$
f\big(\bar w(x)\big) l\big(|\nabla \bar w(x)|\big) \le 2C f\big(u(x)\big)l\big(|\nabla u(x)|\big) \qquad \forall \, x \in U_\delta,
$$
and thus
$$
\Delta_\varphi \bar w \le \frac{1}{2C} b(x)f(\bar w)l(|\nabla \bar w|) \le b(x)f(u)l(|\nabla u|) \le \Delta_\varphi u
$$
on $U_\delta$, with $u = \bar w + \delta$ on $\partial U_\delta$. By the comparison Proposition \ref{prop_comparison}, $u \le  \bar w + \delta$ on $U_\delta$, contradicting the very definition of $U_\delta$ and concluding the proof.\\
The reverse implication $\csp \Rightarrow \eqref{KO_zero}$, under the further assumptions $l \in \lip_\loc\big((0,\xi_0)\big)$ and $f(0)l(0)=0$, is a direct consequence of Theorem \ref{teo_necessitybello}. 
\end{proof}

We next specialize Theorem \ref{teo_CSP_nuovo}, and we prove Theorem \ref{cor_csp_specialized_intro}, that we rewrite for the reader's convenience.

\begin{theorem}\label{cor_csp_specialized}
Let $(M, \metric)$ be a manifold with a pole $\mathcal{O}$ such that, setting $r(x)= \dist(x,\mathcal{O})$,  \eqref{arra_CSP_assu} holds for some $\alpha \ge -2$, $\kappa \ge 0$. Consider $\varphi, b, f,l$ satisfying \eqref{assumptions}, \eqref{assumptions_bfl}, \eqref{assum_secODE_altreL_intro} and \eqref{Cincre_fl_CSP_intro}. Fix $\chi, \mu \in \R$ with
\begin{equation}\label{condi_chimuomega_CSP}
\chi>0, \qquad  \mu \le \chi - \frac{\alpha}{2}
\end{equation}
and assume that
\begin{equation}\label{assu_poli_csp}
\begin{array}{ll}
l(t) \asymp t^{1-\chi}\varphi'(t) & \quad \text{for } \, t \in (0,1), \\[0.2cm]
b(x) \ge C_1\big( 1+r(x)\big)^{-\mu} & \quad \text{for } \, r(x) \ge r_0,
\end{array}
\end{equation}
for some constant $C_1>0$. If there exists a constant $c_F \ge 1$ such that
\begin{equation}\label{condi_F}
F(t)^{\frac{\chi}{\chi+1}} \le c_F f(t) \qquad \text{for each } \, t \in (0, \eta_0), 
\end{equation}
then, 
\begin{equation}\label{laequi_KO}
\text{$\csp$ holds for $(P_\ge)$} \qquad \Longleftrightarrow \qquad \eqref{KO_zero}.
\end{equation}
\end{theorem}
\begin{proof}
First, we note that requirements \eqref{assumptions}, \eqref{assumptions_bfl}, \eqref{assum_secODE_altreL_intro} and \eqref{Cincre_fl_CSP_intro} on $\varphi,f,l$ correspond to \eqref{assumptions_CSP_necessity}, \eqref{phiel_solozero}, \eqref{assu_perCSP} and 
$$
f>0 \ \text{ on } \, (0,\eta_0), \qquad l \in \lip_\loc((0,\xi_0)), \qquad f(0)l(0)=0. 
$$
We can therefore apply Theorem \ref{teo_necessitybello} to deduce the validity of implication $\csp \Rightarrow \eqref{KO_zero}$.\\
Viceversa, assume \eqref{KO_zero}, that in view of \eqref{assu_poli_csp} is equivalent to 
\begin{equation}\label{KO_poli_CSP}
F^{-\frac{1}{\chi+1}} \in L^1(0^+).
\end{equation}
We first observe that it is enough to prove $\csp$ when
$$
\begin{array}{ll}
\disp l(t)= C_2 t^{1-\chi}\varphi'(t) &\quad \text{if } \, t \in (0,1), \\[0.2cm]
b(x) = C_1\big( 1+r(x)\big)^{-\mu} & \quad \text{if } \, r(x) \ge r_0,
\end{array}
$$
for constants $C_1,C_2 >0$. Indeed, the Keller-Osserman condition for $l(t)= C_2 t^{1-\chi}\varphi'(t)$ is still \eqref{KO_poli_CSP}. As underlined in Remark \ref{rem_utile}, although  $t^{1-\chi} \varphi'(t) \in L^\infty_\loc(\R^+_0)$ might fail to be continuous at $t=0$ we can still apply Proposition \ref{prop_CSP_radial}, and thus Theorem \ref{teo_CSP_nuovo}, once we check the validity of the remaining assumptions: $\chi > 0$ imply both \eqref{phiel_solozero} and $(C_3)$, \eqref{condi_chimuomega_CSP} imply $(C_1), (C_2)$, and \eqref{condi_F} is equivalent to $(C_4)$. On the other hand, the function
$$
\bar \beta(t) = c(1+t)^{-\chi-1}, \qquad \text{with } \quad c < K_\infty
$$
satisfies $(\beta_1)$ and $(\beta_2)$. To conclude, note that \eqref{assu_subeta} is equivalent to 
$$
\mu \le \max \left\{ \chi+1,\chi- \frac\alpha 2\right\}= \chi-\frac \alpha 2.
$$
\end{proof}

\begin{remark}
\emph{In the hypotheses of Theorem \ref{cor_csp_specialized}, set $v(r) = \vol(\partial B_r(\cal O))$.  Using the divergence theorem and coarea formula we deduce
$$
v'(r) = \int_{\partial B_r(\cal O)}   \Delta r  \ge -C_1 r^{\chi-\mu} v(r)
$$
where we used \eqref{laplacomp_csp} with $\alpha/2=\chi-\mu$. A further integration gives
$$
\begin{array}{ll}
\text{if } \, \mu< \chi+1, & \qquad v(r) \ge C_1 e^{-C_2 r^{\chi+1-\mu}} \\[0.2cm]
\text{if } \, \mu = \chi+1, & \qquad v(r) \ge C_1 r^{-C_2}.
\end{array}
$$
for some constants $C_1, C_2> 1$. Assume that $\vol(M)< \infty$. Integrating the above on $(r,\infty)$, taking logarithms and recalling that
$$
\vol(M)-\vol(B_r(\mathcal{O}))=\int_r^{\infty} v(s) \di s
$$
we deduce that
\begin{equation}\label{volgrowth_CSP}
\begin{array}{ll}
\text{if } \, \mu < \chi+1, & \disp \qquad \limsup_{r \ra \infty} \frac{- \log\big( \vol(M)-\vol(B_r(\cal O))\big)}{r^{1+\chi-\mu}} < \infty; \\[0.5cm]
\text{if } \, \mu = \chi+1, & \disp \qquad \limsup_{r \ra \infty} \frac{- \log\big( \vol(M)-\vol(B_r(\cal O))\big)}{\log r} < \infty.
\end{array}
\end{equation}
It is interesting to compare \eqref{volgrowth_CSP} with conditions \eqref{volgrowth_sigmazero}. In view of  Theorem \ref{teo_main_2}, one might wonder whether the $\csp$ could be proved under a volume growth assumption like \eqref{volgrowth_CSP}. Indeed, the problem seems to be quite hard, and one of the main reasons lies in the fact that a manifold satisfies $\csp$ if and only if each of its ends does. This forces \eqref{volgrowth_CSP} to be satisfied on \emph{each} end, otherwise an end with big volume would be sufficient for \eqref{volgrowth_CSP} to hold independently of the behaviour of the others. However, an approach via integral estimates like in Theorem \ref{teo_main_2}, loosely speaking, seems unable to distinguish among different ends, and thus, at least, it needs to be complemented by new techniques.
}
\end{remark}

\begin{remark}
\emph{Both $\fmp$ and $\csp$ can be considered for more general inequalities, including the prototype ones
$$
\Delta_p u = u^\omega \pm |\nabla u|^q
$$
for some $\omega,q>0$. In fact, the Keller-Osserman condition changes according to whether $q \ge p-1$ or $q< p-1$, see \cite{pucciserrinzou} and \cite{felmermontenegroquaas} for a detailed account. In particular, in \cite{felmermontenegroquaas} the authors propose suitable Keller-Osserman conditions for  the case when the terms $u^\omega$ and $|\nabla u|^q$ strongly interact. The sharpness of these conditions for general nonlinearities in $u$ and $|\nabla u|$ is, to our knowledge, still an open problem.
}
\end{remark}



%
%

\subsubsection{A second ODE lemma: locating the support}

Proposition \ref{prop_CSP_radial} guarantees the existence of $R_1>R$ such that the supersolution $w$ in  \eqref{proprie_w_CSP} vanishes outside of $B_{R_1}$. However, the proof gives loose indication on the distance between $R_1$ and $R$. Although this further information is not needed in the results that we present here,  nevertheless we feel worth to underline that the construction of $w$ can be modified in such a way to locate $R_1$, say to have $R_1 = 2R$. More importantly, this new method, that works under a set of assumptions which is skew with respect to that  in Proposition \ref{prop_CSP_radial}, allows for weights $b(x)$ that may oscillate between two different polynomial type decays. In the sequel, we need

\begin{itemize}
\item[$(C_2)'$] there exist constants $d_1>0$ and $c_1>0$ such that
$$
\begin{array}{ll}
\varphi'(st) \le d_1 \varphi'(s) \varphi'(t) & \quad \text{for each } \, s,t \in (0,1]; \\[0.3cm]
\disp l(s)l(t) \le c_1 l(st) & \quad \text{for each } \, s,t \in (0,1];
\end{array}
$$
\end{itemize}

\noindent instead of the weaker $(C_2)$ (cf. Lemma \ref{teclemmac2}). On the other hand, we will not need $(C_3)$. Regarding $\beta$, we assume that $\beta \equiv \bar \beta$, that $\beta$ vanishes at infinity, and a further condition $(\beta_3)$, namely we require

\begin{itemize}
\item[$(\beta_1)'$] $0<\beta \in C^1([r_0,\infty))$, $\quad \beta' \le 0$ for $t \ge r_0$, $\quad \beta(t) \ra 0$ as $t \ra \infty$;
\item[$(\beta_2)'$] there exists a constant $c_\beta \ge 1$ such that
$$
\frac{-\beta'(t)}{K^{-1}(\beta(t))} \le c_\beta \beta(t) \qquad \text{for each } \, t \in [r_0, \infty).
$$
\item[$(\beta_3)$] There exists a constant $\hat c_\beta >0$ such that
$$
\limsup_{t \ra \infty} \frac{-t \beta'(t)}{\beta(t)} \ge \hat c_\beta.
$$
\end{itemize}

\begin{remark}
\emph{Up to choosing $r_0$ large enough, by $(\beta_1)'$ we can assume $\beta(t) \in [0, K_\infty)$, thus $(\beta_2)'$ is meaningful.
}
\end{remark}

\begin{example}\label{ex_CSPpoli_2}
\emph{Referring to Example \ref{ex_CSPpoli}, $\varphi$ and $l$ satisfy $(C_1), (C_2)'$ for each $\chi \ge 0$, while $(C_4)$ is met for $\omega \le \chi$. If further $\beta(t) = (1+t)^{-\mu}$, $(\beta_1)'$ and $(\beta_3)$ require $\mu>0$ to be both satisfied, and $(\beta_2)'$ needs $\mu \le \chi+1$.
}
\end{example}

We are ready to state our second main ODE result, to be compared to Proposition \ref{prop_CSP_radial}. Its delicate proof originates from the paper \cite{pucciserrin2}, later refined in \cite{PuRS}, \cite{rigolisalvatorivignati} and \cite{rigolisalvatorivignati_5}, and to help readability we postpone it  to the end of this section.

\begin{proposition} \label{lem_primo_CSP}
Let $\varphi,f,l$ satisfy \eqref{assu_perCSP} and \eqref{phiel_solozero}, and assume the validity of $(C_1),(C_2)',(C_4)$ and $(\beta_1)', (\beta_2)', (\beta_3)$, for some $r_0>0$. Having fixed a non-negative $\theta \in C([r_0,\infty))$, suppose that
\begin{equation}\label{solu_radialCSP_second}
\disp \limsup_{R \ra \infty} K\left(\frac{1}{R K^{-1}(\beta(2R))}\right)R\theta(R) < \infty.
\end{equation}
Then, there exists a diverging sequence $\{R_j\}$ such that the following holds: if 
\begin{equation}\label{KO_zero_22}\tag{$\mathrm{KO}_0$}
\frac{1}{K^{-1} \circ F} \in L^1(0^+),
\end{equation}
then for each $\epsilon \in (0,\xi_0)$, there exist $\lambda \in (0, \eta_0)$ and, for each $R \in \{R_j\}$, a function $z$ with the following properties:
\begin{equation}\label{proprie_z_CSP}
\left\{ \begin{array}{l}
z \in C^1([R, \infty)), \quad \text{and $C^2$ except possibly at $2R$} \\[0.2cm]
0 \le z \le \lambda, \qquad z(R)= \lambda, \qquad z \equiv 0 \quad \text{on } \, [2R, \infty), \\[0.2cm]
z' <0 \quad \text{on } \, [R, 2R), \qquad |z'| \le \epsilon \quad \text{on } \, [R, \infty), \\[0.2cm]
\big(\varphi(z')\big)'-\theta(t) \varphi(z') \le \epsilon \beta(t) f(z) l(|z'|) \qquad {\rm on} \, [R,\infty)
\end{array}\right.
\end{equation}
\end{proposition}

\begin{remark}\label{rem_Rj}
\emph{The two $\limsup$ in $(\beta_3)$, \eqref{solu_radialCSP_second} could be simultaneously replaced by $\liminf$. Indeed, the sequence $\{R_j\}$ is just required to satisfy 
\begin{equation}\label{condi_R_j}
\left\{\begin{array}{l}
\disp R_1 \ge 2r_0 \\[0.2cm]
\disp \frac{- R_j \beta'(R_j)}{\beta(R_j)} \ge \frac{\hat c_\beta}{2} \\[0.4cm]
\disp K\left(\frac{1}{R_j K^{-1}(\beta(2R_j))}\right)R_j\theta(R_j) \le B_2,
\end{array}\right.
\end{equation}
for some $B_2>0$. Observe that, in the ``double liminf" case, the vanishing of $\beta$ is automatic by integrating $(\beta_3)$, hence $(\beta_1)'$ coincides with $(\beta_1)$.}
\end{remark}


As a direct corollary, we have the following result, whose proof follows verbatim that of Theorem \ref{teo_CSP_nuovo} replacing, in the argument, Proposition \ref{prop_CSP_radial} with Proposition \ref{lem_primo_CSP}.

\begin{theorem}\label{teo_CSP_general_1}
Let $(M, \metric)$ be a manifold possessing a pole $\mathcal{O}$, and let $r(x) = \dist(x, \mathcal{O})$. Suppose that \eqref{arra_CSP_assu} holds, consider $0<b \in C(M)$ and let $\beta \in C^1([r_0,\infty))$ such that
$$
b(x) \ge \beta\big(r(x)\big) \qquad \text{ for } \, r(x) \ge r_0.
$$
Let $\varphi, f,l$ satisfy \eqref{assumptions_CSP_necessity}, \eqref{phiel_solozero} and \eqref{assu_perCSP}, and assume the validity of $(C_1), (C_2)', (C_4)$ and $(\beta_1)', (\beta_2)', (\beta_3)$. Suppose that
$$
\disp \limsup_{R \ra \infty} K\left(\frac{1}{R K^{-1}(\beta(2R))}\right)R^{1+ \frac{\alpha}{2}} < \infty.
$$
Then,
$$
\eqref{KO_zero} \qquad \Longrightarrow \qquad \text{$\csp$ holds for \eqref{csp_problem}.}
$$
If moreover 
$$
l \in \lip_\loc\big((0,\xi_0)\big), \qquad f(0)l(0)=0, 
$$
then
$$
\eqref{KO_zero} \qquad \Longleftrightarrow \qquad \text{$\csp$ holds for \eqref{csp_problem}.}
$$
\end{theorem}

Specified to power-like $\varphi,f,l$, Theorem \ref{teo_CSP_general_1} has the next corollary for general weights $b$, where we used Remark \ref{rem_Rj}.

\begin{corollary}\label{cor_CSP_general_plapla}
Let $M$ be a complete manifold with a pole $\mathcal{O}$ and satisfying \eqref{arra_CSP_assu}. Suppose that $\varphi,f,l$ satisfy \eqref{assumptions_CSP_necessity}, \eqref{phiel_solozero}, \eqref{assu_perCSP} and $l \in \lip_\loc((0,\xi_0))$. Moreover, assume that for some $p,\chi,\omega \in \R$ satisfying
$$
p>1, \qquad 0 < \chi \le p-1, \qquad \omega>0,
$$
it holds
$$
\varphi'(t) \asymp t^{p-2}, \qquad l(t) \asymp t^{p-1-\chi}, \qquad f(t) \asymp t^{\omega}
$$
in a right neighbourhood of $t=0$. Suppose that there exist $r_0>0$ and $0<\beta \in C^1([r_0,\infty))$ satisfying $(\beta_1)$ and 
\begin{equation}\label{proprie_beta_plapla}
\begin{array}{l}
\disp - \beta'(t) \le B \big[\beta(t)\big]^{ \frac{\chi+2}{\chi+1}} \qquad \text{on } \, [r_0,\infty); \\[0.3cm]
\disp\liminf_{t \to \infty}\frac{-t \beta' (t)}{\beta(t)} >0, \qquad \disp \liminf_{t \ra \infty} \frac{t^{\frac{\alpha}{2}-\chi}}{\beta(t)} < \infty,
\end{array}
\end{equation}
for some constant $B>0$. If $0< b \in C(M)$ satisfies
$$
b(x) \ge \beta(r(x)) \qquad {\rm for} \ r(x) \ge r_0,
$$
then
$$
\text{$\csp$ holds for \eqref{csp_problem}} \qquad \Longleftrightarrow \qquad \omega < \chi.
$$
\end{corollary}

\begin{remark}
\emph{A careful analysis of \eqref{proprie_beta_plapla} shows that the above corollary allows for bounds $\beta$ that oscillate between the polynomial decays $t^{-1-\chi}$ and $t^{\alpha/2-\chi}$.
}
\end{remark}

\subsubsection{Non-empty cut-locus: the $p$-Laplacian case}

With the help of Section \ref{sec_exhaustion} we now remove the pole condition in the particular case of the $p$-Laplace operator, exploiting the fake distance function $\varrho$. Let $\Omega$ be an end of $M$, and we consider a solution $u$ of
\begin{equation}\label{CSP_plaplacian}
\left\{ \begin{array}{l}
\Delta_p u \ge b(x)f(u)|\nabla u|^{p-1-\chi} \qquad \text{on } \, \Omega, \\[0.2cm]
\disp u \ge 0, \qquad \lim_{x \in \Omega, \, x \ra \infty} u(x) = 0.
\end{array}\right.
\end{equation}
Where $0 \le \chi \le p-1$. Since $\varphi(t) = t^{p-1}$ and $l(t)=t^{p-1-\chi}$, the function $K$ in \eqref{def_K_CSP} automatically satisfies $(C_1)$ and $(C_2)'$. Eventually, we assume $(C_4)$, that in the present case can be written as follows:
\begin{itemize}
\item[$(C_4)$] there exists $c_F \ge 1$ such that 
$$
c_F f(t) \ge F(t)^{\frac{\chi}{\chi+1}} \qquad \text{for } \, t \in [0, \eta_0).
$$
\end{itemize}
Let $(M^m, \metric)$ be a complete Riemannian manifold satisfying, 
\begin{equation}\label{iporicci_csp}
\Ricc(\nabla r, \nabla r) \ge -(m-1)\kappa^2 \qquad \text{on } \, \mathcal{D}_o,
\end{equation}
for some origin $o$ and constant $\kappa>0$. Suppose that $\Delta_p$ is non-parabolic on $M$, and define the fake distance $\varrho$ as in \eqref{def_bxy} associated to the hyperbolic space of curvature $-\kappa^2$ (that is, $g(r) = \kappa^{-1} \sinh(\kappa r)$). 
To be able to radialize with respect to $\varrho$, we shall assume 
\begin{equation}\label{condi_bbeta_plap}
b(x) \ge \beta\big( \varrho(x)\big) \qquad \text{on } \, M, 
\end{equation}
for some function $\beta$ matching the necessary assumptions to apply Proposition \ref{prop_CSP_radial}. In our case of interest, we can restrict to non-increasing $\beta$ and to $\bar \beta = \beta^\gamma$, for some $\gamma \ge 1$. Then, $(\beta_1)$ and $(\beta_2)$ amount to the requests 
\begin{itemize}
\item[$(\beta_1)$] $0<\beta \in C^1(\R^+_0)$, $\beta' \le 0$ on $\R^+$.
\item[$(\beta_2)$] For some $\gamma \ge 1$,
$$
-\beta'(t) \le c_\beta \beta(t)^\frac{\chi+1+\gamma}{\chi+1} \qquad \text{on } \, [1, +\infty). 
$$
\end{itemize}
The prototype example is given by the choice $\beta(t) = (1+t)^{-\mu}$ with $\mu \in [0,\chi+1]$. 
\begin{remark}
\emph{Condition \eqref{condi_bbeta_plap} is expressed in terms of level sets of the Green kernel, and it would be desirable to obtain a corresponding decay of $b$ in terms of the distance function $r$. However, finding good geometric conditions to estimate $\varrho$ from \emph{below} by $r$ at infinity is a delicate issue, especially when $\Ricc$ is negative somewhere. A particular, relevant case is when $\Ricc \ge 0$ and $M$ has maximal volume growth, for which, in case $p=2$, we have that $\rho \asymp r$. A further discussion appears at the end of the section. 
}
\end{remark}

Our main result is the following 

\begin{theorem}\label{teo_CSP_plapla}
Let $(M^m, \metric)$ be a complete Riemannian manifold such that, for some origin $o$,
\begin{equation}\label{iporicci_csp}
\Ricc(\nabla r, \nabla r) \ge -(m-1)\kappa^2 \qquad \text{on } \, \mathcal{D}_o,
\end{equation}
for some constant $\kappa>0$. Let $p \in (1,\infty)$; suppose that $\Delta_p$ is non-parabolic on $M$ and that the minimal positive Green kernel $\gr_p(x,o)$ with pole at $o$ satisfies
\begin{equation}\label{properness_b}
\gr_p(x,o) \ra 0 \quad \text{as $x$ diverges.}
\end{equation}
Let $f$ satisfy
\begin{equation}\label{f_C_increasing}
f \in C(\R), \qquad f>0 \quad \text{and $C$-increasing on } \, (0, \eta_0)
\end{equation}
and, for $\chi \in (0, p-1]$, assume $(C_4)$. Let $0< b \in C(M)$ satisfy \eqref{condi_bbeta_plap}, for some $\beta$ matching $(\beta_1),(\beta_2)$ above. If, either
\begin{itemize}
\item[(i)] $\chi = p-1$ and $f(0)=0$, or
\item[(ii)] $\chi \in (0, p-1)$ and $p \in (1,2]$,
\end{itemize}
then
\begin{equation}\label{inteinzero_plap}
\text{$\csp$ holds for \eqref{CSP_plaplacian}} \qquad \Longleftrightarrow \qquad F^{-\frac{1}{\chi+1}} \in L^1(0^+).  
\end{equation}
\end{theorem}

\begin{remark}
\emph{In view of Corollary \ref{cor_holopainen}, Theorem \ref{teo_sobolev} and Examples \ref{ex_minimal}, \ref{prop_hebey} and \ref{ex_roughisometric}, the vanishing of $\gr$ is granted provided either one of the next conditions holds:
\begin{itemize}
\item[(i)] $p \in (1,\infty)$ and $\Ricc \ge 0$ on $M$;
\item[(ii)] $p \in (1,\infty)$ and $M$ supports the Sobolev inequality \eqref{sobolev};
\item[(iii)] $p \in (1,m)$ and $M$ is minimally immersed into a Cartan-Hadamard manifold;
\item[(iv)] $p \in (1,m)$, 
$$
\Ricc \ge -(m-1) \kappa^2 \metric, \qquad \inf_{x \in M} \vol\big(B_1(x)\big) >0
$$
and $M$ supports the Poincar\'e inequality \eqref{poincare};
\item[(v)] $p \in (1,m)$, $M$ is roughly isometric to $\R^m$, and 
$$
\Ricc \ge -(m-1) \kappa^2 \metric, \qquad \mathrm{inj}(M)>0.
$$
\end{itemize}  
}
\end{remark}

\begin{proof}
Because of \eqref{properness_b}, the fake distance $\varrho$ defined in \eqref{def_bxy} and associated to the hyperbolic space of curvature $-\kappa^2$ is proper. Furthermore, by \eqref{deri_bG} and Proposition \ref{prop_gradienbounded}, 
\begin{equation}\label{belleproprieta}
|\nabla \varrho| \le C_\varrho, \quad \Delta_p \varrho \ge 0 \qquad \text{on } \, \{\varrho \ge 1\},
\end{equation}
for some constant $C_\varrho$. For $s>0$, write $D_s$ to denote the fake ball $\{\varrho < s\}$, and define $\Omega_s = \Omega \cap D_s$. \\
We first prove the implication $\eqref{KO_zero} \Rightarrow \csp$. Since \eqref{solu_radialCSP} is automatically met for $\theta(t) \equiv 0$, for each fixed $\eps>0$ Proposition \ref{prop_CSP_radial} guarantees the existence of $\lambda \in (0, \eta_0)$ sufficiently small and $r_0 > 1$ such that, for each $R> r_0$, we can find $R_1 > R$ and a solution $w$ of 
\begin{equation}\label{proprie_w_CSP_mainteo_concut}
\left\{ \begin{array}{l}
w \in C^1([R,\infty)) \quad \text{and $C^2$ except possibly at $R_1$;} \\[0.2cm]
w \ge 0, \quad \text{on } \, [R, \infty), \qquad w(R) = \lambda, \qquad w \equiv 0 \quad \text{on } \, [R_1, \infty), \\[0.3cm]
w'<0 \quad \text{on } \, [R, R_1), \qquad |w'| \le \eps \quad \text{on } \, [R, \infty),  \\[0.2cm]
\disp \big( |w'|^{p-2}w')\big)' \le \eps \beta(r)f(w)|w'|^{p-1-\chi} \qquad \text{on } \, [R, \infty),
\end{array}
\right. 
\end{equation}
Choose $r_0$ such that $u < \lambda$ on $\Omega_{r_0}$, and $R>r_0$. Since $R_1$ depends continuously on $\eps, R$, up to changing them a little bit we can assume that $R_1$ is a regular value of $\varrho$. Therefore, $\partial \Omega_{R_1}$ is smooth (we recall that the Green kernel $\gr_p$, being $p$-harmonic outside of $o$, is smooth where $\nabla \gr$ does not vanish, and therefore so is $\varrho$).
%
The value of $\eps$ will be specified later, depending just on $C_\varrho$ in \eqref{belleproprieta} and on the $C$-increasing constant. Define $\bar w(x) = w( \varrho(x))$. If we combine \eqref{bella!!!}, $w' \le 0$ and \eqref{deri_bG}, by \eqref{belleproprieta} we deduce 
\begin{equation}\label{supersol_w_CSP_mainteo_concut}
\left\{ \begin{array}{l}
\begin{array}{lcl}
\Delta_p \bar w & = & \disp \left[ \big(|w'|^{p-2}w'\big)' + \frac{v_g'}{v_g} |w'|^{p-2}w'\right](\varrho) |\nabla \varrho|^p \le \big(|w'|^{p-2}w'\big)'|\nabla \varrho|^p \\[0.4cm]
& \le & \disp \eps \beta(\varrho)f(\bar w)|w'(\varrho)|^{p-1-\chi}|\nabla \varrho|^p  \qquad \text{on } \, M \backslash D_R, 
\end{array} \\[0.8cm]
\bar w \in C^1(M \backslash D_R) \quad \text{and is in fact $C^2$ except possibly on $\partial D_{R_1}$;} \\[0.2cm]
\bar w \ge 0, \quad \text{on } \, M \backslash D_R, \qquad \bar w \equiv 0 \quad \text{on } \, M \backslash D_{R_1}, \\[0.3cm]
\bar w = \lambda \quad \text{on } \, \partial D_R, \qquad |\nabla \bar w| \le \eps C_\varrho \quad \text{on } \, D_{R_1}\backslash D_R. \\[0.2cm]
\end{array}
\right. 
\end{equation}

To prove $\csp$ we show, as before, that $u \le \bar w$ on $\Omega_R$: we proceed by contradiction, assuming $u> \bar w$ somewhere, and we look at the set 
$$
\Gamma = \big\{ u- \bar w = c\big\} \Subset \Omega_R, \qquad \text{with} \quad c = \max_{\Omega_R}(u-\bar w)>0.
$$
We will then apply a comparison theorem on an upper level set 
$$
U_\delta = \{ u- \bar w > \delta\} \cap \Omega_R 
$$
for $\delta$ close enough to $c$. To this aim, in Theorem \ref{teo_CSP_nuovo} we obtained $\Delta_p u \ge \Delta_p \bar w$ on $U_\delta$, crucially using $|\nabla \bar w| >0$ on $\Gamma$. However, now $\bar w$ is radial with respect to the fake distance $\varrho$, which differently from $r$ may possess stationary points. This forces us to use a different argument when $\chi< p-1$, that is, in case $(ii)$. On the other hand, in case $(i)$ the gradient term disappears and the argument goes straightforwardly: first choose $\eps$ so that $\eps C C_\varrho^p \le 1$, then observe that $0 \le u, \bar w \le \lambda < \eta_0$ on $\Omega_R$. Therefore, using that $f$ is $C$-increasing, on $U_\delta$ we have
$$
\Delta_p u \ge b(x)f(u) \ge \frac{1}{C} \beta(\varrho) f(\bar w) \ge \frac{1}{C\eps C_\varrho^p} \eps \beta(\varrho)f(\bar w)|\nabla \varrho|^p \ge \eps \beta(\varrho)f(\bar w)|\nabla \varrho|^p \ge \Delta_p \bar w,
$$
and by comparison we conclude $u \le \bar w + \delta$ on $U_\delta$, contradiction.\par
Case $(ii)$ is more delicate. Although we cannot guarantee that $|\nabla \bar w|>0$ on $\Gamma$, nevertheless, as a first step we still claim that
\begin{equation}\label{prim}
\Gamma \Subset D_{R_1}. 
\end{equation}
Indeed, if by contradiction there exists $x \in \Gamma \cap \inte(\Omega_{R_1})$, let $V$ be the connected component of $\Omega_{R_1}$ containing $x$. By the second in \eqref{belleproprieta} and $R_1 \ge 1$, we deduce that $V$ is necessarily unbounded by the maximum principle, being a component of the upper level  set $\{\varrho \ge R_1\}$ of a $p$-subharmonic function. Arguing as in the proof of Theorem \ref{teo_CSP_nuovo}, case $(i)$ we deduce $u \equiv c$ on $V$, which contradicts the vanishing of $u$ at infinity. Therefore, $u<c$ on $\inte (\Omega_{R_1})$, and since $\partial \Omega_{R_1}$ is smooth we can then apply the boundary maximum principle as in $(ii)$ of Theorem \ref{teo_CSP_nuovo} to get $u<c$ also on $\partial \Omega_{R_1}$. This proves \eqref{prim}. As a consequence, from \eqref{proprie_w_CSP_mainteo_concut} the inequality $|w'(\varrho)|>0$ holds on $\Gamma$. Coupling with $|\nabla u| = |\nabla \bar w| \le \eps C_\varrho$ on $\Gamma$, if $\eps C_\varrho < 1/2$ we can choose $\delta_0$ close enough to $c$ in such a way that 
\begin{equation}\label{interesting_csp}
|w'(\varrho)| >0, \qquad  |\nabla u|+ |\nabla \bar w| < 1 \qquad \text{on } \, \overline{U}_{\delta_0}.
\end{equation}
Note that $\delta_0$ depends on $\eps$. We come back to the differential inequality for $\bar w$, which on $U_{\delta_0}$ implies 
\begin{equation}\label{barwdelta}
\disp \Delta_p \bar w \le \eps \beta(\varrho)f(\bar w)|w'(\varrho)|^{p-1-\chi}|\nabla \varrho|^p = \eps \beta(\varrho)|w'(\varrho)|^{-1-\chi} f(\bar w)|\nabla \bar w|^{p}.
\end{equation}
For $\delta \in (\delta_0,c)$, we consider the open sets 
$$
E_\delta = U_\delta \cap \big\{|\nabla u| < |w'(\varrho)| \big\}, \qquad \hat E_\delta = U_\delta \cap \big\{ |\nabla u| > |w'(\varrho)|/2 \big\}.
$$
On $E_\delta$, using the $C$-increasing property of $f$ and $u>\bar w$, whenever $\eps \le C^{-1}$ we deduce that $u$ is a weak solution of 
\begin{equation}\label{Edelta}
\begin{array}{lcl}
\disp \Delta_p u & \ge & \disp \beta(\varrho) f(u) |\nabla u|^{p} |\nabla u|^{-1-\chi} \ge \frac{\beta(\varrho)}{C} f(\bar w) |\nabla u|^{p} |w'(\varrho)|^{-1-\chi} \\[0.3cm]
& \ge & \disp \eps \beta(\varrho)|w'(\varrho)|^{-1-\chi} f(\bar w)|\nabla u|^{p}. 
\end{array}
\end{equation}
On the other hand, on $\Gamma \cap \hat E_\delta$ we have $|\nabla u| = |\nabla \bar w|=|w'(\varrho)||\nabla \varrho|$, hence $|\nabla \varrho| > 1/2$. By continuity, we can therefore choose $\delta$ sufficiently close to $c$ so that
$$
\frac{|\nabla u|}{|w'(\varrho)|} \le 2|\nabla \varrho| \le 2C_\varrho \qquad \text{on } \, \hat E_\delta.
$$
As a consequence, if $\eps \le [C(2C_\varrho)^{\chi+1}]^{-1}$, on $\hat E_\delta$ the function $u$ weakly solves 
\begin{equation}\label{Edelta_hat}
\begin{array}{lcl}
\disp \Delta_p u & \ge & \disp \beta(\varrho) f(u) |\nabla u|^{p} |\nabla u|^{-1-\chi} \ge \frac{\beta(\varrho)}{C} f(\bar w) |\nabla u|^{p} |w'(\varrho)|^{-1-\chi} \frac{1}{(2C_\varrho)^{\chi+1}} \\[0.3cm]
& \ge & \disp \eps \beta(\varrho)|w'(\varrho)|^{-1-\chi} f(\bar w)|\nabla u|^{p}. 
\end{array}
\end{equation}
Putting together \eqref{barwdelta}, \eqref{Edelta} and \eqref{Edelta_hat}, for $\eps$ small enough depending only on $C, C_\varrho,\chi$ and setting $\bar w_\delta = \bar w + \delta$, the following inequalities hold on $U_\delta$:
\begin{equation}\label{compafinale}
\left\{\begin{array}{l}
\disp \Delta_p u \ge \disp \Big[\eps \beta(\varrho)|w'|^{-1-\chi} f(\bar w)\Big]|\nabla u|^{p} ; \\[0.3cm] 
\Delta_p \bar w_\delta \le \disp \disp \Big[\eps \beta(\varrho)|w'|^{-1-\chi}f(\bar w)\Big]|\nabla \bar w_\delta|^{p}, \\[0.3cm]
u = \bar w_\delta \qquad \text{on } \, \partial U_\delta.
\end{array}\right.
\end{equation}
To conclude we observe that, since $p \le 2$, the $p$-Laplacian is non-degenerate elliptic. We claim that we can apply the comparison Theorem 5.3 in \cite{antoninimugnaipucci} (the manifold version of \cite[Thm. 3.5.1]{pucciserrin}) with the choice 
$$
B(x,z,\xi) = \eps \beta(\varrho)|w'(\varrho)|^{-1-\chi} f(\bar w)|\xi|^{p}
$$
to deduce $u \le \bar w_\delta$ on $U_\delta$, a contradiction. To ensure the applicability of the above theorem, we shall check that $B$ is \emph{regular}, in the sense specified in \cite{antoninimugnaipucci}: for each compact set $K \Subset \R \times TU_\delta$, there exists a constant $L_K>0$ such that
$$
\big|B(x,z,\xi) - B(x,z,\eta)\big| \le L_K|\xi -\eta| \qquad \forall \, (x,z,\xi), (x,z,\eta) \in K. 
$$
Let $A$ be such that $|\xi|+|\eta| \le A$ for $(x,z, \xi) \in K$. From
$$
\left||\xi|^p-|\eta|^p\right| = \left| \int_{|\eta|}^{|\xi|}pt^{p-1}\di t\right| \le pA^{p-1} \left| \int_{|\eta|}^{|\xi|}\di t\right| = pA^{p-1}\big| |\xi|-|\eta|\big| \le pA^{p-1}|\xi -\eta|, 
$$
we obtain 
$$
\begin{array}{lcl}
\disp |B(x,z,\xi)-B(x,z,\eta)| & \le & \disp \eps \| \beta(\varrho) w'(\varrho)^{-1-\chi}f(\bar w)\|_\infty \big||\xi|^{p}-|\eta|^{p}\big| \\[0.3cm]
& \le & \bar C\eps |\xi-\eta|,
\end{array}
$$
for some $\bar C>0$, as claimed. This concludes the proof that $u \le \bar w$.\\[0.2cm]
We next show that $\csp \Rightarrow \eqref{KO_zero}$. Suppose the failure of \eqref{KO_zero}. Both cases $(i)$ and $(ii)$ imply $f(0)l(0)=0$ since $l(t)=t^{p-1-\chi}$; thus, having fixed $r_0>0$ and $\xi \in (0,\xi_0)$, we can apply Theorem \ref{teo_necessitybello} to deduce the existence of $\eta$ small enough and of  a solution $u_1 \in \lip(M \backslash B_{r_0})$, ($B_{r_0}$ centered at some fixed origin $o$) of 
$$
\left\{\begin{array}{l}
\Delta_p u_1 \ge b(x) f(u_1) |\nabla u_1|^{p-1-\chi} \qquad \text{weakly on } \, M \backslash B_{r_0}, \\[0.2cm]
0 < u_1 \le \eta \quad \text{on } \, M \backslash B_{r_0},  \\[0.2cm]
u_1 = \eta \quad \text{on } \, \partial B_{r_0}, \qquad  u_1(x) \ra 0 \quad \text{as } \, r(x) \ra \infty, \\[0.2cm]
|\nabla u_1| < \xi \qquad \text{on } \, M \backslash B_{r_0}.
\end{array} \right.
$$
Set $u_2(x) = \gr_p(x,o)$. Up to decreasing $\eta$, we can suppose that $u_1 \le u_2$ on $\partial B_{r_0}$, and thus, by comparison (being $\Delta_p u_1 \ge 0$) we deduce $u_1 \le u_2$ on $M \backslash B_{r_0}$.  Since $u_2$ trivially solves $(P_\le)$, by the subsolution-supersolution method (cf. \cite[Thm. 4.4]{kura}) there exists a solution $u$ of
$$
\left\{ \begin{array}{ll}
\Delta_p u = b(x)f(u)|\nabla u|^{p-1-\chi} & \quad \text{on } \, M \backslash B_{r_0} \\[0.2cm]
u_1 \le u \le u_2 & \quad \text{on } \, M \backslash B_{r_0}
\end{array}\right.
$$ 
The regularity theorem in \cite[Thm. 1]{tolksdorf} ensures that $u \in C^{1,\beta}_\loc(M \backslash B_{r_0})$. Since, by assumption, $u_2$ vanishes at infinity, $u$ shows the failure of $\csp$, concluding our proof.
\end{proof}

To conclude this section, we specialize to manifolds with non-negative Ricci tensor and we restrict to the linear case $p=2$. If $\Ricc \ge 0$, by \cite[Thm. 5.2]{liyau_acta} $M$ is non-parabolic if and only if $r/\vol(B_r) \in L^1(\infty)$ (that forces $m \ge 3$ by Bishop-Gromov comparison), and more precisely the Green kernel $\gr_2$ of $\Delta$ satisfies
\begin{equation}\label{liyau}
C^{-1} \int_{2r(x)}^{\infty} \frac{s\, \di s}{\vol(B_s)} \le \gr_2(x,o) \le C \int_{2r(x)}^{\infty} \frac{s \, \di s}{\vol(B_s)} \qquad \forall \, x \in M \backslash \{o\},
\end{equation}
for some constant $C>0$. Note that, by Theorem \ref{teo_holopainen}, the same holds for the $p$-Laplacian but $x$ is restricted to lie in $\partial M(r)$.\\ 
If we define $\varrho$ in \eqref{def_bxy} associated to the Euclidean space, from \eqref{liyau} we readily have 
\begin{equation}\label{def_h(x)}
\varrho(x) \asymp h(x) = \left[ \int_{2r(x)}^{\infty} \frac{s \, \di s}{\vol(B_s)} \right]^{\frac{1}{2-m}} \qquad \text{ for } \, r(x) \ge 1.
\end{equation}
In particular, if we suppose that $M$ has maximal volume growth 
\begin{equation}\label{maxgrowth}
\vol(B_r) \ge c r^{m} \qquad \forall \, r>0,
\end{equation}
for some $c>0$, then $\varrho \asymp r$. Condition \eqref{condi_bbeta_plap} can be stated in terms of the more manageable $h(x)$ and $r(x)$, giving the following result. 

\begin{theorem}\label{teo_CSP_plapla_riccimagzero}
Let $(M^m, \metric)$ be a complete Riemannian manifold with $\Ricc \ge 0$ and dimension $m \ge 3$. Fix an origin $o$, suppose that $r/\vol(B_r) \in L^1(\infty)$ and define
$$
h(x) = \left[ \int_{2r(x)}^{\infty} \frac{s \, \di s}{\vol(B_s)} \right]^{\frac{1}{2-m}}.
$$
Fix $\chi \in (0, 1]$, $\mu \in [0, \chi+1]$ and let $f$ satisfy
$$
f \in C(\R), \qquad f \ \text{ is positive and $C$-increasing on } \, (0, \eta_0),
$$
for some $\eta_0>0$. Assume $(C_4)$ and, if $\chi=1$, assume also that $f(0)=0$. Then, $\csp$ holds for solutions of 
\begin{equation}\label{CSP_2laplacian_2}
\left\{ \begin{array}{l}
\Delta u \ge \big( 1+ h(x)\big)^{-\mu} f(u)|\nabla u|^{1-\chi} \qquad \text{on } \, \Omega \, \text{ end of $M$;} \\[0.2cm]
\disp u \ge 0, \qquad \lim_{x \in \Omega, \, x \ra \infty} u(x) = 0.
\end{array}\right.
\end{equation}
if and only if 
\begin{equation}\label{inteinzero_plap_2}
F^{-\frac{1}{\chi+1}} \in L^1(0^+).
\end{equation}
\end{theorem}
 
\begin{proof}
Let $\varrho$ be the fake distance associated to the Euclidean space model. From \eqref{def_h(x)},
$$
b(x) = \big( 1+ h(x)\big)^{-\mu} \ge C\big(1+ \varrho(x)\big)^{-\mu},
$$
for some constant $C>0$, and $\varrho$ is proper. To apply Theorem \ref{teo_CSP_plapla}, we just need to ensure $|\nabla \varrho| \le C_\varrho$ (Proposition \ref{prop_gradienbounded} deals only with the case where $\Ricc$ has a negative lower bound). In the present setting, the sharp bound $|\nabla \varrho| \le 1$ has been shown in \cite{colding}.
\end{proof} 

The above theorem could be extended to cover the $p$-Laplace operator, simply replacing $\chi \in (0,1]$ with $\chi \in (0,p-1]$, provided the validity of the following conditions: 
\begin{itemize}
\item[(i)] the inequalities \eqref{bellastima} in Theorem \ref{teo_holopainen} hold for each $x \in \partial B_r$ and not only for $x \in \partial M(r)$. This is, to our knowledge, still an open problem;
\item[(ii)] estimate $|\nabla \varrho| \le C_\varrho$ holds when the reference model is the Euclidean space. In \cite{maririgolisetti_mono}, the authors claim that the stronger inequality $|\nabla \varrho| \le 1$ is true for each $p \in (1,m)$.
\end{itemize} 

\subsubsection{A further fake distance and the Feller property}\label{sec_feller}
When $l$ is constant and $\Delta_\varphi$ is the Laplace-Beltrami operator, a different fake distance $\varsigma$ recently constructed in \cite[Thm 2.1]{bianchisetti} turns out to be effective to improve Theorem \ref{teo_CSP_plapla}:

\begin{theorem}[\cite{bianchisetti}]\label{teo_bianchisetti}
Let $(M, \metric)$ be a complete manifold of dimension $m \ge 2$ satisfying 
\begin{equation}\label{riccibound_intero}
\Ricc \ge - (m-1) \kappa^2\big(1+r^2\big)^{\alpha/2} \metric \qquad \text{on } \, M, 
\end{equation}
for some $\kappa>0$ and $\alpha \in [-2,2]$. Fix an origin $o$ with associated distance $r(x) = \mathrm{dist}(x,o)$. Then, there exists constants $C_1,C_2>0$ depending on $m,\kappa,\alpha, o$ and a function $\varsigma \in C^\infty(M)$ such that
\begin{equation}\label{properties_varsigma}
\begin{array}{ll}
C_1^{-1} \disp \big(1+r(x)\big)^{1+ \frac{\alpha}{2}} \le \varsigma(x) \le C_1 \disp \big(1+r(x)\big)^{1+ \frac{\alpha}{2}} & \quad \text{if } \, \alpha \in (-2,2], \\[0.3cm]
C_1^{-1} \disp \log\big(2+r(x)\big) \le \varsigma(x) \le C_1 \log\big(2+r(x)\big) & \quad \text{if } \, \alpha = -2, \\[0.4cm]
\max\Big\{ |\nabla \varsigma|^2, |\Delta \varsigma| \Big\} \le C_2 (1+r)^{\alpha} & \quad \text{on } \, M.
\end{array}
\end{equation}
\end{theorem}

The proof of the existence of $\varsigma$ is delicate and inspired by the one in \cite{schoenyau}, and we refer the reader to both references for details. Here, we show ho to use $\varsigma$ to prove the compact support principle for solutions of 
$$
\Delta u \ge (1+r)^{-\mu} f(u)
$$
under the only geometric requirement \eqref{riccibound_intero}, for each $\alpha \in (-2,2]$ and 
\begin{equation}\label{ipo_mu_gradientless}
\mu \le 1 - \frac{\alpha}{2}, 
\end{equation}
provided that \eqref{KO_0CSP} holds, i.e, if
\begin{equation}
\frac{1}{\sqrt{F(t)}} \in L^1(0^+).
\end{equation}
The core is to construct the radial compactly supported supersolution $\bar w = w(\varsigma)$, for some $w$ that we assume to be $C^2$, convex and strictly decreasing until it touches zero in a $C^1$ way. Using \eqref{properties_varsigma} and $w' < 0$, $w'' \ge 0$, for $\alpha > -2$ we deduce
\begin{equation}\label{enfim_gradientless}
\begin{array}{lcl}
\Delta \bar w & = & \disp w''|\nabla \varsigma|^2 + w' \Delta \varsigma \\[0.2cm]
& \le & \disp C_2(1+r)^{\alpha} \Big\{ w'' - w' \Big\} \le C_3 \varsigma^{\alpha\left(1 + \frac{\alpha}{2}\right)^{-1}} \Big\{ w'' - w' \Big\},
\end{array}
\end{equation}
for some constant $C_3>0$. Now, let $f$ satisfying \eqref{f_C_increasing}, $f(0)=0$ and 
$$
c_F f(t) \ge \sqrt{F(t)} \qquad \text{for } \, t \in [0, \eta_0), 
$$
for some constant $c_F>0$. We apply Proposition \ref{prop_CSP_radial} with the choices
$$
\varphi(t)=t, \ \  l(t) = 1, \ \  \chi = 1, \ \ \theta(t) = 1, \ \ \beta(t) = t^{-(\mu+\alpha)\left(1 + \frac{\alpha}{2}\right)^{-1}}, \quad \bar \beta(t) = t^{-2}
$$
to deduce the existence of $w$ satisfying 
\begin{equation}\label{proprie_w_CSP}
\left\{ \begin{array}{l}
w \in C^1([R,\infty)) \quad \text{and $C^2$ except possibly at $R_1$;} \\[0.2cm]
0 \le w \le \lambda, \qquad w(R) = \lambda, \qquad w \equiv 0 \quad \text{on } \, [R_1, \infty),  \\[0.2cm]
w' < 0 \quad \text{on } \, [R, R_1), \qquad |w'| \le \eps \quad \text{on } \, [R, \infty), \\[0.2cm]
w'' - w' \le \eps t^{-(\mu+ \alpha)\left(1 + \frac{\alpha}{2}\right)^{-1}} f(w) \qquad \text{on } \, [R, \infty).
\end{array}
\right.
\end{equation}
In fact, $K(t) \asymp t^2$ and the growth requirement \eqref{solu_radialCSP} is equivalent to \eqref{ipo_mu_gradientless}. Plugging into \eqref{enfim_gradientless} and using again \eqref{properties_varsigma} we get
\begin{equation}\label{enfim_gradientless}
\begin{array}{lcl}
\Delta \bar w & \le & \disp C_3 \eps \varsigma^{\alpha\left(1 + \frac{\alpha}{2}\right)^{-1}} \varsigma^{-(\mu+ \alpha)\left(1 + \frac{\alpha}{2}\right)^{-1}} f(w) \\[0.4cm]
& \le & \disp C_4 \eps (1+r)^{-\mu} f(\bar w),
\end{array}
\end{equation}
for some $C_4>0$, hence $\bar w$ is the desired supersolution. The proof of $\csp$ now proceeds verbatim as in Theorem \ref{teo_CSP_plapla}, case $(i)$, leading to the following

\begin{theorem}\label{teo_CSP_nogradient}
Let $M$ be a complete $m$-dimensional manifold satisfying \eqref{riccibound_intero}, for some $\kappa>0$ and $\alpha \in (-2,2]$. Let $f$ satisfying \eqref{f_C_increasing}, $f(0)=0$ and 
$$
c_F f(t) \ge \sqrt{F(t)} \qquad \text{for } \, t \in [0, \eta_0), 
$$
for some constant $c_F>0$. Fix $\mu \le 1- \frac{\alpha}{2}$. If 
\begin{equation}
\frac{1}{\sqrt{F(t)}} \in L^1(0^+),
\end{equation}
then $\csp$ holds for solutions of 
\begin{equation}
\left\{ \begin{array}{l}
\Delta u \ge \big( 1+ r(x)\big)^{-\mu} f(u) \qquad \text{on } \, \Omega \, \text{ end of $M$;} \\[0.2cm]
\disp u \ge 0, \qquad \lim_{x \in \Omega, \, x \ra \infty} u(x) = 0.
\end{array}\right.
\end{equation}
\end{theorem}
\begin{remark}
\emph{The same method directly applies to the $p$-Laplacian for each $p>1$, provided that the corresponding of Theorem \ref{teo_bianchisetti} hold. This is likely to be the case, but the construction of $\varsigma$ may reveal subtleties. In this respect, the gradient estimates in \cite{wangzhang} should be useful.
}
\end{remark}
We conclude this section by commenting on Theorem \ref{teo_CSP_nogradient}. Analogously to the link between $\smp$ and \eqref{KO_infinity_intro}, it seems to us that the function-theoretic property that might describe how geometry relates to $\csp$ be the so called Feller property:
\begin{definition}\label{def_feller}
We say that the \emph{Feller property} (shortly, $\feller$) holds if, for every end $\Omega$ of $M$ and every $\lambda \in \R^+$, the minimal positive solution\footnote{Given any fixed exhaustion $\{\Omega_j\}$ of $\overline\Omega$ by smooth, relatively compact open sets containing $\partial \Omega$, $h$ is obtained as a limit of $h_j$ solving $\Delta h_j = \lambda h_j$ on $\Omega_j$, $h_j=1$ on $\partial \Omega$ and $h_j =0$ on $\partial \Omega_j$. By comparison, $h$ is independent of the chosen exhaustion.} $h$ of
$$
\left\{ \begin{array}{ll}
\Delta h = \lambda h & \quad \text{on } \, \Omega, \\[0.2cm]
h=1 & \quad \text{on } \, \partial \Omega.
\end{array}\right.
$$
satisfies $h(x) \ra 0$ as $x$ diverges in $\Omega$.
\end{definition}
Classically, the Feller property is introduced as the $C_0$ conservation property for the heat flow, that is, the fact that the heat semigroup $P_t$ preserves the space $C_0(M)$ of functions on $M$ that vanish at infinity:
$$
\text{if $u(x) \ra 0$ as $x$ diverges, then, for each $t>0$, $(P_t u)(x) \ra 0$ as $x$ diverges.}
$$
Its equivalence with Definition \ref{def_feller} is shown by R. Azencott, cf. \cite{azencott}. Various authors investigated the geometric conditions needed to guarantee the Feller property, notably \cite{yau2, dodziuk, likarp, hsu, davies} and the recent \cite{pigolasetti_feller}. The most general criteria for its validity are, to the best of our knowledge, the following two. For $G \in C(\R^+_0)$, as usual let $g \in C^2(\R^+_0)$ be the solution of Jacobi equation
\begin{equation}\label{eq_Jacobi_fe}
\left\{ \begin{array}{l}
g'' =Gg = 0 \qquad \text{on } \, \R^+ \\[0.2cm]
g(0)=0, \quad g'(0)=1,
\end{array}\right. \qquad \text{and set} \qquad v_g(r) = \omega_{m-1}g(r)^{m-1}.
\end{equation}
\begin{theorem}
Let $M$ be a complete manifold of dimension $m \ge 2$, fix $o \in M$ and let $r(x) = \mathrm{dist}(x,o)$. Then, $M$ is Feller provided that one of the following properties holds:
\begin{itemize}
\item[(i)] \cite[Thms. 3.4 and 5.9]{pigolasetti_feller} $o$ is a pole, 
$$
K_\rad(x) \le -G\big(r(x)\big) \qquad \text{on } \, M \backslash \{o\},
$$ 
for some $G \in C^\infty(\R^+_0)$, and setting $v_g$ as in \eqref{eq_Jacobi_fe},
\begin{equation}\label{impo_feller}
\text{either} \quad \frac{1}{v_g} \in L^1(\infty), \qquad \text{or} \quad \frac{1}{v_g} \not \in L^1(\infty), \ \ \frac{\int^{\infty}_r v_g}{v_g(r)} \not \in L^1(\infty),
\end{equation}
where the last condition is intended to be trivially satisfied if $v_g \not \in L^1(\infty)$.
\item[(ii)] \cite{hsu, hsu2} the Ricci curvature satisfies 
\begin{equation}
\Ricc \ge - G(r) \metric \qquad \text{on } \, M, 
\end{equation}
for some $G \in C^\infty(\R^+_0)$ matching
\begin{equation}\label{ipo_G_fe}
G >0, \ \ G' \ge 0 \ \ \text{ on $ \, \R^+_0, \ $ and} \qquad \frac{1}{\sqrt{G}} \not \in L^1(\infty).
\end{equation}
\end{itemize}
\end{theorem}

In view of \cite[Thm. 3.4]{pigolasetti_feller}, $(i)$ and $(ii)$ are sharp for the Feller property, and indeed \eqref{impo_feller} is both necessary and sufficient for the model manifold $M_g$ to be Feller. Observe that the inequalities in $(ii)$ coincide with those appearing in \eqref{threshold_SMP}, \eqref{threshold_SMP_2} to guarantee the $\smp$, and that the limit polynomial threshold for both $(i)$ and $(ii)$ is $G(r) \asymp 1+r^2$, that is, $\alpha = 2$. Setting $l \equiv 1$, $b \equiv 1$ and restricting to the Laplace-Beltrami operator, cases $(i)$ and $(ii)$ match, respectively, with the geometric conditions \eqref{arra_CSP_assu} in Theorem \ref{teo_CSP_nuovo} and \eqref{riccibound_intero} in Theorem \ref{teo_CSP_nogradient}. In view of these remarks, we feel interesting to study the following

\begin{problem}
Investigate the validity of the implication 
$$
\text{\eqref{KO_0CSP}} \quad + \quad \text{$\feller$} \quad \Longrightarrow \quad \text{$\csp$}
$$
on a (complete) Riemannian manifold, possibly restricting to the inequality 
$$
\Delta_p u \ge f(u)|\nabla u|^{p-1-\chi}.
$$
Could we obtain, for $\csp$, a result analogous to Theorem \ref{teo_SMPeKO} below?
\end{problem}

\subsubsection{Proof of Proposition \ref{lem_primo_CSP}}
We report the statement to help readability.

\begin{proposition} \label{lem_primo_CSP_appe}
Let $\varphi,f,l$ satisfy \eqref{assu_perCSP} and \eqref{phiel_solozero}, and assume the validity of $(C_1),(C_2)',(C_4)$ and $(\beta_1)', (\beta_2)', (\beta_3)$, for some $r_0>0$. Having fixed a non-negative $\theta \in C([r_0,\infty))$, suppose that
\begin{equation}\label{solu_radialCSP_second_appe}
\disp \limsup_{R \ra \infty} K\left(\frac{1}{R K^{-1}(\beta(2R))}\right)R\theta(R) < \infty.
\end{equation}
Then, there exists a diverging sequence $\{R_j\}$ such that the following holds: if 
\begin{equation}\label{KO_zero_22_appe}\tag{$\mathrm{KO}_0$}
\frac{1}{K^{-1} \circ F} \in L^1(0^+),
\end{equation}
then for each $\epsilon \in (0,\xi_0)$, there exist $\lambda \in (0, \eta_0)$ and, for each $R \in \{R_j\}$, a function $z$ with the following properties:
\begin{equation}\label{proprie_z_CSP_appe}
\left\{ \begin{array}{l}
z \in C^1([R, \infty)), \quad \text{and $C^2$ except possibly at $2R$} \\[0.2cm]
0 \le z \le \lambda, \qquad z(R)= \lambda, \qquad z \equiv 0 \quad \text{on } \, [2R, \infty), \\[0.2cm]
z' <0 \quad \text{on } \, [R, 2R), \qquad |z'| \le \epsilon \quad \text{on } \, [R, \infty), \\[0.2cm]
\big(\varphi(z')\big)'-\theta(t) \varphi(z') \le \epsilon \beta(t) f(z) l(|z'|) \qquad {\rm on} \, [R,\infty)
\end{array}\right.
\end{equation}
\end{proposition}

We first need the next simple result, whose proof is by direct integration.

\begin{lemma} \label{teclemmac2}
If $(C_1)$ and $(C_2)'$ hold, then $(C_2)$ holds, and also
\begin{itemize}
\item[$(K_4)$] $\varphi(st) \le d_1 K'(t)l(t)\varphi(s)=d_1 t \varphi'(t)\varphi(s)\quad$ for each $\, s,t \in (0,1]$
 \end{itemize}
\end{lemma}

\begin{proof}[Proof of Proposition \ref{lem_primo_CSP}]
Because of $(\beta_3)$ and \eqref{solu_radialCSP_second_appe}, we can choose a sequence $\{R_j\}$ satisfying \eqref{condi_R_j}, for some $B_2>0$. Let $\lambda \in (0, \eta_0)$ to be specified later. Using \eqref{KO_zero}, the quantity
$$
C_\lambda = \int_0^\lambda \frac{\di s}{K^{-1}(F(s))}
$$
is well-defined, increasing on $(0, \eta_0)$ and $C_\lambda \downarrow 0$ as $\lambda \ra 0^+$. For each fixed $j \in \N$, we set $R = R_j$ and choose $T = T(R, \lambda)$ small enough that
\begin{equation}\label{Tpiccolo}
T \le R, \qquad \int_{2R-T}^{2R} K^{-1}(\beta(s))\di s \le C_\lambda.
\end{equation}
We also set
\begin{equation}\label{def_ID_CSP}
D= D(\lambda,T,R) = \frac{C_\lambda}{\int_{2R-T}^{2R} K^{-1}(\beta(s))\di s},
\end{equation}
and note that $D \ge 1$. Next, we implicitly define $\alpha : [0,TD] \ra [0, \lambda]$ by the formula
\begin{equation}\label{def_alpha_CSP}
\int_0^{\alpha(s)} \frac{\di \tau}{K^{-1}(F(\tau))} = D \int_{2R-\frac{s}{D}}^{2R} K^{-1}\big(\beta(\tau)\big)\di \tau.
\end{equation}
Then, $\alpha$ is increasing and $\alpha(0)=0$, $\alpha(TD)=\lambda$. Hereafter, the subscript $s$ denotes differentiation in the $s$ variable. From 
\begin{equation}\label{derialpha_CSP}
\alpha_s(s) = K^ {-1}\big(F(\alpha(s))\big) K^{-1}\big(\beta(2R-s/D)\big) >0 \qquad \text{on } \, (0,TD],
\end{equation}
we deduce $\alpha_s(0)=0$ and $0 \le \alpha(s) \le \alpha(DT)=\lambda$. Using $K(0)=0$, we choose $\lambda$ small enough that
\begin{equation}\label{Km1Fle1}
K^ {-1}\big(F(\alpha(s))\big) \le 1 \qquad \text{for } \, s \in [0, DT].
\end{equation}
Furthermore, since $2R-s/D \ge 2R-T \ge R$ in view of  \eqref{Tpiccolo}, by $(\beta_1)'$ we can choose $R_2 \ge R_1$ large enough that
\begin{equation}\label{Km1Fle1}
K^ {-1}\big(\beta(2R-s/D)\big) \le 1 \qquad \text{for } \, s \in [0, DT],
\end{equation}
whence
\begin{equation}\label{boundabovealpha}
|\alpha_s(s)|  \le 1 \qquad \text{for each} \ s \in[0,DT]
\end{equation}.

To simplify the writing set
\begin{equation}\label{def_quantities_CSP}
\tilde \beta(s) = \beta(2R-s/D), \qquad \rho(s) = K^{-1}\big(F(\alpha(s))\big), \qquad \tau(s) = K^{-1}\big(\tilde \beta(s)\big),
\end{equation}
and note that $(\beta_2)'$ can be rewritten as
\begin{equation}\label{beta2tilde}
\frac{\tilde \beta_s(s)}{K^{-1}(\tilde \beta(s))} \le \frac{c_\beta}{D} \tilde \beta(s) \le c_\beta \tilde \beta(s),
\end{equation}
the last inequality being a consequence of $D \ge 1$. Equation \eqref{derialpha_CSP} becomes $\alpha_s = \rho \tau$ and therefore
$$
K(\alpha_s) = K(\rho \tau).
$$
Differentiating this latter and using, in order, \eqref{def_quantities_CSP}, $(C_1)$ together with $(C_2)'$ (and so $(C_2)$ by Lemma \ref{teclemmac2}), $(C_1)$, $(\beta_2)'$, $(C_4)$ and \eqref{derialpha_CSP} we obtain
\begin{equation}\label{calculi}
\begin{array}{lcl}
\disp \big(K(\alpha_s)\big)_s & = & \disp K'(\rho \tau)\left[\frac{f(\alpha) \alpha_s \tau}{K'(\rho)} + \frac{\tilde \beta_s \rho}{K'(\tau)} \right] \\[0.4cm]
& \stackrel{(C_2)} {\le} & \disp k_2 K'(\rho)K'(\tau) \left[\frac{f(\alpha) \alpha_s \tau}{K'(\rho)} + \frac{\tilde \beta_s \rho}{K'(\tau)} \right] \\[0.4cm]
& = & \disp k_2 \left[f(\alpha) \alpha_s \tau K'(\tau) + \tilde \beta_s \rho K'(\rho) \right] \\[0.3cm]
& \stackrel{(C_1)} {\le} & \disp k_2k_1 \left[f(\alpha) \alpha_s K(\tau) + \tilde \beta_s K(\rho) \right] \\[0.3cm]
& = & \disp k_2k_1\left[f(\alpha) \alpha_s \tilde \beta + \tilde \beta_s F(\alpha) \right] \\[0.3cm]
& \stackrel{(\beta_2)'} {\le} & \disp k_2k_1 \left[f(\alpha) \alpha_s \tilde \beta + c_\beta \tilde \beta K^{-1}(\tilde \beta) F(\alpha) \right] \\[0.3cm]
& \stackrel{(C_4)} {\le} & \disp k_2k_1 \left[f(\alpha) \alpha_s \tilde \beta + c_F c_\beta \tilde \beta K^{-1}(\tilde \beta) f(\alpha)K^{-1}(F(\alpha)) \right] \\[0.3cm]
& \stackrel{\eqref{derialpha_CSP}} {=} & \disp k_2k_1 \big[1+c_F c_\beta\big]\tilde \beta f(\alpha)\alpha_s \qquad \text{for each } \, s \in (0, DT).
\end{array}
\end{equation}
Therefore, differentiating $K$ we deduce
$$
\frac{\alpha_s \varphi_s(\alpha_s)}{l(\alpha_s)}\alpha_{ss} = \big(K(\alpha_s)\big)_s \le k_2k_1 \big[1+c_F c_\beta\big]\tilde \beta f(\alpha)\alpha_s,
$$
and since $\alpha_s>0$ and $T \le R$, by the monotonicity of $\tilde \beta$ we obtain
\begin{equation}\label{varphi_alphaniania_CSP}
\begin{array}{lcl}
\disp \big(\varphi(\alpha_s)\big)_s & \le & \disp k_2k_1 \big[1+c_F c_\beta\big]\tilde \beta f(\alpha)l(\alpha_s)  \\[0.3cm]
& \le & k_2k_1 \big[1+c_F c_\beta\big]\tilde \beta\left(\frac{Rs}{T}\right) f(\alpha)l(\alpha_s) \qquad \text{on } \, (0, DT).
\end{array}
\end{equation}
We also note that $\big(K(\alpha_s)\big)_s\ge 0$ follows from the first line in \eqref{calculi} and the fact that $\tilde \beta_s \ge 0$, hence $\alpha_{ss} \ge 0$. Integrating \eqref{varphi_alphaniania_CSP} on $(0,s]$, $s \in (0,DT]$ and using $K^ {-1}(0)=0$, the monotonicity of $\tilde \beta_s$ and $\alpha_s$, \eqref{assu_perCSP} and $T \le R$ we get
\begin{equation}\label{varphi_alphania_CSP_1}
\disp \varphi\big(\alpha_s(s)\big) \le \disp k_2k_1 \big[1+c_F c_\beta\big]C^2 TD \tilde \beta\left(\frac{Rs}{T}\right)f\big(\alpha(s)\big)l\big(\alpha_s(s)\big)
\end{equation}
Next, we define the function $z : [R, \infty) \ra [0, \lambda)$ by setting
\begin{equation}\label{def_z_CSP}
z(t) = \left\{ \begin{array}{ll}
\alpha(s), \quad s= DT \left(2 - \frac{t}{R}\right) & \quad \text{if } \, t \in [R, 2R]; \\[0.3cm]
0 & \quad \text{if } \, t > 2R.
\end{array}\right.
\end{equation}
Then, $z \in C^1([R, \infty))$ (actually, $C^2$ with a possible exception at $t=2R$) and $z$ is non-increasing. Furthermore,
\begin{equation}\label{prop_Zbound}
z(R) = \alpha(DT)=\lambda, \qquad z(2R) = z'(2R) = 0, \qquad z'(t) = -\frac{DT}{R} \alpha_s(s).
\end{equation}
We pause for a moment to estimate the quotient $DT/R$. By definition, and since $\beta$ is decreasing,
\begin{equation}\label{pallino}
\frac{DT}{R} = \frac{C_\lambda T}{R\int_{2R-T}^{2R} K^{-1}(\beta(\tau))\di \tau} \le \frac{2C_\lambda}{2RK^{-1}(\beta(2R))}.
\end{equation}
Using $(\beta_3)$ and recalling that $R= R_j$ satisfy \eqref{condi_R_j},
$$
\frac{-2R \beta'(2R)}{\beta(2R)} \ge \frac{\hat{c}_\beta}{2},
$$
whence, applying $(\beta_2)'$,
\begin{equation}\label{rel_K}
\frac{1}{2RK^{-1}(\beta(2R))} \le \frac{c_\beta \beta(2R)}{-2R\beta'(2R)} \le \frac{2 c_\beta}{\hat{c}_\beta},
\end{equation}
and inserting into \eqref{pallino},
\begin{equation}\label{rel_impo_K}
\frac{DT}{R} \le \frac{4 c_\beta C_\lambda}{\hat{c}_\beta}.
\end{equation}
Up to reducing $\lambda$ further, we can guarantee that $DT/R \le \epsilon$, and consequently by \eqref{boundabovealpha} and \eqref{prop_Zbound} 
$$
|z'| \le \epsilon.
$$  
This shows the third relation in \eqref{proprie_z_CSP}. Next, since $|\alpha_s| \le 1$ on $(0, DT)$, applying $(K_4)$ in Lemma \ref{teclemmac2}, \eqref{varphi_alphania_CSP_1}, $(C_2)'$ and since 
$$
\disp\tilde{\beta}\left( \frac{R}{T}s\right)= \tilde{\beta}(2DR-Dt)=\beta(t)
$$ 
we infer
\begin{equation}\label{parte_1}
\begin{array}{rcl}
- \varphi \big(z'(t)\big) & = & \varphi\big(-z'(t)\big) = \varphi\left(\frac{DT}{R}\alpha_s(s)\right) \\[0.3cm]
& \stackrel{(K_4)}{\le} & d_1 K'\left(\frac{DT}{R}\right) l\left(\frac{DT}{R}\right) \varphi\big(\alpha_s(s)\big) \\[0.3cm]
& \stackrel{\eqref{varphi_alphania_CSP_1}}{\le} & d_1k_1k_2\left[1 + c_F c_\beta\right]C^2 DT K'\left(\frac{DT}{R}\right) l\left(\frac{DT}{R}\right) \tilde \beta\left(\frac{Rs}{T}\right) f\big(\alpha(s)\big) l\big(\alpha_s(s)\big) \\[0.3cm]
&\stackrel{(C_2)'}{\le} & c_1d_1k_1k_2\left[1 + c_F c_\beta\right]C^2 DT K'\left(\frac{DT}{R}\right) \beta(t) f\big(\alpha(s)\big) l\big(\frac{DT}{R}\alpha_s(s)\big) \\[0.3cm]
& = & c_1d_1k_1k_2\left[1 + c_F c_\beta\right]C^2 DT K'\left(\frac{DT}{R}\right) \beta(t) f\big(z(t)\big) l\big(|z'(t)|\big) \\[0.3cm]
& \le & c_1d_1k_1^2k_2\left[1 + c_F c_\beta\right]C^2 R K\left(\frac{DT}{R}\right) \beta(t) f\big(z(t)\big) l\big(|z'(t)|\big).
\end{array}
\end{equation}
We next investigate $(\varphi(z'))'$. By definition, and because of $(C_2)'$, $\alpha_{ss} \ge 0$,  \eqref{varphi_alphaniania_CSP} and $(C_1)$, we obtain
\begin{equation}\label{parte_2}
\begin{array}{lcl}
\big(\varphi(z')\big)' & = & \disp \varphi'\left(\frac{DT}{R}\alpha_s(s)\right) \frac{D^2T^2}{R^2} \alpha_{ss}(s) \stackrel{(C_2)'}{\le} d_1 \varphi'\left( \frac{DT}{R}\right) \varphi'(\alpha_s(s)) \frac{D^2T^2}{R^2} \alpha_{ss}(s) \\[0.3cm]
& = & \disp d_1 \varphi'\left( \frac{DT}{R}\right)\frac{D^2T^2}{R^2}  \big(\varphi(\alpha_s(s))\big)_s \\[0.3cm]
& \stackrel{\eqref{varphi_alphaniania_CSP}}{\le} & \disp k_2k_1 \big[1+c_F c_\beta\big]d_1 \varphi'\left( \frac{DT}{R}\right)\frac{D^2T^2}{R^2} \beta(t) f\big(\alpha(s)\big)l\big(\alpha_s(s)\big) \\[0.3cm]
& = & \disp k_2k_1 \big[1+c_F c_\beta\big]d_1 \varphi'\left( \frac{DT}{R}\right)\frac{D^2T^2}{R^2} \beta(t) f\big(z(t)\big)l\left(\frac{R}{DT}|z'(t)|\right) \\[0.3cm]
& = & \disp k_2k_1 \big[1+c_F c_\beta\big]d_1 \frac{DT}{R}K'\left(\frac{DT}{R}\right)l\left(\frac{DT}{R}\right) \beta(t) f\big(z(t)\big)l\left(\frac{R}{DT}|z'(t)|\right) \\[0.3cm]
& \stackrel{(C_1) \ {\rm and} \ (C_2)'}{\le} & \disp c_1k_2k_1^2 \big[1+c_F c_\beta\big]d_1 K\left(\frac{DT}{R}\right) \beta(t) f\big(z(t)\big)l\big(|z'(t)|\big),
\end{array}
\end{equation}
%
Combining \eqref{parte_1} and \eqref{parte_2} and using $z' \le 0$ we get on $ t \in [R,2R]$
\begin{equation}\label{ending}
\begin{array}{l}
 \disp \big(\varphi(z')\big)' -   \theta(t)\varphi(z')  = \big(\varphi(z')\big)' + \theta(t) \varphi(-z')  \\[0.3cm]
\le  \disp \disp c_1d_1k_1^ 2k_2\big[1+c_F c_\beta\big] K\left(\frac{DT}{R}\right)  \left[1 + C^2 R\theta(R) \right] \beta(t) f(z) l\big(|z'|\big).
\end{array}
\end{equation}
Observe now that properties $(C_1)$ and $(C_2)'$ (hence, $(C_2)$ by Lemma \ref{teclemmac2}) guarantee the validity of $(K_1)$ in Lemma \ref{lem_furtherprop}. Possibly reducing $\lambda$ in such a way that  
$$
\max \left(\sqrt{C_\lambda}, K\left(\frac{\sqrt{C_\lambda} 4 c_\beta}{\hat{c}_\beta}\right) \right) \le 1
$$
and using \eqref{pallino}, we get
$$
\begin{array}{lcl}
\disp K\left(\frac{DT}{R}\right) & \le & \disp K\left( \frac{C_\lambda}{RK^{-1}(\beta(2R))}\right) \\[0.6cm]
& \stackrel{(K_1)}{\le} & \disp k_1 k_2 K(\sqrt{C_\lambda}) K\left( \frac{\sqrt{C_\lambda}}{R K^{-1}(\beta(2R))}\right)  \\[0.6cm]
& \le & \disp k_1 k_2 K(\sqrt{C_\lambda}) \min \left[ K\left( \frac{1}{R K^{-1}(\beta(2R))} \right),1 \right].
\end{array}
$$
Inserting into \eqref{ending}, using $C \ge 1$ and \eqref{solu_radialCSP_second_appe},
\begin{equation}\label{ending_2}
\begin{array}{l}
\disp \big(\varphi(z')\big)' -   \theta(t)\varphi(z')   \\[0.5cm]
\le \disp c_1d_1k_1^ 3k_2^2 \big[1+c_F c_\beta\big]C^2 K\Big(\sqrt{C_\lambda} \Big) \left[1+K\left( \frac{1}{R K^{-1}(\beta(2R))}\right) R\theta(R)\right] \beta(t) f(z) l\big(|z'|\big) \\[0.5cm]
 \le  \disp (1+B_2)c_1d_1k_1^ 3k_2^2  \big[1+c_F c_\beta\big] C^2 K\Big(\sqrt{C_\lambda} \Big) \beta(t) f(z) l\big(|z'|\big) \quad {\rm for } \ t \ge R.
 \end{array}
\end{equation}
With a possible smaller choice of $\lambda$, still independent of $R$, we can ensure that
$$
(1+B_2)c_1d_1k_1^3k_2^2 \big[1+c_F c_\beta\big]C^2 K(\sqrt{C_\lambda}) \le \epsilon,
$$
concluding the proof.
\end{proof}

\section{Keller-Osserman, a-priori estimates and $\slio$} \label{sec_SL}

In this section, we relate the Keller-Osserman condition
\begin{equation}\label{KO_2}\tag{$\mathrm{KO}_\infty$}
\frac{1}{K^{-1}\circ F} \in L^1(\infty)
\end{equation}
to the strong Liouville property $\slio$ for solutions of~$(P_\ge)$. It is particularly interesting to see how geometry comes into play via the validity of the weak or the strong maximum principle for $(bl)^{-1} \Delta_\varphi$. Hereafter, we require

\begin{equation}\label{assumptions_SL_necessity}
\left\{\begin{array}{l}
\varphi \in C(\R^+_0) \cap C^1(\R^+), \qquad \varphi(0)=0, \qquad \varphi'>0 \ \text{ on } \, \R^+, \\[0.3cm]
l \in C(\R^+_0), \qquad l>0 \ \text{on } \, \R^+, \\[0.3cm]
f \in C(\R), \\[0.3cm]
\text{$f>0\ \ $ and $C$-increasing on $(\bar \eta_0, \infty)$, for some $\bar \eta_0 \ge 0$,} \\[0.3cm]
\end{array}\right.
\end{equation}
and moreover
\begin{equation}\label{phiel_solozero_SL}
\frac{t \varphi'(t)}{l(t)} \in L^1(0^+) \backslash L^ 1(\infty).
\end{equation}
Having defined $K$ as in \eqref{def_K}, by \eqref{phiel_solozero_SL} $K$ is a homeomorphism of $\R^+_0$ onto itself, thus \eqref{KO} is well defined with
\begin{equation}\label{def_F_SL}
F(t) = \int_{\bar \eta_0}^t f(s) \di s.
\end{equation}
Note that, for the mean curvature operator, if $l \equiv 1$ then $K$ is not surjective on $\R^+$, that is, \eqref{phiel_solozero_SL} does not hold. Indeed, for operators of mean curvature type one is able to guarantee property $\slio$ \emph{without} the need of the Keller-Osserman condition, at least in some instances. In this respect, the following result of Y. Naito and H. Usami \cite{nu} is illustrative\footnote{The statement reported here is slightly different from the original one in Theorems 1,2,3 of \cite{nu}. However, the two are equivalent in view of Lemma \ref{lem_mettimaosigmatau_novo}. Moreover, their notion of solution needs the further condition $|\nabla u|^{-1}\varphi(|\nabla u|)\nabla u \in C^1(\R^m)$, and $\Delta_\varphi u \ge f(u)$ is meant in the pointwise sense.}:

\begin{theorem}[\cite{nu}, Thms. 1,2,3]
Let $\varphi,f$ satisfy \eqref{assumptions_SL_necessity} and 
$$
f(0)=0, \qquad f>0 \quad \text{and non decreasing on } \, \R^+.
$$
Consider a non-negative solution $u \in C^1(\R^m)$ of $\Delta_\varphi u \ge f(u)$ on $\R^m$. 
\begin{itemize}
\item[(i)] If $\varphi(\infty) < \infty$, then $u \equiv 0$ on $\R^m$.
\item[(ii)] If $\varphi(\infty) = \infty$, then the only non-negative solution is $u \equiv 0$ if and only if \eqref{KO} holds.
\end{itemize}
\end{theorem}
\begin{remark}
\emph{As observed in \cite{nu}, when $l \equiv 1$ condition $\varphi(\infty)=\infty$ implies $K_\infty = \infty$ and thus \eqref{KO} is meaningful. This follows from the next inequalities: 
$$
\begin{array}{lcl}
\disp K(t) + \int_0^1 \varphi(s)\di s & = & \disp \int_0^t s \varphi'(s) \di s + \int_0^1\varphi(s) \di s  \\[0.4cm]
& = & \disp t \varphi(t) - \int_1^t \varphi(s)\di s \ge \varphi(t).
\end{array}
$$
In fact, in the Appendix of \cite{nu} it is also proved that if $t \varphi(t) \asymp t^p$ as $t \ra \infty$, for some $p>1$, then also $K(t) \asymp t^{p}$. 
}
\end{remark}
In the last subsection, we will discuss in detail the case of mean curvature type operators, for which we describe appropriate Keller-Osserman conditions (that are necessary in some cases!) for the validity of $\slio$. First, we focus on those operators for which \eqref{phiel_solozero_SL} holds, and begin with considering the implication $\slio \Rightarrow \eqref{KO_2}$.

\subsection{Necessity of \eqref{KO} for the $\slio$ property}
The main result of this section is Theorem \ref{teo_SL_necessary} below: under the failure of \eqref{KO},  we exhibit a non-constant, non-negative solution $u \in C^1(M)$ of $(P_\ge)$ on any complete manifold with a pole $o$ and satisfying the mild curvature restriction
\begin{equation}\label{ipo_Krad_SL}
K_\rad (x) \le -G\big(r(x)\big) \qquad \text{on } \, M \backslash \{o\},
\end{equation}
for some $G \in C(\R^+_0)$ matching 
\begin{equation}\label{Kneser}
t \int_t^{\infty} G_-(s) \di s \le \frac{1}{4} \qquad \forall \, t \in \R^+ 
\end{equation}
with $G_- = - \min\{G,0\}$.

\begin{remark}\label{rem_kneser}
\emph{Condition \eqref{Kneser} is met, for instance, by any $G \in C(\R^+_0)$ satisfying $G(t) \ge -(4t^2)^{-1}$. Therefore,   
\begin{equation}\label{Kneser_ex}
K_\rad(x) \le \frac{1}{4r(x)^2} \qquad \text{on } \, M \backslash \{o\}
\end{equation}
is sufficient for the validity of \eqref{ipo_Krad_SL} and \eqref{Kneser}. Note that \eqref{Kneser_ex} includes both the Euclidean and the hyperbolic spaces, as well as models with a mild positive curvature that, at infinity, open like paraboloids. 
}
\end{remark}

It is worth to stress that \eqref{ipo_Krad_SL} and \eqref{Kneser} are only used to guarantee that the model to be compared to $M$ is defined on the entire $\R^n$, see Remarks \ref{rem_kneser} and \ref{rem_kneser_2}, and seems somehow to be merely technical (although, we believe, challenging to remove). Loosely speaking, this would suggest that there is no geometric obstruction to ensure that \eqref{KO_2} be necessary for $\slio$.  

\begin{theorem}\label{teo_SL_necessary}
Let $M^m$ be a complete Riemannian manifold with a pole $o \in M$, and assume \eqref{ipo_Krad_SL} for some $G \in C(\R^+_0)$ enjoying \eqref{Kneser}. Let $\varphi, f, l$ satisfy \eqref{assumptions_SL_necessity} and \eqref{phiel_solozero_SL}, and 
$$
f(0)=0, \qquad f>0 \qquad \text{on } \, \R^+.
$$
If
\begin{equation}\tag{$\neg$KO$_\infty$}
\frac{1}{K^ {-1} \circ F} \not \in L^1(\infty), 
\end{equation}
Then, for each $b \in C(M)$, $b>0$ on $M$ there exists a non-constant, non-negative $u \in C^1(M)$ satisfying $(P_\ge)$, that is,
$$
\Delta_\varphi u \ge b(x) f(u) l(|\nabla u|) \qquad \text{on } \, M.
$$
In particular, $\slio$ does not hold on $M$.
\end{theorem} 
 
\begin{proof}
Let $g \in C^2(\R^+_0)$ be the solution of
\begin{equation}\label{def_g_Jaco_SL}
\left\{ \begin{array}{l}
g'' - Gg = 0 \quad \text{on } \, \R^+, \\[0.2cm]
g(0)>0, \quad g'(0)=1. 
\end{array}\right.
\end{equation}
Assumption \eqref{Kneser} guarantees that $g>0$ and $g'>0$ on $\R^+$, see \cite[Prop. 1.21]{bmr2}. Let $\wp(r) = \vol(\Sph^{m-1})g(r)^{m-1}$ be the volume growth of spheres of the model $M_g$. Define $a \in C(\R^+_0)$ and $\bar l \in C(\R^+_0)$ in such a way that
$$
\begin{array}{l}
\disp b(x) \le a\big(r(x)\big) \qquad \forall \, x \in M, \\[0.2cm]
\disp \bar l \ge l \quad \text{on } \, \R^+_0, \qquad \bar l(0)>0, \qquad \bar l = l \quad \text{on } \, [1,\infty).
\end{array}
$$
Then, all of the assumptions in Proposition \ref{prop_Rinfty} are satisfied with $\bar l$  replacing $l$, and there exists a non-constant function $w \in C^1(\R^+_0)$ solving
$$
\left\{ \begin{array}{l}
\big[\wp \varphi(w')\big]' = a \wp f(w)\bar l(|w'|) \qquad \text{on } \, \R^+, \\[0.2cm]
w'(0)=0, \quad w \ge 0, \quad w' \ge 0 \quad \text{on } \, \R^+, \\[0.2cm]
\end{array} \right.
$$ 
Set $u(x) = w\big(r(x)\big)$. By the Laplacian comparison theorem from below (i.e. taking traces in \eqref{lowerbound_hess} of Theorem \ref{teo_hessiancomp}), and using that $\bar l \ge l$, $w'\ge 0$ and $\varphi(w') \in C^1(\R^+_0)$, we have
\begin{equation}\label{eq_lowerb_neu}
\begin{array}{lcl}
\Delta_\varphi u & = & \disp \big[\varphi(w')\big]' + \varphi(w')\Delta r \ge \big[\varphi(w')\big]' + \varphi(w')\frac{\wp'}{\wp} = \wp^{-1} \big[ \wp \varphi(w')\big]' \\[0.3cm]
& = & a f(w) \bar l(|w'|) \ge b(x) f(u) l(|\nabla u|).
\end{array} 
\end{equation}
Since $o$ is a pole and $w'(0)=0$, $u \in C^1(M)$ and provides the desired solution. 
\end{proof}

\begin{remark}\label{rem_kneser_2}
\emph{Evidently, \eqref{Kneser} can be replaced by the only requirement that the solution $g$ of the Jacobi equation \eqref{def_g_Jaco_SL} is positive and non-decreasing on $\R^+$.
}
\end{remark}

The use of the mixed Dirichlet-Neumann problem to prove Theorem \ref{teo_SL_necessary} is inspired by the very recent \cite{bordofilipucci}: in Theorem 1.1 therein, the authors prove existence under \eqref{notKO} in the setting of the Heisenberg group, for each $l \in C(\R^+_0)$ with $l >0$ on $\R^+_0$. In particular, no $C$-monotonicity of $l$ is needed. To the best of our knowledge, in the literature the existence of entire solutions of $(P_\ge)$ under the failure of \eqref{KO} has been shown just in few further special cases, see for instance \cite[p. 694-696]{maririgolisetti}, \cite[Thm. 1.3]{mmmr}, \cite[Cor. 1.1]{fprgrad}. Differently from \cite{bordofilipucci}, in the constructions in \cite{maririgolisetti, mmmr, fprgrad} the radial function $u$ is defined implicitly by a direct use of \eqref{KO}, in a way analogous to that in Proposition \ref{prop_CSP_radial}. This method needs various structural assumptions on $\varphi, b, l$, that considerably restrict the range of the operators. On the contrary, the use of the Dirichlet-Neumann problem allows to make a clean, simpler proof of the existence of solutions of $(P_\ge)$ and, at the same time, to remove the unnecessary conditions on $\varphi, f, l$: in particular, neither we assume a controlled growth of $b$ nor the $C$-monotonicity (increasing or decreasing) of $l$. However, the presence of a pole, intimately related to the use of the comparison theorem from below, is unavoidable for our method to work, as well as for those in the above references. It would be interesting to investigate the following


\begin{question}
Can one prove the necessity of \eqref{KO} for property $\slio$ on a general complete manifold, at least for some classes of $\varphi$ and $l$?  
\end{question}
Note that, in the $p$-Laplacian case, a direct use of the fake distance $\varrho$ as in the proof of Theorem \ref{teo_CSP_plapla_intro} is not enough to conclude. Indeed, from \eqref{bella!!!} and taking into account \eqref{eq_lowerb_neu}, for $u$ to solve $(P_\ge)$ we need a \emph{global lower bound} for $|\nabla \varrho|$. Despite the fact that lower bounds for $|\nabla \varrho|$ seem very difficult to achieve, their existence coupled with the properness of $\varrho$ would still force, by Morse theory, topological restrictions on $M$.

\subsection{Sufficiency of \eqref{KO_2} for the $\slio$ property}
\label{subsec_SL_suffi}

The investigation of the sufficiency of \eqref{KO_2} for $\slio$ in a manifold setting began with the pioneering \cite{chengyau, yau2}, for the prototype semilinear example $\Delta u \ge f(u)$. There, geometry is taken into account via a constant lower bound on the Ricci tensor of $M$. The Liouville theorems therein proved to be remarkably effective in a wealth of different geometric problems. Among them, we stress a striking proof of the generalized Schwarz Lemma for maps between Kahler manifolds in \cite{yau_schwarz}, and the Bernstein theorem for maximal hypersurfaces in Minkovski space in \cite{chengyau_minkovski}. Since then, various authors studied possible useful generalizations, notably \cite{motomiya} for the inequality
$$
\Delta u \ge \varphi(u, |\nabla u|).
$$
The topic has first been considered from a general perspective in \cite{prsmemoirs}, that also contains a detailed account of the previous literature, and later more specifically in \cite{maririgolisetti} for quasilinear equations including $(P_\ge)$ with non-constant $l$. As usual, the geometric requirements range from a control on the Ricci to a growth estimate for the volume of geodesic balls. In the next subsections, we will describe improvements of the results therein, as well as new theorems, and discuss  their sharpness. Of particular interest for us is the case of mean curvature type operators, for which interesting specific phenomena appear. Typically, but not exclusively, we will consider a gradient nonlinearity of the type
$$
l(t) \asymp \frac{\varphi(t)}{t^\chi},
$$
that, for the mean curvature operator, vanishes both as $t \ra 0$ and as $t \ra \infty$ when $\chi \in (0,1)$. When $l$ is allowed to vanish both at $t=0$ and at $t=\infty$ we cannot rely on the existing literature, because all of the results that we know require a $C$-monotonicity of $l$, either increasing (cf. \cite{maririgolisetti}) or decreasing (cf. \cite{bordofilipucci}).\par
We begin with the following result that considers homogeneous operators and a power-like gradient dependence. In this case, we can give a very simple proof of the next implication:
\begin{equation}\label{smpekosl}
\smp \quad + \quad \eqref{KO} \ \ \Longrightarrow \ \ {\slio}.
\end{equation}
The argument naturally splits into two steps. First, the combination of $\smp$ and \eqref{KO} guarantees that each solution of~$(P_\ge)$ is, in fact, bounded from above; in the second step we are left  to prove the validity of~$\lio$.

\begin{theorem}\label{teo_SMPeKO}
Let $M$ be a Riemannian manifold. Fix $p>1$, $\chi \in (-1,p-1]$, and let $f \in C(\R)$ satisfying 
\begin{equation}\label{technical}
\left\{\begin{array}{l}
f>0 \quad \text{on } \, (\bar \eta_0,\infty) \\[0.2cm]
F(t)^{\frac{\chi}{\chi+1}} \le c_F f(t), \qquad \text{on $(\bar \eta_0,\infty)$, for some constant $c_F>0$.} 
\end{array}\right.
\end{equation}
Let $u \in C^1(M)$ solve $\Delta_p u \ge f(u)|\nabla u|^{p-1-\chi}$ on $M$. If
\begin{itemize}
\item[1.] $l^{-1}\Delta_p$ satisfies $\smp$ with $l(t) = t^{p-1-\chi}$, and
\item[2.] the Keller-Osserman condition \eqref{KO} holds, that is, 
$$
F(t)^{-\frac{1}{\chi+1}} \in L^1(\infty),
$$
\end{itemize}
then
\begin{equation}\label{beautiful}
u^* = \sup_M u  < \infty \qquad \text{and} \qquad f(u^*) \le 0.
\end{equation}
\end{theorem}
\begin{remark}
\emph{Note that the second in \eqref{technical} is the equivalent, at infinity, of condition $(C_4)$, repeatedly used in the proof of the compact support principle. It is easy to see that the condition is, for general $f$, unrelated to \eqref{KO}.
}
\end{remark}
\begin{proof}
In our assumptions, the function $K$ defined in \eqref{def_K} satisfies $K(t)\asymp t^{\chi+1}$ as $t \ra \infty$, thus $K_\infty = \infty$ and \eqref{KO} is meaningful. Let $g\in C^2(\R)$ be such that
$$
g' > 0 \quad \text{on } \, \R, \qquad g(t) = \int_{\bar \eta_0}^t \frac{\di s}{K^{-1}(F(s))} \qquad \text{for } \, t\ge \bar \eta_0+1.
$$
Suppose by contradiction that $u^*=\infty$, so that $\Omega_{\bar \eta_0+1} = \{x\in M : u(x)>\bar \eta_0+1\}$ is non-empty. Set $h(x)=g(u(x))$.  Then, $h \in C^1(M)$ and $h^*< \infty$ because of  \eqref{KO}. Computations show that, on $\Omega_{\bar \eta_0+1}$,
\begin{equation}\label{simpleeq}
\begin{aligned}
\disp \nabla h & = \disp g'(u) \nabla u = \frac{\nabla u}{K^{-1}(F(u))} \\
\disp \Delta_p h & = \disp (p-1)(g')^{p-2}g''|\nabla u|^p + (g')^{p-1}\Delta_p u \\
& \ge  \left[-|\nabla h|^p K^{-1}(F(u))^{p-1-\chi} + |\nabla u|^{p-1-\chi}\right] \frac{f(u)}{K^{-1}(F(u))^{p-1}} \\
&= \left[1-|\nabla h|^{\chi+1}\right] \frac{f(u)|\nabla h|^{p-1-\chi}}{K^{-1}(F(u))^{\chi}}
\end{aligned}
\end{equation}
in the weak sense.
Next, by \eqref{technical}, 
\begin{equation}\label{almost_done}
\frac{f(u)}{K^{-1}(F(u))^{\chi}} \ge \hat c_F >0 \qquad \text{on } \, (\bar \eta_0,\infty),
\end{equation}
for some constant $\hat c_F>0$. Then, for each $\eps \in (0,1)$ on the open, non-empty set 
$$
\Omega_{\eta, \eps}= \big\{x \in M \ : \ h(x)>\eta \ {\rm and} \ |\nabla h(x)|< \eps\big\}
$$
the function $h$ solves
$$
\Delta_p h \ge \hat c_F(1-\eps^p)|\nabla h|^{p-1-\chi},
$$
contradicting the validity of $\smp$ for $l^{-1}\Delta_p$. Therefore, $u^*<\infty$ and, since $\smp$ is in force, we can apply Proposition \ref{prop_equivalence_2} to deduce that $f(u^*) \le 0$. This concludes the proof.
\end{proof}


\begin{corollary}\label{cor_SMPeKO}
In the assumptions of Theorem \ref{teo_SMPeKO}, $\slio$ holds for each $f>0$ on $\R^+$ satisfying \eqref{technical}.
\end{corollary}


\begin{remark}
\emph{Theorem \ref{teo_SMPeKO} should be compared with \cite[Thm. 1.31]{prsmemoirs}, that improves on previous results of Cheng and Yau \cite{chengyau} and Motomiya \cite{motomiya}. In \cite{prsmemoirs}, the authors consider solutions of 
$$
\Delta u \ge \varphi(u, |\nabla u|) \qquad \text{on } \, M
$$
under suitable assumptions on $\varphi$ that are skew with those in Theorem \ref{teo_SMPeKO}, and infer the bound $f(u^*) \le 0$ under the validity of $\smp$ for the Laplace-Beltrami operator $\Delta$.
}
\end{remark}

\begin{remark}
\rm{We stress that Theorem \ref{teo_SMPeKO} also holds for $l(t)=t^{p-1-\chi}$ when $\chi \in (-1,0)$. However, this range is not included in Theorem \ref{teo_main_2} and, indeed, currently we do not know which geometric conditions ensure $\smp$ or even $\wmp$ for $l^{-1}\Delta_p$ and negative $\chi$. For further related comments, we refer to Remark \ref{rem_chiminorzero} below.
}
\end{remark}

As a prototype example of applicability of the above theorem, we give a quick proof of the following classical result. The first part of Theorem~\ref{teo_osserman} below is due to R. Osserman who introduced \eqref{KO}, as we have mentioned in the introduction, to prove this result; the second is a restatement of the classical Schwarz lemma for complex analysis.

\begin{theorem}[\cite{osserman}]\label{teo_osserman}
Let $(M,g)$ be a complete, non-compact, simply connected Riemann surface whose sectional curvature satisfies $K(x) \le -1$. Then, $M$ is conformally equivalent to the Poincar\'e disk $(\mathbb{D}, g_{\HH})$ via some conformal diffeomorphism $\varphi: \mathbb{D} \ra M$ satisfying $\varphi^*g \le g_{\HH}$.
\end{theorem}
\begin{proof}
We recall that, if $g = e^{2u} h$ is a conformal deformation of a metric $h$ on a surface $M$, then $u$ turns out to satisfy the Yamabe equation
$$
\Delta u = -K_ge^{2u} + K_h,
$$
where $K_g$ and $K_h$ are the Gaussian curvatures of~$g$ and $h$, and $\Delta$ is the Laplace-Beltrami operator of~$h$. By the Riemann-K\"obe uniformization theorem, $M$ is conformal either to the Euclidean plane $\R^2$ with its flat metric $g_\R$, or to the Poincar\'e disk $(\mathbb{D},g_\HH)$ with the hyperbolic metric $g_\HH$ of, say, sectional curvature $-1$. Both on $\R^2$ and on $\mathbb{D}$, the operator $\Delta$ satisfies $\smp$ because of Theorem \ref{teo_SMP}. Suppose by contradiction that $g = e^{2u} g_\R$.  Then, $u$ satisfies
\begin{equation}\label{yama_eu}
\Delta u = -K_g(x)e^{2u}\ge e^{2u} \qquad \text{on } \, \R^2.
\end{equation}
We therefore apply  Theorem~\ref{teo_SMPeKO} with $f(t)= e^{2t}$ to deduce that $u$ is bounded above and $f(u^*) \le 0$, that is clearly impossible. Hence, $M$ is conformally the Poincar\'e disk and setting $g= e^{2u}g_\HH$, from $K_{g_\HH}=-1$ and $K_g \le -1$ the function $u$ satisfies
\begin{equation}\label{yama_hyp}
\Delta u = -K_g e^{2u} -1 \ge e^{2u}-1.
\end{equation}
Setting $f(t) = e^{2t}-1$, $f$ still satisfies \eqref{KO_infinity_intro} and \eqref{technical}. Furthermore, on the Poincar\'e disk $\Delta$ satisfies $\smp$, and again by Theorem~\ref{teo_SMPeKO} we get $f(u^*)\le 0$, that is $u^* \le 0$. Therefore, $g = e^{2u}g_{\HH} \le g_{\HH}$.
\end{proof}
We pause for a moment to comment on the necessity to require $\smp$ in Theorem~\ref{teo_SMPeKO}. The validity of~$\lio$ for $f>0$ on $\R^+$ is granted under the sole assumption $\wmp$ by Proposition~\ref{prop_equivalence_2}. Thus, one might wonder whether the implication
$$
\smp \quad + \quad \eqref{KO} \ \ \Longrightarrow \ \ \slio
$$
for $f>0$ on $\R^+$ satisfying \eqref{technical}, could be improved to
$$
\wmp \quad + \quad \eqref{KO} \ \ \Longrightarrow \ \ \slio.
$$
This amounts to showing that the combination $\wmp + \eqref{KO}$ guarantees global $L^\infty$-estimates for solutions of~$\Delta_p u \ge f(u)$. This is generally false, as the following example shows.

\begin{example}\label{ex_importante}
{\rm Consider the punctured Euclidean space $\R^m\backslash \{0\}$ with its flat metric. Then, it is easy to see that $\Delta$ satisfies $\wmp$. Indeed, define a function $w \in C^2(\R^m \backslash \{0\})$ as follows:
$$
\begin{array}{ll}
w(x) = -\log|x| + |x|^2 & \quad \text{if } \, m=2, \\[0.2cm]
w(x) = |x|^{2-m} + |x|^2 & \quad \text{if } \, m \ge 3.
\end{array}
$$
then $w$ is an exhaustion on $\R^m \backslash \{0\}$, i.e., it has relatively compact sublevel sets, and $\Delta w = 2m \le w$ outside a compact set. The function $w$ is therefore a good Khasminskii potential, whose existence implies $\wmp$, see for instance \cite{grigoryan} and \cite[Prop. 3.2]{prsmemoirs}. In fact, the existence of such a $w$ is equivalent to $\wmp$, cf. \cite{marivaltorta, maripessoa} for details. Now, consider
$$
\sigma \in \left(1, \frac{m+2}{m}\right), \quad \beta = \frac{2}{\sigma-1}, \quad u(x) = |x|^{-\beta}.
$$
Then, a computation shows that
$$
\Delta u = \beta(\beta-m) u^\sigma.
$$
Hence,  $\Delta u \ge cu^\sigma$ for some constant $c>0$ in our range on $\sigma$. On the other hand, it is easy to check  that the function $f(t)=ct^\sigma$ satisfies $f>0$ on $\R^+$ together with \eqref{technical} and \eqref{KO_2}. Therefore, $\wmp$ is not enough to conclude $\slio$ even for functions  matching \eqref{KO}.}
\end{example}
%

\begin{remark}
\emph{Example \ref{ex_importante} is on a geodesically incomplete manifold. As suggested in the Introduction, it would be very interesting to find an analogous phenomenon for a complete manifold. It is likely that the technique in \cite{borbely_counter} to construct a complete manifold such that $\Delta$ satisfies $\wmp$, but not $\smp$, be useful to produce a complete example.
}
\end{remark}

\subsection{Ricci curvature and $\slio$}

When the operator is not of $p$-Laplacian type, the straightforward method described in Theorem \ref{teo_SMPeKO} does not work, and we cannot directly infer implication \eqref{smpekosl}. Nevertheless, in what follows we will describe how to obtain a sharp result for a much larger class of operators including various of geometric interest. The geometric assumption is given in term of a control on the Ricci curvature at infinity, the same as the one in Theorem \ref{teo_SMP} to guarantee the validity of the $\smp$. Similarly, the approach is inspired by the Phr\'agmen-Lindel\"off method. To construct the relevant supersolutions, we need some mild conditions relating $\varphi$ and $l$. We assume \eqref{assumptions_SL_necessity} and \eqref{phiel_solozero_SL}, and furthermore that

\begin{equation}\label{assu_perSL}
\text{$l$ is $C$-increasing on $\R^+_0$.}
\end{equation}

Moreover, we require the existence of $\chi_1, \chi_2 \in \R$ such that
$$
\begin{array}{lll}
(\chi_1) & \quad \disp t \mapsto \frac{\varphi'(t)}{l(t)} t^{1-\chi_1} & \qquad \text{is $C$-increasing on $\R^+$,} \\[0.4cm]
(\chi_2) & \quad \disp t \mapsto \frac{\varphi(t)}{l(t)} t^{-\chi_2} & \qquad \text{is $C$-increasing on $\R^+$.}
\end{array}
$$

Concerning $\beta$ and $\bar \beta$, we require
\begin{itemize}
\item[$(\beta \bar\beta)$] $\beta \in C([r_0,\infty))$, $\quad \bar \beta \in C^1([r_0,\infty))$, 
$$
\bar \beta' \le 0, \qquad \bar\beta \not \in L^1(\infty). 
$$
\end{itemize}


\begin{example}
\emph{If $\varphi(t) = t^{p-1}$ and $l(t) \asymp t^{p-1-\chi}$, then $(\chi_1),(\chi_2)$ are both satisfied provided that 
$$
\chi_j \le \chi \quad \text{for each } \, j \in \{1,2\}. 
$$
If $\varphi(t) = t/\sqrt{1+t^2}$ is the mean curvature operator, and $l(t) \asymp \varphi(t)/t^\chi$ for some $\chi \in \R$, then $(\chi_1),(\chi_2)$ hold provided that 
$$
\chi_1 \le \chi-2, \qquad \chi_2 \le \chi.
$$
Finally, if $\varphi(t) = t e^{t^2}$ is the operator of exponentially harmonic functions and $l(t) \asymp t^q$, $(\chi_1)$ and $(\chi_2)$ hold whenever
$$
\max\{\chi_1,\chi_2\} \le 1-q.
$$
}
\end{example}
%


%
%
%

We shall first deduce some useful properties from the validity of $(\chi_1)$ and $(\chi_2)$.

\begin{lemma}
\label{proposition_theta_vpell}
Assume that $\varphi$ and $l$ satisfy \eqref{assumptions_SL_necessity}. Then, $(\chi_1)$ with $\chi_1> -1$ implies \eqref{phiel_solozero_SL}. Moreover, $(\chi_2)$ with $\chi_2 \ge 0$ implies
\begin{equation}\label{philzero}
t^{-\chi_2}\frac{\varphi(t)}{l(t)} \in L^\infty\big((0,1)\big).
\end{equation}
In particular, $\varphi(t)/l(t) \ra 0$ as $t \ra 0$ whenever $\chi_2>0$.
\end{lemma}

\begin{proof}
Assume $(\chi_1)$. By definition, there exists $C\geq 1$ such that
\begin{equation*}
0<s^{1-\chi_1}\frac{\vp'(st)}{l(st)} \leq C
\frac{\vp'(t)}{l(t)} \quad \forall t\in \R^+\,\, s\in
(0,1],
\end{equation*}
or, equivalently,
\begin{equation}
\label{prop3.2_eq1}
s^{1-\chi_1}\frac{\vp'(st)}{l(st)} \geq C^{-1}
\frac{\vp'(t)}{l(t)} \quad \forall t\in \R^+\,\, s\in
[1, \infty).
\end{equation}
Setting $t=1$, if $\chi_1 > -1$ we deduce that $\frac{s\vp'(s)}{l(s)}\in
L^1(0+)\setminus L^1(\infty)$, that is, \eqref{phiel_solozero_SL}. In an similar way, if $(\chi_2)$ holds then 
\begin{equation*}
\frac{\vp(st)}{l(st)}s^{-\chi_2}  \leq C
\frac{\vp(t)}{l(t)} \quad \forall
t\in \R^+\,\, s\in (0,1],
\end{equation*}
and \eqref{philzero} follows by setting $t=1$.
\end{proof}


\begin{lemma}\label{prop 3.2bis}
Assume that $\varphi$ and $l$ satisfy \eqref{assumptions_SL_necessity}, and let $F$ be a positive function
defined on $(\bar \eta_0,\infty)$. If $(\chi_1)$ holds with $\chi_1> -1$, then there
exists a constant  $B\geq 1$ such that, for every $\sigma \leq 1$,
\begin{equation}
\label{eq_prop3.2bis_1}
\frac{\sigma^{\frac{1}{\chi_1+1}}}{K^{-1}(\sigma F(t))}\leq
\frac B{K^{-1}(F(t))}
\quad \text{ on } (\bar \eta_0,\infty).
\end{equation}
\end{lemma}

\begin{proof}
According to Lemma \ref{proposition_theta_vpell}, $(\chi_1)$ with $\chi_1>-1$ implies \eqref{phiel_solozero_SL}, so $K^{-1}$ is well defined on $\R^+_0$. Changing variables in the definition of $K$, and  using \eqref{prop3.2_eq1} above,
for every $\lambda \geq 1$ and $t\in \R^+$ we have
\begin{equation*}
\begin{split}
K(\lambda t) &= \int_0^{\lambda t} s\frac{\vp'(s)}{l(s)} \di s =
\lambda ^2 \int _0^t s\frac{\vp'(\lambda s)}{l(\lambda s)}\di s \\
&\geq
C^{-1} \lambda ^{\chi_1+1} \int_0^{t} s\frac{\vp'(s)}{l(s)}\di s
= C^{-1} \lambda ^{\chi_1+1} K(t),
\end{split}
\end{equation*}
where $C\geq 1$ is the constant in $(\chi_1)$. Applying $K^{-1}$ to both sides of the above inequality, and setting
$t= K^{-1}(\sigma F(s))$ we deduce
\begin{equation*}
\lambda K^{-1}(\sigma F(s)) \geq K^{-1}(\lambda^{\chi_1+1}\sigma
C^{-1} F(s)),
\end{equation*}
whence, setting $\lambda =(C/\sigma)^{1/(\chi_1+1)}\geq
1$, the required conclusion follows with $B=C^{1/(\chi_1+1)}$.
\end{proof}



We are ready to construct blowing-up supersolutions that remain close to the constant $\bar \eta_0$ in \eqref{assumptions_SL_necessity} on an arbitrarily fixed annulus. The construction is an improvement of the one in \cite[Prop. 3.4]{maririgolisetti}, and at the same a simplification of it. We recall that $F$ is defined as in \eqref{def_F_SL}.

\begin{proposition}\label{prop3.7}
Assume \eqref{assumptions_SL_necessity} and \eqref{assu_perSL}. Suppose further $(\chi_1)$ and $(\chi_2)$ with
$$
\chi_1 > 0, \qquad \chi_2 >0,
$$
and let $\beta, \bar \beta$ satisfy $(\beta \bar\beta)$. Fix $\theta \in C([r_0,\infty))$ with the property that 
\begin{equation}\label{ipo_theta_SL}
\left\{\begin{array}{l}
\disp \frac{\bar \beta(r)^{\chi_1+1}}{\beta(r)} \in L^\infty([r_0, \infty)), \\[0.5cm]
\disp \frac{\theta(r) \bar\beta(r)^{\chi_2}}{\beta(r)} \in L^\infty\big( [r_0, \infty)\big).
\end{array}\right.
\end{equation}
If \eqref{KO_2} holds, then for each $\eps >0$, $0 < \delta < \lambda$ and $r_1>r_0$, there exist $R_1 > r_1$ and a $C^1$
function $w : [r_0, R_1)\to [\bar \eta_0 + \delta, \infty)$ solving
\begin{equation}\label{eq3.9}
\left\{\begin{array}{l}
\disp \big( \varphi(w')\big)' + \theta(r) \varphi(w') \leq \eps \beta(r) f(w) l(w') \qquad \text{on }\, [r_0, R_1)\\[0.3cm]
\disp w'>0 \ \ \text{ on } \, [r_0,R_ 1), \qquad w(r) \ra  +\infty \ \ \text{ as } r\to R_1^- \\[0.3cm]
\disp \bar \eta_0 + \delta \le w \le \bar \eta_0 + \lambda \qquad \text{on } \, [r_0,r_1].
\end{array}\right.
\end{equation}
\end{proposition}

\begin{proof}
Note first of all that \eqref{KO} is meaningful because, by Lemma \ref{proposition_theta_vpell}, $(\chi_1)$ with $\chi_1> -1$ implies \eqref{phiel_solozero_SL}. Since $\bar\beta$ is bounded on $[r_0, \infty)$, by rescaling we can assume that $\bar \beta \le 1$. For a given $\sigma\in (0, 1]$ to be specified later, set
\begin{equation}
\label{eq3.11}
C_\sigma = \int_{\bar \eta_0+\delta}^{\infty} \frac{\di s}{K^{-1}(\sigma F(s))},
\end{equation}
which is well defined in view of \eqref{KO}, \eqref{assu_perSL} and Lemma \ref{lem_mettimaosigmatau_novo}. Since $\bar \beta \not\in
L^{1}(\infty)$, there exists $R_\sigma> r_0$ such that
\begin{equation*}
C_\sigma = \int_{r_0}^{R_\sigma} \bar \beta(s)\di s.
\end{equation*}
We note that, by monotone convergence, $C_\sigma \ra \infty$ as $\sigma \ra
0+,$ and we can therefore choose $\sigma>0$ small enough that $R_\sigma
>r_1$. We let $w : [r_0,R_\sigma) \ra [\bar \eta_0+ \delta, \infty)$  be
implicitly defined by the equation
\begin{equation}
\label{eq3.12}
\int_{r}^{R_\sigma} \bar \beta(s) \di s =
\int_{w(r)}^\infty \frac{\di s}{K^{-1}(\sigma F(s))},
\end{equation}
so that, by definition,
\begin{equation*}
w(r_0)= \bar \eta_0 + \delta,\qquad w(r) \ra +\infty \text{ as } \, r \ra R_\sigma^-.
\end{equation*}
Differentiating \eqref{eq3.12} yields
\begin{equation}
\label{eq3.13}
w'(r) = \bar \beta(r)K^{-1}\big(\sigma F(w(r))\big),
\end{equation}
so that $w'>0$ on $[r_0, R_\sigma)$, and
\begin{equation*}
\sigma F(w) = K(w'/\bar \beta).
\end{equation*}
Differentiating once more, using the definition of $K$ and \eqref{eq3.13},  we obtain
\begin{equation}
\sigma f(w)w' = K'(w'/\bar \beta) (w'/\bar \beta)'= \frac{w'}{\bar \beta} \frac
{\varphi'(w'/\bar \beta)}{l(w'/\bar \beta)}
\Bigl(\frac{w'}{\bar \beta}\Bigr)',
\end{equation}
that is, 
\begin{equation}\label{eq3.14}
\varphi'\left( \frac{w'}{\bar \beta}\right) \left( \frac{w'}{\bar \beta}\right)' = \sigma \bar \beta f(w) l\left(\frac{w'}{\bar \beta}\right) \qquad \text{on } \, [r_0, R_\sigma).
\end{equation}
Since $f> 0$ on $(\bar \eta_0,\infty)$ and $w'>0$, we infer that $w'/\bar \beta$ is non-decreasing. Moreover, from $\bar \beta'\leq 0$ we deduce
\begin{equation*}
\Bigl(\frac{w'}{\bar \beta}\Bigr)' =
\frac{w''}{\bar \beta} - \frac{w'\bar \beta'}{\bar \beta^2} \geq \frac{w''}{\bar \beta}.
\end{equation*}
Inserting this into \eqref{eq3.14}, using $\bar \beta \le 1$ and $(\chi_1)$ (in the form of \eqref{prop3.2_eq1}), and rearranging we obtain
\begin{equation}
\label{eq3.15}
\varphi'(w')w'' \leq \left\{C\sigma \frac{\bar \beta^{\chi_1+1}}{\beta} \right\} \beta f(w) l(w')\qquad \text{on }\, [r_0, R_\sigma).
\end{equation}
We next integrate \eqref{eq3.14} on $[r_0,r]$, and we use \eqref{assu_perSL} coupled with the monotonicity of both $w$ and $w/\bar \beta$ to deduce 
\begin{equation}\label{eq3.15_inte}
\varphi\left(\frac{w'}{\bar \beta}\right) \leq  \varphi\left(\frac{w'}{\bar \beta}\right)(r_0)+ C^2\sigma f(w)l \left(\frac{w'}{\bar \beta}\right) \int_{r_0}^r \bar \beta \di s,
\end{equation}
and thus using $(\chi_2)$, $\bar \beta\leq
1$ and $w(r_0)=\bar \eta_0+\delta$ we get
\begin{equation}\label{eq3.19}
\begin{split}
\frac{\varphi(w')}{l(w')} & \stackrel{(\chi_2)}{\leq} C^2\bar \beta^{\chi_2} \frac{\varphi(w'/\bar \beta)}{l(w'/\bar \beta)}
\\&
\leq C^2 \bar\beta^{\chi_2}
\Bigl[\frac{\varphi(w'/\bar\beta)(r_0)}{l(w'/\bar \beta)} + C^2\sigma f(w)\int_{r_0}^r \bar \beta \di s\Bigr]\\
& \stackrel{\eqref{assumptions_SL_necessity} + \eqref{assu_perSL}}{\le} C^2 \frac{\bar \beta^{\chi_2}}{\beta}
\Bigl[C^2\frac{\vp(w'/\bar \beta)(r_0)}{f(\bar \eta_0+\delta) l(w'/\bar \beta)(r_0)} + C^2 \sigma \int_{r_0}^r \bar \beta \di s \Bigr] \beta f(w) \\
& \stackrel{\eqref{eq3.12}}{\leq}
C^4 \frac{\bar \beta^{\chi_2}}{\beta}
\Bigl[\frac{\vp(w'/\bar \beta)(r_0)}{f(\bar \eta_0+\delta) l(w'/\bar \beta)(r_0)} + \sigma \int_{\bar \eta_0 + \delta }^\infty \frac{\di s}{K^{-1}(\sigma F(s))}\Bigr] \beta f(w).
\end{split}
\end{equation}
Next, we use \eqref{eq3.13} with $r=r_0$ and the fact that, by Lemma \ref{prop 3.2bis},
$$
\sigma \int_{\bar \eta_0 + \delta }^\infty \frac{\di s}{K^{-1}(\sigma F(s))} \le B \sigma^{ \frac{\chi_1}{\chi_1+1}} \int_{\bar \eta_0 + \delta }^\infty \frac{\di s}{K^{-1}(F(s))}
$$
to obtain 
\begin{equation}\label{eq3.19_dopo}
\begin{split}
\varphi(w') & \le C_1 \frac{\bar \beta^{\chi_2}}{\beta}
\Bigl[\frac{\varphi( K^{-1}(\sigma F(\bar \eta_0+\delta)))}{f(\bar \eta_0+\delta) l( K^{-1}(\sigma F(\bar \eta_0+ \delta)))} + \sigma^{\frac{\chi_1}{\chi_1+1}} \int_{\bar \eta_0 + \delta }^\infty \frac{\di s}{K^{-1}(F(s))}\Bigr] \beta f(w)l(w') \\
& = C_1 \hat C_\sigma \frac{\bar \beta^{\chi_2}}{\beta} \beta f(w)l(w') 
\end{split}
\end{equation}
for some $C_1>0$, and where we set
\begin{equation}\label{hatCsigma}
\hat C_\sigma = \frac{\varphi( K^{-1}(\sigma F(\bar \eta_0+\delta)))}{f(\bar \eta_0+\delta) l( K^{-1}(\sigma F(\bar \eta_0+\delta)))} + \sigma^{\frac{\chi_1}{\chi_1+1}} \int_{\bar \eta_0 + \delta }^\infty \frac{\di s}{K^{-1}(F(s))}.
\end{equation}
Putting together \eqref{eq3.15} and \eqref{eq3.19_dopo}, and using \eqref{ipo_theta_SL}, we obtain
\begin{equation}\label{eq3.19final}
\begin{array}{l}
\disp \big( \varphi(w')\big)' + \theta(r) \varphi(w') \\[0.3cm]
\quad \le \disp \left\{ C\sigma \left\|\frac{\bar \beta^{\chi_1+1}}{\beta}\right\|_{L^\infty([r_0, \infty))} + C_1 \hat C_\sigma \left\|\frac{\theta \bar \beta^{\chi_2}}{\beta}\right\|_{L^\infty([r_0, \infty))} \right\} \beta f(w) l(w').
\end{array}
\end{equation}
Because of $(\chi_2)$ with $\chi_2>0$ and Lemma \ref{proposition_theta_vpell}, $\hat C_\sigma \ra 0$ as $\sigma \ra 0$, and we can choose $\sigma$ small enough such that the differential inequality in \eqref{eq3.9} is satisfied. To prove the last condition in \eqref{eq3.9}, simply observe that by \eqref{eq3.12}
$$
\int_{r_0}^{r_1} \bar \beta(s) \di s = \int_{\bar \eta_0 + \delta}^{w(r_1)} \frac{\di s}{K^{-1}(\sigma F(s))},
$$
thus $w(r_1) \ra \bar \eta_0 + \delta$ as $\sigma \ra 0$. Since $w$ is increasing, it is enough to choose $\sigma$ in such a way that 
$$
\int_{r_0}^{r_1} \bar \beta(s) \di s < \int_{\bar \eta_0 + \delta}^{\bar \eta_0 + \lambda} \frac{\di s}{K^{-1}(\sigma F(s))}.
$$
\end{proof}

\begin{remark}\label{rem_useful_SL}
\emph{If we weaken our assumptions on $\chi_1,\chi_2$ to 
$$
\chi_1 \ge 0, \quad \chi_2 \ge 0 \quad \text{and} \quad \chi_1\chi_2=0, 
$$
assume further
\begin{equation}\label{fphil_borderline_SL}
\limsup_{t \ra + \infty} f(t) = +\infty. 
\end{equation}
Then, one can still guarantee the existence of a \emph{divergent} sequence $\delta_j \ra \infty$ such that, for each $\delta \in \{\delta_j\}$ and $\lambda>\delta$, there exists $w$ solving \eqref{eq3.9}. Indeed, all is needed to conclude the proof is that $\hat C_\sigma$ can be made arbitrarily small for suitable $\sigma$ and $\delta$. First, using $(\chi_2)$ and Lemma \ref{proposition_theta_vpell}, choose $\delta_j$ so that
$$
\frac{1}{f(\bar \eta_0+ \delta_j)} \left\|\frac{\varphi}{l}\right\|_{L^\infty([0,1])} + \int_{\bar \eta_0 + \delta_j }^\infty \frac{\di s}{K^{-1}(F(s))}
$$
is small enough, and then choose $\sigma$ in such a way that $K^{-1}(\sigma F(\bar \eta_0+\delta)) \le 1$. The coupling of these two conditions guarantee the smallness of $\hat C_\sigma$.
}
\end{remark}

\begin{remark}
\emph{It is interesting to compare Proposition \ref{prop3.7} with the corresponding Proposition \ref{prop_CSP_radial} for the compact support principle. Although their underlying idea is the same, the two constructions are \emph{not specular}, neither are the conditions on $\varphi, f, l$. The reason is that, while the monotonicity of the supersolution changes, the weight $\bar \beta$ is still decreasing. To grasp the core of the technical problem, we invite the interested reader to try to prove Proposition \ref{prop_CSP_radial} by following the estimates in Proposition \ref{prop3.7}, suitably replacing $(\chi_1)$ and $(\chi_2)$ in a neighbourhood of zero.
}
\end{remark}

We can now investigate the validity of $\slio$.

\begin{theorem}\label{teo_SMP_SL}
Let $M^m$ be a complete Riemannian manifold of dimension $m \ge 2$ such that, for some origin $o \in M$, the distance function $r(x)$ from $o$ satisfies
\begin{equation}\label{ricciassu_SL}
\Ricc (\nabla r, \nabla r) \ge -(m-1)\kappa^2\big( 1+r^2\big)^{\alpha/2} \qquad \text{on } \, \mathcal{D}_o,
\end{equation}
for some $\kappa \ge 0$ and $\alpha \ge -2$. Let $\varphi,f,l$ meet \eqref{assumptions_SL_necessity}, \eqref{assu_perSL}, and assume $(\chi_1)$ and $(\chi_2)$ with 
\begin{equation}\label{ipo_chi1chi2_SL}
\chi_1 >0, \qquad \chi_2 >0.
\end{equation}
Consider $0< b \in C(M)$ such that
$$
b(x) \ge C_1\big(1+ r(x)\big)^{-\mu} \qquad \text{on } \, M,
$$
for some constants $C_1>0$, $\mu \in \R$ satisfying
\begin{equation}\label{mualphachi12}
\mu \le \min \left\{ \chi_1+1, \chi_2 - \frac{\alpha}{2}\right\}.
\end{equation}
If \eqref{KO} holds, then any non-constant solution $u \in C^1(M)$ of $(P_\ge)$ is bounded above and $f(u^*) \le 0$. In particular, if $f>0$ on $\R^+$ then $\slio$ holds for $C^1$ solutions.
\end{theorem}


\begin{proof}
We first prove that $u$ is bounded above. By contradiction, let $u^* = \infty$ and consider a geodesic ball $B_{r_0}$ centered at $o$, with $r_0>0$. Fix $0< \delta < \lambda$ and $\bar x \not \in B_{r_0}$ enjoying 
\begin{equation}\label{bni}
u_0^* = \max_{B_{r_0}} u \le \bar \eta_0 + \delta, \qquad u(\bar x) > \bar \eta_0 + \lambda,
\end{equation}
and choose $r_1$ in such a way that $\bar x \in B_{r_1}$. From \eqref{ricciassu_SL} and the Laplacian comparison theorem, 
\begin{equation}\label{deltar_SL}
\Delta r \le A r^{\alpha/2} \qquad \text{weakly on } \, M \backslash B_{\frac{r_0}{2}},
\end{equation}
for some constant $A=A(r_0,\alpha,\kappa,m)>0$. Applying Proposition \ref{prop3.7} with 
$$
\theta(r) = Ar^{\alpha/2}, \qquad \bar \beta(r) = (1+r)^{-1}, \qquad \beta(r) = C_1(1+r)^{-\mu}, \qquad \eps = \frac{1}{2C}
$$
we deduce the existence of $w$ satisfying \eqref{eq3.9} (note that \eqref{mualphachi12} guarantees \eqref{ipo_theta_SL}). Setting $\bar w(x) = w(r(x))$ and taking into account $w'>0$, \eqref{deltar_SL} and $b \ge \beta(r)$, $\bar w$ solves 
\begin{equation}\label{eq_perbarw}
\left\{\begin{array}{l}
\Delta_\varphi \bar w \leq \frac{1}{2C} b(x) f(\bar w) l(w'(r)) \qquad \text{on }\, B_{R_1} \backslash B_{r_0}\\[0.3cm]
\disp w'(r)>0 \ \ \text{ on } \, B_{R_1}\backslash B_{r_0}, \qquad \bar w \ra  +\infty \ \ \text{ as } x \ra \partial B_{R_1} \\[0.3cm]
\disp \bar \eta_0 + \delta \le \bar w \le \bar \eta_0 + \lambda \qquad \text{on } \, B_{r_1}\backslash B_{r_0}.
\end{array}\right.
\end{equation}
We compare $u$ and $\bar w$ on $B_{R_1} \backslash B_{r_0}$. By construction, $u \le \bar w$ on $\partial B_{r_0}$ and $u-\bar w \ra -\infty$ approaching $\partial B_{R_1}$. On the other hand, 
$$
u(\bar x) > \bar \eta_0 + \lambda \ge \bar w(\bar x),
$$
and thus $c = \max\{ u-\bar w\}$ is positive and attained on some compact set $\Gamma = \{ u-\bar w = c\} \Subset B_{R_1} \backslash B_{r_0}$. For $\eta \in (0,c)$, consider $U_\eta = \{u-\bar w > \eta\} \Subset B_{R_1} \backslash B_{r_0}$. If $x \in \Gamma \backslash \cut(o)$, then $\bar w \in C^1$ around $x$ and therefore
$$
\nabla u(x) = \nabla \bar w(x) = w'(r(x)).
$$
The same relation also holds if $x \in \cut(o)$ by using Calabi's trick (see the proof of Theorem \ref{teo_SMP}). Since both $\nabla u$ and $w'(r)$ are continuous, for $\eta$ close enough to $c$ the inequality $|\nabla u| \ge w'(r)/2$ holds on $U_\eta$. From
$$
\Delta_\varphi u \ge b(x) f(u) l(|\nabla u|) \ge \frac{1}{2C} b(x) f(\bar w) l\big(w'(r)\big) \ge \Delta_\varphi \bar w = \Delta_\varphi(\bar w + \eta),  
$$
we deduce by comparison that $u \le \bar w + \eta$ on $U_\eta$, contradiction.\par
%
It remains to prove that $f(u^*) \le 0$. If $f(u^*) > K>0$, then choose an upper level set $\Omega_\eta$ with $\eta < u^*$ in such a way that $f(u)> K$ on $\Omega_\eta$, and a continuous function $\bar f \le f$ with 
$$
\begin{array}{l}
\disp \bar f(\eta)=0, \qquad \bar f \ \text{  is positive and $C$-increasing on $(\eta, \infty)$}, \\[0.2cm]
\bar f \le f \quad \text{on } \, (\eta, u^*), \qquad \bar f = f \quad \text{on } \, \big( \max\{\bar \eta_0,u^*\}, \infty).
\end{array}
$$
Then, $\bar u = \max\{u, \eta\}$ satisfies
$$
\Delta_\varphi \bar u \ge b(x)\bar f(\bar u) l(|\nabla\bar u|) \qquad \text{on } \, M
$$
(observe that $\bar f(\eta)l(0)=0$). Let $\bar u_0^* = \sup_{B_{r_0}} \bar u$, and note that $\bar u_0^* < \bar u^*$ in view of the finite maximum principle applied to $\bar u^*-\bar u$. Fix $\lambda>0$ such that $\bar u_0^* + 2\lambda < \bar u^* - 2\lambda$, set $\delta = \lambda/2$, let $\bar x$ satisfying $\bar u(\bar x) > \bar u^*-\lambda$ and choose $r_1$ big enough that $\bar x \in B_{r_1}$. With our choices, $\bar x$ belongs to the relatively compact set $\{\bar u> \bar w\}$ and, consequently, the desired contradiction is achieved by proceeding verbatim as in the case $u^* = \infty$, with $\bar u$ replacing $u$, $\bar u_0^*$ replacing $u_0^*$ and with the same function $w$.
\end{proof}
Once the bound $u^*< \infty$ is shown, an alternative way to conclude $f(u^*) \le 0$ in Theorem \ref{teo_SMP_SL} is to use Theorem \ref{teo_WMP_intro} to ensure the validity of $\wmp$. However, we should require the extra condition \eqref{assu_WMP_intro_growth}, that is avoided in the above argument. Nevertheless, this second approach is needed to deal with the relevant, borderline case when either $\chi_1$ or $\chi_2$ vanishes, that is considered in the next
\begin{theorem}\label{teo_SMP_SL_border}
In the assumptions of Theorem \ref{teo_SMP_SL}, suppose that \eqref{ipo_chi1chi2_SL} is replaced by
$$
\chi_1 \ge 0, \qquad \chi_2 \ge 0, \qquad \chi_1\chi_2 =0
$$
and the validity of 
\begin{equation}\label{allafine_SL}
\limsup_{t \ra + \infty} f(t) = +\infty. 
\end{equation}
Assume that 
\begin{equation}\label{ipo_SL_border_varphi}
\left\{\begin{array}{ll}
\disp l(t) \ge C_1 \frac{\varphi(t)}{t^{\chi_2}} & \qquad \text{on $\R^+$, for some $C_1>0$}, \\[0.3cm]
\disp \varphi(t) \le C_2 t^{p-1} & \qquad \text{on $[0,1]$, for some } \, C_2>0, \ p>1, \\[0.3cm]
\disp \varphi(t) \le \bar C_2 t^{\bar p-1} & \qquad \text{on $[1,\infty)$, for some } \, \bar C_2>0, \ \bar p>1,
\end{array}\right.
\end{equation}
and that, besides \eqref{mualphachi12}, one of the following conditions is met:
\begin{equation}\label{mualphachi12_border}
\left\{ \begin{array}{ll}
\alpha \ge -2, \quad \chi_2>0, \quad \text{or} \\[0.2cm]
\alpha \ge -2, \quad \chi_2 = 0, \quad \mu < - \frac{\alpha}{2}, \quad \text{or } \\[0.2cm]
\alpha > -2, \quad \chi_2 = 0, \quad \mu = - \frac{\alpha}{2}, \quad V_\infty = 0, \quad \text{or } \\[0.2cm]
\alpha = -2, \quad \chi_2 = 0, \quad \mu = - \frac{\alpha}{2}, \quad V_\infty \le p,
\end{array}\right.
\end{equation}
where
$$
V_\infty = \left\{ \begin{array}{ll} 
\disp\liminf_{r \ra \infty} \frac{\log \vol(B_r)}{r^{1+ \alpha/2}} & \quad \text{if } \, \alpha \ge -2, \\[0.4cm]
\disp\liminf_{r \ra \infty} \frac{\log \vol(B_r)}{\log r} & \quad \text{if } \, \alpha = -2.
\end{array}\right.
$$
Then, if \eqref{KO} holds, any non-constant solution $u \in C^1(M)$ of $(P_\ge)$ on $M$ is bounded above and $f(u^*) \le 0$. In particular, if $f>0$ on $\R^+$, $\slio$ holds for $C^1$ solutions.
\end{theorem}

\begin{proof}
Because of Remark \ref{rem_useful_SL}, for $\delta$ large and suitably chosen we can still produce a solution $w$ of \eqref{eq3.9}, that gives rise to $\bar w$ solving \eqref{eq_perbarw}. Following the proof of Theorem \ref{teo_SMP_SL} we obtain $u^* < \infty$. To conclude, we observe that we are in the position to apply Theorem \ref{teo_WMP_intro} with the choice $\chi = \chi_2$: in this respect, note that the first requirement in \eqref{assum_WMP_intro} corresponds to the first in \eqref{ipo_SL_border_varphi}. The conclusion $f(u^*) \le 0$ follows from the validity of $\wmp$ and Proposition \ref{prop_equivalence}.
\end{proof}

\begin{remark}\label{rem_barkappa}
\emph{By the Bishop-Gromov theorem (Theorem \ref{teo_compavol_appe} and the subsequent remarks in the Appendix), \eqref{ricciassu_SL} implies $V_\infty < \infty$ and in particular, if $\alpha = -2$, 
\begin{equation}\label{def_barkappa}
V_\infty \le (m-1) \bar \kappa + 1 \qquad \text{with} \qquad \bar \kappa = \frac{1 + \sqrt{1+4\kappa^2}}{2}. 
\end{equation}
Therefore, condition $V_\infty \le p$ in the last of \eqref{mualphachi12_border} is implied by $\bar \kappa \le \frac{p-1}{m-1}$.
}
\end{remark}

\begin{remark}
\emph{One could alternatively use Theorem \ref{teo_SMP_intro} to conclude $f(u^*) \le 0$. Doing so, on the one hand \eqref{ipo_SL_border_varphi} would weaken to \eqref{assu_SMP_intro_growth}, requiring $\varphi,l$ only on $[0,1]$, but on the other hand the conclusion in the case $(\chi_2 =) \ \chi =0$ in \eqref{condi_volumericci_intro} is only possible under the Euclidean type behaviour $\alpha = -2$. 
}
\end{remark}

We conclude this section with some comments on Theorem \ref{teo_SMP_SL}. We begin with the following corollary for the $p$-Laplace operator, a slight improvement of \cite[Cor. A1]{maririgolisetti}.

\begin{corollary}\label{teo_SMP_SL_plap}
Let $M^m$ be complete and satisfying 
$$
\Ricc (\nabla r, \nabla r) \ge -(m-1)\kappa^2\big( 1+r^2\big)^{\alpha/2} \qquad \text{on } \, \mathcal{D}_o,
$$
for some $\kappa \ge 0$, $\alpha \ge -2$ and some origin $o$. Fix $p>1$ and $\chi \in (0, p-1]$. Consider $0< b \in C(M)$ and $f \in C(\R)$ such that
$$
\begin{array}{l}
\disp b(x) \ge C_1\big(1+ r(x)\big)^{-\mu} \qquad \text{on } \, M, \\[0.3cm]
f(0)=0, \qquad f>0 \ \text{ and $C$-increasing on } \, \R^+.
\end{array}
$$
for some constants $C,C_1>0$ and 
$$
\mu \le \chi - \frac{\alpha}{2}.
$$
If the Keller-Osserman condition
$$
F(t)^{-\frac{1}{\chi+1}} \in L^1(\infty)
$$
is met, then $\slio$ holds for $C^1$ solutions of 
$$
\Delta_p u \ge b(x)f(u)|\nabla u|^{p-1-\chi}.
$$
The same conclusion holds if $\chi=0$, provided that $M$ satisfies one of the next further conditions: either 
$$
\begin{array}{ll}
\alpha > -2, & \qquad \disp \liminf_{r \ra \infty} \frac{\log\vol(B_r)}{r^{1+\alpha/2}} = 0, \qquad \text{or } \\[0.4cm]
\alpha = -2, & \disp \qquad \bar \kappa \le \frac{p-1}{m-1},  
\end{array}
$$
with $\bar \kappa$ as in \eqref{def_barkappa}.
\end{corollary}

\begin{proof}
It is a direct application of Theorems \ref{teo_SMP_SL}, \ref{teo_SMP_SL_border} and Remark \ref{rem_barkappa}. If $\chi=0$, note that \eqref{allafine_SL} follows from $F^{-\frac{1}{\chi+1}} \in L^1(\infty)$. 
\end{proof}

Moving to consider the mean curvature operator, a substantial problem arises: in view of \eqref{assu_perSL}, $l$ is bounded from below in a neighbourhood of infinity, and thus $K_\infty < \infty$ and \eqref{KO_infinity_intro} is meaningless. To overcome the problem and be able to include inequalities of the type  
\begin{equation}\label{eq_MC_old}
\diver\left( \frac{\nabla u}{\sqrt{1+|\nabla u|^2}} \right) \ge b(x)f(u)|\nabla u|^q,
\end{equation}
taking into account that the mean curvature operator satisfies 
\begin{equation}\label{varphi_allamean}
t \varphi'(t) \le C \varphi(t) \qquad \text{on $\R^+$}
\end{equation}
for some $C>0$, the authors in \cite[Sect. 4]{maririgolisetti} propose to replace $t\varphi'(t)$ with $\varphi(t)$ in the definition of $K$ in \eqref{def_K}. In this way, the corresponding \eqref{KO_infinity_intro} makes sense for some classes of $C$-increasing $l$. As we shall see in a moment, this seemingly ``rough" replacement allows indeed to obtain a sharp result, but in the course of the proof in \cite{maririgolisetti} the authors lose optimality in some inequalities, and consequently their main result (Corollary A2 therein) is not sharp. We now describe how to achieve the optimal range of parameters. Clearly, the bulk is to get an analogue of Proposition \ref{prop3.7} for mean curvature type operators and \emph{not requiring that $l$ be $C$-increasing.} Note that the $C$-monotonicity of $l$ is essential to obtain inequality \eqref{eq3.15_inte}. We restrict to consider the relevant case of operators satisfying \eqref{varphi_allamean} and gradient terms $l$ of the type
$$
l(t) = \frac{\varphi(t)}{t^\chi},
$$
for $\chi >0$ small enough to make $l$ continuous at $t=0$. Observe that $\varphi$ may vanish both at $t=0$ and at infinity. Following the idea in \cite{maririgolisetti}, we set
$$
K(t) = \int_0^t \frac{\varphi(s)}{l(s)} \asymp t^{\chi+1}
$$
and the Keller-Osserman condition becomes $F^{- \frac{1}{\chi+1}} \in L^1(\infty)$, with $F$ as in \eqref{def_F_SL}.

\begin{proposition}\label{prop3.7_mc}
Let $\varphi$ satisfy \eqref{assumptions_SL_necessity} and \eqref{varphi_allamean}. Fix $\chi > 0$, $f \in C(\R)$ satisfying
$$
f> 0 \ \ \  \text{ and $C$-increasing on $(\bar \eta_0, \infty)$, for some $\bar \eta_0>0$,}
$$
and $\beta, \bar \beta$ satisfying $(\beta \bar\beta)$. Let $\theta \in C([r_0,\infty))$ with the property that 
\begin{equation}\label{ipo_theta_SL_mc}
\disp \frac{\max\{ \bar \beta(r),\theta(r)\} \cdot \bar \beta(r)^{\chi}}{\beta(r)} \in L^\infty([r_0, \infty)).
\end{equation}
If the Keller-Osserman condition
$$
F^{-\frac{1}{\chi+1}} \in L^1(\infty)
$$
holds, then for each $\eps >0$, $0 < \delta < \lambda$ and $r_1>r_0$, there exist $R_1 > r_1$ and a $C^1$
function $w : [r_0, R_1)\to [\bar \eta_0 + \delta, \infty)$ solving
\begin{equation}\label{eq3.9_mc}
\left\{\begin{array}{l}
\disp \big( \varphi(w')\big)' + \theta(r) \varphi(w') \leq \eps \beta(r) f(w) \frac{\varphi(w')}{[w']^\chi} \qquad \text{on }\, [r_0, R_1)\\[0.3cm]
\disp w'>0 \ \ \text{ on } \, [r_0,R_ 1), \qquad w(r) \ra  +\infty \ \ \text{ as } r\to R_1^- \\[0.3cm]
\disp \bar \eta_0 + \delta \le w \le \bar \eta_0 + \lambda \qquad \text{on } \, [r_0,r_1].
\end{array}\right.
\end{equation}
\end{proposition}

\begin{proof}
We proceed as in Proposition \eqref{prop3.7}, so we skip some of the details and just concentrate on the main differences. For $\sigma\in (0, 1]$ to be specified later, set
\begin{equation}\label{eq3.11_mc}
C_\sigma = \int_{\bar \eta_0+\delta}^{\infty} [\sigma F(s)]^{-\frac{1}{\chi+1}}\di s,
\end{equation}
and since $\bar \beta \not\in
L^{1}(\infty)$, pick $R_\sigma> r_0$ such that
\begin{equation*}
C_\sigma = \int_{r_0}^{R_\sigma} \bar \beta(s)\di s.
\end{equation*}
We can choose $\sigma>0$ small enough that $R_\sigma
>r_1$. We let $w : [r_0,R_\sigma) \ra [\bar \eta_0+ \delta, \infty)$  be
implicitly defined by the equation
\begin{equation}\label{eq3.12_mc}
\int_{r}^{R_\sigma} \bar \beta(s) \di s = 
\int_{w(r)}^\infty [\sigma F(s)]^{-\frac{1}{\chi+1}}\di s.
\end{equation}
Differentiating and rearranging, 
\begin{equation*}
\sigma F(w) = [w'/\bar \beta]^{\chi+1}.
\end{equation*}
Set $l(t) = \varphi(t)/t^\chi$. A second differentiation gives
\begin{equation}\label{novam}
\begin{array}{lcl}
\disp \sigma f(w)w' & = & \disp (\chi+1)(w'/\bar\beta)^\chi (w'/\bar \beta)' = (\chi+1)\frac{\varphi(w')}{l(w')} \bar \beta^{-\chi} (w'/\bar \beta)' \\[0.3cm]
& \ge & \disp (\chi+1)\frac{\varphi(w')}{l(w')} w'' \bar \beta^{-\chi-1},
\end{array}
\end{equation}
where we used that $w'/\beta$ is increasing by the first equality in \eqref{novam}, and $\bar \beta' \le 0$ by $(\beta \bar\beta)$. We next use \eqref{varphi_allamean} and simplify  to deduce
\begin{equation}\label{bonit_mc}
\varphi'(w') w'' \le \left\{ c_1 \sigma \frac{\bar \beta^{\chi+1}}{\beta}\right\} \beta f(w)l(w'),
\end{equation}
for some constant $c_1>0$ independent of $\sigma$. On the other hand, from the first equality in \eqref{novam} we deduce
$$
\sigma f(w)\bar \beta = (\chi+1)(w'/\bar\beta)^{\chi-1}(w'/\bar \beta)',
$$
thus integrating on $[r_0,r]$ and using the $C$-monotonicity of $f$ we get
$$
\begin{array}{lcl}
\disp \frac{\chi+1}{\chi} (w'/\bar\beta)^\chi & = & \disp \frac{\chi+1}{\chi} (w'/\bar\beta)(r_0)^\chi + \sigma \int_{r_0}^r f(w) \bar \beta \\[0.4cm]
& \le & \disp\frac{\chi+1}{\chi} (w'/\bar\beta)(r_0)^\chi + C\sigma f(w(r)) \int_{r_0}^r \bar \beta.  
\end{array}
$$
Therefore, from $\bar \beta \le 1$ (we can always assumed it, up to rescaling) we get
$$
\frac{\varphi(w')}{l(w')} = [w']^\chi \le \bar \beta^\chi (w'/\bar\beta)(r_0)^\chi + c_2 \sigma \bar \beta^\chi f(w) \int_{r_0}^r \bar \beta,   
$$
for some constant $c_2>0$ independent of $\sigma$. Rearranging, by \eqref{eq3.12_mc} and the $C$-monotonicity of $f$ we deduce
$$
\begin{array}{lcl}
\disp \varphi(w') & \le & \disp \frac{\bar \beta^\chi}{\beta} \left\{ C \frac{(w'/\bar\beta)(r_0)^\chi}{f(\bar \eta_0+\delta)} + c_2 \sigma \int_{r_0}^r \bar \beta \right\} \beta f(w)l(w') \\[0.5cm]
 & \le & \disp \frac{\bar \beta^\chi}{\beta} \left\{ C \frac{(w'/\bar\beta)(r_0)^\chi}{f(\bar \eta_0+\delta)} + c_2 \sigma \int_{\bar \eta_0 + \delta}^\infty [\sigma F(s)]^{-\frac{1}{\chi+1}}\di s\right\} \beta f(w)l(w') \\[0.5cm] 
\end{array}
$$
Coupling with \eqref{bonit_mc} and using \eqref{ipo_theta_SL_mc}, we finally infer
$$
\begin{array}{lcl}
\disp \varphi'(w')w'' + \theta(r) \varphi(w') & \le & \disp \left\{ c_1 \sigma \left\| \frac{\bar\beta^{\chi+1}}{\beta}\right\|_{L^\infty([r_0,\infty))} + \left[ C \frac{(w'/\bar\beta)(r_0)^\chi}{f(\bar \eta_0+\delta)} \right. \right. \\[0.5cm]
& & \disp \left. \left. + c_2 \sigma^{\frac{\chi}{\chi+1}}\int_{\bar \eta_0 + \delta}^\infty [F(s)]^{-\frac{1}{\chi+1}}\di s \right] \left\|\frac{\theta \bar\beta^{\chi}}{\beta}\right\|_{L^\infty([r_0,\infty))}\right\} \beta f(w)l(w').
\end{array}
$$
The desired conclusions now follow verbatim from the arguments in Proposition \ref{prop3.7}. 
\end{proof}

\begin{remark}
\emph{We point out that $l(t) = \varphi(t)/t^\chi$ can be singular at $t=0$. Indeed, by construction $w'>0$ on $[r_0, R_1)$ and the continuity of $l$ at $t=0$ is not needed.
}
\end{remark}

Once Proposition \ref{prop3.7_mc} is established, we proceed as in Theorem \ref{teo_SMP_SL} to obtain the following result, that also applies to mean curvature type operators. In particular, a direct application of the next result yields Theorem \ref{teo_SL_ricci_intro} in the Introduction. 

\begin{theorem}\label{teo_SMP_SL_mc}
Let $M^m$ be complete and satisfying 
\begin{equation}\label{ricciassu_SL_mc}
\Ricc (\nabla r, \nabla r) \ge -(m-1)\kappa^2\big( 1+r^2\big)^{\alpha/2} \qquad \text{on } \, \mathcal{D}_o,
\end{equation}
for some $\kappa \ge 0$, $\alpha \ge -2$ and some origin $o$. Let $\varphi,l$ meet \eqref{assumptions_SL_necessity} and 
\begin{equation}\label{allamean}
t \varphi'(t) \le C_2\varphi(t), \qquad l(t) \ge C_1 \frac{\varphi(t)}{t^\chi} \qquad \text{on } \, \R^+,
\end{equation}
for some constants $C_1,C_2>0$ and $\chi >0$. Consider $0< b \in C(M)$ such that
$$
b(x) \ge C_3\big(1+ r(x)\big)^{-\mu} \qquad \text{on } \, M,
$$
for some constants $C_3>0$, $\mu \in \R$ satisfying
\begin{equation}\label{mualphachi12_mc}
\mu \le \chi - \frac{\alpha}{2}.
\end{equation}
Then, under the validity of the Keller-Osserman condition 
\begin{equation}\label{KO_mc}
F^{-\frac{1}{\chi+1}} \in L^1(\infty)
\end{equation}
with $F$ as in \eqref{def_F_SL}, any non-constant solution $u \in C^1(M)$ of $(P_\ge)$ on $M$ is bounded above and $f(u^*) \le 0$. In particular, if $f>0$ on $\R^+$ then $\slio$ holds for $C^1$-solutions.
\end{theorem}

\begin{remark}
\emph{The continuity of $l$ at $t=0$ forces an upper bound on $\chi$ by \eqref{allamean}. If we appropriately define solutions of $(P_\ge)$ when $l$ has a singularity, it is likely that the upper bound on $\chi$ be removable or, at least, weakened. We will not pursue this issue here, and leave it to the interested reader. 
}
\end{remark}

\begin{remark}\label{rem_impo_mc}
\emph{The above proof of Proposition \ref{prop3.7_mc} fails if $\chi=0$, and thus, in this borderline case the possible validity of an analogous of Remark \ref{rem_useful_SL} and of Theorem \ref{teo_SMP_SL_mc} is yet to be investigated. 
}
\end{remark}

\subsection{Sharpness}\label{sec_sharpness_SL}

We now conclude this section by discussing the sharpness of Theorem \ref{teo_SMP_SL_mc}. Consider the polynomial case $f(t) = t^\omega$, for some $\omega \ge 0$. Then, \eqref{KO_mc} becomes 
\begin{equation}\label{omc}
\omega > \chi.
\end{equation}
We are going to contradict $\slio$ under the failure of \eqref{omc}, on a suitable manifold and for $\varphi, l, b, \chi,\mu,\alpha$ meeting all of the remaining requirements in Theorem \ref{teo_SMP_SL_mc}. Let $(M,\di s_g^2)$ be a model manifold as in Subsection \ref{sec_counter}, and suppose further that $\varphi' \ge 0$ on $\R^+$. Note that, because of \eqref{esti_lapla_model} and the asymptotic behaviour $\Delta r \sim (m-1)/r$ as $r \ra 0$, 
\begin{equation}\label{laplar_sotto_mc}
\Delta r \ge c(1+r^2)^{\alpha/4} \qquad \text{on } \, M,
\end{equation}
for some constant $c>0$. For $\sigma>1$ define the smooth function $u = w(r) = (1+r^2)^\sigma$. A direct computation using $\alpha \ge -2$ and $\varphi'\ge 0$, $w',w'' \ge 0$ gives
$$
\Delta_\varphi u = \varphi'(w') w'' + \varphi(w') \Delta r \ge c(1+r^2)^{\alpha/4} \varphi(w').
$$
Therefore, $u$ solves
\begin{equation}\label{asd}
\Delta_\varphi u \ge C\big(1+r(x)\big)^{-\mu} u^\omega \frac{\varphi(|\nabla u|)}{|\nabla u|^\chi} \qquad \text{on } \, M,
\end{equation}
for some $C>0$, if and only if
\begin{equation}\label{rel_omegachiecc}
2\sigma (\omega-\chi) \le \frac{\alpha}{2} + \mu-\chi.
\end{equation}
Since the right-hand side is non-positive because of \eqref{mualphachi12_mc}, \eqref{rel_omegachiecc} is always satisfied for some $\sigma$ large enough if and only if 
\begin{equation}\label{bonito_mc}
\left\{ \begin{array}{lll}
\omega< \chi, & \quad \text{for each} & \mu \le \chi - \frac{\alpha}{2}, \quad \text{or } \\[0.3cm]
\omega = \chi & \quad \text{and} & \mu = \chi- \frac{\alpha}{2},
\end{array}\right.
\end{equation}
that proves the sharpness of \eqref{omc}. Also, the last restriction in the second of \eqref{bonito_mc} is optimal: in fact, in Euclidean setting $\alpha = -2$, if $\omega = \chi$ and $\mu < \chi+1$ then entire solutions of \eqref{asd} are constant if they have polynomial growth, see \cite[Thm. 12]{farinaserrin2} and also Example 4 at p. 4402 therein.

\begin{remark}\label{rem_impo_mc_2}
\emph{Differently from Theorem \ref{teo_SMP_SL_mc}, the above counterexample also works if $\chi=0$ and $\omega=0$.
}
\end{remark}

\subsection{Volume growth and $\slio$}
In this section, we study property $\slio$ for solutions $u$ of $(P_\ge)$ when the condition on the Ricci curvature is replaced by a volume growth requirement, in the particular case when $f(t) \asymp t^\omega$ and 
$$
l(t) \asymp \frac{\varphi(t)}{t^\chi}.
$$
In this setting, Theorem \ref{teo_SMP_SL_mc} and the subsequent remarks show that a sharp Keller-Osserman condition to guarantee the boundedness of $u$ and $f(u^*) \le 0$ is \eqref{KO_mc}, that is, $\omega>\chi$. The condition is optimal also for the mean curvature operator. However, a quite interesting phenomenon happens in this case: we begin by commenting on the following Liouville theorem for solutions of $(P_\ge)$, specific to mean curvature type operators and polynomial volume growths, where \emph{no Keller-Osserman condition} is needed on $f$ nor growth requirements are imposed on $u$. The result considers $(P_\ge)$ with a borderline gradient dependence $l(t) \ge C_2 \varphi(t)$ on $\R^+$. Its proof is inspired by the original one due to Tkachev in \cite{tkachev} for $b\equiv 1$, $l \equiv 1$, later extended in \cite{Serrin_4}.



\begin{theorem}\label{teo_tkachev}
Let $M$ be a complete Riemannian manifold, and consider 
$$
\left\{\begin{array}{l}
\varphi \in C(\R^+_0), \qquad 0 \le \varphi \le C_1 \quad \text{on } \, \R^+_0; \\[0.2cm]
f \in C(\R), \qquad f \ \text{ non-decreasing on } \, \R; \\[0.2cm]
l \in C(\R^+_0), \qquad l(t) \ge C_2 \varphi(t) \quad \text{on } \, \R^+_0.
\end{array}\right.
$$
for some constant $C_1,C_2>0$. Fix $b \in C(M)$ satisfying  
$$
b(x) \ge C \big( 1+ r(x)\big)^{-\mu} \qquad \text{on } \, M, 
$$
for some constants $C>0$, $\mu < 1$. Let $u \in \lip_\loc(M)$ be a non-constant solution of  
\begin{equation}\label{borderline_eq_tkachev}
\Delta_\varphi u \ge b(x) f(u) l(|\nabla u|) \qquad \text{on } \, M.
\end{equation}
If 
\begin{equation}\label{polinomial_tkachev}
\liminf_{r \ra \infty} \frac{\log \vol(B_r)}{\log r} < \infty
\end{equation}
Then 
$$
f(u)\varphi(|\nabla u|) \le 0 \qquad \text{on } \, M.
$$
In particular, if $\varphi>0$ on $\R^+$, then $f(u) \le 0$ on $M$.\par
Furthermore, under the same assumptions, if $u \in \lip_\loc(M)$ is a non-constant solution of $(P_=)$ then 
$$
f(u)\varphi(|\nabla u|) \equiv 0 \qquad \text{on } \, M,
$$
and $f(u) \equiv 0$ provided that $\varphi>0$ on $\R^+$.
\end{theorem}

\begin{proof}
Let $\{f_k\}$ be a sequence of locally Lipschitz functions converging pointwise to $f$ from below: for instance, one can choose
$$
f_k(t) = \inf_{y \in [t-1,t+1]} \Big\{ f(y) + k|t-y|\Big\}.
$$
Since $f$ is increasing, up to replacing $f_k$ with $\bar f_k(t) = \sup_{(-\infty,t)} f_k$ we can further suppose that $f_k$ is increasing for each $k$. From \eqref{borderline_eq_tkachev} we deduce
\begin{equation}\label{borderline_eq_tkachev_2}
\Delta_\varphi u \ge b(x) f_k(u) l(|\nabla u|) \qquad \text{on } \, M.
\end{equation}
Fix a divergent sequence $\{R_j\}$ such that $\{2R_j\}$ realizes the liminf in \eqref{polinomial_tkachev}, and let $d_0,C$ be positive constants such that 
\begin{equation}\label{ipovol_poli_tkachev}
\vol\big(B_{2R_j}\big) \le C R_j^{d_0} \qquad \text{for each } \, j. 
\end{equation}
Suppose that the set $U = \{x : f(u(x)) >0\}$ is non-empty, otherwise the thesis directly follows. We are going to prove that $f(u)\varphi(|\nabla u|) = 0$ on $U$. Fix a cut-off function $0 \le \psi \in \lip_c(M)$ whose support intersects $U$, and choose to test the weak formulation of \eqref{borderline_eq_tkachev_2} the function
$$
\phi = \big(f_k(u)\big)_+^{\alpha-1} \psi,
$$
with $(f_k)_+ = \max\{f_k,0\}$ the positive part of $f_k$, and with $\alpha$ a fixed number satisfying
$$
\alpha > \max\left\{4, d_0, \frac{d_0 -\mu}{1-\mu}\right\}.
$$
Define the open set $U_k = \{ x : f_k(u(x))>0\}$ and note that $U_k \uparrow U$ by the monotone convergence of $f_k$, thus $U_k \neq \emptyset$ and $\phi \not \equiv 0$ on $U_k$ for large $k$. Using $l(t) \ge C_2 \varphi(t)$, we obtain
$$
\begin{array}{lcl}
\disp C_2 \int_{U_k} b|f_k(u)|^\alpha \varphi(|\nabla u|)\psi & \le & \disp \int_{U_k} b|f_k(u)|^\alpha l(|\nabla u|)\psi \\[0.5cm]
& \le & \disp  \disp - \int_{U_k} f_k(u)^{\alpha-1} \langle \frac{\varphi(|\nabla u|)}{|\nabla u|} \nabla u, \nabla \psi\rangle \\[0.5cm]
& & \disp - (\alpha-1)\int_{U_k} f_k(u)^{\alpha-2}f_k'(u) \psi \varphi(|\nabla u|)|\nabla u| \\[0.5cm]
& \le & \disp - \int_{U_k} f_k(u)^{\alpha-1} \langle \frac{\varphi(|\nabla u|)}{|\nabla u|} \nabla u, \nabla \psi\rangle,
\end{array}
$$
where we used $f_k'\ge 0$, $\psi \ge 0$ and $\varphi \ge 0$ to get rid of the second integral in the right hand-side. Applying the Cauchy-Schwarz and H\"older inequalities, we thus get
$$
\begin{aligned}
\disp C_2 \int_{U_k} b|f_k(u)|^\alpha \varphi(|\nabla u|)\psi & \le \disp \int_{U_k} |f_k(u)|^{\alpha-1}\varphi(|\nabla u)|\nabla \psi| \\
& \le \disp \left\{ \int_{U_k} b|f_k(u)|^\alpha \varphi(|\nabla u|)\psi\right\}^{\frac{\alpha-1}{\alpha}} \left\{\int_{U_k} \varphi(|\nabla u|) b^{1-\alpha} \frac{|\nabla \psi|^\alpha}{\psi^{\alpha-1}}\right\}^{{1}/
{\alpha}},
\end{aligned}
$$
whence, rearranging and using the boundedness of $\varphi$,
$$
\int_{U_k} b|f_k(u)|^\alpha \varphi(|\nabla u|)\psi \le C_3 \int_{U_k} b^{1-\alpha} \frac{|\nabla \psi|^\alpha}{\psi^{\alpha-1}} \le C_3 \int_{M} b^{1-\alpha} \frac{|\nabla \psi|^\alpha}{\psi^{\alpha-1}} 
$$
for some constant $C_3>0$ depending on $\alpha$.
Let $\psi(x)= \psi_j(x) = \gamma(r(x)/R_j)$, where $\gamma \in \lip(\R)$ is such that
\begin{equation}\label{def_cutoff_tkachev}
\gamma(t) = \left\{ \begin{array}{ll}
1 & \quad \text{on } \, [0,1] \\[0.2cm]
(2-t)^\alpha & \quad \text{on } \, [1,2) \\[0.2cm]
0 & \quad \text{on } \, (2, \infty).
\end{array}\right.
\end{equation}
Note that $\psi_j \ra 1$ locally uniformly on $M$, and that $|\gamma'|^\alpha/\gamma^{\alpha-1} = \alpha^\alpha$ is bounded on $[1,2]$. Using our bounds on $b$, the coarea formula and integrating by parts, we deduce
\begin{equation}\label{buona}
\begin{array}{lcl}
\disp \int_{U_k} b|f_k(u)|^\alpha \varphi(|\nabla u|)\psi_j & \le & \disp \frac{C_4}{R_j^\alpha} \int_{R_j}^{2R_j} \vol(\partial B_t) (1+t)^{\mu(\alpha-1)}\di t \\[0.5cm]
& = & \disp \frac{C_4}{R_j^\alpha} \left\{ \left[\vol(B_t)(1+t)^{\mu(\alpha-1)}\right]_{R_j}^{2R_j} \right. \\[0.5cm]
& & \disp \left. - \mu(\alpha-1) \int_{R_j}^{2R_j} \vol(B_t) (1+t)^{\mu(\alpha-1)-1}\di t \right\}.
\end{array}
\end{equation}
for some constant $C_4>0$. From \eqref{ipovol_poli_tkachev} we get
\begin{equation}\label{nonnnn}
\begin{array}{l}
\disp \int_{U_k} b|f_k(u)|^\alpha \varphi(|\nabla u|)\psi_j  \\[0.5cm]
\qquad \le \disp \ \ \frac{C_4}{R_j^\alpha} \left\{ C_5 R_j^{d_0 + \mu(\alpha-1)} - \mu(\alpha-1) \int_{R_j}^{2R_j} \vol(B_t) (1+t)^{\mu(\alpha-1)-1}\di t \right\}.
\end{array}
\end{equation}
If $\mu\ge 0$ we get rid of the integral in brackets, while if $\mu <0$ we use inequality $\vol(B_t) \le \vol(B_{2R_j})$, integrate $(1+t)^{\mu(\alpha-1)-1}$ and exploit  \eqref{ipovol_poli_tkachev}. In both of the cases, from \eqref{nonnnn} we infer the existence of a constant $C_6>0$ such that 
$$
\begin{array}{l}
\disp \int_{U_k} b|f_k(u)|^\alpha \varphi(|\nabla u|)\psi_j \le C_6 R_j^{d_0 + \mu(\alpha-1)-\alpha},
\end{array}
$$
and letting $k \ra \infty$ we get
\begin{equation}\label{nonnnn}
\begin{array}{l}
\disp \int_{U} b|f(u)|^\alpha \varphi(|\nabla u|)\psi_j \le C_6 R_j^{d_0 + \mu(\alpha-1)-\alpha}.
\end{array}
\end{equation}
Because of our choice of $\alpha$, the exponent of $R_j$ is negative. Letting $j \ra \infty$ and using $b>0$, $\psi_j \ra 1$ we deduce $f(u) \varphi(|\nabla u|) \equiv 0$ on $U$, as claimed. Moreover, from  $f(u)>0$ on $U$ we get $\varphi(|\nabla u|)=0$ on $U$. Next, if $\varphi>0$ on $\R^+$ then $\nabla u =0$ on $U$, that is, $u$ is constant on connected components of $U$. We claim that this is impossible unless $U$ is empty. Indeed, if $\partial U = \emptyset$ we deduce that $u$ must be globally constant, contradicting our assumption. On the other hand, if $\partial U \neq \emptyset$ then by continuity $f(u)=0$ on $\partial U$, and thus $f(u)=0$ on the entire $U$, contradicting the very definition of $U$. In conclusion, if $\varphi>0$ on $\R^+$ then $U$ is empty, that is, $f(u) \le 0$ on $M$. \par
If $u$ solves $(P_=)$ and is non-constant, we apply the first part of Theorem \ref{teo_tkachev} both to $u$ and to $v = -u$, which solves 
$$
\Delta_\varphi v \ge b(x) \bar f(v) l(|\nabla v|) \qquad \text{with } \, \bar f(t) = -f(-t),
$$
to deduce both $f(u)\varphi(|\nabla u|)\le 0$ and $\bar f(v) \varphi(|\nabla v|) \le 0$ on $M$. The conclusion follows since $\bar f(v) = -f(u)$. 
\end{proof}

\begin{remark}\label{rem_ventkachev}
\emph{Since $\varphi$ is bounded, choosing $l \equiv 1$ in Theorem \ref{teo_tkachev} we include solutions of 
\begin{equation}\label{nongradient}
\Delta_\varphi u \ge b(x) f(u) \qquad \text{on } \, M.
\end{equation}
However, a minor modification of the above proof shows that, in fact, if $u$ solves \eqref{nongradient} then the stronger $f(u) \le 0$ holds on $M$, regardless of the behaviour of $\varphi$. With the equality sign, \eqref{nongradient} has been considered in \cite{tkachev}, see also \cite{nu}, while in \cite{Serrin_4} the author investigated more general equalities of the type 
$$
\diver \mathbf{A}(x,u,\nabla u) = b(x) f(u),
$$ 
where $\mathbf{A}(x,u,\nabla u) \le C r(x)^\lambda$, cf. also \cite{DAmbrMit, farinaserrin1, farinaserrin2}.
}
\end{remark}

It is instructive to compare Theorem \ref{teo_tkachev} with Theorem \ref{teo_SMP_SL_mc} and Corollary \ref{cor_strano}. First, we observe that if $f\le 0$ on $\R$ the conclusion of Theorem \ref{teo_tkachev} is straightworward. Otherwise, since $f$ is increasing, there exists a constant $C>0$ such that $f(t) \ge C >0$ for $t >>1$. Hence, Theorem \ref{teo_tkachev} considers the range 
$$
\omega= \chi = 0, 
$$
that is not covered by Theorem \ref{teo_SMP_SL_mc} (cf. Remark \ref{rem_impo_mc}). The sharpness of $\mu<1$ in Theorem \ref{teo_tkachev} follows from the counterexample in Subsection \ref{sec_sharpness_SL}: otherwise, if $\mu=1$, we can choose $\alpha=-2$ (hence, $M$ of polynomial growth) and $\chi=\omega =0$ (by Remark \ref{rem_impo_mc_2}) to produce a non-constant smooth solution of
$$
\diver \left( \frac{\nabla u}{\sqrt{1+|\nabla u|^2}} \right) \ge C\big( 1+r(x)\big)^{-1} \frac{|\nabla u|}{\sqrt{1+|\nabla u|^2}} \qquad \text{on } \, M.
$$
Also, Theorem \ref{teo_tkachev} is specific to operators of mean curvature type, that is, those satisfying $\varphi \le C_1$ on $\R$. To see it, suppose that $\varphi$ is unbounded, more precisely that
\begin{equation}\label{varphi_notmc}
t \varphi'(t) \ge c_1\varphi(t) \qquad \text{on } \, \R^+,
\end{equation}
for some constant $c_1>0$. Note that, by integration, $\varphi(t) \ge c_2 t^{c_1}$ for some positive $c_2$. For such $\varphi$, we are going to produce
\begin{itemize}
\item[$(i)$] a manifold $M$ satisfying \eqref{ricciassu_SL_mc}, for any chosen $\alpha \ge -2$ (in particular, for $\alpha = -2$, geodesic balls in $M$ grow polynomially), and
\item[$(ii)$] for each $\mu \in \R$, a $\lip_\loc$, non-negative unbounded solution $u$ of 
\begin{equation}\label{eq_borderline_mc}
\Delta_\varphi u \ge C \big( 1+ r(x)\big)^{-\mu} \varphi(|\nabla u|) \qquad \text{on } \, M,
\end{equation}
\end{itemize}
for some constant $C>0$. The combination of $(i)$ and $(ii)$ with $\alpha = -2$ show the failure of Theorem \ref{teo_tkachev} for operators satisfying \eqref{varphi_notmc}. Consider the model manifold in Subsection \ref{sec_sharpness_SL}. For a smooth, radial function $u = w(r)$ with $w$ convex and strictly increasing, by \eqref{laplar_sotto_mc} we compute 
$$
\begin{array}{lcl}
\Delta_\varphi u & = & \disp \varphi'(w')w'' + \varphi(w') \Delta r \ge \left[ c_1 \frac{w''}{w'} + \Delta r \right] \varphi(w') \\[0.4cm]
& \ge & \disp \left[ c_1 \frac{w''}{w'} + c(1+ r^2)^{\alpha/4}\right] \varphi(w'),
\end{array}
$$
for some constant $c>0$. Therefore, if we choose 
$$
w(r) = \int_0^r \exp\left\{ (1+t^2)^\sigma\right\} \di t,
$$
then 
$$
\Delta_\varphi u \ge c_3(1+ r^2)^{\max\{\sigma-1, \alpha/4\}} \varphi(|\nabla u|),
$$
and $u$ solves \eqref{eq_borderline_mc} whenever $\sigma \ge 1 -\mu/2$, as claimed. \\[0.2cm]
\noindent 
%
%
%
In the next result, we show how the technique in Theorem \ref{teo_tkachev} can be adapted to handle $(P_\ge)$ with a more general gradient term $l(|\nabla u|)$ that is not necessary borderline, and with no bound on the decay of $b$. In this case, however, a slow volume growth is needed. 

\begin{theorem}\label{teo_tkachev_pocogrowth}
Let $M$ be a complete Riemannian manifold, consider 
$$
\left\{\begin{array}{l}
\varphi \in C(\R^+_0), \qquad 0 \le \varphi \le C_1 \quad \text{on } \, \R^+_0; \\[0.2cm]
f \in C(\R), \qquad f  \ \text{ non-decreasing on } \, \R;\\[0.2cm]
l \in C(\R^+_0), \qquad l \ge 0 \quad \text{on } \, \R^+; \\[0.2cm]
b \in C(M), \qquad b>0 \ \text{ on } \, M, 
\end{array}\right.
$$
for some constant $C_1>0$. Let $u \in \lip_\loc(M)$ be a non-constant solution of 
\begin{equation}\label{nonbord_eq_tkachev}
\Delta_\varphi u \ge b(x) f(u) l(|\nabla u|) \qquad \text{on } \, M.
\end{equation}
If
\begin{equation}\label{slowvolume_tkachev}
\liminf_{r \ra \infty} \frac{\vol(B_r)}{r} = 0,
\end{equation}
then $f(u)l(|\nabla u|) \le 0$ on $M$. If $u$ is non-constant and solves $(P_=)$, then $f(u)l(|\nabla u|) \equiv 0$ on $M$.
\end{theorem}

\begin{proof}
We proceed as in the proof of Theorem \ref{teo_tkachev}: let $\{f_k\}$ a sequence of increasing, locally Lipschitz functions converging to $f$ from below, and note that $u$ solves
\begin{equation}\label{nonbord}
\Delta_\varphi u \ge b(x) f_k(u) l(|\nabla u|) \qquad \text{on } \, M.
\end{equation}
For $\eps>0$ we define
$$
\eta_\eps(t) = \frac{\big(f_k(t)\big)_+}{\sqrt{\big(f_k(t)\big)_+^2+\eps^2}}.
$$
The monotonicity of $f_k$ implies that $\eta_\eps' \ge 0$. Define $U_k = \{ f_k(u) >0\}$ and $U = \{ f(u)>0\}$, and assume that $U \neq \emptyset$, otherwise the conclusion is immediate. Fix a cut-off function $\psi \in \lip_c(M)$ to be chosen later, insert
$$
\phi = \eta_\eps(u) \psi \in \lip_c(M)
$$
in the weak definition of \eqref{nonbord} and apply Cauchy-Schwarz inequality to deduce
$$
\disp \int b \eta_\eps(u)f_k(u) l(|\nabla u|)\psi \le \disp - \int \eta_\eps(u) \langle \frac{\varphi(|\nabla u|)}{|\nabla u|} \nabla u, \nabla \psi\rangle \le \disp \int \eta_\eps(u) \varphi(|\nabla u|)|\nabla \psi|.
$$
Letting $\eps \ra 0$, using Lebesgue convergence theorem and the boundedness of $\varphi$ we get
\begin{equation}\label{tka}
\disp \int_{U_k} b \big(f_k(u)\big)_+ l(|\nabla u|)\psi \le \int_{U_k} \varphi(|\nabla u|)|\nabla \psi| \le C_1 \int_M |\nabla \psi|.
\end{equation}
Fix a diverging sequence $\{R_j\}$ such that $\{2R_j\}$ realizes the liminf in \eqref{slowvolume_tkachev}, and define $\psi(x)= \psi_j(x) = \gamma(r(x)/R_j)$, where $\gamma \in \lip(\R)$ satisfies
$$
\gamma =1 \quad \text{on } \, [0,1), \qquad \gamma = 0 \quad \text{on } \, (2,\infty), \qquad \gamma(t) = 2-t \quad \text{on } \, [1,2].
$$
Evaluating \eqref{tka} with $\psi = \psi_j$ and letting $k \ra \infty$ we obtain
$$
\disp \int_U b \big(f(u)\big)_+l(|\nabla u|)\psi_j \le \frac{C_1}{R_j} \vol\big(B_{2R_j}\big).
$$
The conclusion follows by letting $j \ra \infty$, and the case of equality is handled as in Theorem \ref{teo_tkachev}.
\end{proof}

We next consider inequalities $(P_\ge)$ under the validity of the Keller-Osserman condition 
$$
F^{-\frac{1}{\chi+1}} \in L^1(\infty),
$$
when just a volume growth upper bound is imposed on $M$. The main result of this section, Theorem \ref{teo_main}, improves on \cite{prs_gafa,prsmemoirs} (see also \cite{puccirigoli}, Thm. 1.3). Although the proof is still based on the delicate iteration argument in \cite{prs_gafa,prsmemoirs}, the presence of a nontrivial gradient term $l$ calls for new estimates, inspired by recent work in \cite{farinaserrin2}.\par
In this section we assume
\begin{equation}\label{assu_WPC_solop1}
\varphi (t) \le Ct^{p-1} \qquad \text{for some } p>1, \ C>0 \quad {\rm and} \ t \in \R^+.
\end{equation}
\begin{theorem}\label{teo_main}
Let $M$ be a complete Riemannian manifold, and consider $\varphi, b,f,l$ meeting assumptions \eqref{assumptions}, \eqref{assumptions_bfl} and \eqref{assu_WPC_solop1}, for some $p>1$. Assume that, for some $\mu,\chi, \omega \in \R$ with
\begin{equation}\label{pararange_2_infty}
\chi \ge 0, \qquad \mu \le \chi+1, \qquad \omega >\chi
\end{equation}
the following inequalities are satisfied:
\begin{equation}\label{assum_main_2_fondo}
\begin{array}{ll}
b(x) \ge C\big(1+r(x)\big)^{-\mu} & \quad \text{on } \, M, \\[0.2cm]
f(t) \ge Ct^\omega & \quad \text{for } \, t \gg 1\\[0.2cm]
\disp l(t) \ge C\frac{\varphi(t)}{t^{\chi}} & \quad \text{on } \, \R^+,
\end{array}
\end{equation}
for some constant $C>0$. Let $u \in \lip_\loc(M)$ be a non-constant solution of $(P_\ge)$ on $M$, and suppose that either 
\begin{equation}\label{volgrowth_sigmazero_Linfty_inthetheorem}
\begin{array}{lll}
\mu < \chi+1 & \text{and} & \disp \qquad \liminf_{r \ra \infty} \frac{\log\vol(B_r)}{r^{\chi+1-\mu}} < \infty \quad \text{($=0$ if $\chi=0$)}; \\[0.4cm]
\text{or}\\[0.1cm]
\mu = \chi+1 & \text{and} & \disp \qquad \liminf_{r \ra \infty} \frac{\log\vol(B_r)}{\log r} < \infty \quad \text{($\le p$ if $\chi=0$)}.
\end{array}
\end{equation}
Then, $u$ is bounded above and $f(u^*) \le 0$. In particular in case $f>0$ on $\R^+$, $\slio$ holds.\\
\end{theorem}


\begin{remark}\label{rem_Linftyestimates}
\emph{In the Euclidean space $\R^m$, and when the third in \eqref{assum_main_2_fondo} is replaced with the stronger $l(t) \ge C t^{p-1-\chi}$, Liouville type results covering some of the cases in Theorem \ref{teo_main} have been obtained by various authors (in some instances, even for more general quasilinear operators). Among them, we stress Thm 1 in \cite{farinaserrin2}, that considers the entire range \eqref{pararange_2_infty}. However, if $\mu = \chi+1$, the authors further need $p>m$ independently of the value of $\chi$, a quite stronger requirement than the second in \eqref{volgrowth_sigmazero_Linfty_inthetheorem}. Previous work in \cite{filippucci} considered the case $0<\chi \le p-1$, $\omega>\chi$ and $\mu<\chi+1$ under the restriction\footnote{See Corollaries 1.3 and 1.4 in \cite{filippucci}; the bound $p \in (1,m)$ is assumed at p.2904.} $p \in (1,m)$, for operators close either to the $p$-Laplacian or to the mean curvature ones.\\
The existence of a Liouville theorem for $\mu = \chi+1$ and $l(t) \equiv 1$ was conjectured by Mitidieri-Pohozaev in \cite[Sect. 14 Ch. 1]{mitpoho}, and has previously been proved in \cite{nu} (for the $p$-Laplace operator) and \cite{usami} (for the mean curvature operator), in both cases on $\R^m$.
}
\end{remark}


Theorem \ref{teo_main} is a consequence of Theorem \ref{teo_main_2} and of the next

\begin{proposition}\label{lem_importante}
Let $M$ be a complete Riemannian manifold, and let $\varphi$ satisfy \eqref{assumptions}, and \eqref{assu_WPC_solop1} with $p>1$. Fix $\mu, \omega, \chi \in \R$ satisfying
\begin{equation}\label{assu_musigmachi_33}
\chi \ge 0, \qquad \mu \le \chi+1, \qquad \omega > \chi,
\end{equation}
and assume either one of the following requirements:
\begin{equation}\label{volgrowth_sigmazero_Linfty}
\begin{array}{lll}
\mu < \chi+1 & \text{and} & \disp \qquad \liminf_{r \ra \infty} \frac{\log\vol(B_r)}{r^{\chi+1-\mu}} < \infty \\[0.2cm]
\text{or} \\[0.2cm]
\mu = \chi+1 & \text{and} & \disp \qquad \liminf_{r \ra \infty} \frac{\log\vol(B_r)}{\log r} < \infty \quad \text{($\le p$ if $\chi=0$)}.
\end{array}
\end{equation}
If $u \in \lip_\loc(M)$ 
\begin{equation}\label{ineq_superlevel_33}
\Delta_\varphi u \ge K(1+r)^{-\mu}u^\omega \frac{\varphi(|\gru|)}{|\gru|^{\chi}} \qquad \text{on } \, \Omega_{\eta} = \{x\in M \ : \ u(x)>\eta\} \neq \emptyset,
\end{equation}
for some $\eta>0$, then $u$ is bounded above.
\end{proposition}

\begin{remark}
\emph{Although we require no upper bound on $\varphi(t)/t^\chi$ in a neighbourhood of zero, the weak inequality \eqref{ineq_superlevel_33} implicitly assumes the term $\varphi(|\nabla u|)/|\nabla u|^\chi$ to be locally integrable on $\Omega_\eta$.
}
\end{remark}

\begin{proof}
Suppose by contradiction that $u^*=\infty$. Fix $\gamma > \eta$, and take $\lambda \in C^1(\R)$ such that
$$
0 \le \lambda \le 1, \quad \lambda'\ge 0, \quad \lambda\equiv 0 \, \text{ on } \, (-\infty,\gamma], \quad  \lambda>0 \, \text{ on } \, (\gamma, \infty).
$$
Let $\psi \in C^\infty_c(M)$ be a cut-off function, and let $\varsigma,\alpha>1$ to be specified later. We plug the non-negative test function
$$
\phi = \psi^\varsigma \lambda(u)u^\alpha \in \lip_c(M)
$$
in the weak definition of \eqref{ineq_superlevel_33} to deduce, using $\lambda' \ge 0$ and \eqref{assum_main_2_fondo}, 
\begin{equation}\label{beginning}
\begin{aligned}
\disp K\int \psi^\varsigma \lambda \frac{u^{\alpha+\omega}}{(1+r)^\mu} \frac{\varphi(|\gru|)}{|\gru|^ \chi} & \le  - \disp \int \frac{\varphi(|\gru|)}{|\gru|} \langle \gru, \nabla (\psi^\varsigma \lambda u^\alpha) \rangle \\
& \le \disp \varsigma \int \psi^{\varsigma-1} \lambda u^\alpha \varphi(|\gru|)|\nabla \psi| \\
 &\qquad -\alpha \int \psi^\varsigma \lambda u^{\alpha-1} \varphi(|\gru|) |\gru|.
\end{aligned}
\end{equation}
We divide the proof into several steps: \\[0.2cm]
\noindent \textbf{Step 1: basic growth estimates.}\\
The following inequalities hold:
\begin{itemize}
\item[-] If $\mu < \chi+1$, then for each $q >0$ there exists $\alpha_q>1$ and a constant $C_q$ depending on $p,q,\chi,\mu,\omega$ such that, if $\alpha \ge \alpha_q$,
\begin{equation}\label{esti_step1}
\int_{B_R\cap \Omega_\gamma} \lambda \frac{u^{\alpha+\omega}}{(1+r)^{\mu}}\frac{\varphi(|\gru|)}{|\gru|^ \chi} \le C_q \frac{\vol(B_{2R})}{R^q}.
\end{equation}
\item[-] If $\mu = \chi+1$, then there exists a constant $C$ depending on $p,\chi,\mu, \omega$ such that
\begin{equation}\label{esti_step1_poli}
\int_{B_R\cap \Omega_\gamma} \lambda \frac{u^{\alpha+\omega}}{(1+r)^{\mu}}\frac{\varphi(|\gru|)}{|\gru|^ \chi} \le C \frac{\vol(B_{2R})}{R^p}.
\end{equation}
\end{itemize}

\begin{proof}[Proof of Step 1] The argument is an adaptation of Lemma 2.2 in \cite{farinaserrin2}, and rests on the use of the triple
Young inequality to the first term on the right-hand side of \eqref{beginning}: we need to find $z_1,z_2,z_3>1$ satisfying
\begin{equation}\label{expoholder}
\frac{1}{z_1} + \frac{1}{z_2} + \frac{1}{z_3} = 1
\end{equation}
and $\tau,\bar C>0$ such that
\begin{equation}\label{desejavel}
\varsigma \psi^{\varsigma-1} \lambda u^\alpha \varphi(|\gru|)|\nabla\psi| = \mathcal{J}_1^{\frac{1}{z_1}}\mathcal{J}_2^{\frac{1}{z_2}}\mathcal{J}_3^{\frac{1}{z_3}},
\end{equation}
with
\begin{equation}\label{choices}
\begin{array}{lcl}
\mathcal{J}_1 & = & \disp \frac{K}{2} \psi^\varsigma \lambda \frac{u^{\alpha + \omega}}{(1+r)^\mu} \frac{\varphi(|\gru|)}{|\gru|^\chi} \\[0.5cm]
\mathcal{J}_2 & = & \disp \alpha \psi^\varsigma \lambda u^{\alpha-1} \varphi(|\gru|)|\gru| \\[0.5cm]
\mathcal{J}_3 & = & \disp \bar C (1+r)^{\tau} \left[\frac{\varphi(|\gru|)}{|\gru|^{p-1}}\right]|\nabla \psi|^{z_3}.
\end{array}
\end{equation}
considering powers of $u$, $|\gru|$, $r$ and $\psi$, to obtain \eqref{desejavel} we need the following balancing:
$$
\begin{array}{rllrll}
i) & \text{powers of $u$:} & \ \disp \alpha = \frac{\alpha + \omega}{z_1} + \frac{\alpha-1}{z_2} & \quad ii) & \text{powers of $|\gru|$:} & \ \disp 0 = -\frac{\chi}{z_1} + \frac{1}{z_2} - \frac{p-1}{z_3}\\[0.5cm]
iii) & \text{powers of $r$:} &  \ \disp 0 = - \frac{\mu}{z_1} + \frac{\tau}{z_3} & \quad iv) & \text{powers of $\psi$:} & \ \disp \varsigma -1 = \frac{\varsigma}{z_1} + \frac{\varsigma}{z_2}.
\end{array}
$$
To find $z_1,z_2,z_3$ note that, by \eqref{expoholder}, the equality for $|\nabla u|$ can be rewritten as
$$
p-1 = \frac{p-1-\chi}{z_1} + \frac{p}{z_2}
$$
Thus, solving the equations for $u, |\nabla u|$ with respect to $z_1$ and $z_2$, and then recovering $z_3$ from \eqref{expoholder}, we get
$$
\begin{array}{c}
\disp \frac{1}{z_1} = \frac{\alpha + p-1}{(\chi+1)(\alpha-1)+ p(\omega+1)}, \qquad \disp \frac{1}{z_2} = \frac{\chi\alpha +(p-1)\omega}{(\chi+1)(\alpha-1)+ p(\omega+1)}, \\[0.5cm]
\disp \qquad \frac{1}{z_3} = \frac{\omega-\chi}{(\chi+1)(\alpha-1)+ p(\omega+1)}
\end{array}
$$
(these are positive numbers less than $1$ because of \eqref{assu_musigmachi_33}), and from the last two equations,
$$
\tau = \mu \frac{z_3}{z_1} = \mu \frac{\alpha+p-1}{\omega-\chi}, \qquad \varsigma = z_3.
$$
The constant $\bar C$ is then uniquely determined by \eqref{desejavel}. Having found the right parameters, from the triple Young inequality
$$
\mathcal{J}_1^{\frac{1}{z_1}}\mathcal{J}_2^{\frac{1}{z_2}}\mathcal{J}_3^{\frac{1}{z_3}} \le \mathcal{J}_1 + \mathcal{J}_2 + \mathcal{J}_3,
$$
and plugging into \eqref{desejavel} and \eqref{choices} we deduce
$$
\begin{array}{lcl}
\disp \varsigma \psi^{\varsigma-1} \lambda u^\alpha \varphi(|\gru|)|\nabla \psi| & \le & \disp \frac{K}{2} \psi^\varsigma \lambda \frac{u^{\alpha + \omega}}{(1+r)^\mu} \frac{\varphi(|\gru|)}{|\gru|^\chi} + \alpha \psi^\varsigma \lambda u^{\alpha-1} \varphi(|\gru|)|\gru| \\[0.5cm]
& & + \disp \bar C (1+r)^{\mu \frac{z_3}{z_1}} \left[\frac{\varphi(|\gru|)}{|\gru|^{p-1}}\right]^{z_3}|\nabla  \psi|^{z_3}.
\end{array}
$$
Inserting into \eqref{beginning} and using \eqref{assu_WPC_solop1} we get
\begin{equation}\label{beginning_3}
\begin{aligned}
\disp \frac{K}{2} \int \psi^\varsigma \lambda \frac{u^{\alpha+\omega}}{(1+r)^\mu}\frac{\varphi(|\gru|)}{|\gru|^\chi} & \le C_1 \int  (1+r)^{\mu \frac{z_3}{z_1}} |\nabla  \psi|^{z_3}.
\end{aligned}
\end{equation}
For large $R>1$, we choose $\psi \in C^\infty_c(M)$ satisfying
\begin{equation}\label{propriepsi_1}
0 \le \psi \le 1, \quad \psi \equiv 1 \ \ \text{ on } \, B_R, \quad \psi \equiv 0 \ \ \text{ on } \, M\backslash B_{2R}, \quad |\nabla \psi| \le \frac{C}{R},
\end{equation}
for an absolute constant $C$. Using \eqref{propriepsi_1} and the fact that $\lambda=0$ when $u \le \gamma$,  we obtain
\begin{equation}\label{farise}
\begin{array}{lcl}
\disp \frac{K}{2} \int_{B_R \cap \Omega_\gamma} \lambda \frac{u^{\alpha+\omega}}{(1+r)^\mu}\frac{\varphi(|\gru|)}{|\gru|^\chi} & \le & \disp \frac{K}{2} \int \psi^\varsigma \lambda \frac{u^{\alpha+\omega}}{(1+r)^\mu}\frac{\varphi(|\gru|)}{|\gru|^\chi} \\[0.5cm]
& \le & \disp \frac{C_2}{R^{z_3}}\int_{B_{2R}} (1+r)^{\mu \frac{z_3}{z_1}} \\[0.5cm]
& \le & C_3 R^{\mu \frac{z_3}{z_1} - z_3} \vol(B_{2R}).
\end{array}
\end{equation}
The exponent of $R$ in \eqref{farise} can be written as
\begin{equation}\label{beautiful}
\mu \frac{z_3}{z_1} - z_3 = \frac{\alpha + p-1}{\omega-\chi}(\mu-\chi-1) - p.
\end{equation}
We examine the two cases, according to whether $\mu< \chi+1$ or $\mu = \chi+1$.
\begin{itemize}
\item[-] If $\mu < \chi+1$, then for any given $q>0$ we can choose $\alpha_q$ sufficiently large that, for $\alpha \ge \alpha_q$,
$$
\mu \frac{z_3}{z_1} - z_3 \le -q.
$$
Having fixed such $\alpha_q$, from \eqref{farise} we get
\begin{equation}
\int_{B_R \cap \Omega_\gamma} \lambda \frac{u^{\alpha+\omega}}{(1+r)^{\mu}}\frac{\varphi(|\gru|)}{|\gru|^\chi} \le \disp \frac{C_q}{K} \frac{\vol(B_{2R})}{R^q},
\end{equation}
and the thesis follows.
\item[-] If $\mu = \chi+1$, then by \eqref{beautiful} the exponent of $R$ in \eqref{farise} is $-p$ independently of $\alpha$, and we obtain
\begin{equation}
\int_{B_R \cap \Omega_\gamma} \lambda\frac{u^{\alpha+\omega}}{(1+r)^{\mu}}\frac{\varphi(|\gru|)}{|\gru|^\chi} \le \disp \frac{C_5}{K} \frac{\vol(B_{2R})}{R^p}.
\end{equation}
as claimed.
\end{itemize}
\end{proof}

%
\noindent \textbf{Step 2: a preliminary inequality.}\\
We consider again \eqref{beginning}, but we are going to choose $\alpha, \varsigma > \chi+1$ in a way different to the one in Step 1. \\[0.1cm]
\noindent \textbf{Case 1: $\chi>0$}.\\
We use Young's inequality with exponents $\chi+1$ and $(\chi+1)/\chi$ to remove the second term in the right-hand side of \eqref{beginning}: for each $\eps>0$, we get
\begin{equation}\label{young}\begin{aligned}
\varsigma\int \psi^{\varsigma-1} \lambda u^\alpha \varphi(|\gru|)|\nabla \psi| \le \frac{\varsigma}{(\chi+1)\,\eps^{\chi+1}}& \int \psi^{\varsigma-\chi-1} \lambda u^{\alpha+ \chi} \frac{\varphi(|\gru|)}{|\gru|^\chi}|\nabla \psi|^{\chi+1}\\
& + \frac{\chi\varsigma\eps^{\frac{\chi+1}{\chi}}}{\chi+1} \int \psi^\varsigma \lambda u^{\alpha-1} \varphi(|\gru|)|\gru|.
\end{aligned}\end{equation}
choosing $\eps$ such that
$$
\frac{\chi\varsigma\eps^{\frac{\chi+1}{\chi}}}{\chi+1}= \alpha, \quad \text{that is,} \quad \eps = \left(\frac{\alpha (\chi+1)}{\chi\varsigma}\right)^{\frac{\chi}{\chi+1}} ,
$$
and inserting \label{young} into \eqref{beginning}, we obtain
\begin{equation}\label{passobase_22}
\disp K\int \psi^\varsigma \lambda \frac{u^{\alpha+\omega}}{(1+r)^\mu}\frac{\varphi(|\gru|)}{|\gru|^\chi} \le \disp \frac{C_1\varsigma^{\chi+1}}{\alpha^{\chi}} \int \psi^{\varsigma-\chi-1} \lambda u^{\alpha+ \chi} \frac{\varphi(|\gru|)}{|\gru|^\chi}|\nabla \psi|^{\chi+1}
\end{equation}
for some constant $C_1= C_1(\chi)>0$.\\[0.1cm]
\noindent \textbf{Case 2: $\chi=0$}.\\
In this case, \eqref{passobase_22} with $\chi=0$ and $C_1=1$ directly follows from \eqref{beginning}, getting rid of the second term on the right-hand side.\\[0.2cm]

\noindent \textbf{Step 3: induction for $\mu < \chi+1$.}\\
If $\mu < \chi+1$, the following inductive relation holds:
\begin{equation}\label{thefinal}
\int_{B_R\cap \Omega_\gamma} \lambda (1+r)^{-\mu} \frac{\varphi(|\gru|)}{|\gru|^\chi} \le 2^{-BR^\theta}\left[\int_{B_{2R}\cap \Omega_\gamma} \lambda (1+r)^{-\mu} \frac{\varphi(|\gru|)}{|\gru|^\chi}\right],
\end{equation}
where
\begin{equation}\label{Btheta}
B = \frac{C_6\gamma^{\omega-\chi}}{\omega-\chi}, \qquad \theta = \chi+1-\mu >0,
\end{equation}
and $C_6= C_6(K,\omega, \chi,\mu)$ is a positive constant independent of $\gamma, R$.
\begin{proof}[Proof of Step 3] Fix $\xi>1$ close enough to $1$ in order to satisfy
\begin{equation}\label{ipoxi}
\omega-\chi - (\chi+1)\left(1-\frac{1}{\xi}\right) > 0
\end{equation}
and, for $R \ge 2$ choose a cut-off function $\psi \in \lip_c(B_{2R})$ such that
\begin{equation}\label{ipopsi}
0 \le \psi \le 1, \qquad \psi \equiv 1 \ \text{ on } B_R, \qquad |\nabla\psi| \le \frac{C}{R}\psi^{1/\xi},
\end{equation}
for some $C=C(\xi)$. Note that this is possible since $\xi>1$ (for instance, one can take the cut-off in \eqref{propriepsi_1}, call it $\psi_0$, and consider $\psi = \psi_0^{\xi/(\xi-1)}$). \\
Choose $\alpha$ and $\varsigma$ in order to satisfy
\begin{equation}\label{ipoalphaeta}
\varsigma = \alpha + \omega, \qquad \alpha > \max\left\{\omega, \chi+1-\omega\right\}.
\end{equation}
However, for the ease of notation we feel convenient to keep $\varsigma$ and $\alpha$ independent in the next computations.
%
By Step 2, inequality \eqref{passobase_22} holds for each $\chi \ge 0$. Using then \eqref{ipopsi}, and since $\{\nabla\psi \neq 0\} \subset B_{2R}\backslash B_R$, $R \ge 2$, from \eqref{passobase_22} we deduce
\begin{equation}\label{starting_2}
\begin{array}{l}
\disp K\int \psi^\varsigma \lambda \frac{u^{\alpha+\omega}}{(1+r)^\mu}\frac{\varphi(|\gru|)}{|\gru|^\chi} \\[0.5cm]
\qquad \qquad \disp \le \disp \frac{C_2\varsigma^{\chi+1}}{\alpha^{\chi}R^{\chi+1}} \int_{\{\nabla \psi \neq 0\}} \psi^{\varsigma-(\chi+1)\left(1-\frac{1}{\xi}\right)} \lambda u^{\alpha+\chi}\frac{\varphi(|\gru|)}{|\gru|^\chi} \\[0.5cm]
\qquad \qquad \disp \le \disp \frac{C_3\varsigma^{\chi+1}}{\alpha^{\chi}R^{\chi+1-\mu}} \int \psi^{\varsigma-(\chi+1)\left(1-\frac{1}{\xi}\right)} \lambda \frac{u^{\alpha+\chi}}{(1+r)^\mu}\frac{\varphi(|\gru|)}{|\gru|^\chi}
\end{array}
\end{equation}
Since $\omega>\chi$, we can apply H\"older's inequality to the RHS with exponents
\begin{equation}\label{exprepq}
 q = \frac{\alpha+\omega}{\omega-\chi},\qquad q' = \frac{\alpha+\omega}{\alpha+\chi}
\end{equation}
and get
$$
\begin{array}{lcl}
\disp \int \psi^{\varsigma-(\chi+1)\left(1-\frac{1}{\xi}\right)} \lambda \frac{u^{\alpha+\chi}}{(1+r)^\mu}\frac{\varphi(|\gru|)}{|\gru|^\chi} & \le & \disp \left(\int \psi^\varsigma\lambda \frac{u^{q'(\alpha+\chi)}}{(1+r)^\mu}\frac{\varphi(|\gru|)}{|\gru|^\chi}\right)^{1/q'} \\[0.5cm]
& & \disp \cdot \left(\int \psi^{\varsigma-(\chi+1)\,q(1-1/{\xi})} \lambda (1+r)^{-\mu} \frac{\varphi(|\gru|)}{|\gru|^\chi}\right)^{1/q}
\end{array}
$$
Inserting into \eqref{starting_2} we obtain
\begin{equation}\label{basic_passo}
\begin{array}{l}
\disp \disp \int \psi^\varsigma \lambda \frac{u^{\alpha+\omega}}{(1+r)^\mu}\frac{\varphi(|\gru|)}{|\gru|^\chi} \\[0.5cm]
\qquad \qquad \disp \le \disp \left(\frac{C_4\varsigma^{\chi+1}}{\alpha^{\chi}R^{\chi+1-\mu}}\right)^q \int \psi^{\varsigma-(\chi+1)q\left(1-\frac{1}{\xi}\right)} \lambda (1+r)^{-\mu} \frac{\varphi(|\gru|)}{|\gru|^\chi}
\end{array}
\end{equation}
for some $C_4(\chi, \mu, K)$. Now, by \eqref{ipoxi}, \eqref{ipoalphaeta} and \eqref{exprepq},
$$
\varsigma-(\chi+1)\,q\left(1-\frac{1}{\xi}\right) = \frac{\alpha+\omega}{\omega-\chi} \left[ \omega-\chi - (\chi+1)\left(1-\frac{1}{\xi}\right)\right]>0,
$$
hence the term with $\psi$ on the right-hand side of \eqref{basic_passo} can be estimated with one on $B_{2R}$. Together with condition $\alpha>\omega$ in \eqref{ipoalphaeta}, this gives
\begin{equation}\label{basic_passo_2}
\begin{array}{l}
\disp \disp \int \psi^\varsigma \lambda \frac{u^{\alpha+\omega}}{(1+r)^\mu}\frac{\varphi(|\gru|)}{|\gru|^\chi} \\[0.5cm]
\qquad \qquad \disp \le \disp \left(\frac{C_5\alpha}{R^{\chi+1-\mu}}\right)^q \int_{B_{2R}\cap \Omega_\gamma} \lambda (1+r)^{-\mu} \frac{\varphi(|\gru|)}{|\gru|^\chi}
\end{array}
\end{equation}
Now, as $u \ge \gamma$ on the domain where $\lambda(u)$ is positive and not zero, using again the properties of~$\psi$ and the definition of $q$ we finally infer
\begin{equation}\label{basic_passo_2}
\begin{array}{l}
\disp \disp \int_{B_R} \lambda (1+r)^{-\mu} \frac{\varphi(|\gru|)}{|\gru|^\chi} \\[0.5cm]
\qquad \qquad \disp \le \disp \left(\frac{C_5\alpha}{R^{\chi+1-\mu}\gamma^{\omega -\chi}}\right)^{\frac{\alpha+\omega}{\omega-\chi}} \int_{B_{2R}} \lambda (1+r)^{-\mu} \frac{\varphi(|\gru|)}{|\gru|^\chi}
\end{array}
\end{equation}
Choose $\alpha$ in such a way that
$$
\frac{C_5\alpha}{R^{\chi+1-\mu}\gamma^{\omega-\chi}} = \frac{1}{2},
$$
that is,
$$
\alpha = \alpha(R) = \frac{\gamma^{\omega-\chi}}{2C_5}R^{\chi+1-\mu} = C_6\gamma^{\omega-\chi}R^{\chi+1-\mu}
$$
Since $\chi+1-\mu>0$ by assumption, if $R$ is big enough then $\alpha$ satisfies \eqref{ipoalphaeta}. Then, setting
$$
\haus(R) = \int_{B_R} \lambda (1+r)^{-\mu} \frac{\varphi(|\gru|)}{|\gru|^\chi},
$$
we have
\begin{equation}\label{thefinal}
\haus(R) \le \disp 2^{-\frac{\alpha+\omega}{\omega-\chi}} \haus(2R) \le 2^{-\frac{\alpha}{\omega-\chi}} \haus(2R) = 2^{-BR^\theta}\haus(2R),
\end{equation}
where
\begin{equation}\label{Btheta}
B = \frac{C_6\gamma^{\omega-\chi}}{\omega-\chi}, \qquad \theta = \chi+1-\mu >0.\\[0.2cm]
\end{equation}
\end{proof}
\noindent \textbf{Step 4: iteration and conclusion for $\mu < \chi+1$.}\\[0.2cm]
Fix $R_0$ big and $\bar R>2R_0$ such that $u$ is not constant on $\Omega_\gamma \cap B_{\bar R}$. Then, $\haus(\bar R)>0$. Consider $R_j=2^j\bar R$, and let $k$ be the integer satisfying $R_k < R \le R_{k+1}$. Iterating \eqref{thefinal} $k$-times and taking the logarithm, we get
$$
\disp \log \haus(\bar R) \le \disp -\left( \sum_{j=0}^{k-1} \bar R^\theta 2^{j\theta} \right) B\log 2 + \log \haus(R_k)  \le -\left( \sum_{j=0}^{k-1} R_{k}^\theta 2^{(j-k)\theta} \right) B\log 2 + \log \haus(R).
$$
Now,
$$
\sum_{j=0}^{k-1} R_k^\theta 2^{(j-k)\theta} = \frac{R_{k+1}^\theta}{2^\theta} \sum_{j=0}^{k-1}2^{(j-k)\theta} = \frac{R_{k+1}^\theta}{2^\theta}  \left( \frac{2^{-\theta} -2^{-(k+1)\theta}}{1-2^{-\theta}}\right) \ge R^{\theta} C_\theta,
$$
for some constant $C_\theta>0$, thus
\begin{equation}\label{almost!!}
\log \haus(\bar R) \le - R^\theta BC_\theta \log 2 + \log \haus(R),
\end{equation}
or in other words,
\begin{equation}\label{done}
\frac{\log \haus(R)}{R^\theta} \ge \frac{\log \haus(\bar R)}{R^\theta} + BC_\theta \log 2.
\end{equation}
By Step 1, for fixed $q=2$ there exists $\alpha_2$ such that, for $\alpha \ge \alpha_2$,
\begin{equation}\label{iiimmm}
\haus(R) \le \frac{1}{\gamma^{\alpha+\omega}} \int_{B_r \cap \Omega_\gamma} \lambda \frac{u^{\alpha+\omega}}{(1+r)^{\mu}} \frac{\varphi(|\gru|)}{|\gru|^\chi} \le \frac{C_q}{\gamma^{\alpha+\omega}}\frac{\vol(B_{2R})}{R^2}.
\end{equation}
Now, choosing $\bar R$ large enough that $\alpha(\bar R) \ge \alpha_2$, plugging \eqref{iiimmm} into \eqref{almost!!}, letting $R \ra \infty$ and using the definition of $B$, because of \eqref{volgrowth_sigmazero_Linfty} we get
$$
\frac{C_6C_\theta \log 2}{\omega-\chi} \gamma^{\omega-\chi} \le \liminf_{R \ra \infty} \frac{\log\vol(B_{2R})}{R^ \theta} < \infty.
$$
However, the assumption $\omega>\chi$ leads to a contradiction provided that $\gamma$ is chosen to be large enough. Therefore, $u^*< \infty$, concluding the proof.\\[0.2cm]

\noindent \textbf{Step 5: conclusion for $\mu = \chi+1$, $\chi=0$.}\\
In this case, let $\{R_j\}$ be a divergent sequence such that $\{2R_j\}$ satisfies the liminf condition in  \eqref{volgrowth_sigmazero_Linfty}. Inserting into \eqref{esti_step1_poli}, using that $\lambda$ is supported on $\{u \ge \gamma\}$ and letting $k \ra \infty$ we deduce
$$
\gamma^{\alpha+\omega} \int_{\Omega_\gamma} \frac{\lambda}{(1+r)^{\mu}}\frac{\varphi(|\gru|)}{|\gru|^\chi} \le \int_{\Omega_\gamma} \lambda \frac{u^{\alpha+\omega}}
{(1+r)^{\mu}}\frac{\varphi(|\gru|)}{|\gru|^\chi} \le \disp C \lim_{j \ra \infty} \frac{\vol(B_{2R_j})}{R_j^p} = C_1.
$$
Without loss of generality, we can suppose $\gamma>1$. The above inequality is then contradicted if $\alpha$ is large enough. Therefore, $u^* < \infty$, as claimed.\\[0.2cm]

\noindent \textbf{Step 6: iteration and conclusion for $\mu = \chi+1$, $\chi>0$.}\\
We begin again with \eqref{passobase_22}, but we fix $\varsigma = \chi+2$. Choosing as $\psi$ a cut-off satisfying
$$
\psi \equiv 1 \quad \text{on } \, B_R, \qquad \psi \equiv 0 \quad \text{on } M \backslash B_{2R}, \qquad |\nabla \psi| \le \frac{2}{R},
$$
we obtain, since $\mu = \chi+1$,
\begin{equation}\label{variante}
\begin{array}{lcl}
\disp \haus_u(R) \doteq \int_{B_R \cap \Omega_\gamma} \lambda \frac{u^{\alpha+\omega}}{(1+r)^\mu}\frac{\varphi(|\gru|)}{|\gru|^\chi} & \le & \disp \frac{C_2}{K\alpha^{\chi}R^{\chi+1}} \int_{B_{2R}\cap \Omega_\gamma} \lambda u^{\alpha+ \chi} \frac{\varphi(|\gru|)}{|\gru|^\chi} \\[0.5cm]
& \le & \disp \disp \frac{C_3}{K\alpha^{\chi}\gamma^{\omega-\chi}} \int_{B_{2R}\cap \Omega_\gamma} \lambda \frac{u^{\alpha+ \omega}}{(1+r)^{\mu}} \frac{\varphi(|\gru|)}{|\gru|^\chi},
\end{array}
\end{equation}
for some $C_3(\chi)$. Since $\chi>0$, for fixed $S>0$ we can choose $\alpha$ large enough to satisfy
$$
\haus_u(R) \le 2^{-S} \haus_u(2R).
$$
Fix $R_0$ big, $\bar R > 2R_0$, $R_i = 2^i \bar R$. For $R > \bar R$, let $k \in \mathbb{N}$ be such that $R_k < R \le R_{k+1}$. Iterating $k$-times and taking the logarithm we deduce
$$
\log \haus_u(\bar R) \le -kS \log 2 + \log \haus_u(R_k) \le -kS \log 2 + \log \haus_u(R)
$$
dividing by $\log R$ and using that $\log R \le (k+1)\log 2 + \log \bar R \le 2k \log 2$ for large enough $R$, we deduce the following inequality:
\begin{equation}\label{andaaa}
\disp \frac{\log \haus_u(\bar R)}{\log R} \le \disp - \frac{Sk}{\log R} \log 2 + \frac{\log \haus_u(R)}{\log R} \le - \frac{S}{2} + \frac{\log \haus_u(R)}{\log R}.
\end{equation}
Now, because of Step 1, $\haus_u(R) \le C_\alpha \vol(B_R)/R^p$, where the constant $C_\alpha$ depends on $\alpha$. If $\{R_j\}$ is a sequence realizing the liminf in \eqref{volgrowth_sigmazero_Linfty},
$$
\limsup_{j \ra \infty} \frac{\log \haus_u(R_j)}{\log R_j} \le \lim_{j \ra \infty} \frac{\log \vol(B_{R_j})}{\log R_j} - p \doteq C_* < \infty.
$$
Inserting into \eqref{andaaa} and letting $j\ra \infty$ we obtain
$$
0 \le - \frac{S}{2} + C_*,
$$
that leads to a contradiction provided that $S$ is chosen large enough. Note that this conclusion does not need $\gamma$ to be large enough, in other words, we showed that $u$ cannot be a non-constant solution of  \eqref{ineq_superlevel} on any upper level set $\Omega_\gamma$.
\end{proof}

We are now ready to prove Theorem \ref{teo_main}.

\begin{proof}[Proof (of Theorem \ref{teo_main})] We first show that $u$ is bounded above. If not, using \eqref{assum_main_2_fondo} we deduce that, for $\eta > 0$ sufficiently large, $u$ would be a non-constant solution of 
\begin{equation}\label{rrrrr}
\Delta_\varphi u  \ge K (1+r)^{-\mu}u^\omega \frac{\varphi(|\nabla u|)}{|\nabla u|^\chi} \qquad \text{on } \, \Omega_\eta = \{x \in M \ : \ u(x)>\eta\} \neq \emptyset,
\end{equation}
for some $K>0$, contradicting Proposition \ref{lem_importante}. Next, we invoke Theorem \ref{teo_main_2} with $\sigma =0$ to deduce that $f(u^*) \le 0$, concluding the proof. If $\mu=\chi+1$ and $\chi>0$, we could argue $f(u^*) \le 0$ directly from Step 6: indeed, if $f(u^*)>0$, then by continuity $u$ would be a non-constant solution of \eqref{rrrrr} on $\Omega_\eta$ for a suitable $K>0$ and $\eta$ close enough to $u^*$. This is impossible by what we observed at the end of Step 6.
\end{proof}

\subsubsection{Yamabe type equations}
With the aid of Theorem \ref{teo_main}, we are able to improve on various geometric corollaries of \cite[Thm. 4.8]{prsmemoirs}. By a way of example, we consider the following conformal rigidity result for manifolds with negative scalar curvature, first investigated by M. Obata in \cite{obata_conf} (in the compact case) and S.T. Yau in \cite{yau_conf}. The geometric conditions in their main theorems have later been  substantially weakened in \cite[Thm. 4.9]{prsmemoirs}, and our next corollary is a mild generalization of it.

\begin{corollary}
Let $(M,\metric)$ be a complete Riemannian manifold of dimension $m \ge 2$ whose scalar curvature $R(x)$ satisfies 
$$
R(x) \le - C\big( 1+r(x)\big)^{-\mu} \qquad \text{on } \, M, 
$$
for some constants $\mu \in \R$, $C \in \R^+$. If either
$$
\begin{array}{lll}
\mu < 2 & \text{and} & \disp \qquad \liminf_{r \ra \infty} \frac{\log\vol(B_r)}{r^{2-\mu}} < \infty, \qquad \text{or} \\[0.4cm]
\mu = 2 & \text{and} & \disp \qquad \liminf_{r \ra \infty} \frac{\log\vol(B_r)}{\log r} < \infty, 
\end{array}
$$
then any conformal diffeomorphism of $M$ preserving $R$ is an isometry.
\end{corollary}

\begin{proof}
Let $T : (M,\metric) \ra (M,\metric)$ be a conformal diffeomorphism, and let $\metricN = T^*\metric = \lambda^2 \metric$ be the conformally deformed metric, for $0< \lambda \in C^\infty(M)$. If $m \ge 3$, writing $\lambda= u^{\frac{2}{m-2}}$ then it is well known that $u$ solves
$$
\Delta u = \frac{R}{c_m}u - \frac{\bar R}{c_m} u^{\frac{m+2}{m-2}} \qquad \text{on } \, M, 
$$
where $\bar R$ is the scalar curvature of $\metricN$, $\Delta$ is the Laplacian of the background metric $\metric$, and $c_m = \frac{4(m-1)}{m-2}$. On the other hand, if $m= 2$, writing $\lambda = e^u$ it holds
$$
2\Delta u = R - \bar R e^{2u} \qquad \text{on } \, M.
$$
Therefore, if $T$ preserves the scalar curvature, 
$$
\Delta u = -R(x)f(u), \qquad \text{with} \qquad f(u) = \left\{ \begin{array}{ll}
\disp \frac{1}{c_m} \left[ u^{\frac{m+2}{m-2}}-u\right] & \quad \text{if } \, m\ge 3, \\[0.4cm]
\disp \frac{1}{2} \left[ e^{2u}-1\right] & \quad \text{if } \, m=2.
\end{array}\right.
$$
We now apply Theorem \ref{teo_main} with $b(x) = -R(x)$, $\varphi(t)=t$ and $\chi=1$ both to $u$ and to $-u$ to deduce that $u$ is bounded and $f(u^*) \le 0 \le f(u_*)$. Hence, $u \equiv 1$ if $m \ge 3$, respectively $u\equiv 0$ if $m =2$, and $T$ is therefore an isometry. 
\end{proof}

For many other applications to Geometry, we refer the reader to \cite{prsmemoirs, AMR_book}. Next, we focus on the mean curvature operator.

\subsubsection{The capillarity equation}\label{subsub_capillarity}

As observed in the Introduction, global solutions $u : \R^m \ra \R$ of the capillary equation

\begin{equation}\label{capillarity_2}
\diver \left( \frac{\nabla u}{\sqrt{1+|\nabla u|^2}}\right) = \kappa(x) u
\end{equation}
have been considered in \cite{tkachev, nu}, with subsequent improvements in \cite{Serrin_4, farinaserrin1}. Combing their results, $u$ must vanish identically provided that 
\begin{equation}\label{lower_kappa_capi}
\kappa(x) \ge C\big( 1+ r(x)\big)^{-\mu}
\end{equation}
on $\R^m$, for some constants $C>0$ and $\mu<2$. In fact, in \cite{farinaserrin1} the authors investigated a more general class of equations including
\begin{equation}\label{capillarity_general}
\diver \left( \frac{\nabla u}{\sqrt{1+|\nabla u|^2}}\right) = \kappa(x) |u|^{\omega-1}u
\end{equation}
on $\R^m$, with $\omega>0$ and $\kappa(x)$ enjoying \eqref{ipo_k_capillarity}, see also Section 5 in \cite{pucciserrin_2}. Applying the Corollary at p. 4387 in \cite{farinaserrin1}, $u \equiv 0$ on $\R^m$ whenever either
\begin{equation}\label{bound_capi_general}
\left\{ \begin{array}{l}
\omega>1, \quad \mu \le 2, \qquad \text{or} \\[0.3cm]
\omega \in (0,1], \quad \mu < \omega+1.
\end{array}\right.
\end{equation}
The upper bound $\mu<2$ is readily recovered for the capillarity equation ($\omega=1$). In a manifold setting, the case $\omega>1$ and $\mu<2$ was already considered in \cite[Thm. 4.8]{prsmemoirs}: with the aid of Theorem \ref{teo_main}, we can improve on it by describing the full range $\omega >0$. In particular, specifying the next theorem to the capillarity problem yields Theorem \ref{teo_capillarity_2} in the Introduction.

\begin{theorem}\label{teo_capillarity_2}
Suppose that $M$ is complete, fix $\omega >0$ and let $\kappa \in C(M)$ satisfying 
\begin{equation}\label{ipo_k_capillarity_2}
\kappa(x) \ge C\big( 1+ r(x)\big)^{-\mu} \qquad \text{on } \, M, 
\end{equation}
for some constants $C >0$ and $\mu \in \R$. Then, the only solution of \eqref{capillarity_general} on $M$ is $u \equiv 0$ whenever one of the following cases occur: 
\begin{equation}\label{ipo_volumecapillarity_2}
\begin{array}{rlll}
(i) & \disp \omega> 1, & \quad \mu < 2 & \quad \text{and} \qquad \disp \liminf_{r \ra \infty} \frac{\log\vol B_r}{r^{2-\mu}} < \infty; \\[0.5cm]
(ii) & \disp \omega> 1, & \quad \mu = 2 & \quad \text{and} \qquad \disp \liminf_{r \ra \infty} \frac{\log\vol B_r}{\log r} < \infty; \\[0.5cm]
(iii) & \disp \omega \in (0,1], & \quad \mu < \omega+1 & \quad \text{and} \qquad \disp \liminf_{r \ra \infty} \frac{\log\vol B_r}{r^{\omega+1 - \mu -\eps}} < \infty,
\end{array}
\end{equation}
for some $\eps>0$. 
\end{theorem}
\begin{remark}
\emph{Case $(i)$ is due to \cite[Thm. 4.8]{prsmemoirs}. From \eqref{ipo_volumecapillarity_2}, we readily  deduce \eqref{bound_capi_general} in the Euclidean setting.
}
\end{remark}

\begin{proof}
Clearly, $u \equiv 0$ is the only constant solution. Suppose that \eqref{capillarity_general} admits a non-constant solution, set $p=2$ and define $\chi=1$ if $\omega>1$, while $\chi = \omega -\eps$ if $\omega \in (0,1]$. Up to reducing $\eps$, we can assume that $\chi \in (0,1]$. Therefore, the boundedness of $t^\chi/\sqrt{1+t^2}$ on $\R$ guarantees the existence of a constant $C_1>0$ depending on $\chi$ such that 
$$
\diver \left( \frac{\nabla u}{\sqrt{1+|\nabla u|^2}}\right) = \kappa(x) |u|^{\omega-1}u \ge C_1 \kappa(x) |u|^{\omega-1}u \frac{|\nabla u|^{1-\chi}}{\sqrt{1+|\nabla u|^2}} 
$$
on $M$. Since $\omega > \chi$, applying Theorem \ref{teo_main} with the choices $b(x)= C_1 \kappa(x)$, $f(t) = |t|^{\omega-1}t$  we deduce $u^* \le 0$. To conclude we then apply Theorem \ref{teo_main} to $-u$ to get $u \equiv 0$, contradiction. 
\end{proof}

\subsection{Other ranges of parameters}
In our investigation of problem $(P_\ge)$, we mostly assumed \eqref{assum_main_2_fondo} in the parameter range
$$
\chi \ge 0, \qquad \mu \le \chi+1.
$$
The main reason for this choice was the possibility to obtain maximum principles at infinity for the operator $(bl)^{-1}\Delta_\varphi$. However, in the recent literature some interesting results in Euclidean space give subtle hints to grasp how Geometry comes into play for other ranges of $\chi,\mu$. To our knowledge, the problem is still completely open in a manifold setting. 
\begin{remark}[\textbf{The range $\mu > \chi+1$}]
\emph{This case is considered in \cite{farinaserrin1, farinaserrin2, pucciserrin_2}. In particular, we quote \cite[Thm. 3]{farinaserrin2} where the authors establish a Liouville theorem under the restriction
\begin{equation}\label{conditions_finali}
\omega> \max\{\chi,0\}, \qquad \frac{\mu-\chi -1}{\omega- \chi} < \frac{p-m}{p-1},
\end{equation}
see also Thm. 2 and Ex. 3 in \cite{farinaserrin1}. Note that $\mu > \chi+1$ may enjoy \eqref{conditions_finali} only if $p> m$. Further results for large $\mu$ can be found in Theorems 4, 8 and 12 in \cite{farinaserrin2}, Thm. C in \cite{farinaserrin1}, Thms. 1.3 and 5.3 in \cite{pucciserrin_2}.
}
\end{remark}
\begin{remark}[\textbf{The range $\chi < 0$}]\label{rem_chiminorzero}
\emph{This corresponds to a gradient dependence $l$ that is allowed to vanish with high order in $t=0$, and we quote \cite[Thm. 11.4]{dambrosiomitidieri_2}. There, the conclusions of Theorem \ref{teo_main} are shown to hold when \eqref{assum_main_2_fondo} holds with $l(t) \ge Ct^{p-1-\chi}$ and $\omega = 0$, provided that
\begin{equation}\label{paracomchi}
\mu < 1, \qquad -\left[\frac{1-\mu}{m-1}\right](p-1) \le \chi < 0.
\end{equation}
Note that, as shown in Remark 11.8 of \cite{dambrosiomitidieri_2}, when $\mu=0$ the value $\frac{1-\mu}{m-1}(p-1)$ in \eqref{paracomchi} is sharp. A similar bound also appears in Thms. 2 and 7 in \cite{farinaserrin2}. Related interesting results, for possibly singular $b(x)$ and still in the range $\chi <0$, are given in \cite{lili}.
}
\end{remark}

\section{Appendix: models and comparisons}
Comparison theory in Riemannian geometry allows to deduce the behaviour of relevant geometric quantities on $M$ from the knowledge of the corresponding ones on a simpler, rotationally symmetric model example, provided that the curvatures of $M$ are controlled by those of the model. For $\lambda>0$, let $\Sph^{m-1}_\lambda = (\Sph^{m-1}, \metricN_\lambda)$ denote the round sphere of radius $1/\lambda$. Given $0<g \in C^2(\R^+)$, a model $(M_{g,\lambda}, \di s^2_{g,\lambda})$ is topologically
\begin{equation}\label{model}
M_{g,\lambda} = \R^+ \times \Sph^{m-1}_\lambda,
\end{equation}
endowed with the rotationally symmetric, warped product metric which in coordinates $(r,\theta) \in \R^+\times \Sph^{m-1}_\lambda$ writes
$$
\di s_{g,\lambda}^2 = \di r^2 + g(r)^2 \metricN_\lambda.
$$
When $\lambda =1$, $g(0)=0$ and $g'(0)=1$, $\di s_{g,\lambda}^2$ extends to a $C^2$-metric on the completion of \eqref{model}, which is topologically $\R^m$. Examples include
\begin{itemize}
\item[-] $\R^m$, recovered for $g(r) = r$;
\item[-] the hyperbolic space $\HH^m_\kappa$ of curvature $-\kappa^2$, for which $g(r) = \kappa^{-1} \sinh(\kappa r)$.
\end{itemize}
However, in the study of $\csp$ we shall also use models whose function $g$ extends smoothly to a positive value at $r=0$, that is, the completion $\overline{M}_{g,\lambda}$ is a manifold with boundary. In both the cases $r$ coincides with the distance to $\{r=0\}$, and a direct computation shows
$$
K_\rad = - \frac{g''(r)}{g(r)}, \qquad \nabla \di r = \frac{g'(r)}{g(r)} \Big( \di s_{g,\lambda}^2 - \di r \otimes \di r \Big), \qquad \Delta r = (m-1)\frac{g'(r)}{g(r)}.
$$
In particular, if $\overline{M}_{g,\lambda}$ has a boundary, the second and third imply that second fundamental form $\II_{-\nabla r}$ of $\partial \overline{M}_{g, \lambda}$ and its (unnormalized) mean curvature $H_{-\nabla r}$ in the direction $-\nabla r$ satisfy
$$
\II_{-\nabla r} = \frac{g'(0)}{g(0)} \di s_{g, \lambda}^2, \qquad H_{-\nabla r} = (m-1)\frac{g'(0)}{g(0)}.
$$
Although the comparison theorems are well-known, we prefer to write down explicitly the results with (sketchy) proofs, for the convenience of the reader. We derive our main estimates in the less standard case when the model has a boundary.
%

%
%
%
%
%
%
%
%
Let $(M^m,\metric)$ be a complete manifold, and fix an origin $\mathcal{O} \subset M$. The set $\cal O$ can be either a single point, or a open subset with smooth boundary, and here we focus on the second case. Denote with $D_{\cal O}$ the maximal domain where the normal exponential map
$$
\exp^\perp \ \ :  \ \ \mathscr{D}_\mathcal{O} \subset T \mathcal{O}^\perp \longrightarrow \mathcal{D}_{\mathcal{O}} \subset M \backslash \mathcal{O}
$$
is the inverse of a chart ($T \mathcal{O}^\perp$ is the subset of the normal bundle of $\partial \mathcal{O}$ made of outwards pointing normal vectors). It is known that $\mathcal{D}_\mathcal{O}$ is open and that $r(x) = \dist(x,\cal O)$ is smooth\footnote{Observe that, since $\partial \mathcal{O}$ is a smooth hypersurface, $r$ is smooth up to $\partial \mathcal{O}$.} on $\mathcal{D}_{\cal O}$. Its complementary $M \backslash ( \mathcal{O} \cup D_{\mathcal{O}})$ is a closed set of measure zero, called the cut-locus of $\mathcal{O}$ and denoted with $\cut(\mathcal{O})$. \par
The starting point of comparison theory is the following construction: for each $x \in D_{\cal O}$, let $\gamma : [0, r(x)] \ra M$ be the unique unit speed, minimizing geodesic normal to $\cal O$, starting from $\partial \cal O$ and ending at $x$. Let $R$ denote tha curvature tensor of $M$, fix a parallel, orthonormal basis $\{\gamma', E_2, \ldots E_m\}$ along $\gamma$ and note that $\{E_\alpha(0)\}_{\alpha \ge 2}$ span $T_{\gamma(0)} \partial \cal O$. Differentiating twice the identity $|\nabla r|^2=1$, using the Ricci commutation rules and contracting with respect to $\{E_\alpha\}$, it turns out that the matrix function
$$
B : [0, r(x)] \ra \sym^2(\R^{m-1}), \qquad B_{\alpha\beta}(t) = \nabla \di r\big(E_\alpha(t), E_\beta(t)\big)
$$
solves the matrix Riccati equation
\begin{equation}\label{matrixriccati}
B' + B^2 + R_\gamma = 0, \qquad \text{where} \qquad (R_\gamma)_{\alpha \beta}(t) = R\big(\nabla r, E_\alpha(t), \nabla r, E_\beta(t)\big),
\end{equation}
with initial condition
\begin{equation}\label{initialcond_riccati}
B(0)_{\alpha\beta} = \II_{-\nabla r}\big( E_\alpha(0), E_\beta(0)\big),
\end{equation}
where $\II_{-\nabla r}$ is the second fundamental form of $\partial \cal O$ in the inward pointing direction $-\nabla r$. We recall the matrix Riccati comparison theorem, as stated in \cite{eschenburghheinze} (see also \cite[Thm. 1.14]{bmr2}).

\begin{theorem}\label{teo_matrixricca}
Let $R_1, R_2 : [0,T] \ra \sym^2(\R^{m-1})$ be continuous, and let $B_1,B_2 : (0,T] \ra \sym^2(\R^n)$ solve
$$
B_1' + B_1^2 + R_1 \le 0, \qquad B_2' + B_2^2 + R_2 \ge 0 \quad \text{on } \, (0,T],
$$
with initial condition $(B_1-B_2)'(0^+) \le 0$. If $R_1 \ge R_2$ on $[0,T]$, then
$$
B_1 \le B_2 \qquad \text{on } \, (0,T],
$$
and $\dim \ker(B_2-B_1)$ is non-increasing. In particular, if $B_1(t_0)=B_2(t_0)$ for some $t_0$, then $B_1\equiv B_2$ on $[0, t_0]$.
\end{theorem}

The Hessian comparison theorem is a direct corollary. We recall that the radial sectional curvature $K_\rad$ of $M$ is the sectional curvature restricted to $2$-planes containing $\nabla r$.
\begin{theorem}[Hessian comparison from below]\label{teo_hessiancomp}
Let $(M^ m,\metric)$ be a complete manifold, let $\cal O$, $D_{\cal O}$, $r$ be as above and suppose that
\begin{equation}\label{ipo_sect_compa}
\mathrm{K}_\rad(\pi_x) \le - G\big(r(x)\big) \qquad \forall  \, x \in D_{\cal O}, \ \pi_x \subseteq T_xM \ \text{ $2$-plane containing $\nabla r(x)$,}
\end{equation}
for some $G \in C^2(\R^+_0)$. Fix $\underline{\lambda} \in \R$ such that
\begin{equation}\label{ipo_second_compa}
\inf_{\partial \cal O} \II_{-\nabla r} \ge \underline{\lambda},
\end{equation}
consider the solution $g$ of
\begin{equation}\label{compa_sect}
\left\{ \begin{array}{l}
g'' - Gg \le 0 \qquad \text{on } \, \R^+ \\[0.2cm]
g(0)=1, \quad g'(0^+) \le \underline{\lambda},
\end{array}\right.
\end{equation}
and let $[0,R)$ be maximal interval where $g >0$. Then,
\begin{equation}\label{lowerbound_hess}
\nabla \di r(x) \ge \frac{g'(r(x))}{g(r(x))} \Big( \metric - \di r \otimes \di r \Big) \qquad \text{for } \, x \in D_{\cal O} \cap B_R(\cal O).
\end{equation}
\end{theorem}

\begin{proof}
Let $x \in D_{\cal O}$ and let $\gamma, B, R_\gamma$ be as above. Clearly, by \eqref{ipo_sect_compa} $R_\gamma \ge G(r)I_{m-1}$ and $B(0)\ge \underline{\lambda} I_{m-1}$. Since the function $\bar B = g'/g I_{m-1}$ solves
\begin{equation}
\left\{ \begin{array}{l}
\bar B' + \bar B^2 - G(t) \le 0 \qquad \text{on } \, (0,R), \\[0.2cm]
\bar B(0^+) \le \underline{\lambda} I_{m-1},
\end{array}\right.
\end{equation}
by Riccati comparison we get $B \ge \bar B$ on $[0, \min\{r(x),R\})$. In other words, $\nabla \di r \ge \frac{g'(r)}{g(r)} \metric$ on $\nabla r^\perp$. Taking into account that $\nabla \di r(\nabla r, \cdot)=0$ (differentiate $|\nabla r|^ 2=1$), the estimate \eqref{lowerbound_hess} follows at once.
\end{proof}

The Laplacian comparison from below simply follows by taking traces in \eqref{lowerbound_hess}, and the Hessian comparison from above by reversing all the inequalities in \eqref{ipo_sect_compa}, \eqref{ipo_second_compa} (that is, assume $\sup_{\partial \cal O} \II_{-\nabla r} \le \underline{\lambda}$), \eqref{compa_sect} and \eqref{lowerbound_hess}. As a matter of fact, for the Hessian comparison from above, one can also prove that $D_{\cal O} \subset B_R(\cal O)$ and that \eqref{lowerbound_hess} (with the reversed sign) holds on all of $M \backslash \cal O$ in the support sense (Calabi sense, see \cite{petersen}).\\
The Laplacian comparison from above, on the other hand, requires a milder curvature requirement and an initial estimate just involving the unnormalized mean curvature $H_{-\nabla r}$ of $\partial \cal O$.

\begin{theorem}[Laplacian comparison from above]\label{teo_laplaciancomp}
Let $(M^ m,\metric)$ be a complete manifold, let $\cal O$, $D_{\cal O}$, $r$ be as above and suppose that
\begin{equation}\label{ipo_ricc_compa}
\Ricc(\nabla r, \nabla r)(x) \ge - (m-1)G\big(r(x)\big) \qquad \forall  \, x \in D_{\cal O},
\end{equation}
for some $G \in C^2(\R^+_0)$. Fix $\overline{\lambda} \in \R$ such that
\begin{equation}\label{ipo_media_compa}
\sup_{\partial \cal O} H_{-\nabla r} \le (m-1)\overline{\lambda},
\end{equation}
consider the solution $g$ of
\begin{equation}\label{compa_ricci}
\left\{ \begin{array}{l}
g'' - Gg \ge 0 \qquad \text{on } \, \R^+ \\[0.2cm]
g(0)=1, \quad g'(0^+) \ge \overline{\lambda},
\end{array}\right.
\end{equation}
and let $[0,R)$ be maximal interval where $g >0$. Then, $D_{\cal O} \subset B_R(\cal O)$ and
\begin{equation}\label{upperbound_lapla}
\Delta r(x) \le (m-1)\frac{g'(r(x))}{g(r(x))}
\end{equation}
holds pointwise on $D_{\cal O}$ and weakly on $M \backslash \mathcal{O}$.
\end{theorem}

\begin{proof}
Taking traces in \eqref{matrixriccati} and applying Newton's inequality $\mathrm{Tr}(B^2) \ge \frac{(\mathrm{Tr}(B))^2}{m-1}$ one deduces that the function
$$
u(t) = \frac{\mathrm{Tr} B(t)}{m-1} = \frac{\Delta r(\gamma(t))}{m-1}
$$
solves
$$
u' + u^2 + \Ricc(\gamma',\gamma') \le 0, \qquad u(0) = \frac{1}{m-1}H_{-\nabla r}\big(\gamma(0)\big) \le \overline{\lambda}.
$$
On the other hand, $\bar u = g'/g$ satisfy
$$
\bar u' + \bar u^2 -G \bar u \ge 0, \qquad \bar u(0^+) \ge \overline{\lambda}.
$$
Riccati comparison now applied to $B_1 = uI_{m-1}$ and $B_2 = \bar u I_{m-1}$ implies $u \le \bar u$ on $[0,\min\{R, r(x)\})$, whence \eqref{upperbound_lapla} holds on $D_{\cal O} \cap B_R$. However, $\bar u \ra -\infty$ as $t \ra R^-$, so necessarily $u$ is unbounded from below as $t \ra R^-$, which imples $r(x) < \bar R$. Hence, $D_{\cal O} \subset B_R$. The weak inequality can be proved as in \cite[Lem. 2.5]{prs} (see also \cite[Thm. 1.19]{bmr2})
\end{proof}


\begin{example} \label{ex_modellishrinking}
\emph{The initial condition satisfied by $g(r)$ is crucial for the validity of the Hessian and Laplacian comparison theorems, as illustrated by the following example. Fix $\delta \ge 1$ and consider the model $M_\delta$ with metric
$$
\di s_\delta^2 = \di t^2+ g_\delta(t)^2 \metricN_1 \qquad \text{where} \quad \left\{ \begin{array}{ll}
g_\delta \in C^2(\R^+_0) & g_\delta>0 \quad \text{on } \, \R^+ \\[0.2cm]
g_\delta(t)=t & \text{if } \, t \le 1/2 \\[0.2cm]
g_\delta(t) = \exp\{-t^\delta\} & \text{if } \, t \ge 1.
\end{array}\right.
$$
Define $\mathcal{O}= \{t<1\}$ and $G_\delta = g_\delta''(t)/g_\delta(t)$. Note that $r=t-1$ is the distance from $\cal O$, and that on $M \backslash \mathcal{O}$,
$$
\begin{array}{l}
\disp \II_{-\nabla r} = -\delta \di s_\delta^2, \qquad \Delta_\delta r = (m-1) \frac{g_\delta'(1+r)}{g_\delta(1+r)} = -\delta(m-1)(1+r)^{\delta-1} \\[0.3cm]
\Ricc_\delta (\nabla_\delta r, \nabla_\delta r) = (m-1)K_{\rad}(r) \\[0.2cm]
\qquad \qquad \qquad \quad =  -(m-1)G_\delta(r) = -(m-1)\delta \big[ -(\delta-1)(1+r)^{\delta-2}+\delta(1+r)^{2\delta-2}\big].
\end{array}
$$
Observe that $\Ricc_\delta(\nabla_\delta r,\nabla_\delta r)$ is a decreasing function of $\delta$, but also $\Delta_\delta r$ is so. This is, however, not in contradiction with Theorem \ref{teo_laplaciancomp}. Indeed, to apply the latter with $M = M_\delta$ and $G = G_{\bar \delta}$, $\delta \neq \bar \delta$, condition \eqref{ipo_ricc_compa} would imply $\delta < \bar \delta$, while \eqref{ipo_media_compa} gives $-\delta \le \overline{\lambda}$. In this way, the function $g_{\bar \delta}$ does not solve \eqref{compa_ricci} because $g_{\bar \delta}'(0) = -\bar \delta < \overline{\lambda}$.
}
\end{example}

For most manifolds the trace of inequality \eqref{lowerbound_hess}, that is,
\begin{equation}\label{Lapladasotto}
\Delta r \ge (m-1) \frac{g'(r)}{g(r)}
\end{equation}
does not hold weakly on $M \backslash \mathcal{O}$ even if $g'(r)/g(r)$ is well defined on $M \backslash \mathcal{O}$, that is, if $R = \infty$, since the singular part of the distribution $\Delta r$ acts as a negative Radon measure concentrated on the cut-locus $\cut(\mathcal{O})$. More precisely, by \cite{mantemasceural} the distribution $\Delta r$ is a Radon measure that can be written as
\begin{equation}\label{carat_cuto}
\Delta r = (\Delta r)_{AC} \di V - |\nabla^+r(x) - \nabla^-r(x)| \haus^{m-1}\llcorner\cut(\mathcal{O}),
\end{equation}
where $\haus^{m-1}$ is the $(m-1)$-dimensional Haurdorff measure, and
\begin{itemize}
\item $(\Delta r)_{AC}$ coincides with the $L^1_\loc(M\backslash \mathcal{O})$ function given by $\Delta r$ outside $\cut(\mathcal{O})$;
\item $|\nabla^+r(x) - \nabla^-r(x)|$ is a function defined on the normal cut-locus, that is, the set of non-conjugate points $x \in \cut(\mathcal{O})$ where exactly $2$ minimizing geodesics meet, and $\nabla^+r(x)$ and $\nabla^-r(x)$ are the tangent vectors of the two
geodesics at $x$.
\end{itemize}

\begin{remark}
\emph{The result in \cite{mantemasceural} is stated for $\mathcal{O}$ being a point, but the proof for smooth $\mathcal{O}$ follows verbatim, see also \cite{mantegazzamennucci} and Section 3.9 of \cite{ambrosiofuscopallara}.
}
\end{remark}

As a consequence of work of various authors (see the account in Section 1.1 of \cite{bmr2}), the complementary of the normal cut-locus has Hausdorff dimension at most $(m-2)$, and the normal cut-locus is dense in the set of non-conjugate (i.e., non-focal) cut-points. Therefore, $\Delta r$ is an absolutely continuous measure if and only if $\cut(\mathcal{O})$ consists only of conjugate points, and in this case \eqref{Lapladasotto} holds weakly on the whole of $M\backslash\mathcal{O}$. In the Introduction, and in particular in Theorem \ref{teo_CSP}, we claimed that \eqref{M1} implies that $o$ be a pole of $M$. We prove this statement in the next

%

\begin{proposition}\label{prop_pole}
In the above notation, suppose that the negative part $(\Delta r)_-$ of the measure $\Delta r$ satisfies
\begin{equation}\label{weak_laplasotto}
(\Delta r)_-  \in L^\infty_\loc(M\backslash \mathcal{O}).
\end{equation}
Then, $\mathcal{O}$ is a pole of~$M$.
\end{proposition}

\begin{proof}
Inequality \eqref{upperbound_lapla} coming from the Laplacian comparison from above implies that the positive part $(\Delta r)_+ \in L^\infty_\loc(M \backslash \mathcal{O})$. Because of \eqref{weak_laplasotto}, $\Delta r$ is absolutely continuous and represented by a locally bounded function and thus, by \eqref{carat_cuto} and the discussion above, $\cut(\mathcal{O})$ has just conjugate points. On the other hand, if $x_0 \in \cut(\mathcal{O})$ is conjugate to $\mathcal{O}$ and denoting with $g(r,\theta)$ the determinant of $\metric$ in normal coordinates $(r, \theta) \in \mathscr{D}_\mathcal{O} \subset \R^+_0 \times  \partial \mathcal{O}$ for $\mathcal{D}_\mathcal{O}$, by the identity
$$
\Delta r = \frac{1}{2} \partial_r \log g(r, \theta)
$$
we see that $\Delta r(y) \ra -\infty$ as $y \in \mathcal{D}_\mathcal{O}$, $y \ra x$. Hence, $\Delta r$ is not bounded in a neighbourhood of $x$, a contradiction which shows that $\cut(\mathcal{O})$ is in fact empty, equivalently, that $\mathcal{O}$ is a pole.
\end{proof}

The volume comparison theorems can be deduced by integration. For $R>0$, we define
$$
\partial B_R(\mathcal{O}) = \Big\{ x \in M : r(x) =R \Big\}, \qquad  B_R(\mathcal{O}) = \Big\{ x \in M : r(x) \in (0,R)\Big\}.
$$
Note that $\mathcal{O} \not \subseteq B_R(\mathcal{O})$, and thus $\vol(B_R(\mathcal{O})) \ra 0$ as $R \ra 0$. Given $0<g \in C^2(\R^+_0)$ we set
$$
v_{g,\lambda}(t) = \vol(\Sph^{m-1}_\lambda) g(t)^{m-1}, \qquad \text{and} \qquad V_{g,\lambda}(t) = \int_0^t v_{g,\lambda}(s)\di s,
$$
which are, respectively, the volume of a geodesic sphere $\{r=t\}$ and ball $\{r < t\}$ in $M_{g,\lambda}$. The proof of the next result follows verbatim the version in \cite[Thm. 2.14]{prs} (see also \cite[Thm. 1.24]{bmr2}).

\begin{theorem}[Volume comparison]\label{teo_compavol_appe}
Let $(M^m, \metric)$ be a complete manifold, $\mathcal{O}, D_{\mathcal{O}}, r$ be as above.
\begin{itemize}
\item[(1)] In the assumptions of Theorem \ref{teo_hessiancomp}, the functions
$$
\frac{\vol(\partial B_r(\cal O))}{v_{g,\lambda}(r)}, \qquad \frac{\vol(B_r(\mathcal{O}))-\vol(B_{r_0}(\mathcal{O}))}{V_{g,\lambda}(r)-V_{g,\lambda}(r_0)}
$$
are non-decreasing in $r$ provided that $r_0 \le r < R$ and $B_r(\cal O) \subset D_{\cal O}$. In particular, there exists $C>0$ such that for all such $r$
$$
\vol\big( \partial B_r(\mathcal{O})\big) \ge C v_{g,\lambda}(r), \qquad \vol\big( B_r(\mathcal{O})\big) - \vol\big( B_{r_0}(\mathcal{O})\big) \ge C \big(V_{g,\lambda}(r)-V_{g,\lambda}(r_0)\big).
$$
\item[(2)] In the assumptions of Theorem \ref{teo_laplaciancomp}, the functions
$$
\frac{\vol(\partial B_r(\cal O))}{v_{g,\lambda}(r)}, \qquad \frac{\vol(B_r(\mathcal{O}))-\vol(B_{r_0}(\mathcal{O}))}{V_{g,\lambda}(r)-V_{g,\lambda}(r_0)}
$$
are non-increasing in $r$ for each $r \ge r_0$ (a.e. $r$ for the first one). In particular, there exists $C>0$ such that for all such $r$
$$
\vol\big( \partial B_r(\mathcal{O})\big) \le C v_{g,\lambda}(r), \qquad \vol\big( B_r(\mathcal{O})\big)-\vol\big( B_{r_0}(\mathcal{O})\big) \le C \big(V_{g,\lambda}(r)-V_{g,\lambda}(r_0)\big).
$$
\end{itemize}
\end{theorem}

The comparison theorem from above, in $(2)$, is due to Bishop-Gromov, see \cite[Sec.2]{prs} for references. In the particular case 
\begin{equation}\label{ricciassu_appe}
\Ricc (\nabla r, \nabla r) \ge -(m-1)\kappa^2\big( 1+r^2\big)^{\alpha/2},
\end{equation}
for some $\kappa \ge 0$ and $\alpha \ge -2$, and when $\mathcal{O}$ is a pole, detailed computations of the asymptotic behaviour of a suitable solution $g$ of 
$$
\left\{ \begin{array}{l}
g'' - \kappa^2(1+t^2)^{\alpha/2} \ge 0 \qquad \text{on } \, \R^+, \\[0.2cm]
g(0)=0, \qquad g'(0) \ge 1 
\end{array}\right.
$$
can be found in \cite[Prop. 2.1]{prs}: more precisely, the above inequality admits a solution $g$ with 
\begin{equation}\label{asintog_appe}
g(r) \asymp \left\{ \begin{array}{ll}
\exp \left\{ \frac{2\kappa}{2+\alpha} (1+r)^{ 1+ \frac{\alpha}{2}}\right\} & \qquad \text{if } \, \alpha \ge 0 \\[0.4cm]
r^{-\frac{\alpha}{4}} \exp \left\{ \frac{2\kappa}{2+\alpha} r^{ 1+ \frac{\alpha}{2}}\right\} & \qquad \text{if } \, \alpha \in (-2,0) \\[0.4cm]
r^{\bar\kappa }, \quad \bar\kappa  = \frac{1+ \sqrt{1+4\kappa^2}}{2} & \qquad \text{if } \, \alpha = -2
\end{array}
\right.
\end{equation}

as $r \ra \infty$. In particular, setting $v_g(r)$ as above, from
\begin{equation}\label{crescitevol_appe}
\log \int_{r_0}^r v_g \sim \left\{ \begin{array}{ll} \frac{2\kappa(m-1)}{2+\alpha} r^{1+ \frac \alpha 2}& \quad \text{if } \, \alpha > -2; \\[0.2cm]
\frac{2\kappa(m-1)}{2+\alpha} r^{1+ \frac \alpha 2} - \frac{\alpha(m-1)}{4} \log r & \quad \text{if } \,\alpha \in (-2,0); \\[0.2cm]
\big[(m-1)\bar\kappa +1\big] \log r & \quad \text{if } \, \alpha = -2
\end{array}\right.
\end{equation}

and Theorem \ref{teo_compavol_appe} we deduce
\begin{equation}\label{riccievol_appe}
\begin{array}{ll}
\disp \limsup_{r \ra \infty} \frac{\log\vol B_r}{r^{1+\alpha/2}} < \infty & \quad \text{if } \, \alpha > -2, \\[0.4cm]
\disp \limsup_{r \ra \infty} \frac{\log\vol B_r}{\log r} \le (m-1)\bar\kappa +1 & \quad \text{if } \, \alpha = -2.
\end{array}
\end{equation}

We conclude by extending the examples in \eqref{asintog_appe} to a larger class of solutions of \eqref{compa_sect} and \eqref{compa_ricci}, that enables to include more general initial conditions. When $\mathcal{O}$ reduced to a point, further examples can be found in the appendix of \cite{bmr3}. The proof of the next lemma is by a direct computation.
\begin{lemma}\label{lem_ODE_infty_increasing}
Let $G \in C^1(\R^+)\cap C(\R^+_0)$ be non-negative, set
$$
\theta_* = \inf_{\R^+} \frac{G'}{2G^{3/2}}, \qquad \theta^* = \sup_{\R^+} \frac{G'}{2G^{3/2}}, \qquad D_\pm (t) = \frac{1}{2}\left(-t \pm \sqrt{t^2+4}\right).
$$
For constants $C>0$, $D \in \R$ consider the function
\begin{equation}\label{def_g_increasing}
g(t) = 1 + C \left\{ \exp\left( D \int_0^t \sqrt{G(s)} \di s\right) -1 \right\}.
\end{equation}
Then, for a fixed $\lambda \in \R$,
\begin{itemize}
\item[(1)] $g$ solves
$$
\left\{ \begin{array}{l}
g'' -Gg \ge 0 \qquad \text{on } \, \R^+ \\[0.1cm]
g(0)=1, \qquad g'(0) \ge \lambda
\end{array}\right.
$$
provided that
$$
C \ge 1, \qquad C D \sqrt{G(0)} \ge \lambda, \qquad D \in \big( -\infty, D_-(\theta^*)\big] \ \cup \ \big[ D_+(\theta_*), +\infty \big);
$$
\item[(2)] $g$ solves
$$
\left\{ \begin{array}{l}
g'' -Gg \le 0 \qquad \text{on } \, \R^+ \\[0.1cm]
g(0)=1, \qquad g'(0) \le \lambda
\end{array}\right.
$$
provided that
$$
C \in (0,1], \qquad C D \sqrt{G(0)} \le \lambda, \qquad D \in \big[ D_-(\theta_*) , D_+(\theta^*) \big].
$$
\end{itemize}
In both $(1)$ and $(2)$, if $\theta^*$ or $\theta_*$ are infinite then $D_\pm(\theta_*), D_\pm(\theta^*)$ are intended in the limit sense and shall be excluded from the range of $D$.
\end{lemma}

\section*{Acknowledgments}
The second author is supported by the grants SNS17\_B\_MARI and SNS\_RB\_MARI of the Scuola Normale Superiore. The third  author is partly supported by the Italian MIUR project titled
{\em Variational and perturbative aspects of nonlinear differential problems}
(201274FYK7) and is a member of the {\em Gruppo Nazionale per
l'Analisi Matematica, la Probabilit\`a e le loro Applicazioni} (GNAMPA)
of the {\em Istituto Nazionale di Alta Matematica} (INdAM). The
manuscript was realized within the auspices of the INdAM--GNAMPA Project
2015 titled {\em Modelli ed equazioni nonlocali di tipo frazionario} (Prot\_2015\_000368).

\bibliographystyle{amsplain}

\end{document}